\documentclass[11pt,a4,twoside]{book}
\usepackage[latin1]{inputenc}

\newtheorem{defeng}{Definition}[chapter] 

\usepackage{pstricks, pst-coil, pst-node, pst-tree, multido}
\usepackage{amssymb,amsmath,stmaryrd,amscd,bbm,marvosym}
\usepackage{epsf}

\usepackage{latexsym} 
\usepackage{verbatim} 
\usepackage{ifthen}
\usepackage{theorem} 
\usepackage{psboxit} 
	\PScommands 
        \usepackage{graphicx} 
\usepackage{graphics} 
\usepackage{pifont}
\usepackage{pstricks,pst-node,pst-text,pst-3d}

\usepackage{color}
\usepackage{clrscode} 

\usepackage{url} 

\usepackage{enumerate} 

\newboolean{perso} 
\setboolean{perso}{false} 

\newboolean{avecinfotech} 
\setboolean{avecinfotech}{false} 

\newcommand{\policecom}{} 

{\ifthenelse{\boolean{perso}}{\policecom}{\comment}}%
{\ifthenelse{\boolean{perso}}{\normalfont}{\endcomment}}

\newcommand{\pexo}[1]{\noindent\addtocounter{exercice}{1}%
\setcounter{q}{0}%
{\sc Exercice~\arabic{exercice}. }%
{\ifthenelse{\boolean{avecinfotech}}{\hfill#1}{}}
\input{#1}}

\newcommand{\exo}[1]{\mbox{}\\\mbox{}\noindent\addtocounter{exercice}{1}%
\setcounter{q}{0}%
{\sc Exercice~\arabic{exercice}. }%
{\ifthenelse{\boolean{avecinfotech}}{\hfill#1}{}}
\input{#1}}

\newcounter{exercice}
\newcounter{q}
\newcounter{sq}

\newcommand{\espR}{\vspace{3.5ex}}
\newcommand{\espSR}{\vspace{1ex}}
\newcommand{\espSSR}{\vspace{.3ex}}

\DeclareMathOperator{\gap}{gap}

\newcommand{\sm}{\setminus} 

\newcommand{\qed}{\relax\ifmmode\hskip2em\Box\else\unskip\nobreak\hfill$\Box$\fi}

\newcommand{\R}{\ensuremath{\mathbb{R}}}

\newcommand{\N}{\ensuremath{\mathbb{N}}}

\newcommand{\suchthat}{\ |\ }

\newcommand{\bp}{}
\newcommand{\tp}{\!-\!}
\newcommand{\ep}{}

\newcounter{claim}

\newtheorem{theoreme}[defeng]{Théorème}
\newtheorem{theorem}[defeng]{Theorem}

\newtheorem{lemme}[defeng]{Lemme}
\newtheorem{lemma}[defeng]{Lemma}

\newtheorem{conjecture}[defeng]{Conjecture}

\newtheorem{exercise}[defeng]{Exercise}
{\theorembodyfont{\rmfamily} }
{\theorembodyfont{\rmfamily} \newtheorem{exercice}[defeng]{Exercice\setcounter{q}{0}}}
{\theorembodyfont{\rmfamily} }
{\theorembodyfont{\rmfamily} \newtheorem{question}[defeng]{Question}}
{\theoremstyle{break}\theorembodyfont{\rmfamily} }
{\theoremstyle{break}\theorembodyfont{\rmfamily} }

\newenvironment{preuve}{{\setcounter{claim}{0}\noindent\sc preuve
    ---}}{\hfill$\Box$\vspace{2ex}} 

\newenvironment{claim}[1][]%
{\refstepcounter{claim}\vspace{1ex}\noindent{(\it\arabic{claim}){#1}{}}\it}{\vspace{1ex}}

\newenvironment{proofclaim}[1][]%
	{\noindent {}{#1}{}}{ This proves~(\arabic{claim}).\vspace{1ex}}

 \newenvironment{proof}[1][]%
 {\noindent {\setcounter{claim}{0}\sc proof ---
    }{#1}{}}{\hfill$\Box$\vspace{2ex}}

\newcommand{\avantiteml}{}
\newcommand{\apresiteml}{\hspace{\labelsep}}
\newcommand{\iteml}{%
\ifthenelse{\equal{\value{enumii}}{0}}{\avantiteml\addtocounter{enumi}{1}\labelenumi\apresiteml}{%
\ifthenelse{\equal{\value{enumiii}}{0}}{\avantiteml\addtocounter{enumii}{1}\labelenumii\apresiteml}{%
\ifthenelse{\equal{\value{enumiv}}{0}}{\avantiteml\addtocounter{enumiii}{1}\labelenumiii\apresiteml}{%
\avantiteml\addtocounter{enumii}{1}\labelenumii\apresiteml%
}%
}%
}%
}
\newcommand{\piteml}{%
\ifthenelse{\equal{\value{enumii}}{0}}{\avantiteml\labelenumi\apresiteml}{%
\ifthenelse{\equal{\value{enumiii}}{0}}{\avantiteml\labelenumii\apresiteml}{%
\ifthenelse{\equal{\value{enumiv}}{0}}{\avantiteml\labelenumiii\apresiteml}{%
\avantiteml\labelenumii\apresiteml%
}%
}%
}%
}

\newcounter{myenumerate}

\newcommand{\insererfigure}[3]%
{
  \begin{figure}[htb]
    \center
    \includegraphics{#1}
    \caption{#2}
    \ifthenelse{\boolean{avecinfotech}}{Fichier: #1}{}
    \label{#3}
  \end{figure}
}

\newcommand{\insererfigurepage}[3]%
{
  \begin{figure}[p]
    \center
    \includegraphics{#1}
    \caption{#2}
    \ifthenelse{\boolean{avecinfotech}}{Fichier: #1}{}
    \label{#3}
  \end{figure}
}

\newcommand{\insererfiguret}[3]%
{
  \center
  \resizebox{#3}{!}{\includegraphics{#1}}
  \\{#2}
  \ifthenelse{\boolean{avecinfotech}}{\\Fichier: #1}{}
}

\usepackage{ifpdf}

\ifpdf
\fi

\usepackage{fancyhdr} 
\fancypagestyle{plain}{ 
\fancyhead{} 
}

 \usepackage{ifpdf}
 \usepackage{makeidx}
 \ifpdf
 \fi

\makeindex

\begin{document}

\thispagestyle{empty}
{{
\huge
{

\noindent \rule{1.5cm}{0cm}Habilitation à diriger des recherches\\}

\vspace{2ex}

\Large
\noindent\rule{1.5cm}{0cm}Mémoire présenté par {\sc Nicolas Trotignon}\\
\noindent\rule{1.5cm}{0cm}CNRS, LIAFA, Université Paris 7, Paris Diderot
}\\

\vspace{2ex}
{
\vspace{10ex}\mbox{}\\
\huge
\noindent\rule{1.5cm}{0cm}Structure~des~classes~de~graphes~définies\\
\noindent\rule{1.5cm}{0cm}par l'exclusion de sous-graphes induits\\
\huge
\vspace{5ex}\mbox{}\\
}

{
\Large

\vfill

\enlargethispage*{50cm}

\vspace{6cm}

\begin{minipage}{10cm}
\noindent\rule{1.5cm}{0cm}Soutenue 15  décembre 2009\\ 

\noindent\rule{1.5cm}{0cm}Jury~: \\

\noindent\rule{1.5cm}{0cm}{\sc Maria Chudnovsky} (rapporteur)\\
\noindent\rule{1.5cm}{0cm}{\sc Michele Conforti} (rapporteur)\\
\noindent\rule{1.5cm}{0cm}{\sc Jean-Paul Delahaye} (examinateur)\\
\noindent\rule{1.5cm}{0cm}{\sc Michel Habib} (rapporteur interne)\\
\noindent\rule{1.5cm}{0cm}{\sc Frédéric Maffray} (examinateur)\\
\noindent\rule{1.5cm}{0cm}{\sc Stéphan Thomassé} (rapporteur) \\
\end{minipage}}

\newpage
\thispagestyle{empty}
\mbox{}\newpage
\thispagestyle{empty}
\mbox{}

\vspace{13ex}

{

\leftmargin0cm
{\noindent\bf\hfill\Large Résumé --- Abstract\hfill}

\vspace{8ex}

  \noindent Ce document présente les recherches de l'auteur durant ces dix
  dernieres années, sur les classes de graphes définies en excluant
  des sous-graphes induits.
  
\vspace{3ex}

  \noindent This document presents the work of the author over the last ten
  years on classes of graphs defined by forbidding induced subgraphs.

}\newpage
\thispagestyle{empty}
\chapter*{Remerciements}
\thispagestyle{empty}
Je remercie Maria Chudnovsky, Michele Conforti et Stéphan Thomassé
d'avoir accepté d'être les rapporteurs externes de ce travail.  Merci
également aux autres membres du jury : Jean-Paul Delahaye, Michel
Habib et  Frédéric Maffray. 

Merci à Christelle Petit, Jean-Sébastien Sereni et Rachel Wieviorka
qui m'ont aidé à améliorer ce document.  Merci surtout à Juraj Stacho
et Béatrice Trotignon qui en ont relu de nombreuses pages.

Merci à mon épouse Christelle et à mes enfants Émile, Alice et Coline
pour leur soutien et leur affection.

\setcounter{page}{0}
\tableofcontents
\nocite{lovasz:exercices2}
\nocite{DBLP:conf/bcc/2005}
{\chapter*{Introduction (en français)}
\addcontentsline{toc}{chapter}{Introduction (en français)}
\renewcommand{\leftmark}{Introduction (en français)}
\renewcommand{\rightmark}{Introduction (en français)}

Ce document présente mon travail de ces dix dernières années en
théorie des graphes.  Le premier chapitre expose une étude sur le
``gap'' d'un graphe, c'est-à-dire l'écart entre la taille d'une plus
grande de ses cliques et son nombre chromatique. C'est l'occasion
d'introduire des notions importantes pour le reste du document, en
particulier la notion de \emph{fonction majorante} (\emph{bounding
  function}) due initialement à Andr\'as Gy\'arf\'as.  Ce chapitre
utilise des outils de plusieurs branches de la théorie des graphes,
comme la théorie de Ramsey et celle des couplages.  J'espère que ce
chapitre divertira tout ceux qui apprécient la théorie des graphes,
peut-être même ceux qui ne l'apprécient pas.

Le reste du document se concentre sur ce que je fais ordinairement, à
savoir étudier des classes de graphes définies par l'exclusion de
certains sous-graphes induits, donner des théorèmes de décomposition
et des algorithmes pour ces classes.  Pourquoi exclure des
sous-graphes induits~?  Essayons de donner une réponse meilleure que
``parce qu'il y a 10 ans, mon directeur de thèse, Frédéric Maffray,
m'a dit de le faire'' ou ``parce que il y a beaucoup des sous-graphes
induits possibles, ce qui conduira à publier beaucoup d'articles''.
Notons tout d'abord que tout classe fermée pour la relation
``sous-graphe induit'' est nécessairement définie de manière
équivalente par une liste de sous-graphes induits exclus.  Et cette
relation ``sous-graphe induit'' est mathématiquement ``naturelle'' en
ce qu'elle correspond à la notion classique de ``sous-structure''
présente partout en algèbre et dans toutes les branches des
mathématiques.  De plus, dans plusieurs modèles de Recherche
Opérationnelle, des problèmes pratiques sont modélisés par des
graphes.  Les objets considérés sont les sommets du graphe tandis que
les contraintes entre objets sont représentées par les arêtes.
Souvent, la classe de graphes résultant d'une telle modélisation est
fermée par sous-graphes induits.  Car supprimer des objets dans le
monde réel correspond à supprimer des sommets du graphes, tandis que
supprimer des contraintes entre objets ne corres\-pond à aucune
opération du monde réel ; de telle sorte que toutes les arêtes qui
existent entre les sommets non-suprimés du graphe doivent demeurer.

Dans les années 1960, le travail pionnier de Gabriel Dirac sur les
graphes chordaux, de Tibor Gallai sur les graphes de comparabilité et
les deux conjectures des graphes parfaits de Claude Berge ont inauguré
le domaine.  Au long des quarante années qui ont suivi, beaucoup de
recherches ont été consacrées aux graphes parfaits et à d'autres
classes des graphes, avec de nombreux succès, comme la preuve de la
conjecture faible des graphes parfaits par Laszlo Lov\'asz, les
résultats de Va\v sek Chv\'atal et Delbert Fulkerson sur le lien entre
les graphes parfaits et la programmation linéaire, et la preuve de la
conjecture forte des graphes parfaits par Maria Chudnovsky, Neil
Robertson, Paul Seymour et Robin Thomas.

À partir des années 1980, Neil Robertson et Paul Seymour ont développé
le ``Graph Minor Project''.  Il s'agit d'une théorie très profonde qui
décrit toutes les classes de graphes fermées pour la relation de
mineur (et non pas de sous-graphe induit).  Il est naturel de se
demander si une telle théorie pourrait exister pour la relation de
sous-graphe induit.  Jusqu'à présent, il semblerait que la réponse
soit négative.  Les classes fermées pour la relation de sous-graphe
induit ne semblent pas assez régulières pour être décrites par une
théorie unifiée.  Un indice parmi d'autres est donné au
chapitre~\ref{chap:reco} de ce travail, où plusieurs classes de
graphes sont poynomiales ou NP-complètes à reconnaître, en fonction de
changements apparemment insignifiants dans leur définition.  Pourtant,
il pourrait y avoir des caractéristiques partagées par toutes les
classes de graphes définies en interdisant un sous-graphe
induit (voir par exemple la conjecture de Erd\H os-Hajnal), ou par
beaucoup de classes définies plus généralement (voir par exemple la
conjecture de Gy\'arf\'as sur les arbres,
Conjecture~\ref{conj:gyarfas}, ou la conjecture de Scott,
Conjecture~\ref{conj:scott}).  La question qui m'intéresse le plus est
de comprendre comment les classes de graphes fermées par sous-graphes
induits peuvent être décrites de la manière la plus générale possible.
Mais comme le lecteur le constatera, la plupart des résultats
présentés ici concerne des classes particulières.

\section*{Plan du document}
\addcontentsline{toc}{section}{Plan du document}

Rappelons que le chapitre~\ref{chap:gap} présente des résultats à
propos de l'écart entre le nombre chromatique et la taille d'une plus
grande clique d'un graphe.  Ce chapitre s'inspire
de~\cite{nicolas.gyarfas.sebo:gap}.

L'outil le plus puissant ces dernières années pour l'étude des classes
fermées par sous-graphe induit est l'approche structurelle, qui
consiste en la description des classes de graphes à travers des
\emph{théorèmes de décomposition}.  Ceci est expliqué au
chapitre~\ref{chap:6th} où six théorèmes de décomposition sont
présentés.  Ces théorèmes sont tous assez simples, mais la motivation
de ce chapitre est de présenter un échantillon typique de théorie
structurelle des graphes à destination de lecteurs ne voulant pas se
lancer dans la lecture d'articles ou d'analyses de cas trop longs. Les
preuves sont toutes assez courtes, mais elles tentent d'illustrer des
idées qui seront utilisées dans le reste du document. Ce chapitre peut
aussi être utilisé pour enseigner la théorie structurelle des graphes.
Ce chapitre ne s'inspire pas d'un article particulier, il reprend des
théorèmes divers, donne parfois des nouvelles preuves, et aussi
quelques résultats originaux.

Le chapitre~\ref{chap:reco} présente un aspect important des classes
de graphes définies par sous-graphes induits exclus.  Comment décider
avec un algorithme si tel graphe est dans telle classe~?  L'approche
la plus simple semble de voir comment on peut détecter des
sous-graphes induits dans un graphe donné.  Nous montrerons que
certains problèmes de ce type sont polynomiaux, d'autres NP-complets.
Les outils pour la NP-complétude proviennent tous d'une construction
de Bienstock.  L'outil le plus général pour montrer la polynomialité
semble être l'algorithme dit ``three-in-a-tree'' dû à Chudnovsky et
Seymour.  Nous donnerons des variantes de cet algorithme.  Ce chapitre
s'inspire de~\cite{leveque.lmt:detect}, \cite{maffray.t:reco},
\cite{nicolas.wei:kTree} et \cite{nicolas.d.p:fourTree}.

Le chapitre~\ref{chap:2dec} est consacré à deux classes de graphes~:
les graphes qui ne contiennent pas de cycle avec une seule corde, et
les graphes qui ne contiennent pas de subdivision induite de $K_4$.
La motiviation initiale pour l'étude de ces classes était leur
reconnaissance en temps polynomial, un problème issu du chapitre
précédant.  Mais leur étude nous a conduit à des résultats sur le
nombre chromatique.  Ce chapitre s'inspire
de~\cite{nicolas.kristina:one} et \cite{nicolas:isk4}.

Le chapitre~\ref{chap:Berge} est consacré aux graphes de Berge.  On
procède d'abord à un survol des résultats importants concernant leur
structure. Nous donnons un théorème de structure pour les graphes de
Berge sans partition antisymétrique paire (balanced skew partition).
Comme application de ce théorème, nous donnons un algorithme de
coloration des graphes de Berge sans partition antisymétrique paire et
sans paire homogène.  Ce chapitre s'inspire de~\cite{nicolas:bsp} et
\cite{nicolas.kristina:2-join}.

Les annexes A à F présentent des données obligatoires pour tout
mémoire d'habilitation.  L'annexe~F donne la liste de mes publications
et indique où trouver dans ce mémoire le contenu de tel ou tel
article.

La plupart des travaux présentés ci-après ont été réalisés en
collaboration.  Les contributions de chacun seront précisées au fur et
à mesure, mais je suis heureux de donner maintenant la liste de mes
co-auteurs~: Amine Abdelkader, Nicolas Dehry, Sylvain Gravier,
Andr\'as Gy\'arf\'as, Benjamin L\'ev\^eque, David Lin, Christophe
Picouleau, J\'er\^ome Renault, Andr\'as Seb\H o, Juraj Stacho et Liu
Wei.  Je voudrais remercier plus particulièrement deux co-auteurs avec
qui mes collaborations ont été très proches et enrichissantes~:
Fr\'ed\'eric Maffray et Kristina Vu\v skovi\'c.

Le lecteur francophone qui a fait l'effort de me lire jusqu'ici,
constatera s'il poursuit que ce travail a été rédigé en anglais, ce
qui le situe à la marge de la légalité.  Avec le conseil scientifique
de mon UFR, il a été convenu qu'une dizaine de pages en français
devrait suffir et constituer officiellement mon mémoire, le reste
étant une ``annexe''.  Ne pouvant pas raisonnablement faire passer ce
qui précède pour une ``dizaine'', ni même pour une ``petite'' dizaine
de pages, je me propose de donner ci-dessous la traduction de la
section de ce document dont la lecture est la plus profitable selon
moi, la section~\ref{sec:wt}.  On a d'abord besoin de quelques
rappels.

\section*{Rappels}

Un \emph{trou} dans un graphe est un cycle induit de longueur au moins
4.  Un \emph{antitrou} est un trou du graphe complémentaire.  Un
\emph{graphe de Berge} est un graphe qui ne contient ni trou impair ni
antitrou impair.  Notons qu'on utilise le mot \emph{contenir} au sens
des sous-graphes induits.  Une \emph{clique} est un graphe dont tous
les sommets sont reliés (deux à deux).  Par $\chi(G)$ on dénote le
nombre chromatique de $G$, par $\omega(G)$ le plus grand nombre de
sommets deux à deux adjacents de $G$.  Claude Berge a conjecturé au
début des années soixante que tout graphe de Berge $G$ satisfait ce
qu'il appelait la \emph{belle propriété} : $\chi(G) = \omega(G)$.
Ceci est devenu la célèbre \emph{conjecture forte des graphes
  parfaits}, prouvée par Chudnovsky, Robertson, Seymour et Thomas en
2002.  Il est facile de voir que les trous impairs et les antitrous
impairs ne satisfont pas la belle propriété, mais cette remarque
n'aide pas tellement à montrer la conjecture\dots\ Un trou ou un
antitrou est \emph{long} s'il contient au moins 5 sommets.  Un graphe
est dit \emph{parfait} si pour tout sous-graphe induit $G'$ on a
$\chi(G') = \omega(G')$.

Le théorème de décomposition le plus simple pour une classe de graphes
est sans doute le suivant.

\begin{theoreme}[folklore]
  \label{th:P3F}
  Un graphe ne contient aucun $P_3$ induit si et seulement si c'est
  une union disjointe de cliques. 
\end{theoreme}

\begin{preuve}
  Une union disjointe de cliques est manifestement sans $P_3$.
  Réciproquement, considérons une composante connexe $C$ d'un graphe
  sans $P_3$ et supposons en vue d'une contradiction que deux sommets
  $u, v$ de $C$ ne sont pas adjacents.  Un plus court chemin de $C$
  joignant $u$ à $v$ contient un $P_3$, contradiction (ou comme eût
  dit Claude Berge, \emph{d'où l'absurdité}).
\end{preuve}

\section*{Section \ref{sec:wt} traduite en français} 
\addcontentsline{toc}{section}{Section \ref{sec:wt} traduite en français}

Un graphe est \emph{faiblement triangulé} s'il ne contient ni trou
long ni antitrou long.  Les graphes faiblement triangulés sont donc de
Berge, et nous allons montrer qu'ils satisfont la belle propriété, ce qui
prouve leur perfection et donne une version affaiblie du théorème fort
des graphes parfaits.  Les graphes faiblement triangulés ont été
étudiés par Chv\'atal et Hayward dans les années 1980 et le but de
cette section est de convaincre le lecteur qu'ils constituent l'une
des classes de graphes de Berge les plus intéressantes.  Car avec eux,
on peut comprendre plusieurs concepts en lisant seulement six pages~:
les lemmes du type Roussel-et-Rubio, comment traiter les antitrous,
comment des décompositions désagréables comme les partitions
antisymétriques (skew partitions) entrent en jeux, en quoi elle sont
désagréables, comment on peut parfois s'en débarrasser grâce à des
sommets spéciaux (les paires d'amis, en anglais ``even pairs'') et obtenir
des algorithmes de coloration simples et efficaces.

Le lemme de Roussel et Rubio~\cite{roussel.rubio:01} est un outil
technique important pour la preuve du théorème fort des graphes
parfaits.  L'équipe qui a prouvé ce théorème a redécouvert le lemme
indépendamment de ses auteurs et l'a baptisé le ``wonderful lemma'' en
raison de ses nombreuses applications.  Il dit qu'en un sens, tout
ensemble anticonnexe de sommets d'un graphe de Berge se comporte comme
un sommet (\emph{anticonnexe} signifie connexe dans le
complémentaire).  Comment un sommet se ``comporte''-t-il dans un
graphe de Berge~?  Si un chemin de longueur impaire (au moins 3) a ses
deux extrémités adjacentes à $v$, alors $v$ doit avoir d'autres
voisins dans le chemin car sinon il y a un trou impair.  Un ensemble
anticonnexe $T$ de sommets se comporte de la même manière~: si un
chemin de longueur impaire (au moins~3) a ses deux extrémités
complètes à $T$, alors au moins un sommet intérieur du chemin doit
aussi être complet à $T$.  En fait, il y a deux exceptions à cet
énoncé et le lemme de Roussel et Rubio est un peu plus compliqué.  
Nous ne donnons pas ici son énoncé exact.  Pour plus d'informations,
variantes et preuves courtes notamment, voir le Chapitre~4
de~\cite{nicolas:these}.

Voici un lemme qui peut être vu comme une version du lemme de Roussel
et Rubio~\cite{roussel.rubio:01} pour les graphes faiblement triangulés.

\begin{lemme}[avec Maffray \cite{nicolas:artemis}]
  \label{l:rrwtF}
  Soit $G$ un graphe faiblement triangulé.  Soit $P = \bp x \tp \cdots
  \tp y \ep$ un chemin de $G$ de longueur au moins 3 et $T \subseteq
  V(G)$, disjoint de $V(P)$ et tel que $G[T]$ soit anticonnexe, et
  les extrémités de $P$ sont $T$-complètes.  Alors $P$ contient un sommet
  intérieur qui est $T$-complet.
\end{lemme}

\begin{preuve}
  Noter qu'aucun sommet $t \in T$ peut être non-adjacent à deux
  sommets consécutifs de $P$ car alors $V(P) \cup \{t\}$ contient un
  trou long.  Soit $z$ un sommet intérieur de $P$ adjacent à un nombre
  maximum de sommets de $T$.  Supposons, en vue d'une contradiction,
  qu'il existe un sommet $u \in T \setminus N(z)$.  Soient $x'$ et
  $y'$ les voisins de $z$ sur $P$, nommés de sorte que $x, x', z, y',
  y$ apparaissent dans cet ordre le long de $P$.  Alors, d'après la
  première phrase de cette preuve, $ux', uy' \in E(G)$.  À un
  renommage près de $x$ et $y$, on suppose que $x' \neq x$.  D'après
  le choix de $z$, puisque $ux' \in E(G)$ et $uz \notin E(G)$, il
  existe un sommet $v \in T$ tel que $vz\in E(G)$ et $vx'\notin E(G)$.
  Puisque $G[T]$ est anticonnexe, il existe un antichemin $Q$ de
  $G[T]$ de $u$ vers $v$, et $u, v$ sont choisis pour que cet
  antichemin soit minimal.  D'après la première phrase de cette
  preuve, les sommets intérieurs de $Q$ sont tous adjacents à $x'$ ou
  $z$ et d'après la minimalité de $Q$, les sommets intérieurs de $Q$
  sont tous adjacents à $x'$ et $z$.  Si $x'x \notin E(G)$ alors $V(Q)
  \cup \{ x, x', z\}$ induit un antitrou long. Donc $x'x \in E(G)$.
  Si $zy \notin E(G)$ alors $V(Q) \cup \{z, x', y\}$ induit un
  antitrou long.  Donc $zy \in E(G)$ et $y = y'$. Mais alors, $V(Q)
  \cup \{x, x', z, y\}$ induit un antitrou long, une contradiction.
\end{preuve}

Quand $T$ est un ensemble de sommets, $C(T)$ dénote l'ensemble des
sommets complets à $T$.

\begin{lemma}
  \label{l:pathWTF}
  Soit $G$ un graphe faiblement triangulé et $T$ un ensemble de
  sommets tel que $G[T]$ est anticonnexe et $C(T)$ contient au moins
  deux sommets non adjacents.  Supposons que $T$ soit maximal au sens
  de l'inclusion avec ces propriétés.  Alors tout chemin de $G\sm T$
  dont les extrémités sont dans $C(T)$ a tous ses sommets dans $C(T)$.
\end{lemma}

\begin{preuve}
  Soit $P$ un chemin de $G\sm T$ dont les extrémité sont dans $C(T)$.
  Si un sommet de $P$ n'est pas dans $C(T)$, alors $P$ contient un
  sous-chemin $P'$ de longueur au moins 2 dont les extrémités sont
  dans $C(T)$ et dont l'intérieur est disjoint de $C(T)$.  Si $P'$ est
  de longueur~2, soit $P'=a\tp t \tp b$, alors $T\cup \{t\}$ est
  ensemble qui contredit la maximalité de $T$.  Si $P'$ est de
  longueur supérieure à 2, alors il contredit le lemme~\ref{l:rrwtF}.
\end{preuve}

Le théorème suivant est dû à Hayward, mais j'en propose ici une
nouvelle preuve.  Ma preuve n'est pas vraiment plus courte que celle
de Hayward, mais elle montre comment des lemmes du type Roussel et
Rubio peuvent être utilisés.  Un \emph{ensemble d'articulation} d'un
graphe est un ensemble $S$ de sommets tel que $G\sm S$ n'est pas
connexe.  Une \emph{étoile} d'un graphe est un ensemble de sommets $S$
qui contient un sommet $c$ tel que $S \subseteq N[c]$.  Une
\emph{étoile d'articulation} est une étoile qui est un ensemble
d'articulation.

\begin{theoreme}[Hayward \cite{hayward:wt}]
  \label{th:wtF}
  Soit $G$ un graphe faiblement triangulé. Alors ou bien~:
  \begin{itemize}
  \item $G$ est une  clique;
  \item $G$ est le complémentaire d'un couplage parfait;
  \item $G$ possède une étoile d'articulation.
  \end{itemize}
\end{theoreme}

\begin{preuve}
  Si $G$ est une union disjointe de cliques, en particulier quand
  $|V(G)| \leq 2$, alors la conclusion du théorème est satisfaite.
  D'après le théoreme~\ref{th:P3F}, on peut donc supposer que $G$
  contient un $P_3$.  Donc, il existe un ensemble $T$ de sommets tel
  que $G[T]$ est anticonnexe et $C(T)$ contient au moins deux sommets
  non-adjacents, parce que le milieu d'un $P_3$ forme un tel ensemble.
  Supposons alors $T$ maximal comme dans le lemme~\ref{l:pathWTF}.
  Puisque $C(T)$ n'est pas une clique, par induction, nous avons deux
  cas à considérer~:

  \noindent{\bf Cas~1:} le graphe induit par $C(T)$ possède une étoile
  d'articulation $S$.  D'après le lemme~\ref{l:pathWTF}, $T\cup S$ est
  une étoile d'articulation de $G$.

  \noindent{\bf Case~2:} le graphe induit par $C(T)$ est le
  complémentaire d'un couplage parfait.

  Supposons d'abord que $V(G) = T \cup C(T)$.  Alors, par induction,
  ou bien $T= \{t\}$, ou bien $T$ induit le complementaire d'un
  couplage parfait, ou bien $T$ possède une étoile d'articulation $S$.
  Mais dans le premier cas, $\{t\} \cup C(T) \sm \{a, b\}$ où $a, b$
  sont non-adjacents dans $C(T)$, est une étoile d'articulation de $G$.
  Dans le second cas, $G$ lui-même est le complémentaire d'un couplage
  parfait.  Dans le troisième cas, $S\cup C(T)$ est une étoile
  d'articulation de $G$.

  Donc, on peut supposer qu'il existe $x \in V(G) \sm (T \cup C(T))$.
  On choisit $x$ avec un voisin $y$ dans $C(T)$, ce qui est possible
  car sinon $T$ avec n'importe quel sommet de $C(T)$ forme une étoile
  d'articulation de $G$.

  Rappelons que $C(T)$ est le complémentaire d'un couplage parfait.
  Soit donc $y'$ le non-voisin de $y$ dans $C(T)$.  Noter que
  $xy'\notin E(G)$ car sinon $T\cup \{x\}$ contredirait la maximalité
  de $T$.  Nous affirmons que $S = T \cup C(T) \sm \{y'\}$ est une
  étoile d'articulation de $G$ séparant $x$ de $y'$.  Tout d'abord,
  c'est une étoile centrée en $y$.  Et c'est un ensemble
  d'articulation car s'il y a un chemin dans $G\sm S$ de $x$ vers
  $y'$, en ajoutant $y$ à ce chemin, on voit que $G\sm T$ contient un
  chemin de $y$ vers $y'$ qui n'est pas inclus dans $C(T)$, une
  contradiction au lemme~\ref{l:pathWTF}.
\end{preuve}

D'après les théorèmes ci-dessus et ci-dessous, les graphes faiblement
triangulés sont parfaits. Pour s'en rendre compte, considérons un
graphe faiblement triangulé, non parfait, et minimal avec ces
propriétés au sens de l'inclusion des sommets.  Donc, c'est un graphe
\emph{minimalement imparfait} (c'est-à-dire un graphe non parfait dont
tous les sous-graphes induits sont parfaits).  Puisque les cliques et
les complémentaires de couplages parfaits sont parfaits, il doit avoir
une étoile d'articulation d'après le théorème~\ref{th:wtF}.  Donc il
contredit le théorème ci-dessous (qui est admis).

\begin{theoreme}[Chv\'atal \cite{chvatal:starcutset}]
  Un graphe minimalement imparfait n'a pas d'étoile d'articulation. 
\end{theoreme}

Le théorème~\ref{th:wtF} a un vice caché~: il utilise l'étoile
d'articulation, l'exemple le plus simple de ce que Kristina Vu\v
skovi\'c appelle les \emph{décompositions fortes}~: des types de
décomposition qui disent très peu sur la structure du graphe.  Pour
s'en rendre compte, notons qu'une étoile d'articulation peut être très
grosse.  Par exemple, ce peut être tout l'ensemble des sommets, sauf
deux.  Et puisque dans l'étoile elle-même, il y a peu de contraintes
sur les arêtes, savoir qu'un graphe a une étoile d'articulation ne
dit pas grand-chose sur sa structure.  Un autre exemple de
décomposition forte que nous rencontrerons est la partition
antisymétrique, voir section~\ref{sec:decompBerge}.

Les décompositions fortes ne donnent pas de théorème de structure.
Pour s'en rendre compte, essayons de voir comment construire un graphe
en re\-collant deux graphes plus petits à l'aide d'une étoile
d'articulation.  Cela ne sera sans doute pas satisfaisant.  Car
trouver la même étoile dans deux graphes distincts est
algorithmiquement assez difficile.  Cela suppose de savoir déterminer
si deux étoiles sont isomorphes, problème aussi difficile que le
célèbre problème de l'isomorphisme.  Évidemment, il n'y a pas de
définition formelle des théorèmes de structure, il peut donc y avoir des
discussions sans fin à ce sujet.  Mais on sent bien que recoller des
graphes avec des étoiles d'articulations est moins automatique qu'avec
des paires d'articulation comme on le fait section~\ref{sec:chordless}.

Autre problème, les décompositions fortes sont difficiles à utiliser
dans des algorithmes en temps polynomial.  Car lorsqu'on construit des
blocs de décompositions, dans le cas malchanceux où l'ensemble
d'articulation est presque aussi gros que le graphe, on doit mettre
presque tout le graphe de départ dans chacun des deux blocs.  Donc,
les algorithmes récursifs qui utilisent des décompositions fortes sont
typiquement en temps exponentiel.  Un méthode délicate, inventée par
Conforti et Rao~\cite{ConfortiR:92,ConfortiR:93}, appelée
\emph{netto\-yage}, ou \emph{cleaning}, permet de donner des algorithmes
de reconnaissance rapides pour des classes de graphes dont les
décompositions sont fortes.  Mais pour des problèmes d'optimisation
combinatoire, il semble que nul ne sache comment utiliser les
décompositions fortes.  Cependant, quand une classe de graphes est
suffisamment complexe pour que les décompositions fortes semblent
inévitables, il y a encore un espoir. En effet, dans certains
théorèmes, l'existence de décompositions est remplacée par l'existence
d'un sommet, ou d'une paire de sommets, avec des propriétés spéciales.
L'exemple le plus ancien est le théorème suivant, à comparer au
théorème~\ref{th:chordal}.  Un sommet est \emph{simplicial} si son
voisinage est une clique.

\begin{theoreme}[Dirac, \cite{dirac:chordal}]
  \label{th:simplicialF}
  Tout graphe chordal possède un sommet simplicial. 
\end{theoreme}

Pour les graphes parfaits, la bonne notion de ``sommets spéciaux''
semble être la paire d'amis.  Une \emph{paire d'amis} est une paire de
sommets telle que tous les chemins les reliant soient de longueur paire.
D'après le théorème suivant, les paires d'amis sont un bon outil pour
prouver la perfection d'une classe de graphes.

\begin{theoreme}[Meyniel \cite{meyniel:87}]
  Un graphe minimalement imparfait ne possède pas de paire d'amis. 
\end{theoreme}

\emph{Contracter} une paire d'amis $a, b$ signifie remplacer $a, b$
par un sommet complet à $N(a) \cup N(b)$.  D'après le théorème
suivant, les paires d'amis sont aussi un bon outil pour la coloration
des graphes.

\begin{theoreme}[Fonlupt et Uhry \cite{fonlupt.uhry:82}]
  Contracter une paire d'amis d'un graphe conserve son nombre
  chromatique et la taille d'une plus grande clique.
\end{theoreme}

Le théorème suivant a d'abord été prouvé par Hayward, Ho\`ang et
Maffray mais la preuve donnée ici a été obtenue en collaboration avec
Maffray, voir~\cite{nicolas:artemis}.  Une \emph{2-paire} $a, b$ est
une paire d'amis particulière: tous les chemins de $a$ vers $b$ sont
de longueur~2.

\begin{theoreme}[Hayward, Ho\`ang et Maffray \cite{hayward.hoang.m:90}]
  \label{th:2pairF}
  Un graphe fai\-ble\-ment triangulé possède une 2-paire ou est une
  clique. 
\end{theoreme}

\begin{preuve}
  Si $G$ est une union disjointe de cliques (en particulier quand
  $|V(G)| \leq 2$) alors le théorème est trivialement satisfait.
  Donc, on peut supposer que $G$ contient un $P_3$.  Donc, il existe
  un ensemble $T$ comme dans le lemme~\ref{l:pathWTF} (commencer avec
  le milieu d'un $P_3$ pour construire $T$).  Puisque $C(T)$ n'est pas
  une clique, par induction, on sait que $C(T)$ possède une 2-paire de
  $G[C(T)]$.  D'après le lemme~\ref{l:pathWTF}, c'est une 2-paire de
  $G$.
\end{preuve}

La technique ci-dessus pour trouver une paire d'amis peut être
retracée jusqu'à l'article fondateur de Henri
Meyniel~\cite{meyniel:87}, voir l'exercice~\ref{ex:meynielF}
ci-dessous.  En utilisant des idées de Cl\'audia Linhares Sales et
Maffray~\cite{linhares.maffray:evenpairsansc4}, cette technique peut
être étendue aux graphes d'Artémis~\cite{nicolas:artemis}, qui sont
une généralisation de plusieurs classes de graphes parfaits connues
pour posséder des paires d'amis (graphes faiblement triangulés,
graphes de Meyniel, graphes parfaitement ordonnables etc; nous ne
définissons pas toutes les classes, un lecteur qui veut consulter le
bestiaire peut lire le Chapitre~3 de~\cite{nicolas:these}).  Les
techniques utilisées dans~\cite{nicolas:artemis} ainsi que d'autres
types de sommets spéciaux et des variantes complexes du lemme de
Roussel et Rubio sont utilisées par Chudnovsky et
Seymour~\cite{chudnovsky.seymour:even} pour raccourcir
significativement la preuve du théorème fort des graphes parfaits.

À partir du théorème~\ref{th:2pairF}, il est facile de déduire un
algorithme de co\-lo\-ration en temps polynomial pour colorier les graphes
faiblement triangulés (en contractant des 2-paires tant qu'il y en a).
Hayward, Spinrad et Sritharan \cite{hayward.S.S:fastWT} ont accéléré
cet algorithme jusqu'à $O(n^3)$.  Puisque la contraction d'une 2-paire
préserve $\chi$ et $\omega$, l'algorithme transforme tout graphe
faiblement triangulé $G$ en une clique $K$ de taille $\omega(G) =
\omega(K) = \chi(K) = \chi(G)$, prouvant ainsi la perfection de $G$
(car tout cela peut être fait pour tout sous-graphe induit de~$G$).
Donc, le théorème~\ref{th:2pairF} donne  une preuve bien plus
courte de la perfection des graphes faiblement triangulés.

D'une certaine manière, les graphes de Berge se comportent comme les
graphes faiblement triangulés.  Le lemme de Roussel et Rubio est un
outil important pour prouver leur perfection, des décompositions fortes
sont utilisées pour les décomposer (les partitions antisymétriques
paires), mais en utilisant les paires d'amis, on peut très nettement
raccourcir la preuve de leur perfection (ceci est fait par Chudnovsky
et Seymour~\cite{chudnovsky.seymour:even}).  Une grande différence est
bien sûr que pour les graphes de Berge en général, les preuves sont
beaucoup plus longues et hautement techniques.  Et jusqu'à présent,
aucun algorithme combinatoire de coloration des graphes de Berge n'est
connu.  Les paires d'amis pourraient être un ingrédient d'un tel
algorithme.  Une série de théorèmes et de conjectures militent pour
cette idée, mais une lourde machinerie de définitions doit précéder
leur simple énoncé, qui est reporté à la section~\ref{sec:SF}.

Une autre question importante à propos des graphes de  Berge est
l'existence d'un véritable théorème de structure les décrivant. 
La question suivante est apparemment plus facile mais reste ouverte à
ce jour. 

\begin{question}
  Trouver un théorème de structure pour les graphes faiblement
  triangulés. 
\end{question}

En fait, je ne serais pas surpris qu'\emph{il n'existe pas de théorème
  de structure pour les graphes de Berge}.  Il serait bon d'avoir un
outil, comme la NP-complétude, pour convaincre les collègues d'énoncés
négatifs de cette sorte, voire mieux encore de les prouver. La
non-existence d'un algorithme polynomial de reconnaissance pourrait
constituer un argument, mais cela ne fonctionne pas pour les graphes
de Berge qui peuvent être reconnus en temps $O(n ^9)$.

Une autre classe bien connue a été à l'origine de beaucoup d'idées
dans la théorie des graphes parfaits~: les graphes de Meyniel.  Un
graphe est dit de \emph{Meyniel} si tous ses cycles induits ont au
moins deux cordes.  La perfection des graphes de Meyniel a été prouvée
très tôt par Meyniel~\cite{meyniel:76}, et ils ont été la première
classe après les graphes triangulés et les graphes sans $P_4$ pour
laquelle un théorème de décomposition a été prouvé (par Burlet et
Fonlupt~\cite{burlet.fonlupt:meyniel}).  En ce qui concerne les paires
d'amis, les graphes de Meyniel ont un comportement similaire aux
graphes faiblement triangulés.  Nous laissons cela en exercice.

\begin{exercice}
  \label{ex:meynielF}
  Donner une version du lemme de Roussel et Rubio pour les graphes
  de Meyniel.  Déduire qu'un graphe de Meyniel différent d'une clique
  possède une paire d'amis.  Pour une solution, voir
  \cite{meyniel:87}.
\end{exercice}

\chapter*{Introduction (in English)}
\addcontentsline{toc}{chapter}{Introduction (in English)}
\renewcommand{\leftmark}{Introduction (in English)}
\renewcommand{\rightmark}{Introduction (in English)}
\label{chap:intro}

Title of the document in english: {\bf Structure of classes of graphs
  defined by forbidding induced subgraphs}

\vspace{2ex}

\noindent This document presents my work over the last ten years in Graph
Theory.  The first chapter presents a study about the gap between the
chromatic number of a graph and the largest size of a clique.  This
notion of gap allows to introduce some key notions for the rest of
this study, in particular the notion of \emph{bounding function} first
defined by Andr\'as Gy\'arf\'as.  This chapter uses tools from several
branches of graph theory like Ramsey Theory and Matching Theory.  I
hope it will be entertaining for all those who enjoy graph theory,
perhaps for some who usually do not.

The rest of the document is more focused on what I do usually:
studying classes of graphs defined by forbidding induced subgraphs,
giving decomposition theorems and algorithms for them.  Why forbidding
induced subgraphs?  I try to give an better answer than ``because 10
years ago, my PhD-adviser, Fr\'ed\'eric Maffray, told me to do so'' or
``because there are many possible induced subgraphs, so this will lead
to publish many papers''.  First note that any class of graphs closed
under taking induced subgraphs can be defined equivalently by
forbidding a list of induced subgraphs.  And the ``induced subgraph''
containment relation is mathematically ``natural''.  It corresponds
for graphs to the classical notion of ``substructure'' that is
everywhere in algebra and all branches of mathematics.  Also, in
several Operation Research models, problems are modeled by graphs.
The objects under consideration are represented by vertices of a graph
while the constraints between them are represented by edges.  Often,
the class of graphs arising by such a kind of model is closed under
taking induced subgraphs.  Because deleting objects in the real world
situation corresponds to deleting vertices in the graph, while
deleting constraints does not make sense in the real world, so that
all the edges between the remaining vertices must stay in the graph.

In the 1960's, pioneer works of Gabriel Dirac on chordal graphs, of
Tibor Gallai on comparability graphs and the two perfect graph
conjectures of Claude Berge really started the field.  Over the next
forty years, many researches were devoted to perfect graphs and other
classes of graphs, with much success like the proof of the weak
perfect graph conjecture by Laszlo Lov\'asz, the results of Va\v sek
Chv\'atal and Delbert Fulkerson on the link between Perfect Graphs and
Linear Programming and the proof of the Strong Perfect Graph
Conjecture by Maria Chudnovsky, Neil Robertson, Paul Seymour and Robin
Thomas.

From the eighties onwards the Graph Minor Project was developed,
mainly by Neil Robertson and Paul Seymour.  It is a very deep general
theory of all the classes of graphs closed under taking minors
(instead of induced subgraphs).  A natural question is whether there
exists such a general theory for classes closed under taking induced
subgraphs.  Up to now, the answer has seemed to be no.  Classes closed
under taking induced subgraphs seem to be a bit too messy to be
described by a unified theory.  As an evidence among others,
Chapter~\ref{chap:reco} of this work gives several classes of graphs
that are polynomial or NP-complete to recognize according to very
slight changes in the excluded induced subgraphs of their definitions.
Yet, there might be features shared by all classes of graphs defined
by forbidding at least one induced subgraph (see for instance the
Erd\H os-Hajnal's Conjecture), or by many classes (see for instance
Gy\'arf\'as' Conjecture~\ref{conj:gyarfas} on trees or Scott's
Conjecture~\ref{conj:scott}).  My main interest in research is to
understand more how classes of graphs closed under taking induced
subgraphs can be described in the most general possible way, and what
properties can be proved about them.  But as the reader will see, most
of the theorems presented here concern particular classes.

\section*{Outline of the document}
\addcontentsline{toc}{section}{Outline of the document}

We recall that Chapter~\ref{chap:gap} presents several results about
the gap between the chromatic number and the size of a largest clique
of a graph.  This chapter presents results
from~\cite{nicolas.gyarfas.sebo:gap}.

The tool that was perhaps the most successful in the last decades in
the study of classes closed under taking induced subgraph is the
structural approach, that is describing structures of classes of
graphs through \emph{decomposition theorems}.  All this will be
explained in Chapter~\ref{chap:6th} where six simple decomposition
theorems for classes defined by forbidding induced subgraphs are
presented.  These theorems are all quite simple.  But the motivation
for this chapter is to present typical structural graph theory for
readers who do not want to go into reading longer proofs, case
analysis and so on.  The proofs are all quite short, but they capture
ideas that will be explained and used more substantially in the rest
of the document. This chapter can also be used for teaching.  This
chapter does not present results from a particular paper, but several
theorems, some new proofs, and several new results.

Chapter~\ref{chap:reco} presents an important aspect of classes of
graphs defined by excluding induced subgraphs: how to decide
algorithmically whether a given graph is in a given class.  The
simplest approach seems to be to study how one can detect an induced
subgraph in some graph.  Several examples of such problems are shown
to be polynomial, some others NP-complete.  Tools for NP-completeness
all follow from a construction of Bienstock.  The most general tool
for polynomiality seems to be the three-in-a-tree algorithm of
Chudnovsky and Seymour.  Variations on this algorithm will be given.
This chapter presents result
from~\cite{leveque.lmt:detect}, \cite{maffray.t:reco},
\cite{nicolas.wei:kTree} and \cite{nicolas.d.p:fourTree}.

Chapter~\ref{chap:2dec} is devoted to two classes of graphs: graphs
that do not contain cycles with a unique chord and graphs that do not
contain induced subdivision of $K_4$.  The initial motivation for
these classes is their recognition, a problem arising from the
previous chapter.  But by studying them, some results were found about
their chromatic number.  This chapter presents results
from~\cite{nicolas.kristina:one} and \cite{nicolas:isk4}.

Chapter~\ref{chap:Berge} is devoted to Berge graphs.  Important
results are surveyed.  A structure theorem for Berge graphs with no
balanced skew partition is given.  As an application, a combinatorial
algorithms for coloring Berge graphs with no balanced skew partition
and no homogeneous pairs is given.  This chapter presents results
from~\cite{nicolas:bsp} et \cite{nicolas.kristina:2-join}.

Appendices A to F present mandatory material for the administration.
In Appendix~F, a list of my publications is given together with the
chapter of this document where results from each papers are to be
found.

Most of the studies presented below were done jointly with coauthors.
Their contributions will be acknowledged precisely, but I am happy to
list them: Amine Abdelkader, Nicolas Dehry, Sylvain Gravier, Andr\'as
Gy\'arf\'as, Benjamin L\'ev\^eque, David Lin, Christophe Picouleau,
J\'er\^ome Renault, Andr\'as Seb\H o, Juraj Stacho and Liu Wei.  I
would like to thank particularly two co-authors with whom
collaboration has been very close, rich and boosting for me:
Fr\'ed\'eric Maffray and Kristina Vu\v skovi\'c.

\section*{How to read this document}
\addcontentsline{toc}{section}{How to read this document}

We assume that the reader is familiar with the most basic concepts in
Graph Theory.  We use more or less standard definitions and notation
from the book of Bondy and Murty~\cite{bondy.murty:book}.  Here are
explained the conventions specific to this document.  By \emph{graph}\index{graph}
we mean simple and finite graph except when specified.  

By $N[v]$ we mean the closed neighborhood of a vertex, that is $N(v)
\cup \{v\}$.  Sometimes we allow a slight confusion between a graph
and its vertex set.  For instance when $G$ and $H$ are two graphs, we
write $G\sm H$ instead of $G[V(G) \sm V(H)]$.  Also, when $v$ is a
vertex of a graph $G$, we write $G\sm \{v\}$ or even $G\sm v$ instead
of $G[V(G) \sm \{v\}]$.  When $S$ is a set of vertices of $G$, we
write $G\sm S$ instead of $G[V(G) \sm S]$.  When we say that $C$ is a
connected component of a graph $G$, we usually do not specify whether
$C$ is an induced subgraph of $G$ or a subset of $V(G)$.  I hope all
this brings about more convenience than confusion.

Since all the document is about induced subgraphs, we say that $G$
\emph{contains}\index{contain!$G$ contains $H$} $H$ when $G$ has an
induced subgraph isomorphic to $H$.  We say that $G$ is
\emph{$H$-free}\index{free!$H$-free} (or \emph{with no $H$}) if it
does not contain $H$.  We call \emph{path}\index{path} any connected
graph with at least one vertex of degree~1 and no vertex of degree
greater than~2. A path has at most two vertices of degree~1, which are
the \emph{ends}\index{ends!of a path} of the path. If $a, b$ are the
ends of a path $P$ we say that $P$ is \emph{from $a$ to~$b$}. The
other vertices are the \emph{interior}\index{interior} vertices of the
path.  When $P$ is a path, we say that $P$ is \emph{a path of $G$} if
$P$ is an induced subgraph of $G$.  To sum up, what we call ``path of
$G$'' is what is called ``induced path of $G$'' in texts not focused
on induced subgraphs.  If $P$ is a path and if $a, b$ are two vertices
of $P$ then we denote by $aPb$ or $a \tp P \tp b$ the only induced
subgraph of $P$ that is path from $a$ to $b$.  The
\emph{length}\index{length of a path} of a path is the number of its
edges.

By $uvw$ we denote the path on vertices $uvw$ with edges $uv$, $vw$.
We also use the notation $u \tp v \tp w$. These notations are formally
equivalent, but we use the second one when we want to emphasize that
the path is an induced subgraph of some graph that we are working on
(that is most of the time).  When $G, G'$ are graphs, we denote by
$G\cup G'$ the graph whose vertex set is $V(G)\cup V(G')$ and whose
edge set is $E(G) \cup E(G')$.

In the document, the reader will find his/her beloved Definitions,
Lemmas, Theorems, Conjectures and Exercises, but also ``Questions''.
These are to be understood as questions on which I plan to work during
the next years.  But I would be happy if someone else solves them.

\section*{Perfect graphs}
\addcontentsline{toc}{section}{Perfect graphs}

Many interesting classes of graphs are defined by forbidding induced
subgraphs, see~\cite{chudnovsky.seymour:excluding} for a survey.  But
one class is worth presenting in the introduction since it has been
studied more than all others and will be present everywhere in this
document: the class of perfect graphs.

A graph $G$ is \emph{perfect}\index{perfect graph} if for every
induced subgraph $G'$ of $G$, the chromatic number of $G'$ is equal to
the maximum size of a clique of~$G'$.  A \emph{hole}\index{hole} in a
graph is an induced cycle of length at least~4.  An
\emph{antihole}\index{antihole} is the complement of a hole.  A graph
is said to be \emph{Berge}\index{Berge graph} if it does not contain
an odd hole nor an odd antihole.  It is easy to check that every
perfect graph must be Berge.  In 1961, Berge~\cite{berge:61}
conjectured that a graph is perfect if and only if its complement is
so.  This was known as the Perfect Graph Conjecture and was proved by
Lov\'asz~\cite{lovasz:nh}.  Berge also conjectured a stronger
statement: every Berge graph is perfect.  This was known as the
\emph{Strong Perfect Graph Conjecture}\index{Strong Perfect Graph
  Conjecture} and was an object of much research until it was finally
proved by Chudnovsky, Robertson, Seymour and Thomas in
2002~\cite{chudnovsky.r.s.t:spgt}.  So Berge graphs and perfect graphs
are the same class of graphs, but we prefer to write ``Berge'' for
results which rely on the structure of the graphs, and ``perfect'' for
results which rely on the properties of their colorings.  To prove
that all Berge graphs are perfect, Chudnovsky et al.\ proved a
decomposition theorem for Berge graphs.  This idea of using
decomposition was promoted in the 1980's and 1990's by Chv\'atal and
several others.  What seems to be now the good decomposition statement
was first guessed by Conforti, Cornu\'ejols and Vu\v skovi\'c who
could prove it in the square-free case~\cite{conforti.c.v:square}.  A
more complete survey can be found in Chapter~\ref{chap:Berge} that is
devoted to the decomposition of Berge graphs.

In 2002, Chudnovsky, Cornu\'ejols, Liu, Seymour and Vu\v
skovi\'c~\cite{chudnovsky.c.l.s.v:reco} gave a polynomial time
algorithm that decides whether a given graph is Berge (and therefore
perfect).  In the 1980's, Gr\"otschel, Lov\'asz and
Schrijver~\cite{grostchel.l.s:color} gave a polynomial time algorithm
that colors any perfect graph.  Their algorithm relies on linear
programming and the ellipsoid method, so it is difficult to implement.
A purely combinatorial algorithm, that would for instance rely on a
decomposition theorem, is still an open
question. Section~\ref{sec:SF} gives theorems and
conjectures about what such an algorithm could be.

There are several surveys on perfect graphs.  The survey of
Lov\'asz~\cite{lovasz:pgsurvey} is a bit old now but is still an very
good reading.  Two books contain many material: \emph{Topics on
  Perfect Graphs}~\cite{berge.chvatal:topics} edited Berge and
Chv\'atal and a more recent one, \emph{Perfect
  graphs}~\cite{livre:perfectgraphs}, edited by Ram{\'{i}}rez
Alfons{\'{i}}n and Reed.  The most recent survey is by the team who
proved the Strong Perfect Graph
Conjecture~\cite{chudvovsky.r.s.t:progress}.  A very entertaining
reading is a paper by Seymour~\cite{seymour:how} that tells the story
of how the Strong Perfect Graph Conjecture was proved.  I would
recommend this one even for non-mathematicians

\chapter{Mind the gap}
\label{chap:gap}
}

The \emph{gap}\index{gap} of a graph is the difference between its
chromatic number and the size of one of its largest cliques.  As the
reader will see in the next chapters, ``understanding a class of
graphs'' often means giving a bound for the gap of the graphs in the
class.  Gy\'arf\'as~\cite{gyarfas:perfect} defined a graph $G$ to be
\emph{$\chi$-bounded}\index{bounded!$\chi$-bounded} with
\emph{$\chi$-bounding function $f$}\index{bounding function} if for
all induced subgraphs $G'$ of $G$ we have $\chi(G') \leq f(\omega(G'))$.
A class of graphs is \emph{$\chi$-bounded} if there exists a
$\chi$-bounding function that holds for all graphs of the class.  Many
interesting classes of graphs are $\chi$-bounded or better, have a gap
bounded by a constant; the rest of this document will provide several
examples.  For instance, perfect graphs are precisely these graphs
whose induced subgraphs have all gap~0, or equivalently have a
$\chi$-bounding function defined by $f(x)=x$.  So the gap of a graph
can be seen as a measure of its perfectness.

Clearly for all graphs $G$ we have $\chi(G) \geq \omega(G)$, so the
gap of a graph is always non-negative.  Very early in the development
of graph theory it has been discovered that graphs of arbitrarily
large gap do exist.  Blanche Descartes gave a construction of graphs
with no 3, 4 and 5-cycles and of arbitrarily large chromatic number,
see~\cite{lovasz:Ex27P69}.  Mycielski~\cite{mycielski:color} gave a
famous construction of triangle-free graphs of arbitrarily large
chromatic number.  Mycielski's construction is sometimes given as the
simplest example of graphs of arbitrary gap.  To obtain a graph of a
given gap $k$, both these constructions require more than $2^k$
vertices.  Erd\H os proved that given two integers $k\geq 1$, $g \geq
3$, there exist graphs of girth at least $g$ and of chromatic number
at least $k$, see Section~VII.1 in~\cite{bollobas:mgt}.
Interestingly, Erd\H os does not give an explicit construction but a
probabilistic argument showing that the desired graph exists.  He
needs $k^{12}$ vertices to obtain a graph of gap $k$.  Erd\H os'
result implies that for any $k\geq 3$, the class of graphs of girth at
least $k$ is not $\chi$-bounded.

After this short survey, a reader might expect that a construction for
all $k\geq 0$ of a graph of gap $k$ whose order is linear into $k$ is
something quite involved, but this is not the case: consider the graph
$G_k$ obtained from $k$ disjoint $C_5$'s and add all edges between
them.  It is a routine matter to check that $\chi(G_k) = 3k$ and
$\omega(G) = 2k$, so with $5k$ vertices, we obtain a graph of gap $k$.
Of course this construction is far less interesting than those
mentioned above because it does not show for instance that
triangle-free graphs are not $\chi$-bounded.  But strangely, to my
knowledge, this trivial little construction is not mentioned in any
textbook.  In fact, I was quite surprised when I found it by accident.
Because when explaining to a mathematician what is a perfect graph,
after showing that odd holes have gap 1, often the question of whether
larger gaps exist pops up\dots I was not aware of such a striking
simple answer to this question.

Luckily, in September 2008, I asked to Andr\'as Gy\'arf\'as who was
visiting Grenoble: ``does gluing $C_5$'s as above gives the smallest
graph of a given gap?''.  The day after, Gy\'arf\'as had an example
showing that the answer is no.

\begin{exercise}
  Do not read further and find an integer $k$ and an example of a
  graph of gap $k$ on less than $5k$ vertices.  
\end{exercise}

Luckily again Andr\'as Seb\H o got interested in the problem and
discovered an unexpected link between this question and Matching
Theory.  All this lead us to several results about the minimum number
of vertices needed to construct a graph of gap $k$.

\section{Basic facts about the gap}

We use the following standard notation.  By $\alpha(G)$ we denote the
size of a largest stable set of $G$.  By $\zeta(G)$ we denote the
minimum number of edges that cover the vertices of $G$.  By $\nu(G)$
we denote the size of a maximum matching of $G$.  By $\theta(G)$ we
denote minimum number of cliques that cover the vertices of $G$.

For reasons that will be clearer later, we prefer to think about the
difference between $\chi$ and $\omega$ in the complement.  Hence the
\emph{gap}\index{gap} of a graph is defined by: $$\gap(G) = \theta(G) -
\alpha(G).$$

A graph $G$ is \emph{gap-critical}\index{gap-critical} if for any
vertex $v$ we have $\gap(G\sm v) < \gap(G)$.  For $t\geq 0$, we denote
by $s(t)$ the order of a smallest graph with a gap of~$t$.  A graph is
\emph{$t$-extremal}\index{extremal!$t$-extremal} if it has gap $t$ and
order $s(t)$.  A graph is
\emph{gap-extremal}\index{extremal!gap-extremal} if it is $t$-extremal
for some $t$.  Sometime, we write \emph{extremal} instead of
gap-extremal.  Note that the empty graph has gap 0, so $s(0) = 0$.
Note that $s(1) = 5$ and $C_5$ is the only 1-extremal graph.  It is
clear that every gap-extremal graph is gap-critical, but the converse
does not hold as shown by $C_7$.

\begin{lemma}
  \label{l:components}
  If a graph $G$ has $k$ connected components $C_1, \dots, C_k$ then
  $\gap(G) = \gap (C_1) + \cdots + \gap(C_k)$.  Every connected
  component of a gap-critical graph is gap-critical. Every connected
  component of a gap-extremal graph is gap-extremal.
\end{lemma}

\begin{proof}
  Clear.
\end{proof}

\begin{lemma}
  $s(t) \leq 5t$.
\end{lemma}

\begin{proof}
  Because the graph obtained by taking $t$ disjoint copies of $C_5$
  has gap $t$.
\end{proof}

\begin{lemma}
  \label{l:removeK}
  Let $G$ be a gap-critical graph and $K$ be a non-empty clique of
  $G$.  Then $\theta(G\sm K) = \theta(G) - 1$, $\alpha(G\sm K) =
  \alpha(G)$ and $\gap(G\sm K) = \gap(G)- 1$.
\end{lemma}

\begin{proof}
  In general, deleting a clique from a graph either leaves $\theta$
  unchanged or decreases it by one.  And deleting a clique from a
  graph either leaves $\alpha$ unchanged or decreases it by one.  But
  since $G$ is critical, the removal of a clique must decrease
  the gap by at least one, so the only option is $\theta(G\sm K) =
  \theta(G) - 1$ and $\alpha(G\sm K) = \alpha (G)$.  So, $\gap(G\sm K)
  = \gap(G) - 1$.
\end{proof}

\begin{lemma}
  \label{l:jumpk}
  If there exists a $(t+1)$-extremal graph $G$ that has a $k$-clique then
  $s(t+1) \geq s(t) +k$.
\end{lemma}

\begin{proof}
  By Lemma~\ref{l:removeK}, $\gap(G\sm K) = \gap(G) - 1$.  So $$s(t)
  \leq |V(G\sm K)| = |V(G)| - k = s(t+1) - k.$$
\end{proof}

\begin{lemma}
  \label{l:jump2}
  $s(t+1) \geq s(t) +2$.
\end{lemma}

\begin{proof}
  Clear by Lemma~\ref{l:jumpk} since any $t$-extremal graph, $t\geq
  1$, obviously contains an edge.
\end{proof}

A vertex of a graph is a \emph{simplicial vertex}\index{simplicial} if its neighbors
induce a complete graph.

\begin{lemma}
  \label{simplicial} 
  If $G$ is a gap-critical graph then $G$ has no simplicial vertex.
\end{lemma}

\begin{proof} 
  Suppose that we have a simplicial vertex $v$ in $G$. Let us put
  $K=N(v)\cup \{v\}$.  Then $$\theta(G)=1+\theta(G\sm K) = 1 +
  \gap(G\sm K) + \alpha(G\sm K) = 1 + \gap(G\sm K) + \alpha(G) -1,$$
  so by Lemma~\ref{l:removeK},
$$\theta(G) = \gap(G) - 1 +
\alpha(G),$$ a contradiction.
\end{proof}

\section{Ramsey Theory and asymptotic results}

By $S_k$ we denote a stable set on $k$ vertices.  The mathematician,
economist and philosopher Franck Ramsey proved the following which has
been the starting point of a very rich theory.

\begin{theorem}[Ramsey \cite{ramsey:fl}]
  For all integers $k, l \geq 1$ there exits a number $R$ such that
  any graph on at least $R$ vertices contains either a clique on $k$
  vertices or a stable set on $l$ vertices.
\end{theorem}

By $R(k, l)$\index{R@$R(k, l)$} we mean the smallest integer $N$ such
that every graph on $n\geq N$ vertices contains a clique on $k$
vertices or a stable set on $l$ vertices.  A graph is \emph{$R(k,
  l)$-extremal}\index{extremal!$R(k, l)$-extremal} if it has $R(k, l)
- 1$ vertices, and contains no clique on $k$ vertices and no stable
set on $l$ vertices.  Each time in the sequel we say ``it is known
from Ramsey Theory that'', the fact that we claim is to be found in
Table~\ref{t:ramsey}.  This table indicates the notation that we use
for several Ramsey extremal graphs, but they will be described when
needed.  All the information that we use along with the corresponding
references can be found in the survey of Radziszowski
\cite{radzi:ramsey}.

\begin{table}
\begin{center}
\begin{tabular}{lllll}
  $(k, l)$ & $R(k, l)$ & Extremal graphs & References\\\hline

  (3, 3) & 6  & 1 graph: $C_5$ & Folklore\\

  (3, 4) & 9 & 3 graphs: $W, W', W''$ &\cite{gg:ramsey,kery:ramsey}\\

  (3, 5) & 14 & 1 graph: $R$ &\cite{gg:ramsey,kery:ramsey} \\

  (3, 6) & 18 & 7 graphs &\cite{kery:ramsey,Kalbfleisch:ramsey}\\

  (3, 7) & 23 & 191 graphs &\cite{Kalbfleisch:ramsey,graverYackel:ramsey,mcKayZhang:ramsey,radziKreher:ramsey}\\

  (3, 8) & 28 &  At least 430215 graphs &\cite{grinsteadRoberts:ramsey,mcKayZhang:ramsey,radzi:ramsey}\\

  (3, 9) & 36 & At least 1 graph &\cite{Kalbfleisch:ramsey,grinsteadRoberts:ramsey} \\

  (4, 4)  & 18   & 1 graph &\cite{gg:ramsey,Kalbfleisch:ramsey,mcKayRadzi:R55}\\

  (4, 5)  & 25   & At least 350904 graphs&\cite{kalbfleisch:edge,mcKayRadzi:ramsey}
\end{tabular}
\caption{Small Ramsey numbers\label{t:ramsey}}
\end{center}
\end{table}

Ramsey Theorem provides the first example of a non-trivial
$\chi$-bounded class of graphs.  Because any $S_3$-free graphs $G$ has
a $\chi$-bounding function defined by $f(x) = R(3, x) - 1$ since
$$\chi(G) \leq |V(G)| \leq R(3, \omega(G))-1.$$

Here is why Ramsey numbers are related to small graphs of large gap.
Suppose that a graph $G$ is $R(3, k)$-extremal.  Since it has no
$K_3$, we have $\theta(G) \geq |V(G)|/2 \geq R(3, k)/2$ while
$\alpha(G) = k$.  So, if Ramsey numbers are large enough, we may
obtain a large gap.  Since $R(3, 5) = 14$, an $R(3, 5)$-extremal graph
provides a graph of gap at least~3 on 13 vertices, which is better
than 3 disjoint copies of $C_5$ (see Figure~\ref{fig:R} page
\pageref{fig:R}).  We prove now that using Ramsey-extremal graphs
gives rather small graphs of a given gap in general.  Jeong Han Kim
proved the following deep result:

\begin{theorem}[Kim \cite{kim:95}]
  \label{th:kim}
  $R(3, p) \geq c(1 - o(1))p^2 / \log(p)$ where $c = 1/162$.
\end{theorem}

\begin{theorem}[with Gy\'arf\'as and Seb\H o \cite{nicolas.gyarfas.sebo:gap}]
  $$2t + 9 \leq s(t) \leq 2t + o(t).$$
\end{theorem}

\begin{proof}
  From Theorem~\ref{th:gapSumUp} on page~\pageref{th:gapSumUp} (sorry for
  this non-circular forward reference), we know that $2t + 9 \leq
  s(t)$ holds for $t\leq 4$.  For greater values of $t$, the
  inequality follows by an easy induction from Lemma~\ref{l:jump2}.
  Let us prove the second inequality.

  Let $H_p$ be an $R(3, p)$-extremal graph and $t_p = \gap(H_p)$.  So,
  since $H_p$ is triangle-free,
  $$
  t_p = \theta(H_p) - \alpha(H_p) \geq |V(H_p)|/2 -p. 
  $$

  Hence,
  $$
  s(t_p) \leq |V(H_p)| \leq 2t_p+p.
  $$

From Theorem~\ref{th:kim},
  we have

\begin{claim}\label{i:gapIneq}
\hfill  $  
  t_p \geq |V(H_p)|/2 -p \geq c(1-o(p)) (p^2 / 2\log(p)) - p,
  $ \hfill
\end{claim}

\noindent which proves in particular that $t_p$ tends to infinity with
$p$.  So, it remains to prove that $p = o(t_p)$ because since $s(t)$
increases with $t$, $s(t_p) \leq 2t_p + o(t_p)$ implies $s(t) \leq 2t
+ o(t)$.

By solving~(\ref{i:gapIneq}), we obtain that $p \leq c(\varepsilon)
t_p^{(1/2) + \varepsilon}$ for any $\varepsilon > 0$, where
$c(\varepsilon)$ is a constant depending only on $\varepsilon$.  This
proves that $p = o(t_p)$ and therefore the theorem.
\end{proof}

\section{Triangle-free gap-critical graphs}

A graph is \emph{factor-critical}\index{factor-critical} if the removal of any vertex yields
a graph with a perfect matching. A proof in English of the following
theorem can be found in \cite{lovasz:Ex26P58}.

\begin{theorem}[Gallai, \cite{gallai:factorCritical}]
  \label{th:gallai}
  If $G$ is connected and  $\nu(G\sm v) = \nu(G)$ for all $v \in
  V(G)$, then $G$ is factor-critical.
\end{theorem}

The following is useful to avoid many case checking in the
sequel.

\begin{theorem}[with Gy\'arf\'as and Seb\H o \cite{nicolas.gyarfas.sebo:gap}]
  \label{l:fcrit}
  If $G$ is triangle-free and gap-critical graph then every component
  of $G$ is factor-critical.
\end{theorem}

\begin{proof}
  Let $H$ be a component of a triangle-free gap-critical graph.  By
  Lemma~\ref{l:components}, all connected components of a gap-critical
  graph are gap-critical, so $H$ is gap-critical.  Since the removal
  of an isolated vertex does not change the gap, $H$ has no isolated
  vertex.  So $\theta(H) = \zeta(H)$ since $H$ is triangle-free.  By
  Lemma~\ref{l:removeK}, for all $v\in V(H)$ we have $\theta (H\sm v)
  = \theta(H) - 1$, so $\zeta (H\sm v) = \zeta(H) - 1$.

  For all graphs $G$ we have $\nu(G) = |V(G)| - \zeta(G)$.  So,
  $\nu(H\sm v) = |V(H\sm v)| - \zeta(H\sm v) = |V(H)| - \zeta(H) =
  \nu(H)$.  Hence, $H$ is factor-critical by Theorem~\ref{th:gallai}.
\end{proof}

\section{Small gap-extremal graphs}

In this section, we compute $s(2)$, $s(3)$ and $s(4)$.  For $s(2)$ and
$s(3)$, we prove that the corresponding gap-extremal graphs are
unique.  The proofs are a bit tedious and we apologize for this.  Yet,
they give a feeling of what is going on (maybe some nicer argument can
shorten them?)  They are all given here also because they are not
published or even submitted.  They have been obtained jointly with
Gy\'arf\'as and Seb\H o.

\subsection*{About graphs on 10 vertices}

There exists a graph on 10 vertices with gap 2: take two disjoint
$C_5$'s.  We denote this graph by $2C_5$.  Our aim in this section is
to prove $s(2)=10$ and to prove that $2C_5$ is the only 2-extremal
graph.

\begin{lemma}
  \label{l:gap2Triangle}
  A graph on at most 10 vertices with a gap of 2 is either $2C_5$ or
  contains a triangle.
\end{lemma}

\begin{proof}
  Suppose that $G$ is a triangle-free graph on at most 10 vertices and
  $\gap(G)=2$.  We suppose $G$ minimal with respect to this property,
  so $G$ is gap-critical.

  By Lemma~\ref{l:fcrit}, all components of $G$ are factor-critical.
  If $G$ is connected, then $G$ is on $2k+1$ vertices and we have
  $\theta(G) = k+1$.  If $k=4$ then $\alpha(G) \geq 4$ because $R(3,
  4) = 9$.  So $\gap(G) \leq 1$, a contradiction.  If $k=3$ then
  $\alpha(G) \geq 3$ because $R(3, 3) = 6$.  So $\gap(G) \leq 1$, a
  contradiction.  If $k=2$, then $G$ is on 5 vertices, so we know that
  $\gap(G) \leq 1$, a contradiction.

  Hence, $G$ has at least two components that are all gap-critical by
  Lemma~\ref{l:components}.  So all components of $G$ have gap 1, and
  since $|V(G)| \leq 10$, there must be two of them, both isomorphic to
  $C_5$.  So $G$ is indeed $2C_5$.
\end{proof}

The \emph{Grotzsch graph}\index{Grotzsch graph}, or \emph{Mycielski
  graph}\index{Mycielski graph}, is a famous graph on eleven vertices
constructed as follows: take an chordless cycle $C = a_1, \dots, a_5$.
For every $i$ take a vertex $b_i$ adjacent to the neighbors of $a_i$
in $C$.  Then add a last vertex $c$ adjacent to $b_1, \dots, b_5$.
Chv\'atal \cite{chvatal:74} proved that the Grotzsch graph is the only
graph on at most eleven vertices such that $\omega=2$ and $\chi\geq
4$.  Rephrased in the complement, this gives:

\begin{lemma}[Chv\'atal, \cite{chvatal:74}]
  \label{l:chvatal}
  If $G$ is a graph on at most 10 vertices such that $\alpha(G) \leq
  2$ then $\gap(G) \leq 1$.
\end{lemma}

\begin{lemma}
  \label{l:s2=10}
  $s(2) = 10$.
\end{lemma}

\begin{proof}
  Because of $2C_5$ we know that $s(2) \leq 10$ so it remains to prove
  that no graph on at most nine vertices has a gap of 2.  So suppose
  for a contradiction that there exists a graph $G$ on at most nine
  vertices and $\gap(G)=2$.  We choose $G$ minimal with respect to
  this property, hence $G$ is gap-critical.  By Lemma~\ref{l:chvatal}
  we know that $\alpha(G)$ is at least 3, so it is sufficient to prove
  $\theta(G) \leq 4$.  By Lemma~\ref{l:gap2Triangle}, we know that $G$
  must contain a triangle $T$.  By Lemma~\ref{l:removeK}, $G\sm T$ has
  gap 1 and is on at most 6 vertices.  So, $G\sm T$ must contain an
  induced $C_5$ (else it is perfect).  If there is $v\in X\setminus
  V(C_5)$ such that $v$ is adjacent to a vertex of $C_5$ or $X=V(C_5)$
  then $\theta(G)\le 4$, a contradiction. Otherwise $v$ is a
  simplicial vertex in $G$ and the contradiction is by
  Lemma~\ref{simplicial}.
\end{proof}

Let $W$ be the Wagner graph, that is the graph on eight vertices $w_1,
\dots, w_8$, with the edges $w_iw_{i+1}$ and $w_iw_{i+4}$, $i=1,
\dots, 8$, where the addition is taken modulo 8.  It is well known
from Ramsey Theory that $R(3, 2) = 9$ and that $\alpha(W) = 3$,
$\omega(W) = 2$.  Note that the Wagner graph is not the unique graph
on eight vertices with $\alpha = 3$, $\omega = 2$.  Two other graphs
exist, that may be obtained from $W$ by removing respectively one and
two well chosen edges.

Since we want to cover a graph $G$ on 10 vertices $\alpha(G) + 1$
cliques, it would be useful to show that every such graph contains
sufficiently many ``big'' cliques or contains a ``big'' stable set.
Many such statements may exist, but let us discuss the following: ``if
$G$ has 10 vertices, then either it contains $K_4$, $S_4$ or two
disjoint triangles''.  Note first that $R(4, 4) = 18$.  Hence, the
output ``two disjoint triangles'' really needs to be there for else,
we may obtain counter-examples up to 17 vertices.  If $G$ is on only 9
vertices, it may fail to have one of the desired subset.  Indeed, if
we add a vertex complete to $W$, we obtain a graph on 9 vertices that
does not satisfy our conclusion.  The output ``clique on four
vertices'' needs to be there also because if we forget it, by adding
to $W$ a $K_2$ complete to $\{w_1, w_2, w_3\}$, we obtain a graph on
ten vertices that does not satisfy the conclusion.  Finally, the
output ``stable set on four vertices'' is needed and best possible as
shown by two disjoint copies of $C_5$, where no stable set on five
vertices exists.  Therefore the following lemma is in a sense best
possible.

A graph $G$ has the \emph{second stable set property}\index{second
  stable set property} if for every stable set $S$ of $G$ there exists
a maximum stable set of $G$ disjoint from $S$.  To check that a graph
has the second stable set property, it is sufficient to check that for
every \emph{maximal} stable set $S$ of $G$ there exists a maximum
stable set of $G$ disjoint from $S$.  But it maybe not sufficient to
check that for every \emph{maximum} stable set $S$ of $G$ there exists
a maximum stable set of $G$ disjoint from $S$.

\begin{lemma}
  \label{threecases}
  If $G$ is a graph on at least 10 vertices then either $G$ contains
  a clique or a stable set on four vertices, or $G$ contains two
  disjoint triangles.
\end{lemma}

\begin{proof}
  Let $G$ be on 10 vertices.  We suppose that $G$ contains no $K_4$
  and no two disjoint triangles.  We look for a stable set of size 4
  in $G$.  Since $R(4, 3) = 9$, we may assume that $G$ contains a
  triangle $T = uvw$.  We put $H = G\sm T$.  Note that since
  $G$ contains no two disjoint triangles, $H$ is triangle-free.  We
  give now three sufficient conditions on $H$ for the existence of an
  $S_4$ in $G$.

  \begin{claim}
    \label{c:bip}
    It is sufficient to prove that $H$ is bipartite.
  \end{claim}

  \begin{proofclaim}
    Clear because then $H$ (and then $G$) contains an $S_4$ since $H$
    is on $7$ vertices.
  \end{proofclaim}

  \begin{claim}
    \label{c:pento}
    It is sufficient to prove that $H$ contains an induced subgraph
    isomorphic to the \emph{pentoline}\index{pentoline} where the pentoline is the graph on
    $\{a, b, c, d, x\}$ whose edge-set is $\{ab, cd\}$.
  \end{claim}

  \begin{proofclaim}
    For suppose that $H$ contains a pentoline on $\{a, b, c, d, x\}$
    with edges $ab, cd$.  Then one vertex of $T$, say $u$, needs to be
    non-adjacent to $x$ for otherwise $G$ contains a $K_4$.  So, if
    $u$ has a non-neighbor in $ab$ and a non-neighbor in $cd$ we have
    an $S_4$.  Hence, we may assume that $u$ is complete to $ab$.  So
    none of $v, w$ can be a non-neighbor of $x$, because then by a
    similar argument, it would be complete to $ab$ or $cd$, and this
    would yield a $K_4$ (if $ab$) or 2 disjoint triangles (if $cd$).
    So both $v, w$ are adjacent to $x$.  Hence, $vwx$ and $uab$ are
    two disjoint triangles, a contradiction.
  \end{proofclaim}

  \begin{claim}
    \label{c:2ndstable10}
    It is sufficient to prove that $H$ has the second stable set
    property.
  \end{claim}

  \begin{proofclaim}
    Because then, suppose that $N_H(u)$ is a stable set of $H$.  Then $H$
    contains a stable set $S$ of size $\alpha(H)$ disjoint from
    $N_H(u)$.  And since $\alpha(H) = 3$ (because $H$ is triangle-free
    on at least $R(3, 3) = 6$ vertices), $S\cup \{u\}$ is a stable set
    of size 4 of $G$.  So $u$, and similarly every vertex of $T$ must
    have an edge in its neighborhood.  But since $G$ contains no two
    disjoint triangles, all these edges must be pairwise intersecting,
    so since $H$ is triangle-free, they must share a common vertex $x$.
    Therefore $T\cup \{x\}$ is a $K_4$, a contradiction.
  \end{proofclaim}

  Now we study more precisely the structure of $H$.  By~(\ref{c:bip})
  we may assume that $H$ is not bipartite, and since it is
  triangle-free, it must contain an induced $C_7$ or $C_5$.  If it
  contains an induced $C_7$, then $H=C_7$.  It easy to see that $C_7$
  has the second stable set property so we are done
  by~(\ref{c:2ndstable10}).  Hence we may assume that $H$ contains an
  induced $C_5$, say $C = z_1 \dots z_5 z_1$, and two other vertices
  $x, y$.

  If $x$ and $y$ have a common neighbor in $C$, say $z_1$, then since
  $H$ is triangle-free, $xy, xz_2, xz_5, yz_2, yz_5 \notin E(G)$.  So,
  $\{x, y, z_2, z_5\}$ is an $S_4$.  Hence, from here on we assume
  that $x$ and $y$ have no common neighbor in $C$.

  We suppose first that both $x, y$ have at most one neighbor in $C$.
  Then we may assume $N_C(x) \subseteq \{z_1\}$.  Up to symmetry we
  also assume $N_C(y) \subseteq \{z_2\}$ or $N_C(y) \subseteq
  \{z_3\}$.  In the second case, either $\{x, y, z_2, z_4\}$ is an
  $S_4$ (when $xy\notin E(G)$) or $\{x, y, z_2, z_4, z_5\}$ is a
  pentoline (when $xy\in E(G)$) and we
  apply~(\ref{c:pento}). Therefore we may assume $N_C(y) \subseteq
  \{z_2\}$.  We have $xz_1, yz_2, xy \in E(G)$ for otherwise $H$
  contains an $S_4$ or a pentoline.  We observe now that $H$ has the
  second stable set property, so we are done by~(\ref{c:2ndstable10}).

  Hence, we may assume from here on that $x$ or $y$ (say $x$) has at
  least 2 neighbors in $C$, and in fact exactly 2 neighbors in $C$ (say
  $z_1, z_3$) because $H$ is triangle free.  If $y$ has no neighbor in
  $C$ then $\{z_2, z_4, z_5, x, y\}$ either contains an $S_4$ or is a
  pentoline.  So, we may assume that $y$ has a neighbor in $C$ and we
  suppose first that $y$ has a unique neighbor in $C$.  Up to symmetry
  there are two case: $yz_2\in E(G)$ and $yz_5 \in E(G)$.  In the
  first case, we have $xy\in E(G)$ or otherwise, $\{z_2, y, x, z_4,
  z_5\}$ is a pentoline.  We observe now that $H$ has the second
  stable set property.  Hence we are in the second case where $yz_5\in
  E(G)$.  We have then $xy\in E(G)$ or $\{z_2, x, y, z_4\}$ is an
  $S_4$.  Now we observe that $z_1 x z_3 z_4 z_5 z_1$ is an induced
  $C_5$ of $H$ and that $z_2, y$ have both two neighbors in this
  $C_5$.  So, up to an isomorphism, it sufficient from here on to
  study the case when both $x, y$ have two neighbors in $C$.

  Since $x, y$ have both 2 neighbors in $C$ and $H$ is triangle-free,
  we may assume $N_C(x) = \{z_1, z_3\}$ and $N_C(y) = \{z_2, z_4\}$.
  If $xy\in E(G)$, we can see that $H$ has the second stable set
  property.  To get convinced, it maybe convenient to notice that $H$
  is in fact isomorphic to $K_{3, 3}$ with one edge subdivided.  So we
  may assume that $xy\notin E(G)$.  Now $H$ contains no pentoline and
  does not have the second stable set property.  Indeed $S = \{x, y,
  z_5\}$ is a stable set of $H$ that intersects every maximum
  stable set of $H$.  Also $S$ is the only stable set with this
  property.   Like in the proof of~(\ref{c:2ndstable10}),  we
  notice that if every vertex among $u, v, w$ sees an edge of $H$,
  this gives a $K_4$ or two disjoint triangles.  So may assume that
  $N_H(u)$ is a stable set, and unless this stable set is $S$, we find
  a stable set of size 3 disjoint from it implying $\alpha(G)\geq 4$.
  Hence, we may assume $N_C(u) = \{x, y, z_5\}$.

  Now, $G[H \cup \{u\}]$ is graph isomorphic to $W$.  To see this, one
  can consider the Hamiltonian cycle $z_1 z_2 y u x z_3 z_4 z_5 z_1$.
  So, we use our notation $w_1, \dots, w_8$ for the vertices of $W =
  G\sm \{v, w\}$.  If $v$ and $w$ both have two edges in their
  neighborhood in $W$ then they must see a common edge or two disjoint
  edges because $W$ is a cubic graph.  So, we have a $K_4$ or two
  disjoint triangles.  Hence, we may assume that $v$ sees at most one
  edge of $W$.  If $v$ sees no edge in $W$, then $v$ sees a stable set
  of $W$, and since $W$ has the second stable set property, there a
  stable set $S$ of size 3 of $W$ disjoint from $N_W(v)$, and $S\cup
  \{v\}$ shows $\alpha(G) \geq 4$.  So in fact, $v$ sees exactly one
  edge of $W$, say $w_1w_2$ (because $W$ is edge transitive).  Now,
  $N_W(v) \subseteq \{w_1, w_2, w_4, w_7\}$.  So, $\{v, w_3, w_5,
  w_8\}$ is a stable set of size 4 of $G$.
\end{proof}

\begin{lemma}\label{unique10}
  The unique 2-extremal graph is $2C_5$.
\end{lemma}

\begin{proof}
  Let $G$ be a 2-extremal graph.  By Lemma~\ref{l:s2=10} we have
  $|V(G)| = 10$.  By Lemma~\ref{l:chvatal}, we may assume that
  $\alpha(G) \geq 3$.  We suppose that $G$ is not $2C_5$ and we shall
  reach a contradiction by proving that $\theta(G)\le \alpha(G)+1$.
  By Lemma~\ref{threecases}, there are three cases:

  \noindent\bf Case 1: \rm $G$ contains a clique $K$ on 4 vertices.
  By Lemma~\ref{l:removeK} we have $\gap(G\sm K) = 1$ so the graph
  $G\sm K$ is not perfect.  Therefore $G\sm K$ contains an induced
  $C_5$.  Let $v\in V(G\sm K)$ be the vertex not on $C_5$.  If there
  is an edge from $v$ to $C_5$, we have a cover of $V(G)$ by $K$ and
  three edges, a contradiction.  Thus $v$ is not adjacent to any
  vertex of $C_5$ so it is a simplicial vertex in $G$.  Therefore
  there is a contradiction by Lemma~\ref{simplicial}.

  \noindent\bf Case 2: \rm $G$ contains a stable set $S$ on $4$
  vertices.  By Lemma~\ref{l:gap2Triangle}, $G$ contains a triangle
  $T$.  Note that in this paragraph, a cover of $G$ with at most 5
  cliques brings a contradiction.  The argument is similar to Case 1.
  Like in Case 1, $G \sm T$ is not perfect, thus it contains either
  $C_5$ or is isomorphic to one of the graphs $C_7$, $\overline{C_7}$.
  The latter cases lead to a contradiction since those graphs can be
  covered by at most four cliques.  Thus there is an induced $C_5$ in
  $G\setminus T$. If the two vertices uncovered by $C_5$ are adjacent,
  or any of them adjacent to a vertex of $C_5$, we have a similar
  contradiction.  Thus both of these uncovered vertices are simplicial
  and we finish by using Lemma \ref{simplicial}.

  \noindent\bf Case 3: \rm $G$ contains two vertex-disjoint triangles,
  $T_1=\{a_1,a_2,a_3\}, T_2=\{b_1,b_2,b_3\}$ and $\alpha(G) =
  \omega(G)=3$. If the remaining four vertices contain a triangle or
  two independent edges, we have $\theta(G)\le 4$, a contradiction.
  Therefore three of these vertices form an independent set
  $C=\{c_1,c_2,c_3\}$ and we have the following subcases according to
  the adjacencies of the last vertex~$d$ (which has a neighbor among
  $c_1, c_2, c_3$ because $\alpha(G) = 3$).

  \bf Subcase 3.1, \rm $dc_i\in E(G)$ for $i=1,2,3$.  Each vertex of
  $T_1$ must have a neighbor in $C$ because $\alpha(G) = 3$.  If
  $a_1c_1, a_2c_1 \in E(G)$ then we must have $a_3c_2 \in E(G)$ or
  $a_3c_3\in E(G)$ because there is no $K_4$. But then, we can cover
  $G$ with two triangles and two edges.  So we proved that no two
  vertices in $T_1$ can have a common neighbor in $C$.  Hence, we may
  assume that the only edges between $T_1$ (and similarly $T_2$) and
  $C$ are $c_ia_i$ (and similarly $c_ib_i$), $i=1,2,3$.  Using that
  $\alpha(G)=3$, it follows that $a_ib_i\in E(G)$ and now
  $a_i,b_i,c_i$ for $i=1,2,3$ give three disjoint triangles showing
  that $\theta(G)\le 4$, a contradiction.

  \bf Subcase 3.2, \rm $dc_3\in E(G),dc_1,dc_2\notin E(G)$.  Suppose
  first that every vertex of $T_1$ has a neighbor in $\{c_1, c_2\}$.
  Since there is no $K_4$ we may assume $a_1c_1, a_2c_2, a_3c_2 \in
  E(G)$, so we can cover $G$ with two triangles and two edges, a
  contradiction.  So there must be a vertex in $T_1$ with no neighbor
  in $\{c_2, c_1\}$, say $a_1$, and by the same argument a similar
  vertex in $T_2$, say $b_1$.  Using five times that $\alpha(G)=3$, we
  get that $a_1c_3,b_1c_3,a_1b_1,da_1,db_1\in E(G)$, a contradiction
  because $\{a_1,b_1,c_3,d\}$ is a clique.

  \bf Subcase 3.3, \rm $dc_2,dc_3\in E(G),dc_1\notin E(G)$.  We claim
  that $c_1$ is nonadjacent to at least two vertices of both
  $T_1,T_2$.  If not, say $c_1$ is adjacent to $a_2, a_3$, then
  $c_2a_1,c_3a_1 \notin E(G)$ otherwise we have a cover with two
  triangles and two edges. Depending on $c_1a_1\in E(G)$ or not, we
  have either a clique or an independent set of size four, a
  contradiction that proves the claim.  Therefore, w.l.o.g.\ $c_1$ is
  non-adjacent to $a_2, a_3, b_2, b_3$.  If $c_1a_1\notin E(G)$ or
  $c_1b_1\notin E(G)$ or $a_1b_1 \in E(G)$ then $c_1$ is a simplicial
  vertex, a contradiction to Lemma~\ref{simplicial}.  Thus $c_1a_1,
  c_1b_1\in E(G), a_1b_1\notin E(G)$.

  Next we note that each of $a_2, a_3$ must have a neighbor in $\{c_2,
  c_3\}$, else there is an $S_4$.  But $a_2, a_3$ may not have a
  common neighbor in $\{c_2, c_3\}$ because then there is a cover with
  two triangles and two edges.  Hence w.l.o.g.\ the only edges between
  $T_1$ and $C$ are $c_1a_1, c_2a_2, c_3a_3$.  Similarly, the only
  edges between $T_2$ and $C$ are $c_1b_1, c_2b_2, c_3b_3$.

  Now $\alpha(G)=3$ implies $a_2b_2,a_3b_3\in E(G)$.  Moreover $da_2,
  da_3, db_2, db_3 \notin E(G)$ otherwise there is a clique cover with
  two triangles and two edges. Then $a_2b_3, a_3b_2 \in E(G)$ for
  otherwise $a_2, b_3, c_1, d$ or $a_3, b_2, c_1, d$ would form an
  independent set.  But now have the final contradiction since
  $a_2, a_3, b_2, b_3$ span a clique.
\end{proof}

\subsection*{About graphs on 13 vertices}

Let $R$ be the graph on $\{r_1, \dots, r_{13}\}$ with the following
edges: $r_ir_{i+1}$ and $r_ir_{i+5}$, $i=1, \dots, 13$, where the
addition taken modulo 13 (see Figure~\ref{fig:R}).  It well known from
Ramsey Theory that $R$ is the largest graph such that $\omega=2$ and
$\alpha=4$.  Note that $\gap(G) = 3$.

\begin{figure}[ht]
  \center
  {\includegraphics{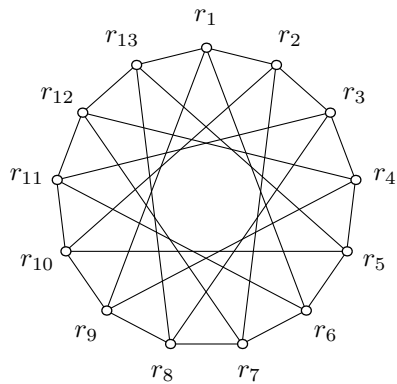}}
 \caption{The $R(3, 5)$-extremal graph\label{fig:R}}
\end{figure}

\begin{lemma}
  \label{l:s313}
  $s(3) = 13$ and every 3-extremal graph is connected.
\end{lemma}

\begin{proof}
  Because of $R$, $s(3) \leq 13$.  Since $s(2)=10$
  (Lemma~\ref{l:s2=10}), $s(1) = 5$ and by Lemma~\ref{l:components},
  it is impossible to have a disconnected graph with gap 3 on less
  than 15 vertices.  So every 3-extremal graph is connected.  By
  Lemma~\ref{l:jump2} and since $s(2) = 10$ we have $s(3) \geq 12$.
  So suppose for a contradiction that $s(3) = 12$ and consider an
  extremal graph on 12 vertices.  By Lemma~\ref{l:jumpk}, $G$ is
  triangle-free.  By Lemma~\ref{l:fcrit}, every component of $G$ is
  factor-critical.  In particular, every component of $G$ has an odd
  number of vertices, so $G$ is not connected a contradiction.
\end{proof}

\begin{lemma}
  \label{l:3exT}
  A 3-extremal graph is either $R$ or contains a triangle.
\end{lemma}

\begin{proof}
  If $G$ is triangle-free and 3-extremal $|V(G)|=13$, $G$ is connected
  by Lemma~\ref{l:s313} and $G$ is factor-critical by
  Lemma~\ref{l:fcrit}.  So $\theta(G) = 7$ and $\alpha(G) = \theta(G)
  - \gap(G) = 4$.  Hence $\omega(G) = 2$ and $\alpha(G) = 4$, so from
  Ramsey Theory, $G$ is isomorphic to $R$.
\end{proof}

\begin{lemma}
  \label{l:twelveTh}
  The unique 3-extremal graph is $R$.
\end{lemma}

\begin{proof}
  Let $G$ be a 3-extremal graph.  So $|V(G)| = 13$.  By
  Lemma~\ref{l:3exT} we may assume that $G$ contains a triangle $T =
  c_1c_2c_3$.  Let us see that this leads to a contradiction.  By
  Lemma~\ref{l:removeK}, $\gap(G\sm T) = 2$ so $G\sm T$ is isomorphic
  to $2C_5$.  So $G\sm T$ contains two disjoint $C_5$, $C = a_1 \dots
  a_5$ and $C' = b_1 \dots b_5$.  Note that $\omega(G) = 3$ for
  otherwise, by removing a $K_4$ we obtain by Lemma~\ref{l:removeK} a
  graph with gap 2 on 9 vertices, a contradiction to $s(2)=10$.

  We see that $G$ admits a clique cover with 7 cliques, so $\alpha(G)
  \leq 4$.  So, $c_i$, $i=1, 2, 3$, has a neighbor in every $S_4$ of
  $G\sm T$.  It follows that $c_i$, $i=1, 2, 3$, must be complete to
  an induced $P_3$ of at least one of $C, C'$.  But since there is no
  $K_4$, these three $P_3$'s must be edge-disjoint.  So w.l.o.g.\ we
  have $c_1a_1, c_1a_2, c_1a_3$, $c_2a_3, c_2a_4, c_2a_5$, $c_3b_1,
  c_3b_2, c_3b_3$, $\in E(G)$.  Now $\{a_1, a_2, c_1\}$, $\{a_4, a_5,
  c_2\}$, $\{a_3\}$, $\{b_1, b_2, c_3\}$, $\{b_3, b_4\}$ and $\{b_5\}$
  are 6 cliques that cover $G$, a contradiction.
\end{proof}

\subsection*{About graphs on 16 vertices}

The aim of this section is to prove $s(4) = 17$.  We need to study $R$
a bit further.

\begin{lemma}
  \label{l:Rsecond}
  If $S$ is a set of vertices of $R$ such that $|S|\leq 4$ then there
  exists a maximum stable set $S'$ of $R$ such that $S\cap S' =
  \emptyset$.  In particular, $R$ has the second stable set property.
\end{lemma}

\begin{proof}
  Note that $S$ is any set (possibly not a stable set).  Since
  $|S|\leq 4$ we have $|V(R\sm S)| \geq 9$.  Since $R(3, 4) = 9$,
  $R\sm S$ contains a stable set of size 4.
\end{proof}

\begin{lemma}
  \label{l:s416}
  $s(4) \geq 16$.
\end{lemma}

\begin{proof}
  By Lemma~\ref{l:jump2} and since $s(3) = 13$ we have $s(4) \geq
  15$.  So our lemma holds unless $s(4) = 15$.  Then let $G$ be a
  4-extremal graph on 15 vertices.  By Lemma~\ref{l:jumpk}, $G$ is
  triangle free, and by Lemma~\ref{l:removeK}, for any edge $xy$,
  $G\sm xy$ has gap 3, so $G\sm xy$ is isomorphic to $R$ by
  Lemma~\ref{l:twelveTh}.  But since $G$ is triangle-free, $N(x)$ is a
  stable set, and since $R$ has the second stable set property, $G$
  contains a stable set $S$ of size four disjoint from $N(x)$.  So
  $\alpha(G) \geq 5$ because of $S\cup \{x\}$.  Since $\theta(G) \leq
  \theta(R) + 1 = 8$, we have $\gap(G) \leq 3$, a contradiction.
\end{proof}

\begin{lemma}
  \label{l:16omega2}
  If $G$ is triangle-free and 4-extremal then $G$ contains at least 17
  vertices.
\end{lemma}

\begin{proof}
  Suppose for a contradiction that there exists a triangle-free
  4-extremal graph $G$ on at most 16 vertices.  Then by
  Lemma~\ref{l:fcrit} every connected component of $G$ is
  factor-critical, so $G$ cannot be connected.  But since $s(1) = 5$,
  $s(2) = 10$, $s(3) = 13$ and $s(4) \geq 16$, it is easy to see by
  Lemma~\ref{l:components} that a disconnected graph on at most 16
  vertices has gap at most 3.
\end{proof}

The following is useful.

\begin{lemma}
  \label{l:3edge}
  Let $G$ be a non-empty graph and $E_1, E_2, E_3 \subseteq E(G)$ be
  three pairwise disjoint sets of edges.  Suppose that there are no
  disjoint edges $e,f$ such that $e\in E_i$, $f\in E_j$ where $1\leq
  i<j \leq 3$.  Then either:
  \begin{itemize}
  \item
    there exists one vertex $v$ and one integer $1\leq i\leq 3$
    such that for all edges $e\in E_i$ we have $v\in e$
  \item
    the graph spanned by $E_1, E_2, E_3$ is a $K_4$ and $E_1, E_2,
    E_3$ are three disjoint perfect matchings of this $K_4$.
  \end{itemize}
\end{lemma}

\begin{proof}
  If $E_i = \emptyset$ we are done since then any vertex of $G$ is
  contained in all edges of $E_i$ by vacuity.  Also we are done
  trivially if $|E_i|=1$ by choosing any vertex in the unique edge of
  $E_i$.  Hence we may assume $|E_i| \geq 2$, $i= 1, 2, 3$.  So let
  $e_1 = uv$ be an edge of $E_1$.  Now every edge of $E_2$ must be
  incident to either $u$ or $v$ by assumption.  We may assume that
  there are edges $ux$, $vy \in E_2$ for otherwise $u$ or $v$ will be
  in every every edge of $E_2$, and we may assume $x\neq y$ for
  otherwise $x$ will be in every edge of $E_2$.  Similarly there is an
  edge $ux'\in E_3$, but $x'=y$ for otherwise $vy, ux'$ contradict our
  assumption.  Also, $vy\in E_3$.  Now there is one more edge in $E_1$
  and it must be $xy$ (or we are done).  Now the graph spanned by
  $E_1, E_2, E_3$ is a $K_4$ and $E_1, E_2, E_3$ are three perfect
  matchings of this $K_4$.  Adding any edge to $E_1 \cup E_2 \cup E_3$
  contradicts our assumption.
\end{proof}

\begin{lemma}
  \label{l:s4leq17}
  The seven $R(3,6)$-extremal graphs have gap 4 and $s(4) \leq 17$.
\end{lemma}

\begin{proof}
  From Ramsey Theory, it is known that there are seven
  $R(3,6)$-extremal graphs on 17 vertices. If $G$ is one one of them
  then $\omega(G) = 2$ and $\alpha(G) = 5$.  So $\theta(G) \geq 9$ and
  $\gap(G) \geq 4$.  Since $s(4) \geq 16$ (Lemma~\ref{l:s416}), we
  know by Lemma~\ref{l:jump2} that $s(5) \geq 18$, so $\gap(G) \leq
  4$.  Hence $\gap(G) = 4$ proving $s(4) \leq 17$.  
\end{proof}

\begin{lemma}
  \label{l:s417}
  $s(4) = 17$. 
\end{lemma}

\begin{proof}
  We know $16 \leq s(4) \leq 17$, so suppose for a contradiction that
  $s(4) = 16$.  Let $G$ be 4-extremal on 16 vertices, so $\gap(G) =
  4$.  Then by Lemma~\ref{l:16omega2} we must have $\omega(G) \geq 3$.
  By Lemma~\ref{l:removeK}, $G$ cannot contain a $K_4$ (else $G\sm
  K_4$ has 12 vertices and gap 3 contradicting $s(3)=13$), hence
  $\omega(G)=3$.  So let $T = t_1t_2t_3$ be a triangle of $G$.  By
  Lemma~\ref{l:removeK}, we have $\gap(G\setminus T) = \gap(G) - 1 =
  3$.  So by Lemma~\ref{l:twelveTh}, we have that $G\setminus T$ is
  isomorphic to $R$.  In fact we proved that for every triangle $T$,
  $G\sm T$ is isomorphic to $R$.  Since $R$ is triangle-free we deduce
  that $G$ does not contain two vertex-disjoint triangles.  Since
  $\theta(R) = 7$, we have $\theta(G) \leq 8$ and since $\alpha(R) =
  4$ we have $\alpha(G) \geq 4$.  But since $\gap(G) = 4$ we must have
  $\theta(G) = 8$ and $\alpha(G) = 4$.  Let us sum up the properties
  of $G$ needed in the end of the proof:

  \begin{claim}
    $\omega(G) = 3$, $\alpha(G) = 4$ and $G$ contains no 2 disjoint
    triangles.
  \end{claim}

  Also:

  \begin{claim}
    \label{c:geq5}
    $|N_R(t_i)| \geq 5$, $i = 1, 2, 3$. 
  \end{claim}
  
  \begin{proofclaim}
    If $|N_R(t_i)| < 5$ then by Lemma~\ref{l:Rsecond}, there exists an
    $S_4$ in $G\sm N_R(t_i)$.  Together with $t_i$, this gives an
    $S_5$ in $G$, a contradiction to the properties of $G$.
  \end{proofclaim}

  Let $N_i = N_R(t_i)$ and $E_i$ be the edge-set of $R[N_i]$, $i=1, 2,
  3$.  It is impossible to have $1 \leq i < j \leq 3$, $e_i\in E_i$,
  $e_j \in E_j$ and $e_i \cap e_j = \emptyset$.  Because then $\{t_i\}
  \cup e_i$ and $\{t_j\} \cup e_j$ are two vertex-disjoint triangles.
  So $E_1, E_2, E_3$ are three disjoint sets of edges of $R$ (disjoint
  because else, there is a $K_4$ in $G$) that satisfies the assumption
  of Lemma~\ref{l:3edge}.  Since $R$ is triangle-free, the only
  possible output of Lemma~\ref{l:3edge} is that some vertex of $v$ of
  $R$ is in all edge of $E_1$ say.  So $S = N_R(t_1) \setminus \{v\}$
  is a stable set of $R$.  Hence, $|N_R(t_1)| \leq 5$.
  By~(\ref{c:geq5}), $|N_R(t_1)| = 5$ so, since $R$ is 4-regular,
  $N_R(t_1) = N[v]$.  Hence since $R$ is vertex-transitive we may
  assume $N_R(t_1) = \{r_1, r_2, r_{13}, r_6, r_9\}$.  Now it is a
  routine matter to check that any edge of $R$ is disjoint from at
  least one edge of $R[\{r_1, r_2, r_{13}, r_6, r_9\}]$.  Since $G$
  does not contain 2 disjoint triangles, this means that $N_R(t_2)$ is
  a stable set, contradicting~(\ref{c:geq5}).
\end{proof}

\section{Conclusion}

From all the lemmas of the previous section, we have:

\begin{theorem}[with Gy\'arf\'as and Seb\H o
  \cite{nicolas.gyarfas.sebo:gap}]
  \label{th:gapSumUp}
  $s(1) = 5$, $s(2) = 10$, $s(3) = 13$ and $s(4) = 17$.
\end{theorem}

It is strange that the hardest part of the work was devoted to $s(2)$
and that $s(3)$ was much easier.  It seems that the difficulty of
computing $s(k+1)$ depends more on the jump $s(k+1)-s(k)$ than on $k$
itself.  Maybe the jump of 5 between $s(2)$ and $s(1)$ does not exist
later in the sequence, and maybe after a while, jumps of 4 disappear
also.  So, a kind of easiness may occur for big numbers, but it could
be of no use because ``bigness'' brings its own kind of trouble.

We believe $s(5) = 21$ and that any gap-extremal on 21 vertices can be
obtained by removing one vertex from an $R(3, 7)$-extremal graph.  With
some tedious checking, this might be provable, but one could get
tired of trying to get there and further\dots\ More generally, it is
tempting to conjecture that all $R(3, p)$-extremal graphs (where
$p\geq 5$ and after possibly removing one vertex) are gap-extremal and
that all the gap-extremal graphs are obtained by removing vertices
from these.  But we are far from a proof.

Proving $s(4) = 17$ without a computer and with a reasonably long
proof (well, perhaps a bit too long\dots) is not so bad since small
extremal objects are often difficult to compute.  Of course, we
cheated a bit: we took advantage from the knowledge of small Ramsey
numbers, and these helped a lot.  The fact that Ramsey numbers helped
suggests that what we did is a kind of dual of the computation of the
$R(3, p)$'s.  Let us explain this.

Ramsey Theory says: when your graph is big, it has some structure
(here a triangle or a big stable set).  Gap Theory (if any) says: when
your graph is small, it has some structure (here, a small gap).  So,
Ramsey Theory is a matter of maximization problems (find the maximum
number of vertices without creating a triangle or a big stable set).
And Gap Theory is a matter of minimisation problems (find a minimum
number of vertices with a given gap).  When a maximization and a
minimisation problem have the same solutions and when moreover the
solution to one helps to find the solution to the other, they are
likely to be dual.

\begin{question}
  Try to give a formal evidence that computing $R(3, p)$-extremal
  graphs and gap-extremal graphs are dual problems.  
\end{question}

Maybe an answer to this question can help to compute bounds on small
Ramsey numbers.

\chapter{Six simple decomposition theorems}
\label{chap:6th}

A \emph{decomposition theorem}\index{decomposition vs structure} for a
class $\cal C$ of mathematical objects is any statement saying that
every object of the class either belongs to some well understood
\emph{basic} class or can broken into pieces according to some well
described rules.

The oldest decomposition theorem perhaps states that any $n$-gon is
either a triangle or can be obtained by gluing a triangle along an
edge to an $(n-1)$-gon.  This theorem is of great practical interest
since it allows computing the area of any $n$-gon by just knowing how to
proceed with triangles.

Another famous and very old decomposition theorem states that any integer
$n\geq 2$ can be obtained uniquely by multiplying primes.  Here, the
practical interest is less direct.  Also, the ``basic class'', i.e.,\
prime numbers, is till today far from being ``well understood''.  Yet,
any mathematician would agree that this theorem describes an essential
aspect of integers.

Decomposition theorems for classes of graphs are just as these above.
Some will have a clear practical interest: allowing to devise fast
algorithms.  Others will be of a more theoretical flavour, but one
feels in front of them that they really describe the essence of the
class.  Some will be more artificial and are designed only for proving
a single theorem.  The notion of basic class also can be different
from one theorem to another, as above.  Some basic classes will be
clearly ``simple'', as cycles or paths.  Others will be simple only
with respect to some questions.  For example, bipartite graphs are
very simple with respect to graph coloring, but are as complicated as
general graphs with respect to the isomorphism problem.

Some decomposition theorems are in fact something stronger, they are
what we call structure theorems.  A \emph{structure theorem} tells how
all objects from a class can be built from basic pieces by gluing them
together.  So, the two examples given above are structure theorems:
all $n$-gons can be built by gluing triangles (with perhaps some
restrictions, like requiring that the triangles do not overlap, but
this is not essential), all integers can be built by multiplying
primes.

It is not obvious to give a simple example of decomposition theorem
that is not a structure theorem.  The famous Bolzano-Weierstrass
Theorem is the following: any bounded sequence of real numbers either
converges to some limit, or contains a subsequence that converges to
some limit.  This is a kind decomposition theorem where ``basic''
sequences are these which converge.  But it does not tell how all
bounded sequences can be built from these that converge, so it is not
a structure theorem for bounded sequences.  A better example from
Graph Theory is Hayward's characterization of weakly triangulated
graphs, see Section~\ref{sec:wt}.

The simplest decomposition theorem for a class of graphs defined by
forbidding induced subgraphs is the following.

\begin{theorem}[folklore]
  \label{th:P3}
  A graph is $P_3$-free if and only if it is a disjoint union of
  cliques.
\end{theorem}

\begin{proof}
  A disjoint union of cliques is obviously $P_3$-free.  Conversely,
  consider a connected component $C$ of a $P_3$-free graph and suppose
  for a contradiction that two vertices $u, v$ of $C$ are not
  adjacent.  A shortest path of $C$ linking $u$ to $v$ contains a
  $P_3$, a contradiction.
\end{proof}

Each of the next six sections is devoted to a simple decomposition (or
structure) theorem illustrating notions that are interesting in a more
general context.

\section{Subdivisions of a paw}

The \emph{paw}\index{paw} is the graph on four vertices, with four edges and that
contains a triangle.  A set $A$ of vertices of a graph is
\emph{complete}\index{complete} (resp.\ \emph{anticomplete}\index{anticomplete}) to a set of vertices $B$
when $A\cap B = \emptyset$ and there are all possible edges
(resp. no edge) between $A$ and $B$.  A \emph{complete $k$-partite}\index{complete $k$-partite}
graph where $k\geq 1$ is a graph made of $k$ disjoint non-empty stable
sets pairwise complete to one another.

\begin{theorem}[with Abdelkader \cite{abdelkader:memoire}]
  \label{th:paw}
  A connected graph $G$ does not contain any subdivision of the paw if
  and only if $G$ is a cycle or $G$ is a complete $k$-partite graph
  where $k\geq 2$ or $G$ is a tree.
\end{theorem}

\begin{proof}
  Clearly a tree and a cycle do not contain a subdivision of the paw.
  In a subdivision of a paw, there is a vertex of degree one and there
  is at least one edge between its non-neighbors.  So, such a vertex
  cannot exist in a complete multipartite graph.  Hence, a complete
  multipartite graph, a tree or a cycle contains no subdivision of the
  paw.  Let us prove the converse by considering a graph $G$ that
  contains no subdivision of the paw.

  If $G$ contains a triangle then let $H$ be a complete $k$-partite
  graph that is an induced subgraph of $G$, where $k\geq 3$.  Let us
  suppose that $H$ is inclusion-wise maximal with respect to this
  property.  So, $V(H)$ can be partitioned into stable sets $H_1$,
  \dots, $H_k$ that are pairwise complete to one another.  If $G=H$ we
  are done, so let us suppose for a contradiction that there is a
  vertex $v$ in $G\sm H$.  Since $G$ is connected, we may assume that
  $v$ has a neighbor $v_1$ in $H_1$ say.  If $v$ has a non-neighbor
  $v_i \in H_i$ and a non-neighbor $v_j\in H_j$ where $1< i < j \leq
  k$ then $\{v, v_1, v_i, v_j\}$ induces a paw, so $v$ must be
  complete to all $H_i$'s, $1<i\leq k$ except possibly one, say $H_2$.
  So $v$ has a neighbor in $H_3$ and by a symmetric argument, $v$ is
  complete to $H_1$.  Suppose that $v$ has a neighbor $v_2 \in H_2$
  and a non-neighbor $v'_2 \in H_2$.  Then $\{v, v_1, v_2, v'_2\}$
  where $v_1\in H_1$ induces a paw, a contradiction.  Hence $v$ is
  either complete or anticomplete to $H_2$.  In either case, $H\cup
  \{v\}$ induces a complete multipartite graph, a contradiction to the
  maximality of $H$.  So, from here on we may assume that $G$ contains
  no triangle.

  If $G$ contains a square then let $H$ be a complete bipartite graph
  with both sides of size at least two and that is an induced subgraph
  of $G$.  Let us suppose that $H$ is inclusion-wise maximal with
  respect to this property.  So, $V(H)$ can be partitioned into stable sets
  $H_1$, $H_2$ that are pairwise complete to one another.  If $G=H$ we
  are done, so let us suppose for a contradiction that there is a
  vertex $v$ in $G\sm H$.  Since $G$ is connected, we may assume that
  $v$ has a neighbor $v_1$ in $H_1$ say.  Since $G$ contains no
  triangle, $v$ has no neighbor in $H_2$. If $v$ has a non-neighbor
  $v'_1 \in H_1$ then $\{v, v_1, v'_1, v_2, v'_2\}$ where $v_2,
  v'_2\in H_2$ induces a subdivision of a paw, a contradiction.  So,
  $v$ is complete to $H_1$.  Hence, $H\cup \{v\}$ induces a complete
  bipartite graph, a contradiction to the maximality of $H$.  So, from
  here on we may assume that $G$ contains no square.

  If $G$ contains a cycle then let $C= c_1\dots c_lc_1$ be a shortest
  cycle.  If $G=C$ we are done, so let us suppose that $v$ is a vertex
  of $G\sm C$.  Since $G$ is connected, we may assume that $v$ has a
  neighbor $c_1$ in $C$ say.  If $v$ has no other neighbor in $C$ then
  $C\cup \{v\}$ is a subdivision of a paw, so $v$ has another neighbor
  $c_i$.  Let us choose such a neighbor with $i$ minimum.  Since $G$
  contains no triangle, $v$ is not adjacent to $c_{i+1}$ and since $G$
  contains no square, $c_{i+1}$ is not adjacent to $c_1$.  Hence,
  $\{v, c_1, \dots, c_{i+1}\}$ induces a subdivision of a paw, a
  contradiction.  So, from here on we may assume that $G$ contains no
  cycle.

  Since $G$ is connected with no cycle, it is a tree. 
\end{proof}

The hardest part in finding and proving the theorem above was maybe to
guess the statement from a bunch of examples.  The proof goes
through~3 steps: when the graph contains a triangle, when it contains
a square, when it contains a cycle.  In each step, it is proved that
the whole graph must be a kind of extension of the considered
subgraph.  Technically, for the sake of writing the proof, it is
convenient to assume that the graph under consideration contains a
``maximal extension'' of the subgraph under consideration.  This is
typical of how proving decomposition theorems usually goes.  The order
in which the subgraphs are considered is the key to short proofs of
simple theorems.  For more complicated theorems, it is simply the key
to the proof.  As an exercise, the reader could try to reprove the
theorem above by first supposing that the graph contains a
sufficiently big tree (more than a claw).  It is a likely that the
proof will be uncomfortable but I would not bet too much on that; it
might as well lead to a shorter proof.

Let us add an informal remark.  By reading carefully the proof, one
can see that there are three basic classes but that they are of
different flavor.  Cycles and complete multipartite graphs form what I
call \emph{connected classes}\index{connected class vs sparse class},
that are basic classes of graphs that are sufficiently rich,
``connected'', so that adding a vertex to them very easily yields an
obstruction.  On the other hand, trees rather form what I call a
\emph{sparse class}\index{sparse class vs connected class},.  So, the
proof goes this way: trying to get rid of as many connected-class
subgraphs as possible (by showing that their presence entails a
decomposition), and then proving that the graph is so impoverished
that it is in a sparse class.  Readers familiar with the proof of the
Strong Perfect Graph Theorem will recognize the line-graphs of
bipartite graph as the main connected class of Berge graphs, while
bipartite graphs form the main sparse class.

The original motivation for Theorem~\ref{th:paw} is Scott's
conjecture.  When $H$ is a graph, we denote by ${\rm
  Forb}(H)$\index{Forb} the class of those graphs that do not contain
$H$.  We denote by ${\rm Forb}^*(H)$\index{Forb$^*$} the class of
those graphs that do not contain any subdivision of $H$ (so ${\rm
  Forb}^*(H)$ is a superclass of ${\rm Forb}(H)$ since we view a graph
as one of its own subdivisions).

\begin{conjecture}[Scott]
  \label{conj:scott}
  For all graphs $H$, ${\rm Forb}^*(H)$ is $\chi$-bounded.
\end{conjecture}

Theorem~\ref{th:paw} implies trivially that Scott's conjecture is true
when $H$ is the paw.  Indeed, from Theorem~\ref{th:paw}, one can
easily check that if $G$ is in ${\rm Forb}^*({\rm paw})$ then either
$\chi(G) =3$ or $\chi(G) = \omega(G)$.  Is there a simpler proof of
this tight bound for $\chi$ that does go through a full description of
the class?  I do not know, but the proof given here is quite simple.
As we will see in the rest of this document, the structural method is
very efficient for giving tight bounds on $\chi$.

Can the method that was successful for the paw prove Scott's
conjecture for other graphs?  Certainly it can for small graphs as
$P_4$ or the square for instance.  Graph in ${\rm Forb}^*({\rm
  square})$ are usually called \emph{chordal graphs}\index{chordal
  graph} and graphs in ${\rm Forb}^*({\rm P_4})$ are simply $P_4$-free
graphs\index{P@$P_4$-free graphs} (because any subdivision of $P_4$
contains $P_4$).  From the following two classical theorems, it
follows by an easy induction that chordal graphs and graphs with no
$P_4$ are perfect.

\begin{theorem}[Dirac, \cite{dirac:chordal}]
  \label{th:chordal}
  Any chordal graph either is a clique or has a clique-cutset. 
\end{theorem}

\begin{theorem}[Seinsche \cite{seinsche:P4}]
  Any $P_4$-free graph is either a vertex, or is disconnected or has a
  disconnected complement.  
\end{theorem}

But for more complicated classes, it would be surprising that the
structural approach solves Scott's conjecture in such a
straightforward way.  Consider a cycle on $a, b, c, d$ (in this
order), add two pending edges to $a$, two pending edges to $c$, and
call $H$ the resulting graph on eight vertices.  In
Chapter~\ref{chap:reco}, it is proved that deciding whether a graph is
in ${\rm Forb}^*(H)$ is NP-complete.  So, it is unlikely that a very
precise structural description of ${\rm Forb}^*(H)$ exists.

Most successful and general attempts for proving Scott's conjecture
such as K\"uhn and Osthus~\cite{kuhnOsthus:04}, who prove it for
graphs with no $K_{s, s}$ or Scott~\cite{scott:tree}, who proves it
for trees, rely a lot on Ramsey Theory or probabilistic method.
These attempts cover many cases but give bad bounds on $\chi$,
typically mixtures of Ramsey number with towers of exponentials (and
worse).  It seems that the structural method gives better bounds but
applies only to several simple classes.  Section~\ref{sec:ISK4} gives
a partial structural description of ${\rm Forb}^*({\rm K_4})$ that
allow proving Scott's conjecture for $K_4$ (Theorem~\ref{th:ScottK4}).  The bound on $\chi$
follows the bound of K\"uhn and Osthus~\cite{kuhnOsthus:04} for $K_{3,
  3}$.  A more precise description of ${\rm Forb}^*({\rm K_4})$ that
gives a sharp bound for Scott's conjecture is still an open question.
Even the following question seems non-trivial to me.  The \emph{bull}\index{bull}
is the graph obtained from the paw by adding a pending vertex to a
vertex of degree two.

\begin{question}
  Give a structural description of ${\rm Forb}^*({\rm bull})$. 
\end{question}

\section{Triangle and $T_{123}$}
\label{sec:juraj}

Warning: the main result of this section is the 8-colorability of a
class of graphs.  Just before the deadline, I learned from a referee
report that Kierstead and Penrice~\cite{kiersteadPenrice:90} proved
4-colorability for the same class, so the results presented here are
not really new or interesting.  However I keep the section because of
the general purpose material presented.  I have no time to find
something more interesting illustrating this material.

Let $T_{123}$ be the tree on $a, b, c, d, e, f, g$ whose edges are
$ab$, $ac$, $cd$, $ae$, $ef$, $fg$.  We give here a decomposition
theorem for graphs with no triangle and no $T_{123}$ motivated by the
following.

\begin{conjecture}[Gy\'arf\'as \cite{gyarfas:perfect}]
 \label{conj:gyarfas}
 For every tree $T$, ${\rm Forb}(T)$ is $\chi$-bounded. 
\end{conjecture}

If the conjecture above is true, then it must be that for all trees
$T$, graphs with no triangle and no $T$ have a chromatic number
bounded by a constant.  In fact, for the particular tree $T_{123}$,
this is known to be true by the following.

\begin{theorem}[Scott \cite{scott:tree}]
  \label{th:scott}
  For every tree $T$, ${\rm Forb}^*(T)$ is $\chi$-bounded. 
\end{theorem}

Since $T_{123}$ has a single vertex of degree~3, it is easy to see
that ${\rm Forb}^*(T_{123}) = {\rm Forb}(T_{123})$, so ${\rm
  Forb}(T_{123})$ is $\chi$-bounded by Theorem~\ref{th:scott}.  But
Theorem~\ref{th:scott} gives bad bounds for $\chi$ and here we show
how a decomposition theorem can provide a better bound.  Note that in
the triangle-free case, Gy\'arf\'as' conjecture is known to be true
for paths with a good bound.

\begin{theorem}[Gy{\'a}rf{\'a}s, Szemer{\'e}di and Tuza \cite{gyarfasSzTuza:tree}]
  \label{th:gSzT}
  For all integers $k$, if $\chi(G)\geq k+1$, then $G$ contains a
  triangle or $P_k$.
\end{theorem}

So, if we want to bound the chromatic number in ${\rm Forb}(T_{123},
{\rm triangle})$, $P_8$-free graphs form a decent basic class.  On the
other hand, the following Lemma shows that long paths (length at
least~8) form a potentially interesting connected basic class for
${\rm Forb}(T_{123}, {\rm triangle})$.

\begin{lemma}[with Stacho \cite{nicolas.juraj:F}]
  \label{l:type}
  Let $G$ be a graph with no triangle and no $T_{123}$, and let
  $P=p_1\tp p_2\tp\cdots\tp p_k$ be a maximal path of $G$.  If $k\geq
  8$ then every vertex $v$ of $G \sm P$ is of one of the following
  types:
  \begin{description}
  \item[type {\sf 0}:] $N(v)\cap P = \emptyset$;
  \item[type {\sf 1}:] $N(v) \cap P= \{p_2\}$ or  $N(v)\cap P = \{p_{k-1}\}$;
  \item[type {\sf 2}:] $N(v) \cap P= \{p_i, p_{i+2}\}$ for some $1\leq i \leq k-2$;
  \item[type {\sf o}:] $N(v) \cap P= \{p_1, p_3, \dots, p_{2\lceil  k/2\rceil-1}\}$;
  \item[type {\sf e}:] $N(v) \cap P= \{p_2, p_4, \dots, p_{2\lfloor  k/2\rfloor}\}$;
  \item[type {\sf c}:]  $N(v) \cap P= \{p_1, p_k\}$.
  \end{description}
\end{lemma}

\begin{proof}
  Note that the assumption $k\geq 8$ is necessary.  Indeed, if $k=7$ and $N_P(v) =
  \{p_1, p_4, p_7\}$ then $v$ fails to be of one of the types.

  Suppose $v$ is not of type~{\sf 0}.  Let $i, j$ be respectively minimal and
  maximal indices such that $vp_i, vp_j \in E(G)$.

  \begin{claim}
    We may assume $j\geq i+3$.
  \end{claim}

  \begin{proofclaim}
    Suppose $j=i$.  If $i=1$ or $i=k$, we have a contradiction with
    the maximality of $P$.  If $i=2$ or $i=k-1$, then $v$ is of
    type~{\sf 1}.  In all other cases, we conclude that $\{p_i, v,
    p_{i-1}, p_{i-2}, p_{i+1}, p_{i+2}, p_{i+3}\}$ or $\{p_i, v,
    p_{i+1}, p_{i+2}, p_{i-1}, p_{i-2}, p_{i-3}\}$ induces a $T_{123}$ in
    $G$.  So, $j>i$.  If $j=i+1$, then there is a triangle.  If
    $j=i+2$ then $v$ is of type~{\sf 2}. Therefore, $j\geq i+3$.
  \end{proofclaim}

  \begin{claim}\label{c:il2}
    $i \leq 2$ and $j\geq k-1$.
  \end{claim}

  \begin{proofclaim}
    Suppose that $i\geq 3$. If $j>i+3$, then $\{p_i, p_{i+1}, p_{i-1},
    p_{i-2}, v, p_j, p_{j-1}\}$ induces a $T_{123}$ in $G$. So,
    $j=i+3$.  If $i>3$, then $\{p_i, p_{i+1}, v, p_j, p_{i-1},
    p_{i-2}, p_{i-3}\}$ induces a $T_{123}$ in $G$.  So, $i=3$ and
    $j=6$.  Now, since $k\geq 8$, it follows that $\{p_i, p_{i+1},
    p_{i-1}, p_{i-2}, v, p_j, p_{j+1}\}$ induces a $T_{123}$ in
    $G$. So, $i\leq 2$ holds and $j\geq k-1$ follows symmetrically.
  \end{proofclaim}

  Suppose now that $v$ is not of type~{\sf o} or~{\sf e}.  This
  implies that there exists an integer $s$ with $i \leq s \leq k-3$
  such that $vp_i, vp_{i+2}, \dots, vp_s \in E(G)$ and $vp_{s+2}
  \notin E(G)$.  Similarly, there exists an integer $t$ with $j \geq t
  \geq 4$ such that $vp_j, vp_{j-2}, \dots, vp_t \in E(G)$ and
  $vp_{t-2} \not\in E(G)$.

  \begin{claim}\mbox{}
    $t-s \geq 4$.
  \end{claim}

  \begin{proofclaim}
    From the definition of $s$ and $t$, we can conclude that $t-s \geq
    3$. Suppose that $t-s=3$.  Since $k\geq 8$, we may assume (up to
    the symmetry between $s$ and~$t$) that $s\geq 3$.  We put $Q =
    p_{s-1} \tp p_{s-2} \tp v \tp p_{s+3} \tp p_{s+2} \tp p_{s+1}$ and
    see that $Q$ is an induced path of $G$.  If $s+3 \neq k$, then
    $Q\cup \{p_{s+4}\}$ induces a $T_{123}$ in $G$.  So, $s+3=k$, which
    implies $s\geq 5$ because $k\geq 8$.  But then $Q\cup \{p_{s-4}\}$
    induces a $T_{123}$ in $G$, a contradiction.
  \end{proofclaim}

  \begin{claim}\mbox{}
    $t-s \geq 5$.
  \end{claim}

  \begin{proofclaim}
    Suppose that $t-s=4$.  Since $k\geq 8$, we may assume (up to the
    symmetry between $s$ and $t$) that $s\geq 3$.  If $s+4 \neq k$,
    then $\{v, p_{s-2}, p_{s+4}, p_{s+5}, p_{s}, p_{s+1}, p_{s+2}\}$
    induces a $T_{123}$ in $G$.  So, $s+4=k$, which implies $s\geq 4$
    because $k\geq 8$.  Hence, $\{v, p_{s+4}, p_{s-2}, p_{s-3}, p_{s},
    p_{s+1}, p_{s+2}\}$ induces a $T_{123}$ in $G$.
  \end{proofclaim}

  If $s\geq 2$, then $\{p_s, p_{s-1}, p_{s+1}, p_{s+2}, v, p_t,
  p_{t-1}\}$ induces a $T_{123}$ in $G$ because $t-s\geq 5$.  So, $s=1$,
  which implies $i=s=1$ and by symmetry, $j=t=k$.  If $v$ has a
  neighbour $p_r$ such that $1<r<k$, then $r\geq 4$ and $r\leq k-3$
  (from $s=1$ and $t=k$).  Also, since $k\geq 8$, we may assume (up to
  symmetry) that $r\geq 5$. Thus $\{v, p_r, p_k, p_{k-1}, p_1, p_2,
  p_3\}$ induces a $T_{123}$ in $G$.  It follows that $v$ has no neighbour
  $p_r$ where $1<r<k$, and hence, $v$ is of type~{\sf c}.  
\end{proof}

The lemma above is what I call an \emph{attachment}\index{attachment!lemma} lemma: it says how
a vertex (or sometimes a subgraph) must attach to a decently connected
substructure in a graph of a given class.  Structural Graph Theory is
full of these attachment lemmas, and some are quite long to prove.
Most of the time (if not all the time) attachment lemmas are proved by
long case by case analysis.

I would like to define an informal notion of
\emph{deepness}\index{deepness of an attachment lemma} of an
attachment lemma (or more generally of a decomposition theorem).  If a
careful reader really checks all the cases in the proof above, it is
likely that (s)he ends up with 10 pictures, each showing a path
together with a vertex that attaches to it in a special way that
entails a $T_{123}$.  Moreover, the proof was designed carefully
enough so that the 10 pictures are pairwise non-isomorphic: on each
picture, the tree appears in a different fashion.  So, I would be very
surprised if one comes up with a significantly shorter proof of
Lemma~\ref{l:type}.  It might be that some unexpected clever argument
solves Gy\'arf\'as conjecture with a short non-structural proof.  This
kind of things sometimes happen in discrete mathematics, for example
Lov\'asz could compute the chromatic number of Kneser graphs thanks to
seemingly unrelated theorems from topology,
see~\cite{matousek:borsuk}.  But it would be strange to prove the
little Lemma~\ref{l:type} without considering the 10 little cases.
So, I suggest to say that Lemma~\ref{l:type} has \emph{deepness 10},
because when proving it, one needs to consider 10 pairwise
non-isomorphic configurations.  Of course, a big deepness may mean a
badly written proof (that considers too many cases), but the fact that
all the cases are pairwise non-isomorphic means that a kind of a shortest
proof has been found.

\begin{question}
  Define formally what is an attachment lemma and the notion of
  deepness.
\end{question}

Another interesting aspect of attachment lemmas is that they are in a
way easy, some would say tedious and not interesting.  In front of
them, one feels that no ``idea'' is needed to prove them, that it is
just a matter of work and techniques to decide whether a given
attachment statement is true or false.  This is good in way: the art
of the mathematician is to convert difficult questions into checkable
computations.  Anyway, this leads to the following question.

\begin{question}
  Once a formal notion of attachment lemma is defined, try to see
  whether there exists an algorithm that decides in finite time whether
  a given attachment statement is true (by giving a proof) or false
  (by outputting a counter-example).
\end{question}

Of course, there are many ways to formalize what is an attachment
lemma.  A model which allows very general statements is likely to
produce some whose truth is undecidable, since from logic (G\"odel's
Theorem\dots), it is known that no algorithm can decide the truth of
too general mathematical statements.  On the other hand, a model that
restricts too much what we mean by attachment lemmas is likely to
produce only trivially decidable statements.  A useful answer to the
question above must be something in between.  The idea behind all this
is to have a theoretical basis for computer assisted proofs in
structural graph theory: give the most tedious part of the job to
computers and keep for us the interesting part.  This would be nice
because the field suffers from too long case analyses.

A partition $(X_1, X_2)$ of the vertex set of a graph $G$ is a
\emph{2-join}, if there exist disjoint subsets $A_1, B_1$ of $X_1$ and
disjoint subsets $A_2, B_2$ of $X_2$ such that every vertex of $A_1$
is adjacent to every vertex of $A_2$, every vertex of $B_1$ is
adjacent to every vertex of $B_2$ and there are no other edges
between $X_1$ and $X_2$.  The 2-join was first defined by Cornu\'ejols
and Cunningham~\cite{cornuejols.cunningham:2join}.

A partition $(X_1,X_2)$ of the vertex set of $G$ is a {\em bipartite
  2-join}\index{2-join!bipartite} of $G$ if $(X_1,X_2)$ is a 2-join
and, in addition, it satisfies the following conditions:
\begin{enumerate}
\item $A_1 \cup B_1 \neq \emptyset$;
\item $A_2 \neq \emptyset$, $B_2 \neq \emptyset$ and $X_2 = A_2 \cup B_2$;
\item $|X_2| \geq 3$;
\item $A_2$ and $B_2$ are independent sets;
\item there exists at least one edge in $G[X_2]$.
\end{enumerate}

Let $x$ be a vertex of $G$ and $H$ a graph vertex-disjoint from~$G$.
The graph obtained from $G$ by {\em replacing $x$ with $H$} has
vertices $V(G\sm x)\cup V(H)$ and edges $E(H)\cup E(G\sm x)\cup
\{uv~|~u\in N(x),~v\in V(H)\}$.  An {\em extension of $G$} is any
graph we obtain from $G$ by replacing each vertex of $G$ with an
independent set (possibly empty).  Here is the decomposition theorem
of this section.

\begin{theorem}[with Stacho \cite{nicolas.juraj:F}]
  \label{th:stacho}
  Every connected graph $G$ with no triangle and no $T_{123}$
  satisfies at least one of the following:
  \begin{enumerate}
  \item $G$ is $P_8$-free;
  \item $G$ is an extension of a cycle or a path;
  \item $G$ admits a bipartite 2-join.
  \end{enumerate}
\end{theorem}

\begin{proof}
  we may assume that $G$ contains an induced path $P = p_1 \tp \cdots
  \tp p_k$ where $k\geq 8$ and we choose $P$ to be an inclusion-wise
  maximal induced path of $G$.

  We say that a set $W\subseteq V(G)$ is a {\em $P$-template}, if
  there exist pairwise disjoint sets of vertices $X_1, \dots, X_k$,
  $X_0$, $A$, $B$ whose union is $W$ and such that:

  \begin{enumerate}
  \item $p_i \in X_i$, for each $i=1, \dots, k$;
  \item $X_i$ is complete to $X_{i+1}$, for each $i=1, \dots, k-1$;
  \item $X_0$ is complete to $X_1$ and $X_k$;
  \item if $k$ is even and $X_0 \neq \emptyset$, then the sets $A$ and
    $B$ are empty;
  \item if $k$ is odd or $X_0 = \emptyset$, then $A$ is complete to
    $X_1, X_3, \dots X_{2\lceil  k/2\rceil-1}$ and $B$ is complete
    to $X_0, X_2, X_4, \dots, X_{2\lfloor  k/2\rfloor}$;
  \item each of the sets $X_1,\ldots,X_k,X_0,A,B$ is an independent
    set; 
  \item there are no other edges in $G[W]$.
  \end{enumerate}

  Let $W$ be an inclusion-wise maximal $P$-template of $G$.  We say
  that a path $Q=q_1\tp q_2 \tp \ldots\tp q_k$ is a {\em snake}, if
  the vertices of $Q$ are consecutively in $X_i$, $X_{i+1}$, \dots,
  $X_{i+k-1}$ for some $i$ where the indices are taken modulo
  $k+1$.\bigskip

  \noindent{\bf Case 1:} $k$ is even and $X_0 \neq \emptyset$. Then,
  by the definition of $W$, the sets $A$ and $B$ are empty.  We claim
  that $G=G[W]$.  For otherwise, pick a vertex $v$ in $G\sm W$
  with at least one neighbour in $W$.  So, $v$ has neighbors in at
  least one snake $Q = q_1 \tp \cdots \tp q_k$ of $W$ and we choose
  such a snake containing as many neighbours of $v$ as possible.

  If $v$ is of type~{\sf 1} with respect to $Q$, say $vq_2 \in E(G)$,
  then there is a snake $Q' = q_0 \tp q_1 \tp q_2 \tp \cdots \tp
  q_{k-1}$ where $q_0\in X_0$ such that $v$ and $Q'$ contradict
  Lemma~\ref{l:type} or the choice of $Q$.  If $v$ is of type~{\sf 2}
  with respect to $Q$, say $vq_3, vq_5 \in E(G)$, then up to
  relabeling we may assume $q_3\in X_3$ and $q_5 \in X_5$, and we
  claim that $v$ is complete to all vertices of $X_3\cup X_5$ and that
  there are no other edges between $v$ and $W$. Indeed, this is easy
  to see, since otherwise we again contradict the choice of $Q$ or
  Lemma~\ref{l:type} for some snake $Q'$.  It follows that $W\cup
  \{v\}$ contradicts the maximality of $W$.

  Next, if $v$ is of type~{\sf o} or~{\sf e}, then up to relabeling,
  we may assume $vq_1, vq_3, \dots, vq_{k-1} \in E(G)$.  Since $k$ is
  even, we obtain a contradiction to Lemma~\ref{l:type} applied to the
  path $q_3 \tp \cdots q_k \tp q_0 \tp q_1$ where $q_0\in X_0$.  This
  proves that $G=G[W]$, and hence, $G$ is an extension of an odd
  cycle.\bigskip

  \noindent{\bf Case 2:} $k$ is odd or $X_0=\emptyset$. Then we claim
  that each vertex $v$ of $G\sm W$ is such that $N(v)\cap W \subseteq
  A\cup B$.  Otherwise, there exists a vertex $v$ of $G\sm W$ whose
  neighbourhood contains some vertex of $X_0 \cup X_1 \cup \cdots \cup
  X_k$.  Again, $v$ has neighbours in at least one snake $Q = q_1 \tp
  \cdots \tp q_k$ of $W$, and we choose a snake containing as many
  neighbours of $v$ as possible.  Now, depending on the type of $v$
  with respect to $Q$, we again contradict either the choice of $Q$,
  or Lemma~\ref{l:type}, or the maximality of $W$.

  It remains to observe that if $A\cup B$ is empty, then we have $G =
  G[W]$, and hence, $G$ is an extension of a path or a cycle.
  Otherwise, we conclude that $(V(G) \sm (X_0 \cup \cdots \cup X_k),
  X_0 \cup \cdots \cup X_k)$ is a bipartite 2-join of $G$.
\end{proof}

There is a difference between Theorem~\ref{th:stacho} and
Theorem~\ref{th:paw}.  Theorem~\ref{th:paw} is in a sense canonical:
it really describes fully the class ${\rm Forb}^*({\rm paw})$ and it
would be surprising that a more precise description exists.
Theorem~\ref{th:stacho} was designed on purpose to bound the chromatic
number, and more precise theorems about the structure of ${\rm
  Forb}^*({\rm triangle}, F_{123})$ should exist (maybe $\chi \leq 7$
is provable by some structural description?).  Finally we obtain the
main theorem of this section.

\begin{theorem}[with Stacho \cite{nicolas.juraj:F}]
  \label{th:8colors}
  Every triangle-free graph $G$ with no $T_{123}$ is 8-colourable.
\end{theorem}

\begin{proof}
  The proof is by induction on the size of $V(G)$.  Let $G$ be a
  triangle-free graph with no induced $T_{123}$. We may assume that
  $G$ is connected.

  If $G$ is not 8-colourable, then by Theorem \ref{th:stacho} and Theorem
  \ref{th:gSzT}, we conclude that $G$ is either an extension of $C_k$
  or $P_k$, or admits a bipartite 2-join.  In the former case, $G$ is
  clearly 3-colourable since $C_k$ and $P_k$ are.  In the latter case,
  we have a bipartite 2-join $(X_1, X_2)$ with notation like in the
  definition.  In particular, $G[X_2]$ is a bipartite graph with
  partite sets $A_2, B_2$ and there exists an edge $xy\in E(G)$ where
  $x\in A_2$ and $y\in B_2$.

  We now consider the graph $G'=G[X_1\cup\{x,y\}]$. Clearly, $G'$
  contains no triangle or induced $T_{123}$ because $G'$ is an induced
  subgraph of $G$.  Also, $|X_2|\geq 3$, and hence, $V(G')<V(G)$.
  Therefore, by the inductive hypothesis, $G'$ is 8-colourable.  It
  remains to observe that we can extend to $G$ any 8-colouring of $G'$
  by assigning to the vertices of $A_2$ the colour of $x$ and by
  assigning to the vertices of $B_2$ the colour of $y$.
\end{proof}

\section{Chordless graphs}
\label{sec:chordless}

A graph is \emph{chordless}\index{chordless graph} if all its cycles
are chordless.  Chordless graphs are clearly closed under taking
induced subgraphs.  A graph $G$ is \emph{sparse}\index{sparse graph}
if for all pairs $u, v$ of vertices of degree at least~3, $uv\notin
E(G)$.  It is clear that all sparse graphs are chordless, because the
ends of a chord of a cycle must be adjacent and of degree at least~3.

\begin{exercise}
  Find a non-sparse chordless 2-connected graph. 
\end{exercise}

On the one hand, sparse graphs are not very ``basic'': one can obtain a
sparse graph by subdividing all edges of any graph, so sparse graphs
are in a sense as complex as general graphs.  

On the other hand, sparse graphs are very ``basic''.  Because of their
``generality'', I believe that they cannot be decomposed further in
some interesting way.  Indeed, any deep structural statement about
sparse graphs, when applied to subdivisions of general graphs, is
likely to say something about all graphs, and so must be really
extremely deep\dots Also, sparse graphs are ``simple'' in some
respects: they are easy to recognize and are all 3-colorable (3-color
all components isomorphic to a cycle, give color 1 to vertices of
degree at least 3, the remaining vertices induce a collection of
disjoint paths, which is 2-colorable).  From this, it can be seen that
sparse graphs form an essential basic class of chordless graphs,
because they capture both ``simple'' and ``complex'' aspects of them.
Thanks to sparse graphs, we can see for instance that the isomorphism
problem will certainly be difficult for chordless graphs, but that
there is a hope to color them.

For any integer $k\ge 0$, a
\emph{$k$-cutset}\index{cutset!$k$-cutset}\index{K@$k$-cutset} in a
graph is a subset $S\subseteq V(G)$ of size $k$ such that $G\sm S$ is
disconnected.  Note that a graph is disconnected if and only if it
admits a 0-cutset.  A $2$-cutset $\{a, b\}$ of a graph $G$ is
\emph{proper}\index{proper!2-cutset}\index{2-cutset!proper} if $a\neq
b$ and:

\begin{itemize}
\item $ab\notin E(G)$;
\item $V(G)\sm \{a,b \}$ can be partitioned into non-empty sets $X$
  and $Y$ so that there is no edge between $X$ and $Y$;
\item each of $G[X \cup \{ a,b \}]$ and $G[Y \cup \{ a,b \}]$ is not
  an $(a,b)$-path.
\end{itemize}

Most of the proof of the following is implicitly given
in~\cite{nicolas.kristina:one} (proof of Theorem~2.2 and Claims~12
and~13 in the proof of Theorem~2.4).  But the result is not stated
explicitly in~\cite{nicolas.kristina:one} and several details differ.

\begin{theorem}[with L\'ev\^eque and Maffray \cite{nicolas:isk4}]
  \label{th:nochord}
  Let $G$ be a chordless graph.  Then either $G$ is sparse or $G$
  admits a $1$-cutset or a proper $2$-cutset.
\end{theorem}

\begin{proof}
  Let us assume that $G$ has no $1$-cutset and no proper $2$-cutset.
  Note that $G$ contains no $K_4$, since a $K_4$ is a cycle with two
  chords.  Moreover,
 
 \begin{claim}
    \label{c:triangle}
    We may assume that $G$ is triangle-free. 
  \end{claim}
  
  \begin{proofclaim}
    For suppose that $G$ contains a triangle $T$.  Then $T$ is a
    maximal clique of $G$ since $G$ contains no $K_4$.  We may assume
    that $G \neq T$ because a triangle is sparse, and that $G$ is
    connected, for otherwise every vertex is a $1$-cutset.  So some
    vertex $a$ of $T$ has a neighbor $x$ in $G\setminus T$.  Since $a$
    is not a $1$-cutset of $G$, there exists a shortest path $P$
    between $x$ and a member $b$ of $T\sm a$.  But then $P\cup T$ is a
    cycle with at least one chord (namely $ab$), a contradiction.
  \end{proofclaim}

  \begin{claim}
    \label{c:cliqueCutset}
    We may assume that $G$ has no clique cutset.
  \end{claim}
  
  \begin{proofclaim}
    Suppose that $K$ is a clique cutset in $G$.  Since $G$ has no
    cutset of size one and there is no clique of size at least three
    by~(\ref{c:triangle}), $K$ has exactly two elements $a$ and $b$.
    Let $X$ and $Y$ be two components of $G\sm \{a, b\}$.  Since none
    of $a$ and $b$ is a 1-cutset of $G$, $X\cup\{a, b\}$ contains a
    path $P_X$ with endvertices $a$ and $b$; and a similar path $P_Y$
    exists in $Y\cup\{a,b\}$.  But then $P_X\cup P_Y$ forms a cycle
    with at least one chord (namely $ab$), a contradiction.
  \end{proofclaim}

  We can now prove that $G$ is sparse.  Suppose on the contrary that
  $G$ has an edge $ab$ with $\deg(a) \geq 3$ and $\deg(b) \geq 3$.
  Let $c, e$ be two neighbors of $a$ different from $b$, and let $d,
  f$ be two neighbors of $b$ different from $a$.  Note that $\{c, e\}$
  and $\{d, f\}$ are disjoint by~(\ref{c:triangle}).
  By~(\ref{c:cliqueCutset}), $\{a, b\}$ is not a cutset, so there is
  in $G\setminus\{a, b\}$ a path between $\{c, e\}$ and $\{d, f\}$ and
  consequently a path $P$ that contains exactly one of $c, e$ and one
  of $d, f$.  Let the endvertices of $P$ be $e$ and $f$ say.  Thus $P
  \cup\{a, b\}$ forms a cycle $C$.  Since $G\setminus\{a, b\}$ is
  connected, there exists a path $Q = c \cdots u$, where $u \in P\cup
  \{d\}$ and no interior vertex of $Q$ is in $C \cup \{d\}$.  But
  $u = d$ implies that $Q\cup C$ forms a cycle with at least
  one chord, namely $ab$, so $u\in P$.  Also since $G \setminus \{a,
  b\}$ is connected, there exists a path $R = d \cdots v$ where $v \in
  P \cup Q$ and no interior vertex of $R$ is in $C \cup Q$.

  If $v$ is in $Q \sm u$, then $bdRvQcaePfb$ is a cycle with at least
  one chord, namely $ab$, a contradiction.  So $v$ is in $P$.  If $e,
  v, u, f$ lie in this order on $P$ and $v\neq u$, then
  $bdRvPeacQuPfb$ is a cycle with at least one chord, namely $ab$, a
  contradiction.  So $e, u, v, f$ lie in this order on $P$ (possibly
  $u=v$).  This restores the symmetry between $c$ and $e$ and between
  $d$ and $f$.  We suppose from here on that the paths $P, Q, R$ are
  chosen subject to the minimality of the length of $uPv$.

  Let $P_e=ePu\sm u$, $Q_c=cQu\sm u$, and $P_b = bPu\sm u$.  We now
  show that $\{a, u\}$ is a proper $2$-cutset of $G$.  Suppose not.
  So there is a path $D=x \cdots y$ in $G \setminus \{a, u\}$ such
  that $x$ lies in $P_e \cup Q_c$, $y$ lies in $P_b \cup R$, and no
  interior vertex of $D$ lies in $P\cup \{a\}\cup Q\cup R$.  Up to
  symmetry we may assume that $x$ is in $Q_c$.  If $y$ is in the
  subpath $uPv$, then, considering path $Q'=cQxDy$, we see that the
  three paths $P, Q', R$ contradict the choice of $P, Q, R$ because
  $y$ and $v$ are closer to each other than $u$ and $v$ along $P$.  So
  $y$ is not in $uPv$, and so, up to symmetry, $y$ is in $R \setminus
  \{ v\}$.  But then $x Q c a e P f b d R y D x$ is a cycle with at least
  one chord (namely $ab$), a contradiction.  This proves that we can
  partition $G\sm\{a, u\}$ into a set $X$ that contains $P_e\cup Q_c$
  and a set $Y$ that contains $P_b\cup R$ such that there is no edge
  between $X$ and $Y$, so $\{a, u\}$ is a 2-cutset.  So,
  by~(\ref{c:cliqueCutset}), $a$ and $u$ are not adjacent. This
  implies that $\{a, u\}$ is proper.
\end{proof}

The theorem above will be generalized in Section~\ref{sec:oneChord},
see Theorem~\ref{th:OneChord}.  It has interesting consequences.  It
shows how a non-sparse chordless graph can be broken into pieces, and
this allows to prove theorems by induction (see for example
Theorem~\ref{th:3cChordless} below).  Usually, to have a usable
induction hypothesis, one needs to apply it to more than just the
components of what remains once the cutset is removed.  To see this
let us define the \emph{blocks of decomposition}\index{blocks of
  decomposition!w.r.t.\ 2-cutset} of a graph $G$ with respect to a
proper 2-cutset and two sets $X$ and $Y$ like in the definition.  They
are $G_X = G[X \cup \{a, b, x\}]$ and $G_Y = G[Y \cup \{a, b, y\}]$
where $x$ and $y$ are a new vertices adjacent to $a$ and $b$ only.  To
see that in some situations \emph{we need} $x$ and $y$, one can try to
prove that a graph with no 1-cutset and with a proper 2-cutset is
chordless if and only if both its blocks are.  With our definition of
blocks, this works but it fails if we omit $x$ and $y$ in the
definition.

\begin{exercise}
  Write an algorithm that decides in polynomial time whether a graph is
  chordless. 
\end{exercise}

\begin{question}
  Is there a linear-time algorithm that recognizes chordless graphs?
  Most of the ideas might be in~\cite{hopcroft.tarjan:3con}.
\end{question}

Theorem~\ref{th:nochord} is a \emph{structure theorem}: it tells how
to build all chordless graphs from basic pieces by gluing them
together along three operations.  All the operations below take two
vertex-disjoint graphs $G_1$ and $G_2$ and output a third graph $G$.

\begin{description}
\item[Operation 1:] put $G = G_1 \cup G_2$.
\item[Operation 2:] choose  $v_1 \in V(G_1)$ and $v_2 \in V(G_2)$.
  Put:
  \begin{itemize}
  \item $V(G) = V(G_1) \cup V(G_2) \sm \{v_2\}$.
  \item $E(G) = E(G_1) \cup E(G_2 \sm v_2) \cup \{v_1v \suchthat
    v\in N_{G_2}(v_2)\}$.
  \end{itemize}
\item[Operation 3:] choose $v_1\in V(G_1)$ and $v_2 \in V(G_2)$ both
  of degree~2 and with 2 non-adjacent neighbors.  Let
  $N_{G_1}(v_1) = \{v'_1, v''_1\}$ and $N_{G_2}(v_2) = \{v'_2,
  v''_2\}$.  Put:
  \begin{itemize}
  \item $V(G) = (V(G_1) \sm \{v_1\}) \cup (V(G_2) \sm \{v_2, v'_2,
    v''_2\})$.
  \item $E(G) = E(G_1 \sm \{v_1\}) \cup E(G_2 \sm \{v_2, v'_2,
    v''_2\}) \cup \{v'_1v \suchthat v\in N_{G_2\sm v_2}(v'_2)\} \cup
    \{v''_1v \suchthat v\in N_{G_2\sm v_2}(v''_2)\}$.
  \end{itemize}
\end{description}  

It is a routine matter to check that when applied to chordless graphs,
Operations 1, 2 and 3 output chordless graphs.  Also, from
Theorem~\ref{th:nochord} we see that if a graph is chordless and not
sparse, it can be obtained from smaller chordless graphs by one of the
operations (Operation 1 corresponds to these ``degenerate'' 1-cutset
that occur in disconnected graphs, Operation 2 corresponds to
1-cutsets, and Operation~3 to 2-cutsets).  So, in fact,
Theorem~\ref{th:nochord} is a structure theorem: it gives a recipe to
build all chordless graphs from sparse graphs by gluing them.  We will
present more complex structure theorems: Theorems~\ref{th:OneChord},
\ref{th:ISK4Wheel} and \ref{th.3}.

Also, Theorem~\ref{th:nochord} implies $\chi$-boundness. The following
theorem is a consequence of a more general one,
Theorem~\ref{th:color1chord}, but here we give a direct proof.  It is
not very difficult, but still one needs to be careful.  A direct
attempt to color recursively along 2-cutset decomposition is likely to
fail because when a 2-cutset exists, assuming by induction that the
two blocks are 3-colorable is in general not enough to reconstruct a 3
coloring of the whole graph.  To prove the 3-colorability of chordless
graphs, some fine tuning on the induction hypothesis has to be done at
some step.  These fine tunings are part of the fun of graph theory
(unless they drive you mad\dots).

\begin{theorem}
  \label{th:3cChordless}
  Every chordless graph is 3-colorable. 
\end{theorem}

\begin{proof}
  We first need the following.

  \begin{claim}
    \label{c:42deg2}
    A chordless graph on at least four vertices and that is not a
    square contains three vertices $u, u', u''$ of degree at most~2
    such that $\{u, u', u''\}$ does not induce a $P_3$.
  \end{claim}

  \begin{proofclaim}
    Note that we really need the graph to be on at least four vertices
    because of $P_3$.  Let $G$ be chordless on at least four vertices,
    not a square.  By Theorem~\ref{th:nochord} there are three cases.

    \noindent{\bf Case 1:} $G$ is sparse.  If $G$ has no vertex of
    degree at least~3, then it is a disjoint union of paths and
    cycles.  Since $G$ is not a square, there exist three vertices in
    $G$ that do not induce a $P_3$.  If $G$ has a vertex of degree at
    least~3, then its neighborhood is made of at least three vertices
    of degree at most~2 that do not induce a $P_3$ for else there is a
    cycle with one chord.

    \noindent{\bf Case 2:} $G$ has a 1-cutset $\{v\}$.  So, $V(G)$ can
    be partitioned into three nonempty sets $X, \{v\}, Y$ and there
    are no edges between $X$ and $Y$.  Up to symmetry we assume $|X|
    \geq 2$.  We claim that $X$ contains two vertices $u, u'$ of
    degree at most~2 of $G$.  This is obvious if $|X| = 2$ and if
    $|X|\geq 3$ it follows from the induction hypothesis applied to
    $G[X \cup \{v\}]$.  Similarly, $Y$ contains one vertex $u''$ of
    degree at most~2 in $G$.  Since there are no edges between $X$ and
    $Y$, we see that $\{u, u', u''\}$ does not induce a $P_3$.

    \noindent{\bf Case 3:} $G$ has no $1$-cutset but has a proper
    2-cutset $\{a, b\}$.  Then let $X, Y$ be sets like in the
    definition of a proper 2-cutset.  We suppose that $\{a, b\}$ has
    been chosen subject to the minimality of $X$.  Let $G_X$ and
    $G_Y$ be the block of decomposition of $G$ with respect to $\{a,
    b\}$, $X$ and $Y$.  Vertices $a, b$ have degree at least~2 in
    $G_X$ and $G_Y$ because else $G$ has a 1-cutset.  Graphs $G_X$ and
    $G_Y$ are chordless, on at least four vertices and are not squares
    because $\{a, b\}$ is proper.  They are both smaller than $G$.
    Vertices $a$ and $b$ have degree at least~3 in $G_X$ because else,
    if $a$ say has degree~2, then it has a unique neighbor $a'$ in $X$
    and $\{a', b\}$ is a proper 2-cutset that contradicts the choice
    of $\{a, b\}$ (minimality of $X$).  Hence, from the induction
    hypothesis applied to $G_X$, there are at least two vertices $u,
    u'$ of degree at most~2 in $X$.  From the induction hypothesis
    applied to $G_Y$, there is at least one vertex $u''$ of degree at
    most~2 in $Y$ (because $\{a, y, b\}$ induces a $P_3$).  Since
    there are no edges between $X$ and $Y$, we see that $\{u, u',
    u''\}$ does not induce a $P_3$ in $G$.
  \end{proofclaim}

  Now we see that any chordless graph $G$ contains a vertex $v$ of
  degree at most~2.  This is obvious when $|V(G)| \leq 3$ or when $G$
  is a square, and it follows from~(\ref{c:42deg2}) otherwise.  So
  3-colorability is obtained by induction: if $G\sm v$ is 3-colorable,
  then so is $G$ because at least one color is available for~$v$.
\end{proof}

There is an idea in the proof above that is often helpful:
\emph{minimally sided decompositions}\index{minimally sided}.  The
notion is self-explanatory, we use it when we choose a 2-cutset $\{a,
b\}$ subject to the minimality of $X$.  In some situations, this leads
to an \emph{extreme decomposition}\index{extreme!decomposition}, that
is a decomposition such that one of the block of decomposition is
basic.  It is not always the case, because once we have a
decomposition, we can always choose it minimally sided, while
sometimes extreme decomposition simply do not exist, examples are
given on Figures~\ref{fige2j} and~\ref{f:noextr}.  Extreme
decompositions are useful when doing combinatorial optimization.
Suppose for instance that to solve a problem, one decomposes a graph
into 2 blocks and has to ask 2 questions for one of the blocks.  This
will lead in general to an exponential number of questions.  But if
the decomposition is extreme, one of the block is basic, so it is
likely that the problem can be solved directly for this block.  Then
there is a hope to obtain a polynomial complexity.  In
Section~\ref{sec:Lovasz}, this idea is applied successfully to color
Berge graphs with no balanced skew partition and no homogeneous pair.

\section{Line-graphs of triangle-free graphs}

The \emph{line-graph}\index{line-graph} of a graph $G=(V, E)$ is the
graph $L(G)$ whose vertex-set if $E(G)$, and such that two vertices of
$L(G)$ are adjacent if and only if the corresponding edges of $G$
share a common end.  In many decomposition theorems, one basic class
is a class of line-graphs; the rest of this work provides several
examples.  This phenomenon is difficult to explain at this point, so
let us wait to the end of this section where more data will be
available for an explanation.

A theorem of Beineke~\cite{beineke:linegraphs} characterizes
line-graphs by a list of nine forbidden induced subgraphs.  We give
here a simpler theorem of Harary and
Holzmann~\cite{harary.holzmann:lgbip} on line-graph of triangle-free
graphs.  Proving this theorem is a good exercise to understand what is
a line-graph, see~\cite{nicolas:these} for a solution.  Here, we show
only how the theorem of Harary and Holzmann follows easily from the
more general theorem of Beineke (note that
\cite{harary.holzmann:lgbip} is very difficult to find).

\begin{figure}[ht]
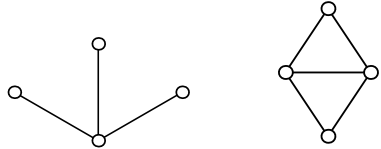

  \center
  {\includegraphics[height=1.5cm]{figHdr.1}}\hspace{1cm}
  \includegraphics[height=2cm]{figHdr.2}
  \caption{Claw and diamond\label{fig:clawDiamond}}
\end{figure}

The \emph{claw}\index{claw} and the \emph{diamond}\index{diamond} are the two graphs represented
Figure~\ref{fig:clawDiamond}.

\begin{theorem}[Harary and Holzmann \cite{harary.holzmann:lgbip}]
  \label{th:HH}
  A graph $G$ is the line-graph of a triangle-free graph if and only
  if $G$ contains no claw and no diamond.
\end{theorem}

\begin{proof}
  Checking that a line-graph of a triangle-free graph contains no claw
  and no diamond is a routine matter.  Conversely, let $G$ be graph
  with no claw and no diamond.  The theorem of
  Beineke~\cite{beineke:linegraphs} states that there exists a list
  $\cal L$ of nine graphs such any graph that does not contain a graph
  from $\cal L$ is a line-graph.  One of the nine graphs is the claw
  and the eight remaining ones all contain a diamond.  So, $G = L(R)$
  for some graph $R$.  Let $R'$ be the graph obtained from $R$ by
  replacing each connected component of $R$ isomorphic to a triangle
  by a claw.  So, $L(R)=L(R') = G$.  We claim that $R'$ is
  triangle-free.  Else let $T$ be a triangle of $R'$.  From the
  construction of $R'$, there is a vertex $v\notin T$ in the connected
  component of $R'$ that contains $T$.  So we may choose $v$ with a
  neighbor in $T$.  Now the edges of $T$ and one edge from $v$ to $T$
  induce a diamond of $G$, a contradiction.
\end{proof}

The Harary and Holzmann's theorem is in a sense a decomposition
theorem for graphs with no claw and no diamond.  It states that all
these graphs are basic.  It is likely that no much deeper structural
description can be found, because Whitney proved that except for the
claw and the triangle (which have the same line-graph), two different
graphs have different line-graphs.  This yields the following corollary: 

\begin{theorem}[Whitney \cite{whitney:graphs}]
  If $G$ is a line-graph of a triangle-free graph then there exists a
  unique graph $R$ such that $G = L(R)$.
\end{theorem}

So, being able to decompose line-graphs of triangle-free graphs would
certainly imply a structural description of all triangle free graphs.
Such a description is a well known difficult open question linked to
difficult Ramsey Theory as suggested in Chapter~\ref{chap:gap}.  More
generally, a statement on all line-graphs is likely to be translatable
into a statement on all graphs.  So, line-graphs behave a bit like
sparse graphs of Section~\ref{sec:chordless}.  Strangely, their
generality make them basic, and once line-graphs are identified as a
substantial subclass of a class, they are likely to form a basic class.
This phenomenon occurs in the perhaps two deepest known decomposition
theorems: the decomposition of Berge graphs by Chudnovsky, Robertson
Seymour and Thomas and the decomposition of claw-free graphs by
Chudnovsky and Seymour, see the next section.

\section{Claw and coclaw}

Chudnovsky and Seymour proved a deep structural description of
claw-free graphs~\cite{DBLP:conf/bcc/ChudnovskyS05}.  The statement is
very long and we will not deal with this result in the rest of this
work.  Here we focus on something much simpler: graphs that are
claw-free and whose complements are claw-free.  As we will see, these
graphs form a very restricted class.

The \emph{coclaw}\index{coclaw} is the complement of the claw.  The line-graph of
$K_{3, 3}$ is represented on Figure~\ref{base.fig.lk33}.  We denote by
$A_6$ the graph obtained from a triangle by adding for each edge of
the triangle, a vertex complete to that edge.  The following theorem
is due to Pouzet and Si Kaddour but the proof is mine.

\begin{theorem}[Pouzet and Si Kaddour]
  The class ${\rm Forb}\{{\rm claw},$ ${\rm co\-claw}\}$ consists of
  $A_6$, of the induced subgraphs of $L(K_{3, 3})$, of graphs whose
  connected components consist of cycles of length at least~4 and
  paths, and of the complements of all these graphs.
\end{theorem}

\begin{proof}
  Let $G$ be in the class. 

  \begin{claim}
    \label{c:conn}
    We may assume that $G$ and $\overline{G}$ are connected.
  \end{claim}

  \begin{proofclaim}
    Else, up to symmetry, $G$ is disconnected.  If $G$ contains a
    vertex $v$ of degree at least $3$, then $N(v)$ contains an edge
    (for otherwise there is a claw), so $G$ contains a triangle.  This
    is a contradiction since by picking a vertex in another component
    we obtain a coclaw.  So all vertices of $G$ are of degree at
    most~2.  It follows that the components of $G$ are cycles (of
    length at least 4, or there is a coclaw) or paths, an outcome of
    the theorem.
  \end{proofclaim}

  \begin{claim}
    \label{c:P6}
    We may assume that $G$ contains no induced $P_6$.
  \end{claim}  

  \begin{proofclaim}
    Else $G$ has an induced subgraph $H$ that is either a path on at
    least 6 vertices or a cycle on at least 7 vertices.  Assume $H$
    maximal with respect to this property.  If $G=H$, an outcome of
    the theorem is satisfied.  Else, by~(\ref{c:conn}), we pick a
    vertex $v$ in $G\sm H$ with at least one neighbor in $H$.  From
    the maximality of $H$, $v$ has a neighbor $p_i$ in the interior of
    some $P_6 = p_1 \tp p_2 \tp p_3 \tp p_4 \tp p_5 \tp p_6$ of $H$.
    Up to symmetry we assume that $v$ has a neighbor $p_i$ where $i\in
    \{2, 3\}$.  So $N(v) \cap \{p_1, p_2, p_3, p_4\}$ contains an edge
    $e$ for otherwise $\{p_i, p_{i-1}, p_{i+1}, v\}$ induces a claw.
    If $e = p_1p_2$ then $v$ must be adjacent to $p_4, p_5, p_6$ for
    otherwise there is a coclaw; so $\{v, p_1, p_4, p_6\}$ induces a
    claw.  If $e = p_2p_3$ then $v$ must be adjacent to $p_5, p_6$ for
    otherwise there is a coclaw, so from the symmetry between $\{p_1,
    p_2\}$ and $\{p_5, p_6\}$ we may rely on the previous case.  If $e
    = p_3p_4$ then $v$ must be adjacent to $p_1, p_6$ for otherwise
    there is a coclaw; so $\{v, p_1, p_3, p_6\}$ induces a claw.  In
    all cases there is a contradiction.
  \end{proofclaim}
 
  \begin{claim}
    \label{c:A6}
    We may assume that $G$ contains no $A_6$ and no $\overline{A_6}$.
 \end{claim}
 
 \begin{proofclaim}
   Else up to a complementation, let $aa'$, $bb'$, $cc'$ be three
   disjoint edges of $G$ such that the only edges between them are
   $ab$, $bc$, $ca$.  If $V(G) = \{a, a', b, b', c, c'\}$, an outcome
   of the theorem is satisfied, so let $v$ be a seventh vertex of $G$.
   We may assume that $av\in E(G)$ (else there is a coclaw).  If
   $a'v\in E(G)$ then $vb', vc'\in E(G)$ (else there is a coclaw) so
   $\{v, a', b', c'\}$ is a claw.  Hence $a'v\notin E(G)$. We have
   $vb\in E(G)$ (or $\{a, a', v, b\}$ is a claw) and similarly $vc\in
   E(G)$.  So $\{a', v, b, c\}$ is a coclaw.
 \end{proofclaim}

  \begin{claim}
    \label{c:diamond}
    We may assume that $G$ contains no diamond.
  \end{claim}
 
  \begin{proofclaim}
    We prefer to think about this in the complement, so suppose for a
    contradiction that $G$ contains a \emph{co-diamond}\index{co-diamond}, that is four
    vertices $a, b, c, d$ that induce only one edge, say $ab$.
    By~(\ref{c:conn}), there is a path $P$ from $\{c, d\}$ to some
    vertex $w$ that has a neighbor in $\{a, b\}$.  We choose such a
    path $P$ minimal and we assume up to symmetry that the path is
    from $c$.

    If $w$ is adjacent to both $a, b$ then $\{a, b, w, d\}$ induces a
    coclaw unless $w$ is adjacent to $d$.  But in this case, $w$ is
    adjacent to $c$ from the minimality of $P$, so $\{w, a, c, d\}$
    induces a claw.  Hence $w$ is adjacent to exactly one of $a, b$,
    say to $a$.  So, $P' = c\tp P \tp w \tp a \tp b$ is a path and for
    convenience we rename its vertices $p_1, \dots, p_k$.  If $d$ has
    a neighbor in $P'$ then, from the minimality of $P$, this neighbor
    must be $p_2$.  So, $\{p_2, p_1, p_3, d\}$ induces a claw.  Hence,
    $d$ has no neighbor in $P'$.

    By~(\ref{c:conn}), there is a path $Q$ from $d$ to some vertex $v$
    that has a neighbor in $P'$.  We choose $Q$ minimal with respect
    to this property.  From the paragraph above, $v\neq d$.  Let $p_i$
    (resp. $p_j$) be the neighbor of $v$ in $P'$ with minimum
    (resp. maximum) index.  If $i=j=1$ then $d \tp Q \tp v \tp p_1 \tp
    P \tp p_k$ is a path on a at least 6 vertices a contradiction
    to~(\ref{c:P6}).  So, if $i=j$ then $i\neq 1$ and symmetrically,
    $i\neq k$, so $\{p_{i-1}, p_i, p_{i+1}, v\}$ is a claw.  Hence
    $i\neq j$.  If $j>i+1$ then $\{v, v', p_i, p_j\}$, where $v'$ is
    the neighbor of $v$ along $Q$, is a claw.  So, $j=i+1$.  So
    $vp_ip_j$ is a triangle.  Hence $P' = p_1 \tp p_2 \tp p_3 \tp
    p_4$, $Q= d \tp v$ and $i=2$, for otherwise there is a coclaw.
    Hence, $P'\cup Q$ form an induced $\overline{A_6}$ of $G$, a
    contradiction to~(\ref{c:A6}).
  \end{proofclaim}  
  
  Now $G$ is connected and contains no claw and no diamond.  So, by
  Theorem~\ref{th:HH}, $G$ is the line-graph of some connected
  triangle-free graph $R$.

  If $R$ contains a vertex $v$ of degree at least~4 then all edges of
  $R$ must be incident to $v$, for else an edge $e$ non-incident to
  $a$ together with three edges of $R$ incident to $a$ and
  non-incident to $e$ form a coclaw in $G$.  So all vertices of $R$
  have degree at most 3.  We may assume that $R$ has a vertex $a$ of
  degree~3 for otherwise $G$ is a path or a cycle.  Let $b, b', b''$
  be the neighbors of $a$.  Since $a$ has degree 3, all edges of $R$
  must be incident to $b, b'$ or $b''$ for otherwise $G$ contains a
  coclaw.

  If one of $b, b', b''$, say $b$, is of degree~3, then $N(b) = \{a,
  a', a''\}$ and all edges of $R$ are incident to one of $a, a', a''$
  (or there is a coclaw).  So $R$ is a subgraph of $K_{3, 3}$.  So,
  $G= L(R)$ is an induced subgraph of $L(K_{3, 3})$, an outcome of the
  theorem.  Hence we assume that $b, b', b''$ are of degree at most 2.
  If $|N(\{b, b', b''\})\sm \{a\}| \geq 3$, then $R$ contains the
  pairwise non-incident edges $bc, b'c', b''c''$ say, and the edges
  $ab, ab', ab'', bc, b'c', b''c''$ are vertices of $G$ that induce an
  $\overline{A_6}$, a contradiction to~(\ref{c:A6}).  So, $|N(\{b, b',
  b''\})\sm \{a\}| \leq 2$ which means again that $R$ is a subgraph of
  $K_{3, 3}$.
\end{proof}

An interesting feature of the theorem above is the presence of
sporadic graphs: basic graphs are almost all in quite regular classes
(here paths and cycles) but there is a finite number of
\emph{sporadic}\index{sporadic graphs} exceptions (here $A_6$ and
$L(K_{3, 3})$).  How sporadic graphs pop up in decomposition theorems
is something quite fascinating.  In graph theory, these exceptions
seem to be member a small club: the five platonic graphs ($K_4$, the
cube, the octahedron, the dodecahedron and the icosahedron), the
Wagner graph, the Petersen graph, the Heawood graph, graphs related to
the Fano plane, $K_{3, 3}$ and its line-graph, and several more.  I
wonder who put them here\dots

\begin{figure}[h]
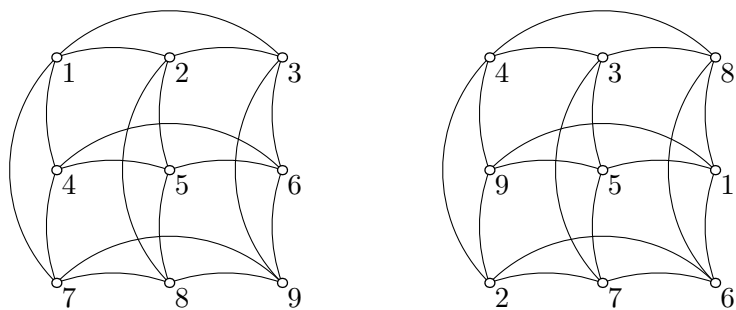

  \center
  \begin{tabular}{ccc}
  \includegraphics{figHdr.8}&\rule{1cm}{0cm}&
  \includegraphics{figHdr.9}
  \end{tabular}
  \caption{$L(K_{3,3})$ and its complement\label{base.fig.lk33}}
\end{figure}

Let me give an entertaining example (from my thesis) of how a sporadic
graph appears in a completely unexpected context.  On
Figure~\ref{base.fig.lk33} is represented $L(K_{3,3})$ with vertices
naturally labeled.  Next to it is represented its complement with the
same labels on the vertices.  Note that the two graphs are
\emph{different}: two vertices are adjacent in the first one if and
only if they are not in the second.  But as one can see, the two
graphs are isomorphic, showing that $L(K_{3, 3})$ is
self-complementary.  The entertaining fact is the following: when one
sums up the numbers from one row, or one column, or one diagonal of
the second square, the result is always~15.  So, the nine numbers on
the second graphs form what is usually called a \emph{magical
  square}\index{magical square},
something very famous in mathematical puzzles.  This magical square
was known by Chinese mathematicians as early as 650 BC (according to
Wikipedia), but they certainly ignored that this was related to the
biggest self-complementary line-graph\dots

Examples of sporadic graphs are present in this document: the Wagner
graph plays an essential role in the proof of Lemma~\ref{threecases},
the cube and $K_4$ with all edges subdivided are sporadic exceptions in
Theorem~\ref{th:main}, Petersen and
Heawood graphs are sporadic in Theorem~\ref{th:OneChord}, the
octahedron is present in Lemmas~\ref{thm:wheel} and
\ref{l:prismdecomp}.

\section{Weakly triangulated graphs}
\label{sec:wt}

A hole or an antihole is said to be \emph{long}\index{long!antihole or hole} if it
contains at least 5 vertices.  A graph is \emph{weakly
  triangulated}\index{weakly triangulated} if it contains no long hole
and no long antihole.  Weakly triangulated graphs were investigated by
Chv\'atal and Hayward in the eighties and the aim of this section is
to convince the reader that they can be considered today as one of the
most interesting classes of perfect graphs.  Because with them, one
can understand several important concepts by just reading six pages:
Roussel-and-Rubio-like Lemmas, how to deal with antiholes, how
uncomfortable decompositions like skew partitions come into the game,
why they are uncomfortable, how we can sometimes get rid of them with
special vertices (namely even pairs) and thus obtain fast and simple
coloring algorithms.

The Roussel and Rubio Lemma~\cite{roussel.rubio:01} is an important
technical tool to prove the Strong Perfect Graph Theorem.  The team
who proved the Strong Perfect Graph Conjecture rediscovered it and
called it at some step the ``wonderful lemma'' because of its many
applications.  It states that in a sense, any anticonnected set of
vertices of a Berge graph behaves like a single vertex (where
\emph{anticonnected}\index{anticonnected} means connected in the
complement).  How a vertex $v$ ``behaves'' in a Berge graph?  If a
path of odd length (at least~3) has both ends adjacent to $v$, then
$v$ must have other neighbors in the path for otherwise there is an
odd hole.  An anticonnected set $T$ of vertices behaves similarly: if
a path of odd length (at least~3) has both ends complete to $T$, then
at least one interior vertex of the path is also complete to $T$.  In
fact, there are two exceptions to this statement, so the Roussel and
Rubio Lemma is slightly more complicated, and we do not state it here
(we do not need it the rest of the document).  For more about Roussel
and Rubio Lemma, in particular equivalent statements and short proofs,
see Chapter~4 of~\cite{nicolas:these}.

Here is a lemma which can be seen as a version of the Roussel and
Rubio Lemma~\cite{roussel.rubio:01} for weakly triangulated graphs.

\begin{lemma}[with Maffray \cite{nicolas:artemis}]
  \label{l:rrwt}
  Let $G$ be a weakly triangulated graph.  Let $P = \bp x \tp \cdots
  \tp y \ep$ be a path of $G$ of length at least 3 and $T \subseteq
  V(G)$ disjoint from $V(P)$ such that $G[T]$ is anticonnected and the
  ends of $P$ are $T$-complete.  Then $P$ has an interior-vertex which
  is $T$-complete.
\end{lemma}

\begin{proof}
  Note that no vertex $t \in T$ can be non-adjacent to two consecutive
  vertices of $P$, for otherwise $V(P) \cup \{t\}$ contains a long
  hole.  Let $z$ be an interior vertex of $P$ adjacent to a maximum
  number of vertices of $T$.  Suppose for a contradiction that there
  exists a vertex $u \in T \setminus N(z)$.  Let $x'$ and $y'$ be the
  neighbors of $z$ along $P$, so that $x, x', z, y', y$ appear in this
  order along $P$.  Then, from the first sentence of this proof, $ux',
  uy' \in E(G)$.  Up to a relabelling of $x$ and $y$, we assume $x'
  \neq x$.  From the choice of $z$, since $ux' \in E(G)$ and $uz
  \notin E(G)$, there exists a vertex $v \in T$ such that $vz\in E(G)$
  and $vx'\notin E(G)$.  Since $G[T]$ is anticonnected, there exists
  an antipath $Q$ of $G[T]$ from $u$ to $v$, and $u, v$ are chosen
  subject to the minimality of this antipath.  From the first
  sentence, interior vertices of $Q$ are all adjacent to $x'$ or $z$
  and from the minimality of $Q$, interior vertices of $Q$ are all
  adjacent to $x'$ and $z$.  If $x'x \notin E(G)$ then $V(Q) \cup \{
  x, x', z\}$ induces a long antihole. So $x'x \in E(G)$.  If $zy
  \notin E(G)$ then $V(Q) \cup \{z, x', y\}$ induces a long antihole.
  So $zy \in E(G)$ and $y = y'$. But then, $V(Q) \cup \{x, x', z, y\}$
  induces a long antihole, a contradiction.
\end{proof}

When $T$ is a set of vertices, $C(T)$ denotes the set of these
vertices that are complete to $T$.

\begin{lemma}
  \label{l:pathWT}
  Let $G$ be a weakly triangulated graph and $T$ a set of vertices
  such that $G[T]$ is anticonnected and $C(T)$ contains at least two
  non-adjacent vertices.  Suppose that $T$ is inclusion-wise maximal
  with respect to these properties.  Then any path of $G\sm T$ whose
  ends are in $C(T)$ has all its vertices in $C(T)$.
\end{lemma}

\begin{proof}
  Let $P$ be a path of $G\sm T$ whose ends are in $C(T)$.  If some
  vertex of $P$ is not in $C(T)$ then $P$ contains a subpath $P'$ of
  length at least 2 whose ends are in $C(T)$ and whose interior is
  disjoint from $C(T)$.  If $P'$ is of length~2, say $P'=a\tp t \tp b$
  then $T\cup \{t\}$ is a set which contradicts the maximality of $T$.
  If $P$ is of length greater than~2, it contradicts
  Lemma~\ref{l:rrwt}.  
\end{proof}

The following theorem is due to Hayward but the proof given here is
mine.  My proof is not really shorter than Hayward's, but it shows how
Roussel-and-Rubio-like lemmas can be used.  A
\emph{cutset}\index{cutset} of a graph is a set $S$ of vertices such
that $G\sm S$ is disconnected.  A \emph{star}\index{star} of a graph
is a set of vertices $S$ that contain a vertex $c$ such that $S
\subseteq N[c]$.  A \emph{star cutset}\index{star cutset} is a star
that is a cutset.

\begin{theorem}[Hayward \cite{hayward:wt}]
  \label{th:wt}
  Let $G$ be a weakly triangulated graph. Then either: 
  \begin{itemize}
  \item $G$ is a clique;
  \item $G$ is the complement of a perfect matching;
  \item $G$ admits a star cutset.
  \end{itemize}
\end{theorem}

\begin{proof}
  If $G$ is a disjoint union of cliques (in particular when $|V(G)|
  \leq 2$) then the theorem holds trivially.  So by
  Theorem~\ref{th:P3}, we may assume that $G$ contains a $P_3$.
  Hence, there exists a set $T$ of vertices such that $G[T]$ is
  anticonnected and $C(T)$ contains at least two non-adjacent vertices
  because the middle of a $P_3$ forms such a set.  Let us assume $T$
  maximal like in Lemma~\ref{l:pathWT}.  Since $C(T)$ is not a clique,
  by induction we have two cases to consider:

  \noindent{\bf Case~1:} the graph induced by $C(T)$ has a star
  cutset $S$.  From Lemma~\ref{l:pathWT}, $T\cup S$ is a star cutset
  of $G$.  

  \noindent{\bf Case~2:} the graph induced by $C(T)$ is the complement
  of a perfect matching.  

  Suppose first that $V(G) = T \cup C(T)$.  Then by induction, either
  $T= \{t\}$, of $T$ induces the complement of a perfect matching or
  $T$ has a star cutset $S$.  But in the first case, $\{t\} \cup C(T)
  \sm \{a, b\}$ where $a, b$ are non-adjacent in $C(T)$, is a star
  cutset of $G$.  In the second case, $G$ itself is the complement of
  a perfect matching. In the third case, $S\cup C(T)$ is a star cutset
  of $G$.  

  So, we may assume that there exists a vertex $x \in V(G) \sm (T \cup
  C(T))$.  We choose $x$ with a neighbor $y$ in $C(T)$, this is
  possible otherwise $T$ together with any vertex of $C(T)$ forms a
  star cutset of $G$.  

  Recall that $C(T)$ is the complement of a perfect matching.  So, let
  $y'$ be the non-neighbor of $y$ in $C(T)$.  Note that $xy'\notin
  E(G)$ for otherwise $T\cup \{x\}$ would contradict the maximality of
  $T$.  We claim that $S = T \cup C(T) \sm \{y'\}$ is a star cutset of
  $G$ separating $x$ from $y'$.  First it is a star centered $y$.  And
  it is a cutset because if there a path in $G\sm S$ from $x$ to $y'$,
  then by appending $y$ to that path we see that $G\sm T$ contains a
  path from $y$ to $y'$ that is not included in $C(T)$, a
  contradiction to Lemma~\ref{l:pathWT}.
\end{proof}

From the theorem above and the following one, weakly triangulated
graphs are perfect.  To see this, consider a non-perfect weakly
triangulated graph, and vertex-inclusionwise minimal with respect to
this property.  So, it is a \emph{minimally imperfect
  graph}\index{minimally imperfect graph} (that is
a non-perfect graph whose all induced subgraphs are perfect).  Since
cliques and complement of perfect matchings are perfect, it must have
a a star cutset by Theorem~\ref{th:wt}.  So, it contradicts the
following (proof is omitted here).

\begin{theorem}[Chv\'atal \cite{chvatal:starcutset}]
  A minimally imperfect graph has no star cutset.
\end{theorem}

Theorem~\ref{th:wt} has a disease: it uses the star cutset which is
the simplest example of what Kristina Vu\v skovi\'c calls \emph{strong
  cutsets}\index{strong cutset}: kinds of cutsets that do not give much structure.  To see
this, note that a star cutset can be very big.  For instance, it can
be the whole vertex-set except two vertices.  And since in the cutset
itself, edges are quite unconstrained, knowing that the graph has a
star cutset tells little about its structure.  Another famous example
of a strong cutset that we will meet and define is the skew
partition, see Section~\ref{sec:decompBerge}.

Strong cutsets do not give structure theorems.  To see this, try to
see how to build a graph by gluing two smaller graphs along a star.
This is likely to be unsatisfactory.  Because finding the same star in
two graphs is algorithmically quite difficult.  It supposes that one
can check whether two stars are isomorphic, a problem as difficult as
the well known open isomorphism problem.  Well, there is no formal
notion of what is a structure theorem, so there can be endless
discussions on this.  But one can feel that gluing along star cutsets
is less automatic than gluing along 2-cutsets like in
Section~\ref{sec:chordless}.

Strong cutset are difficult to use in polynomial time algorithms.
Because to build blocks of decomposition with them, in the unlucky
case where the cutset is almost as big as the graph, one has to keep
almost all the graph in both blocks of decomposition.  So recursive
algorithms using strong cutset typically take exponential time.  A
technically involved method invented by Conforti and
Rao~\cite{ConfortiR:92,ConfortiR:93}, the so-called \emph{cleaning}\index{cleaning},
allows to devise fast recognition algorithms for classes decomposed by
strong cutsets.  But for combinatorial optimization problems, it seems
that no one knows how to use strong cutsets.  Yet, when a class of
graphs is complex enough so that strong cutsets seem unavoidable, there
is still a hope.  Indeed, several theorems replace cutsets with the
existence of a vertex or a pair of vertices with very special
properties.  The oldest example is the following, to be compared with
Theorem~\ref{th:chordal}.  A \emph{simplicial}\index{simplicial} vertex is a vertex
whose neighborhood is a clique.

\begin{theorem}[Dirac, \cite{dirac:chordal}]
  \label{th:simplicial}
  Any chordal graph contains a simplicial vertex. 
\end{theorem}

For perfect graphs, the good notion of special vertices seems to be
the even pair. An \emph{even pair}\index{even!pair} of vertices is a pair of vertices
such that all paths linking them are of even length.  By the following
theorem, even pairs are a good tool for proving perfectness. 

\begin{theorem}[Meyniel \cite{meyniel:87}]
  A minimally imperfect graph has no even pair. 
\end{theorem}

\emph{Contracting}\index{contracting!an even pair} an even pair $a, b$ means replacing $a, b$ by a single
vertex complete to $N(a) \cup N(b)$.  By the following theorem, even
pairs are also a good tool for graph coloring.

\begin{theorem}[Fonlupt and Uhry \cite{fonlupt.uhry:82}]
  Contracting an even pair of a graph preserves its chromatic number
  and the size of a largest clique.
\end{theorem}

The following theorem was first proved by Hayward, Ho\`ang and Maffray
but the proof given here was obtained jointly with Maffray
in~\cite{nicolas:artemis}.  A \emph{2-pair}\index{2-pair} $a, b$ is a special kind
of even pair: all paths from $a$ to $b$ are of length~2.

\begin{theorem}[Hayward, Ho\`ang and Maffray \cite{hayward.hoang.m:90}]
  \label{th:2pair}
  A weakly triangulated graph either is a clique or admits a 2-pair.  
\end{theorem}

\begin{proof}
  If $G$ is a disjoint union of cliques (in particular when $|V(G)|
  \leq 2$) then the theorem holds trivially.  So we may assume that
  $G$ contains a $P_3$.  So, there exists a set $T$ like in
  Lemma~\ref{l:pathWT} (start with the middle of a $P_3$ to build $T$).
  Since $C(T)$ is not a clique, by induction we know that $C(T)$
  admits a 2-pair of $G[C(T)]$.  From Lemma~\ref{l:pathWT} it is a
  2-pair of $G$.
\end{proof}

The technique above to find an even pair can be backtracked to the
seminal paper of Henri Meyniel~\cite{meyniel:87}, see
Exercise~\ref{ex:meyniel} below.  Using ideas of
Cl\'audia Linhares Sales and
Maffray~\cite{linhares.maffray:evenpairsansc4}, this technique can be
extended to Artemis graphs~\cite{nicolas:artemis}, which are a
generalization of several classes of perfect graphs known to have even
pairs (weakly triangulated graphs, Meyniel graphs and perfectly
orderable graphs etc; we shall not define all the classes, a reader who
want to visit the zoo can read Chapter~3 of~\cite{nicolas:these}).
The technicalities in~\cite{nicolas:artemis} together with new kinds
of special vertices and complicated variations on the Roussel and
Rubio Lemma are used by Chudnovsky and
Seymour~\cite{chudnovsky.seymour:even} to shorten significantly the
proof of the Strong Perfect Graph Theorem.

From the Theorem~\ref{th:2pair}, it is easy to deduce a polynomial
time algorithm for coloring weakly triangulated graphs (by contracting
2-pairs as long as there are some).  Hayward, Spinrad and Sritharan
\cite{hayward.S.S:fastWT} could speed this algorithm up to $O(n^3)$.
Since the contraction of a 2-pair preserves both $\chi$ and $\omega$,
the algorithm will transform any weakly triangulated graph $G$ into a
clique $K$ of size $\omega(G) = \omega(K) = \chi(K) = \chi(G)$, thus
proving the perfectness of $G$ (because all this can be done for all
induced subgraphs of~$G$).  So, Theorem~\ref{th:2pair} gives in fact a
much shorter proof of the perfectness of weakly triangulated graphs.

Berge graphs behave in a way like weakly triangulated graphs.  The
Roussel and Rubio Lemma is an important tool to prove their
perfectness, a strong cutset is needed to decompose them (the
so-called balanced skew partition), but by using even pairs, a shorter
proof of perfectness can be found (this was done by Chudnovsky and
Seymour~\cite{chudnovsky.seymour:even}).  A big difference of course
is that for general Berge graphs all the proofs are much longer and
highly technical.  And up to now, no fast combinatorial coloring
algorithm for Berge graphs is known.  But it might be that even pairs
become an ingredient of such an algorithm.  A collection of theorems
and conjectures supports this idea, but a heavy machinery is needed
just to state all them, so we postpone this to Section~\ref{sec:SF}.

Another important open question about Berge graphs is the existence of
a structure theorem for them.  The following is seemingly easier but
remains open.

\begin{question}
  Find a structure theorem for weakly triangulated graphs.
\end{question}

In fact, I would not be surprised that \emph{no structure theorem
  exists for Berge graphs}.  It would be great to have a tool like
NP-completeness to get people convinced of such negative statements,
or better to prove them.  The non-existence of a polynomial time
recognition algorithm could be an argument, but it does not work for
Berge graphs that can be recognized in $O(n ^9)$ as mentioned in the
introduction.

Another well known class was seminal to many ideas in the theory of
perfect graphs: Meyniel graphs.  A graph is
\emph{Meyniel}\index{Meyniel graph} if all its
odd cycles of length at least~5 admit at least two chords.  Meyniel
graphs were proved to be perfect very early by
Meyniel~\cite{meyniel:76}, and were the first class after chordal and
$P_4$-free graphs for which a decomposition theorem could be proved
(by Burlet and Fonlupt~\cite{burlet.fonlupt:meyniel}).  With respect
to even pairs, Meyniel graphs have behavior similar to weakly
triangulated graphs.  We leave this as an exercise.

\begin{exercise}
  \label{ex:meyniel}
  Give a version of the Roussel and Rubio lemma for Meyniel graphs.
  Deduce that a Meyniel graph either is a clique or has an even pair.
  For a solution, see \cite{meyniel:87}.
\end{exercise}

\chapter{Detecting induced subgraphs}
\label{chap:reco}

\emph{Recognizing}\index{recognition of a class of graph} a class of
graph means giving an algorithm that decides whether an input graph is
in the class.  This can be of practical interest (does a network
satisfy this or that condition, maybe an application of what follows
can be cooked \ldots).  But for us, recognizing classes of graphs is
mainly of a theoretical interest.  It is a tool to classify classes:
those that are NP-complete to recognize are unlikely to have a very
precise decomposition theorem.  This criterion is not to be taken too
religiously.  Some classes are trivial to recognize in polynomial time
while their structure is unknown or highly non-trivial, for instance
triangle-free graphs and claw-free graphs.  Some classes have a
decomposition theorem not precise enough to give a polynomial-time
recognition algorithm (for instance odd hole-free graphs, see
Theorem~\ref{th.ccv}).  But it is nice to have a general criterion
about what class is ``understood''.  Together with $\chi$-boundness,
the existence of a fast recognition algorithm is the only criterion
applying to any class that we are aware of.

The classes we are working on are defined by forbidding a list of
induced subgraphs.  So recognizing one of them is equivalent to being
able to decide whether a given graph contains one graph from the list.
In its full generality, this problem is known to be difficult from the
very beginning of Complexity Theory, even for lists of size one.

\begin{theorem}[Cook \cite{cook:np}] The following problem is
  NP-complete.
  \begin{description}
  \item[Instance] Two graphs $G$ and $H$.
  \item[Question] Does $G$ contain an induced subgraph isomorphic to
    $H$?
  \end{description}
\end{theorem}

On the other hand, the following is trivial by a brute force
algorithm.

\begin{theorem} Let $H$ be a graph.  The following problem can be
  solved in time $O(n^{|V(H)|})$.
  \begin{description}
  \item[Instance] A graph $G$.
  \item[Question] Does $G$ contain an induced subgraph isomorphic to
    $H$?
  \end{description}
\end{theorem}

The brute-force method above is the source of many algorithms with
strange running times (the worst I will mention runs in time
$O(n^{18})$, see the end of Section~\ref{sec:shortPD}).  These
complexities are bad, but not as bad as one might expect.  Suppose for
instance that an algorithm uses as a subroutine an enumeration of all
subsets on 18 vertices of an input graph.  So, the complexity is very
bad.  It is sometimes said\footnote{In fact, in 2004, I gave a talk in
  Haifa and someone interrupted my talk to put forward this
  statement.} that $O(n^{18})$-algorithms are useless in practice and
that an algorithm of complexity $O(2^n)$ is better for any reasonable
$n$ (because $2^{30} < 30^{18}$ say).  This can be true for some
particular algorithms, but is completely false for some others.  To
see this, consider the brute-force algorithm to enumerate all subsets
of size 18 from a set on $n$ elements (complexity $O(n^{18})$). Is it
worse than enumerating all the subsets (complexity $O(2^n)$)?
Certainly not: there are many more subsets than subsets on 18
elements!  Because the constant hidden in the $O$ matters, and it is a
mistake to think that this constant can only bring trouble.  In some
situations, the constant is much smaller than 1, and enumerating
subsets on $k$ elements is an example (because ${n\choose k} <
n^{k}$).  Anyway, comparing blindly big-$O$'s for practical reasons is
not very wise.  For practical applications, only the actual running
time matters and complexity theory just gives a possibly misleading
indication.

To produce something between trivial polynomiality and
NP-completeness, we need to consider the problem of detecting an
induced subgraph from a fixed infinite list.  But this is far much too
general.  To see this, consider the following question that should not
be too difficult.  It should suffice to encode any Turing machine as a
graph and to define $\cal L$ as something related to the halting
problem.

\begin{question}
  Define a list $\cal L$ of graphs such that deciding whether a graph
  contains an element of $\cal L$ as an induced subgraph is an
  undecidable problem.
\end{question}

So, to produce something interesting, we need to give some structure
to the list of forbidden induced subgraphs.  In many interesting
problems, the list of forbidden induced subgraphs is obtained by
subdividing prescribed edges of a prescribed graph.  This lead
L\'ev\^eque, Maffray and myself~\cite{leveque.lmt:detect} to the
following definitions.

A \emph{subdivisible graph}\index{subdivisible graph}
(\emph{s-graph}\index{s-graph} for short) is a triple $B = (V, D, F)$
such that $(V, D \cup F)$ is a graph and $D \cap F = \emptyset$.  The
edges in $D$ are said to be \emph{real edges of $B$}\index{real edge
  of an s-graph} while the edges in $F$ are said to be
\emph{subdivisible edges of $B$}\index{subdivisible edge of an
  s-graph}.  A \emph{realisation}\index{realisation of an s-graph} of
$B$ is a graph obtained from $B$ by subdividing edges of $F$ into
paths of arbitrary length (at least one).  The problem $\Pi_B$ is the
decision problem whose input is a graph $G$ and whose question is
"Does $G$ contain a realisation of $B$ as an induced subgraph?''.  On
figures, we depict real edges of an s-graph with straight lines, and
subdivisible edges with dashed lines.

\begin{figure}
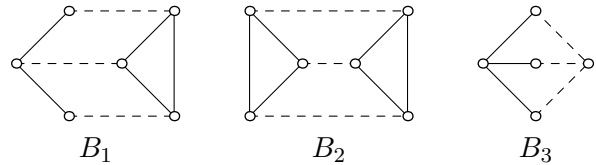

\center
\begin{tabular}{ccccc}
\includegraphics{bigraphs.17}&
\hspace{10em}&
\includegraphics{bigraphs.16}&
\hspace{10em}&
\includegraphics{bigraphs.18}\\
$B_1$&&
$B_2$&&
$B_3$
\end{tabular}
\caption{Pyramids, prisms and thetas\label{fig:ppt}}
\end{figure}

We shall study several instances of $\Pi_B$ and survey the efficient
methods to tackle them.  Some $\Pi_B$'s are NP-complete, some are
polynomial, and many remain open.  As an appetizer, let us give now
several examples.  A \emph{pyramid}\index{pyramid} (resp.  \emph{prism}\index{prism},
\emph{theta}\index{theta}) is any realisation of the s-graph $B_1$ (resp.  $B_2$,
$B_3$) represented on Figure~\ref{fig:ppt}.  Chudnovsky and
Seymour~\cite{chudnovsky.c.l.s.v:reco} gave an $O(n^9)$-time algorithm
for $\Pi_{B_1}$ (or equivalently, for detecting a pyramid).  As far as
we know, that is the first example of a solution to a $\Pi_B$ whose
complexity is non-trivial to settle.  The method used is what we call
\emph{shortest path detector}, see Section~\ref{sec:shortPD}.  In
contrast, we proved jointly with Maffray~\cite{maffray.t:reco} that
$\Pi_{B_2}$ (or detecting a prism) is NP-complete.  The method relies
on a construction due to Bienstock that has many applications, see
Section~\ref{sec:bienstock}.  Chudnovsky and
Seymour~\cite{chudnovsky.seymour:theta} gave an $O(n^{11})$-time
algorithm for $P_{B_3}$ (or detecting a theta).  Their algorithm
relies on the solution of a problem called ``three-in-a-tree'', that
has many applications, see Section~\ref{sec:3-in-tree}.  In
Section~\ref{sec:gen3} we present several variations on
three-in-a-tree.  

Note that pyramids, prisms and thetas are sometimes called
respectively 3PC($\Delta$, $\cdot$), 3PC($\Delta$, $\Delta$),
3PC($\cdot$, $\cdot$) where ``3PC'' means \emph{three paths
  configurations}\index{3PC}.  This is sometimes a very convenient
notation.

Often, forbidding 3PC's yields interesting classes.  Note for instance
that Berge graphs contain no pyramid and even-hole-free graphs contain
no prisms and no thetas.
 
From the examples above one can feel that strangely close s-graphs
lead to drastically different complexities (as far as P$\neq$NP).
Before going into the details, let us give a last example may be more
striking than pyramid/prism/theta~: $\Pi_{B_4}, \Pi_{B_6}$ are
polynomial and $\Pi_{B_5}, \Pi_{B_7}$ are NP-complete, where $B_4,
\dots, B_7$ are the s-graphs represented on Figure~\ref{fig:antenna},
see Theorem~\ref{th:antena}.

\begin{figure}
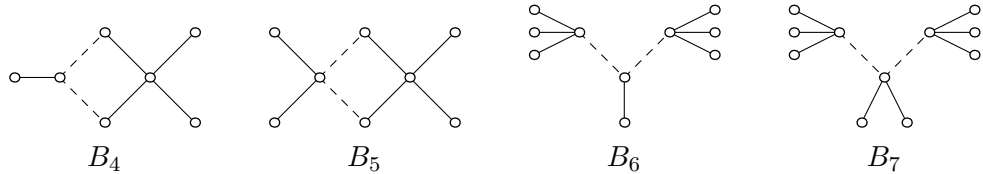

\center
\begin{tabular}{ccccccc}
\includegraphics{bigraphs.15}&
\hspace{10em}&
\includegraphics{bigraphs.14}&
\hspace{10em}&
\includegraphics{bigraphs.22}&
\hspace{10em}&
\includegraphics{bigraphs.21}\\
$B_4$&&
$B_5$&&
$B_6$&&
$B_7$
\end{tabular}
\caption{Some s-graphs with pending edges\label{fig:antenna}}
\end{figure}

\section{A tool for NP-completeness: Bienstock's Construction}
\label{sec:bienstock}

In 1991, Dan Bienstock published a paper~\cite{bienstock:evenpair}
with an NP-completeness proof that can be applied in many situations
(with slight variations).  I am grateful to Bienstock because out of
the ten accepted/published papers I have today, six contain results
relying on Bienstock's construction, and three more mention
NP-completeness results relying on it.  

Let $\cal S$ be a set of graphs and $d$ an integer.  Let $\Gamma_{\cal
  S}^d$ be the problem whose instance is $(G, x, y)$ where $G$ is a
graph such that $\Delta(G) \leq d$, with no induced subgraph in $\cal
S$ and $x, y \in V(G)$ are two non-adjacent vertices of degree~$2$.
The question is ``Does $G$ contain a hole passing through $x, y$?''.
For simplicity, we write $\Gamma_{\cal S}$ instead of $\Gamma_{\cal
  S}^{+\infty}$ (so, the graph in the instance of $\Gamma_{\cal S}$
has unbounded degree).  Also we write $\Gamma^d$ instead of
$\Gamma^d_{\emptyset}$ (so the graph in the instance of $\Gamma^d$ has
no restriction on its induced subgraphs).
Bienstock~\cite{bienstock:evenpair} proved that $\Gamma =
\Gamma_{\emptyset}$ is NP-complete.  Here we give the simplest proof
that we are aware of using Bienstock's construction.  This is from a
paper with Maffray~\cite{maffray.t:reco} and is slightly simpler than
the original construction.

\begin{theorem}
  \label{thm:pinpc}
  Problem $\Gamma_{\{K_3\}}$ is NP-complete
\end{theorem}

\begin{proof} 
  Let us give a polynomial reduction from the problem {\sc
    $3$-Satisfiability} of Boolean functions to problem $\Gamma_{\{K_3\}}$.  Recall
  that a Boolean function with $n$ variables is a mapping $f$ from
  $\{0, 1\}^n$ to $\{0, 1\}$.  A Boolean vector $\xi\in\{0, 1\}^n$ is
  a \emph{truth assignment} for $f$ if $f(\xi)=1$.  For any Boolean
  variable $x$ on $\{0, 1\}$, we write $\overline{x}:=1-x$, and each
  of $x, \overline{x}$ is called a \emph{literal}.  An instance of
  {\sc $3$-Satisfiability} is a Boolean function $f$ given as a
  product of clauses, each clause being the Boolean sum $\vee$ of
  three literals; the question is whether $f$ admits a truth
  assignment.

Let $f$ be an instance  of {\sc $3$-Satisfiability}, consisting of $m$
clauses $C_1, \ldots,  C_m$ on $n$ variables $x_1,  \ldots, x_n$.  Let
us build a graph $G_f$  with two specialized vertices $a,b$, such that
there will be a hole containing both $a,b$ in $G$ if and only if there
exists a truth assignment for $f$.

For each variable $x_i$ ($i=1, \ldots, n$), make a graph $G(x_i)$ with
eight vertices $a_i,  b_i, t_i, f_i, a'_i, b'_i,  t'_i, f'_i,$ and ten
edges  $a_it_i, a_if_i,  b_it_i, b_if_i$  (so that  $\{a_i,  b_i, t_i,
f_i\}$ induces  a hole), $a'_it'_i, a'_if'_i,  b'_it'_i, b'_if'_i$ (so
that  $\{a'_i,  b'_i, t'_i,  f'_i\}$  induces  a  hole) and  $t_if'_i,
t'_if_i$.  See Figure~\ref{fig:gxi}.

\begin{figure}[p]
\begin{center}
\includegraphics{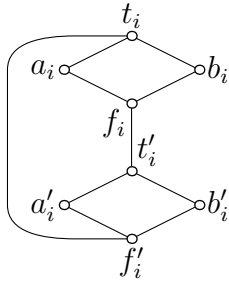}
\end{center}
\caption{Graph $G(x_i)$}\label{fig:gxi}
\end{figure}

\begin{figure}[p]
\begin{center}
\includegraphics{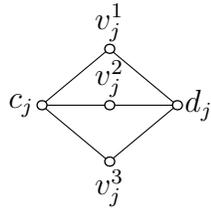}
\end{center}
\caption{Graph $G(C_j)$}\label{fig:gcj}
\end{figure}

\begin{figure}[p]
\begin{center}
\includegraphics{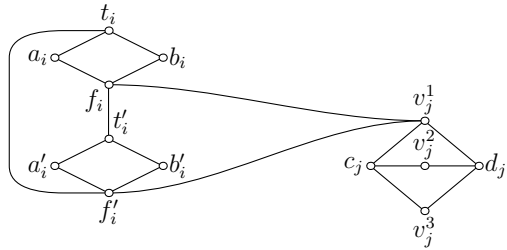}
\end{center}
\caption{The two edges added to $G_f$ in the case $u_j^p=x_i$}\label{fig:gf}
\end{figure}

\begin{figure}[p]
\begin{center}
\includegraphics{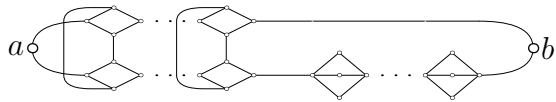}
\end{center}
\caption{Graph $G_f$}\label{fig:gf2}
\end{figure}

For  each  clause  $C_j$   ($j=1,  \ldots,  m$),  with  $C_j=u_j^1\vee
u_j^2\vee u_j^3$, where  each $u_j^p$ ($p=1, 2, 3$)  is a literal from
$\{x_1, \ldots, x_n, \overline{x}_1, \ldots, \overline{x}_n\}$, make a
graph $G(C_j)$ with five vertices  $c_j, d_j, v_j^1, v_j^2, v_j^3$ and
six edges  so that each of $c_j,  d_j$ is adjacent to  each of $v_j^1,
v_j^2,  v_j^3$.   See  Figure~\ref{fig:gcj}.   For  $p=1,  2,  3$,  if
$u_j^p=x_i$  then  add  two  edges  $u_j^pf_i,  u_j^pf'_i$,  while  if
$u_j^p=\overline{x}_i$ then add two edges $u_j^pt_i, u_j^pt'_i$.  See
Figure~\ref{fig:gf}. 

The graph $G_f$ is obtained from the disjoint union of the $G(x_i)$'s
and the $G(C_j)$'s as follows.  For $i=1, \ldots, n-1$, add edges
$b_ia_{i+1}$ and $b'_ia'_{i+1}$.  Add an edge $b'_nc_1$.  For $j=1,
\ldots, m-1$, add an edge $d_jc_{j+1}$.  Introduce the two specialized
vertices $a,b$ and add edges $aa_1, aa'_1$ and $bd_m, bb_n$.  See
Figure~\ref{fig:gf2}.  Clearly the size of $G_f$ is polynomial
(actually linear) in the size $n+m$ of $f$.  Moreover, it is easy to
see that $G_f$ contains no triangle, and that $a,b$ are non-adjacent
and both have degree $2$.

Suppose  that $f$ admits  a truth  assignment $\xi\in\{0,  1\}^n$.  We
build a hole  in $G$ by selecting vertices  as follows.  Select $a,b$.
For  $i=1, \ldots,  n$, select  $a_i, b_i,  a'_i, b'_i$;  moreover, if
$\xi_i=1$ select  $t_i, t'_i$, while if $\xi_i=0$  select $f_i, f'_i$.
For $j=1,  \ldots, m$, since $\xi$  is a truth assignment  for $f$, at
least  one  of the  three  literals  of $C_j$  is  equal  to $1$,  say
$u_j^p=1$  for some  $p\in\{1, 2,  3\}$.  Then  select $c_j,  d_j$ and
$v_j^p$.   Now it  is  a routine  matter  to check  that the  selected
vertices  induce a  cycle $Z$  that contains  $a,b$, and  that  $Z$ is
chordless, so it is a hole.  The  main point is that there is no chord
in $Z$ between some subgraph  $G(C_j)$ and some subgraph $G(x_i)$, for
that  would  be  either  an  edge  $t_iv_j^p$  (or  $t'_iv_j^p$)  with
$u_j^p=x_i$ and  $\xi_i=1$, or, symmetrically, an  edge $f_iv_j^p$ (or
$f'_iv_j^p$) with $u_j^p=\overline{x}_i$ and $\xi_i=0$, in either case
a contradiction to the way the vertices of $Z$ were selected.

Conversely, suppose that $G_f$ admits  a hole $Z$ that contains $a,b$.
Clearly $Z$ contains  $a_1, a'_1$ since these are  the only neighbours
of $a$ in $G_f$.

\begin{claim}\label{clm:zgxi}
For $i=1, \ldots, n$, $Z$ contains exactly six vertices of $G_i$: four
of them are $a_i, a'_i, b_i, b'_i$, and the other two are either $t_i,
t'_i$ or $f_i, f'_i$.
\end{claim}

\begin{proofclaim} 
  First we prove the claim for $i=1$.  Since $a, a_1$ are in $Z$ and
  $a_1$ has only three neighbours $a, t_1, f_1$, exactly one of $t_1,
  f_1$ is in $Z$.  Likewise exactly one of $t'_1, f'_1$ is in $Z$.  If
  $t_1, f'_1$ are in $Z$ then the vertices $a, a_1, a'_1, t_1, f'_1$
  are all in $Z$ and they induce a hole that does not contain $b$, a
  contradiction.  Likewise we do not have both $t'_1, f_1$ in $Z$.
  Therefore, up to symmetry we may assume that $t_1, t'_1$ are in $Z$
  and $f_1, f'_1$ are not.  If a vertex $u_j^p$ of some $G(C_j)$
  ($1\le j\le m$, $1\le p\le 3$) is in $Z$ and is adjacent to $t_1$
  then, since this $u_j^p$ is also adjacent to $t'_1$, we see that the
  vertices $a, a_1, a'_1, t_1, t'_1, u_j^p$ are all in $Z$ and induce
  a hole that does not contain $b$, a contradiction.  Thus the
  neighbour of $t_1$ in $Z\setminus a_1$ is not in any $G(C_j)$ ($1\le
  j\le m$), so that neighbour is $b_1$.  Likewise $b'_1$ is in $Z$.
  So the claim holds for $i=1$.  Since $b_1$ is in $Z$ and exactly one
  of $t_1, f_1$ is in $Z$, and $b_1$ has degree $3$ in $G_f$, we
  obtain that $a_2$ is in $Z$, and similarly $b_2$ is in $Z$.  Now the
  proof of the claim for $i=2$ is essentially the same as for $i=1$,
  and by induction the claim holds up to $i=n$.
\end{proofclaim}

\begin{claim}\label{clm:zgcj}
  For $j=1, \ldots, m$, $Z$ contains $c_j, d_j$ and exactly one of
  $v_j^1, v_j^2, v_j^3$.
\end{claim}

\begin{proofclaim}
  First we prove this claim for $j=1$.  By Claim~\ref{clm:zgxi},
  $b'_n$ is in $Z$ and exactly one of $t'_n, f'_n$ is in $Z$, so
  (since $b'_n$ has degree $3$ in $G_f$) $c_1$ is in $Z$.
  Consequently exactly one of $u_1^1, u_1^2, u_1^3$ is in $Z$, say
  $u_1^1$.  The neighbour of $u_1^1$ in $Z\setminus c_1$ cannot be a
  vertex of some $G(x_i)$ ($1\le i\le n$), for that would be either
  $t_i$ (or $f_i$) and thus, by Claim~\ref{clm:zgxi}, $t'_i$ (or
  $f'_i$) would be a third neighbour of $u_1^1$ in $Z$, a
  contradiction.  Thus the other neighbour of $u_1^1$ in $Z$ is $d_1$,
  and the claim holds for $j=1$.  Since $d_1$ has degree $4$ in $G_f$
  and exactly one of $v_1^1, v_1^2, v_1^3$ is in $Z$, it follows that
  its fourth neighbour $c_2$ is in $Z$.  Now the proof of the claim
  for $j=2$ is the same as for $j=1$, and by induction the claim holds
  up to $j=m$.
\end{proofclaim}

We can now make a Boolean  vector $\xi$ as follows.  For $i=1, \ldots,
n$, if $Z$ contains $t_i, t'_i$ set $\xi_i = 1$; if $Z$ contains $f_i,
f'_i$ set  $\xi_i = 0$.   By~(\ref{clm:zgxi}) this  is consistent.
Consider any  clause $C_j$  ($1\le j\le m$).   By~(\ref{clm:zgcj})
and up to symmetry we may assume  that $v_j^1$ is in $Z$.  If $u_j^1 =
x_i$  for some  $i\in\{1, ..,  n\}$,  then the  construction of  $G_f$
implies that $f_i, f'_i$ are not in $Z$, so $t_i, t'_i$ are in $Z$, so
$\xi_i=1$,  so  clause $C_j$  is  satisfied  by  $x_i$.  If  $u_j^1  =
\overline{x}_i$ for some $i\in\{1,  .., n\}$, then the construction of
$G_f$ implies that  $t_i, t'_i$ are not in $Z$, so  $f_i, f'_i$ are in
$Z$, so  $\xi_i=0$, so clause $C_j$ is  satisfied by $\overline{x}_i$.
Thus $\xi$ is a truth assignment for $f$.  This completes the proof of
the theorem.
\end{proof}

Now we can see how to use the Theorem above to prove the
NP-completeness of a detection problem. 
\begin{theorem}[with Maffray \cite{maffray.t:reco}]
  \label{thm:prismsnpc}
  Problem $\Pi_{B_2}$ (deciding whether a graph contains a prism) is
  NP-complete.
\end{theorem}

\begin{proof}
  Starting from from an instance ($G$, a, b) of $\Gamma_{K_3}$, build
  a graph $G'$ as follows (see Figure~\ref{fig:reco.8}): replace
  vertex $a$ by five vertices $a_1, a_2, a_3, a_4, a_5$ with five
  edges $a_1 a_2$, $a_1 a_3$, $a_2 a_3$, $a_2 a_4$, $a_3 a_5$, and put
  edges $a_4 a'$ and $a_5 a''$.  Do the same with $b$, with five
  vertices named $b_1, \ldots, b_5$ instead of $a_1, \ldots, a_5$ and
  with the analogous edges.  Add an edge $a_1 b_1$.  Since $G$ has no
  triangle, $G'$ has exactly two triangles $\{a_1, a_2, a_3\}$ and
  $\{b_1, b_2, b_3\}$.  Moreover we see that $G'$ contains a prism if
  and only if $G$ contains a hole that contains $a$ and $b$.  So every
  instance of $\Gamma_{\{K_3\}}$ can be reduced polynomially to an
  instance of $\Pi_{B_2}$, which proves that $\Pi_{B_2}$ is
  NP-complete.
\end{proof}

\begin{figure}[!htb]
\begin{center}
\includegraphics{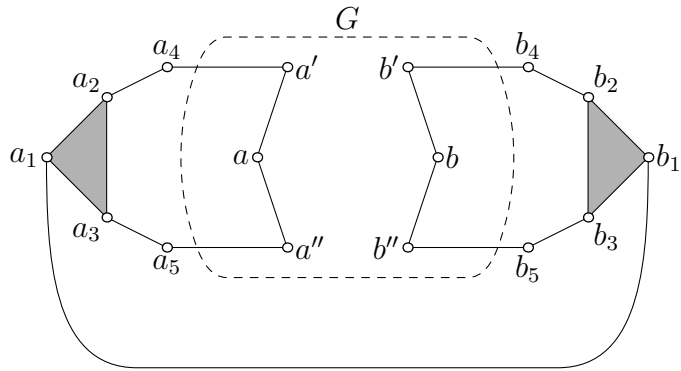}
\end{center}
\caption{Detecting a prism: $G$ and $G'$}
\label{fig:reco.8}
\end{figure}

\begin{exercise}
  Deciding whether a graph contains the line-graph of a subdivision of
  $K_4$ is NP-complete.
\end{exercise}

Adapting the construction of Bienstock, the following can be proved
(it is not so easy to eliminate simultaneously triangles and square,
the proof is very tedious so we omit it).  We denote by $I_l$ ($l\geq
1$) the tree on $l+5$ vertices obtained by taking a path of length $l$
with ends $a, b$, and adding four vertices, two of them adjacent to
$a$, the other two to $b$; see Figure~\ref{fig:I1}.

\begin{figure}
  \center
  \includegraphics{smallgraphs.7}

  \caption{$I_1$\label{fig:I1}}  
\end{figure}

\begin{theorem}[with L\'ev\^eque and Maffray
  \cite{leveque.lmt:detect}]
  \label{th:bienstock}
  Let $k\geq 5$ be an integer.  Then $\Gamma_{\{C_3, \dots, C_k, K_{1,
      6}\}}$ and $\Gamma_{\{I_1, \dots, I_k, C_5, \dots, C_k, K_{1,
      4}\}}$ are NP-complete.
\end{theorem}

The theorem above will be proved in a sense best possible in
Section~\ref{sec:3-in-tree} where a dichotomy criterion will be given
for the problems $\Gamma_{\{H\}}$ where $H$ is a connected graph, see
Theorem~\ref{th:class}.  It has several easy corollaries as the
NP-completeness of $\Pi_{B_5}, \Pi_{B_7}$ mentioned above (see
Theorem~\ref{th:antena} below for a proof).  Also, it can be used to
prove the NP-completeness of the detections of structures analogous to
pyramids, prisms and theta but with four paths instead of three (this
NP-completeness was in fact obtained jointly with Maffray and Vu\v
skovi\'c without Theorem~\ref{th:bienstock}
in~\cite{maffray.t.v:3pcsquare}).  In Section~\ref{sec:oneChord}, an
$O(nm)$ algorithm will be given to detect cycles with a unique chord.
The Theorem above will then be used to prove the NP-completeness of
quite similar detection problems.

The following was proved by David Lin during his Master Thesis, a very
nice and surprising result.  Since publishing NP-completeness results
alone is no more fashionable, we invited Lin in our
paper~\cite{leveque.lmt:detect}\footnote{As a result of this strange
  kind of joint work, Lin is the only coauthor I never met.}.  The
Bienstock-like construction that Lin uses is so modified that we can
say that Lin found a new construction.

\begin{theorem}[Lin \cite{leveque.lmt:detect}]
  \label{th:delta3}
  $\Gamma^3$ is NP-complete.
\end{theorem}

From Theorem~\ref{th:delta3} it is easy to see that detecting whether
a graph contains a subdivision of $K_5$ is NP-complete,
see~\cite{leveque.lmt:detect}.  This result is of interest because a
major breakthrough of Robertson and Seymour in their Graph Minor
Project is that detecting a subdivision of any graph as a possibly
non-induced subgraph is a polynomial problem.  So, the ``induced''
constraint really matters and causes trouble.

\section{A tool for polynomiality (1): shortest path detectors}
\label{sec:shortPD}

The \emph{shortest path detector}\index{shortest path detector} is a method designed by Chudnovsky
and Seymour~\cite{chudnovsky.c.l.s.v:reco} to decide whether a graph
contains a pyramid.  The motivation for detecting pyramids is that
every pyramid contains an odd holes.  So, when trying to recognize
Berge graphs, it is interesting to first look for pyramids.

\begin{theorem}[Chudnovsky and Seymour \cite{chudnovsky.c.l.s.v:reco}]
  \label{thm:pyramid}
  There exists an $O(n^9)$ algorithm for Problem $\Pi_{B_1}$ (deciding
  whether a graph contains a pyramid).
\end{theorem}

To explain the method, we give here an algorithm that detects prisms
in graphs with no pyramids (recall that in general graphs, the problem
is NP-complete).  This algorithm is very simple but is not published
because a faster one exists, see~\cite{maffray.t:reco}.  A shortest
path detector always relies on a Lemma stating in the smallest
substructure of the kind that we are looking for, a path linking two
particular vertices of the substructure can be replaced by any shortest
path.

\begin{lemma}
  \label{l:kpri}
  Let $G$ be a graph with no pyramid.  Let $K$ be a smallest prism in
  $G$.  Suppose that $K$ is a prism, formed by paths $P_1, P_2, P_3$,
  with triangles $\{a_1, a_2, a_3\}$ and $\{b_1, b_2, b_3\}$, so that,
  for $i=1, 2, 3$, path $P_i$ is from $a_i$ to $b_i$.  Then:
  
  If $R_i$ is any shortest path from $a_i$ to $b_i$ whose interior
  vertices are not adjacent to $a_{i+1}$, $a_{i+2}$, $b_{i+1}$ or
  $b_{i+2}$, then $R_{i}, P_{i+1}, P_{i+2}$ form a prism on $|V(K)|$
  vertices in $G$, with triangles $\{a_1, a_2, a_3\}$ and $\{b_1, b_2,
  b_3\}$ (the addition of subscripts is taken modulo 3).
\end{lemma}

\begin{proof}
  Suppose that the lemma fails for say $i=1$.  So, some interior
  vertex of $R$ has neighbors in the interior of $P_2$ or $P_3$.  Let
  $x$ be such a vertex, closest to $a_1$ along $R$.  Let $a'_2$
  (resp.\ $a'_3$) be the neighbor of $a_2$ (resp.\ $a_3$) along $P_2$
  (resp.\ $P_3$).  Let $Q = a'_2 \tp P_2 \tp b_2 \tp b_3 \tp P_3 \tp
  a'_3$.  Let $y$ (resp.\ $z$) be the neighbor of $x$ closest to
  $a'_2$ (resp.\ $a'_3$) along $Q$.

  If $y=z$ then $y\tp x \tp R \tp a_1$, $y \tp Q \tp a'_2 \tp a_2$ and
  $y \tp Q \tp a'_3 \tp a_3$ form a pyramid, a contradiction.  If
  $y\neq z$ and $yz\notin E(G)$ then $x \tp R \tp a_1$, $x\tp y \tp Q
  \tp a'_2 \tp a_2$ and $x \tp z \tp Q \tp a'_3 \tp a_3$ form a
  pyramid, a contradiction.  If $yz\in E(G)$ then $x \tp R \tp a_1$, $y \tp Q
  \tp a'_2 \tp a_2$ and $z \tp Q \tp a'_3 \tp a_3$ form a prism on
  less vertices  than $K$, a contradiction. 
\end{proof}

Now detecting a prism in a graph with no pyramid can be performed as
follows.  For all 6-tuples $(a_1, a_2, a_3, b_1, b_2, b_3)$ compute
three shortest paths $R_i$ in $G\sm ((N[a_{i+1}] \cup N[a_{i+2}] \cup
N[b_{i+1}] \cup N[b_{i+2}]) \sm \{a_i, b_i\})$ from $a_i$ to $b_i$.
Check whether $R_1, R_2, R_3$ form a prism, and if so output it.  If
no triple of paths forms a prism, output that the graph contains no
prism.  If the algorithm outputs a prism, this is obviously a correct
answer: the graph contains a prism.  Suppose conversely that the graph
contains a prism.  Then it contains a smallest prism with triangles
$\{a_1, a_2, a_3\}$ and $\{b_1, b_2, b_3\}$.  At some step, the
algorithm will check this 6-tuple (unless a prism is discovered before,
but then the correctness is proved anyway).  By three applications of
Lemma~\ref{l:kpri}, we see that the three paths $R_1, R_2, R_3$ form a
prism that is output.  All this take time $O(n^8)$.

The method is very beautiful and simple; detecting pyramids in
$O(n^9)$ is more involved.  With several tricks, the detection of
prisms in graphs with no pyramids can be performed in $O(n^5)$, see~\cite{maffray.t:reco}.  Shortest path detectors solve several
other problems for graphs with no odd holes, see~\cite{maffray.t:reco}, such as the detection of line-graphs of
bipartite subdivisions of $K_4$ in $O(n^{18})$, even prisms in
$O(n^9)$ and odd prisms in $O(n^{18})$ (where a prism is \emph{even}
(resp.\ \emph{odd}) when the three path forming it are of even (resp.\
even) length).  These problems are all of interest because these three
substructures are used in the proof of the Strong Perfect Graph
Conjecture, so detecting them could help to read the proof as an
algorithm.  Such an algorithm would take any non-basic graph as input
and would output either and odd hole or and odd antihole or some
decomposition (here \emph{basic} and \emph{decomposition} is to be
understood with respect to some decomposition theorem for Berge
graphs, see Chapter~\ref{chap:Berge}).  Also, classes of graphs of
interest exclude prism or odd prisms, such as Artemis graphs,
see~\cite{maffray.t:reco,nicolas:these}.

\section{A tool for polynomiality (2): three-in-a-tree}
\label{sec:3-in-tree}

Chudnovsky and Seymour proved the following, referred as
\emph{three-in-a-tree}\index{three-in-a-tree}.

\begin{theorem}[Chudnovsky and Seymour
  \cite{chudnovsky.seymour:theta}]
  Let $G$ be a graph and $a, b, c$ three vertices.  Deciding whether
  $G$ contains a tree that goes through $a, b, c$ can be performed in
  time $O(n^4)$.
\end{theorem}

This might seem trivial to people familiar with basics in Graph Theory,
because of a famous theorem stating that any connected graph admits a
spanning tree.  But this theorem has no ``induced constraint'' and
instances where the answer to three-in-tree is ``no'' can be obtained
by gluing three paths to a triangle.  To see that all this is not
trivial, the reader can try the following exercises (of course the
solution relies on Bienstock's construction).

\begin{exercise}
  Three-in-a-path is NP-complete.
\end{exercise}

\begin{exercise}
  Three-in-the-line-graph-of-a-tree is NP-complete. 
\end{exercise}

To solve three-in-a-tree, Chudnovsky and Seymour describe these graphs
where no tree covers three vertices $a, b, c$.  The main ingredient of
their quite complicated description is line-graphs. Because since
line-graphs contain no claws, any induced tree of a line-graph must be
a path.  So, by adding three pending vertices to a line-graph, we
obtain easily many instance of three-in-a-tree for which the answer is
no.  Three-in-a-tree can be used to solve many problems: detecting a
theta in time $O(n^{11})$ (this was the original motivation of
Chudnovsky and Seymour, related to the recognition of even-hole-free
graphs), detecting a pyramid (this is slower than with shorter path
detectors, $O(n^{10})$) and also $\Pi_{B_4}$ and $\Pi_{B_6}$.  Here
below we give an example of how three-in-a-tree can be applied,
together with an NP-completeness proof.

\begin{theorem}[with L\'ev\^eque and Maffray~\cite{leveque.lmt:detect}]
  \label{th:antena}
  There is an $O(n^{13})$-time algorithm for $\Pi_{B_4}$ but
  $\Pi_{B_5}$ is NP-complete.
\end{theorem}

\begin{proof}
  A realisation of $B_4$ has exactly one vertex of degree 3 and one
  vertex of degree 4.  Let us say that the realisation $H$ is
  \emph{short} if the distance between these two vertices in $H$ is at
  most~3.  Detecting short realisations of $B_4$ can be done in time
  $n^9$ as follows: for every 6-tuple $F = (a, b, x_1, x_2, x_3, x_4)$
  such that $G[F]$ has edge-set $\{x_1a, ax_2, x_2b, bx_3, bx_4\}$ and
  for every 7-tuple $F = (a, b, x_1, x_2, x_3, x_4, x_5)$ such that
  $G[F]$ has edge-set $\{x_1a, ax_2, x_2x_3, x_3b, bx_4, bx_5\}$,
  delete $x_1, \dots, x_5$ and their neighbors except $a,b$.  In the
  resulting graph, check whether $a$ and $b$ are in the same
  component.  The answer is YES for at least one 7-or-6-tuple if and
  only if $G$ contains at least one short realisation of ${B_4}$.

  Here is an algorithm for $\Pi_{B_4}$, assuming that the entry graph
  $G$ has no short realisation of $B_4$.  For every 9-tuple $F = (a,
  b, c, x_1, \dots, x_6)$ such that $G[F]$ has edge-set $\{x_1a, bx_2,
  x_2x_3, x_3x_4, cx_5, x_5x_3, x_3x_6\}$ delete $x_1, \dots, x_6$ and
  their neighbors except $a, b, c$.  In the resulting graph, run
  three-in-a-tree for $a, b, c$.  It is easily checked that the answer
  is YES for some 9-tuple if and only if $G$ contains a realisation of
  $B_4$.

  Let us prove that $\Pi_{B_5}$ is NP-complete by a reduction of
  $\Gamma^3$ to $\Pi_{B_5}$.  Since by Theorem~\ref{th:delta3},
  $\Gamma^3$ is NP-complete, this will complete the proof.  Let $(G,
  x, y)$ be an instance of $\Gamma^3$.  Prepare a new graph $G'$: add
  four vertices $x', x'', y', y''$ to $G$ and add four edges $xx',
  xx'', yy', yy''$.  Since $\Delta(G) \leq 3$, it is easily seen that
  $G$ contains a hole passing through $x, y$ if and only if $G'$
  contains a realisation of $B_5$.
\end{proof}

From the proof above, it is easy to see that three-in-tree can be
applied to any situation where the three subdivisible edges of an
s-graph have a common end.  Here is a dichotomy theorem which is a
further evidence of the generality of three-in-a-tree.

\begin{theorem}[with L\'ev\^eque and Maffray~\cite{leveque.lmt:detect}]
  \label{th:class}
  Let $H$ be a connected graph.  Then one of the following holds:
  \begin{itemize}
  \item $H$ is a path or a subdivision of a claw and $\Gamma_{\{H\}}$
    is polynomial.
    \item $H$ contains one of $K_{1, 4}$, $I_k$ for some $k\geq 1$, or
      $C_l$ for some $l\geq 3$ as an induced subgraph and
      $\Gamma_{\{H\}}$ is NP-complete.
  \end{itemize}
\end{theorem}
\begin{proof}
  Let us first prove the following. 

  \begin{claim}
    \label{cl:clawpoly}
    Let $H$ be a graph on $k$ vertices that is either a path or a
    subdivision of a claw.  There is an $O(n^{k})$-time algorithm for
    $\Gamma_{\{H\}}$.
  \end{claim}
  
  \begin{proofclaim}
    Here is an algorithm for $\Gamma_{\{H\}}$.  Let $(G, x, y)$ be an
    instance of $\Gamma_H$.  If $H$ is a path on $k$ vertices then
    every hole in $G$ is on at most $k$ vertices.  Hence, by a
    brute-force search on every $k$-tuple, we will find a hole through
    $x, y$ if there is any.  Now we suppose that $H$ is a subdivision
    of a claw.  So $k\ge 4$.  For convenience, we put $x_1=x$,
    $y_1=y$.  Let $x_0, x_2$ (resp.  $y_0, y_2$) be the two neighbors
    of $x_1$ (resp.  $y_1$).

    First check whether there is in $G$ a hole $C$ through $x_1, y_1$
    such that the distance between $x_1$ and $y_1$ in $C$ is at most
    $k-2$.  If $k=4$ or $k=5$ then $\{x_0, x_1, x_2, y_0, y_1, y_2\}$
    either induces a hole (that we output) or a path $P$ that is
    contained in every hole through $x,y$.  In this last case, the
    existence of a hole through $x, y$ can be decided in linear time
    by deleting the interior of $P$, deleting the neighbors in
    $G\setminus P$ of the interior vertices of $P$ and by checking the
    connectivity of the resulting graph.  Now suppose $k\ge 6$.  For
    every $l$-tuple $(x_3, \dots, x_{l+2})$ of vertices of $G$, with
    $l\leq k-5$, test whether $P = x_0 \tp x_1 \tp \cdots \tp x_{l+2}
    \tp y_2 \tp y_1 \tp y_0$ is an induced path, and if so delete the
    interior vertices of $P$ and their neighbors except $x_0, y_0$,
    and look for a shortest path from $x_0$ to $y_0$.  This will find
    the desired hole if there is one, after possibly swapping $x_0,
    x_2$ and doing the work again.  This takes time $O(n^{k-3})$.

    Now we may assume that in every hole through $x_1,y_1$, the
    distance between $x_1, y_1$ is at least $k-1$.

    Let $k_i$ be the length of the unique path of $H$ from $u$ to
    $v_i$, $i=1, 2, 3$.  Note that $k = k_1 + k_2 + k_3 + 1$.  Let us
    check every $(k-4)$-tuple $z = (x_3, \dots, x_{k_1+1}, y_3, \dots,
    y_{k_2 + k_3})$ of vertices of $G$.  For such a $(k-4)$-tuple,
    test whether $x_0 \tp x_1 \tp \cdots \tp x_{k_1+1}$ and $P = y_0
    \tp y_1 \tp \cdots \tp y_{k_2 + k_3}$ are induced paths of $G$
    with no edge between them except possibly $x_{k_1 + 1}y_{k_2 +
      k_3}$.  If not, go to the next $(k-4)$-tuple, but if yes, delete
    the interior vertices of $P$ and their neighbors except $y_0,
    y_{k_2+k_3}$.  Also delete the neighbors of $x_2, \dots, x_{k_1}$,
    except $x_1, x_2, \dots, x_{k_1}, x_{k_1+1}$.  Call $G_z$ the
    resulting graph and run three-in-a-tree in $G_z$ for the vertices
    $x_1, y_{k_2 + k_3}, y_0$.  We claim that the answer to
    three-in-a-tree is YES for some $(k-4)$-tuple if and only if $G$
    contains a hole through $x_1, y_1$ (after possibly swapping $x_0,
    x_2$ and doing the work again).

    To prove this, first assume that $G$ contains a hole $C$ through
    $x_1, y_1$ then up to a symmetry this hole visits $x_0, x_1, x_2,
    y_2, y_1, y_0$ in this order.  Let us name $x_3, \dots, x_{k_1+1}$
    the vertices of $C$ that follow after $x_1, x_2$ (in this order),
    and let us name $y_3, \dots, y_{k_2+k_3}$ those that follow after
    $y_1, y_2$ (in reverse order).  Note that all these vertices exist
    and are pairwise distinct since in every hole through $x_1, y_1$
    the distance between $x_1, y_1$ is at least $k-1$.  So the path
    from $y_0$ to $y_{k_2+k_3}$ in $C\setminus y_1$ is a tree of $G_z$
    passing through $x_1, y_{k_2 + k_3}, y_0$, where $z$ is the
    $(k-4)$-tuple $(x_3, \dots, x_{k_1+1}, y_3, \dots, y_{k_2 +
      k_3})$.

    Conversely, suppose that $G_z$ contains a tree $T$ passing through
    $x_1, y_{k_2 + k_3}, y_0$, for some $(k-4)$-tuple $z$.  We suppose
    that $T$ is vertex-inclusion-wise minimal.  If $T$ is a path
    visiting $y_0, x_1, y_{k_2 + k_3}$ in this order, then we obtain
    the desired hole of $G$ by adding $y_1, y_2, \dots, y_{k_2+k_3-1}$
    to $T$.  If $T$ is a path visiting $x_1, y_0, y_{k_2 + k_3}$ in
    this order, then we denote by $y_{k_2+k_3+1}$ the neighbor of
    $y_{k_2 + k_3}$ along $T$.  Note that $T$ contains either $x_0$ or
    $x_2$.  If $T$ contains $x_0$, then there are three paths in $G$:
    $y_0 \tp T \tp x_0 \tp x_1 \tp \cdots \tp x_{k_1}$, $y_0 \tp T \tp
    y_{k_2+k_3+1} \tp \cdots \tp y_{k_3+2}$ and $y_0 \tp y_1 \tp
    \cdots \tp y_{k_3}$.  These three paths form a subdivision of a
    claw centered at $y_0$ that is long enough to contain an induced
    subgraph isomorphic to $H$, a contradiction.  If $T$ contains
    $x_{2}$ then the proof works similarly with $y_0 \tp T \tp
    x_{k_1+1} \tp x_{k_1} \tp \cdots \tp x_{1}$ instead of $y_0 \tp T
    \tp x_0 \tp x_1 \tp \cdots \tp x_{k_1}$.  If $T$ is a path
    visiting $x_1, y_{k_2 + k_3}, y_0$ in this order, the proof is
    similar, except that we find a subdivision of a claw centered at
    $y_{k_2+k_3}$.  If $T$ is not a path, then it is a subdivision of
    a claw centered at a vertex $u$ of $G$.  We obtain again an
    induced subgraph of $G$ isomorphic to $H$ by adding to $T$
    sufficiently many vertices of $\{x_0, \dots x_{k_1+1}, y_0, \dots,
    y_{k_2+k_3}\}$.
  \end{proofclaim}

  If $H$ contains one of $K_{1, 4}$, $I_k$ for some $k\geq 1$, or
  $C_l$ for some $l\geq 3$ as an induced subgraph then
  $\Gamma_{\{H\}}$ is NP-complete by Theorem~\ref{th:bienstock}.
  Else, $H$ is a tree since it contains no $C_l$, $l\geq 3$.  If $H$
  has no vertex of degree at least~$3$, then $H$ is a path and
  $\Gamma_{\{H\}}$ is polynomial by~(\ref{cl:clawpoly}).  If $H$ has a
  single vertex of degree at least~$3$, then this vertex has degree 3
  because $H$ contains no $K_{1, 4}$.  So, $H$ is a subdivision of a
  claw and $\Gamma_{\{H\}}$ is polynomial by~(\ref{cl:clawpoly}).  If
  $H$ has at least two vertices of degree at least~3 then $H$ contains
  an $I_l$, where $l$ is the minimum length of a path of $H$ joining
  two such vertices.  This is a contradiction.
\end{proof}

Interestingly, the following analogous result for finding maximum
stable sets in $H$-free graphs was proved by Alekseev:

\begin{theorem}[Alekseev, \cite{alekseev:83}]
  \label{th:alekseev}
  Let $H$ be a connected graph that is not a path nor a subdivision of a
  claw.  Then the problem of finding a maximum stable set in $H$-free
  graphs is NP-hard.
\end{theorem}

But the complexity of the maximum stable set problem is not known in
general for $H$-free graphs when $H$ is a path or a subdivision of a
claw. See~\cite{hertz.lozin:survey} for a survey.

Three-in-a-tree seems to be more general than shortest path detectors.
For instance, shortest path detectors seem to be useless for
detecting a theta or to prove the dichotomy Theorem~\ref{th:class}.
But when there are parity constraints, shortest path detectors are
often the only known method to solve a problem (like these mentioned
at the end of Section~\ref{sec:shortPD}).  This leads us to the
following question.

\begin{question}
  Is there a variant on three-in-a-tree with parity constraints?
\end{question}

A natural question now is whether three-in-a-tree and Bienstock's
construction can decide the complexity of any problem $\Pi_B$ where
$B$ is an s-graph.  With L\'ev\^eque and
Maffray~\cite{leveque.lmt:detect}, we tried to check all instances on
four vertices.  After eliminating symmetries and trivial instances, we
found twelve interesting s-graphs.

For the following two s-graphs, there is a polynomial algorithm using
three-in-a-tree.  The two algorithms are essentially similar to those
for thetas and pyramids (see Figure~\ref{fig:ppt}).

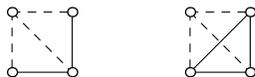
\noindent
\begin{center} 
  \includegraphics{bigraphs.3} 
  \rule{3em}{0ex}
  \includegraphics{bigraphs.8} 
\end{center}

\noindent The next two s-graphs yield NP-complete problems:

\noindent
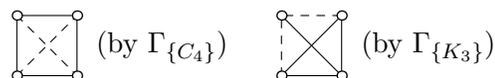
\begin{center} 
  \parbox[c]{1cm}{\includegraphics{bigraphs.4}} (by $\Gamma_{\{C_4\}}$)
  \rule{1em}{0ex}
  \parbox[c]{1cm}{\includegraphics{bigraphs.5}} (by $\Gamma_{\{K_3\}}$)
\end{center}

\noindent For the next seven s-graphs on four vertices, we could not get
an answer, so they are open problems.  We tried quite hard for the
last one, but we could obtain only partial results, see
Section~\ref{sec:ISK4}.

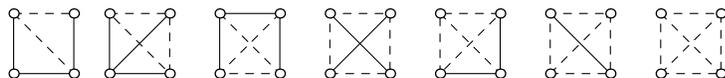
\noindent
\begin{center}
\includegraphics{bigraphs.1}\rule{.5em}{0ex}
\includegraphics{bigraphs.6}\rule{1em}{0ex}
\includegraphics{bigraphs.7}\rule{1em}{0ex}
\includegraphics{bigraphs.9}\rule{1em}{0ex}
\includegraphics{bigraphs.10}\rule{1em}{0ex}
\includegraphics{bigraphs.11}\rule{1em}{0ex}
\includegraphics{bigraphs.12}
\end{center}

\noindent For the last graph represented below, we could not obtain an
answer by three-in-a-tree or Bienstock's construction.  But with Vu\v
skovi\'c~\cite{nicolas.kristina:one}, we could prove that the problem
can be solved in time $O(nm)$, using a method based on
decomposition.  See Section~\ref{sec:oneChord}.  

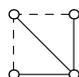
\noindent
\begin{center}
\includegraphics{bigraphs.2}\rule{1em}{0ex}
\end{center}

\section{Generalizing three-in-a-tree}
\label{sec:gen3}

Because of the power and deepness of three-in-a-tree, it would be
interesting to generalise it.  The obvious extension would be
\emph{four-in-a-tree}, or better, $k$-in-a-tree.  But this seems quite
involved.  Most of the trouble comes from small cycles as suggested by
the following exercise.

\begin{exercise}
  Let $G$ be graph of girth at least $k+1$ and let $x_1, \dots, x_k$
  be vertices of $G$.  Let $H$ be an induced subgraph of $G$ that
  contains $x_1, \dots, x_k$ and that is inclusion-wise minimal with
  respect to this property.  Show that $H$ is a tree.
\end{exercise}

So, jointly with Dehry and Picouleau, and later jointly with Liu Wei,
we tried to see whether we could get results by excluding small
cycles.  Finally, we investigated something slightly more general than
the exercise above, $k$-in-a-tree in graphs of girth at least $k$.
For $k=3$, this is a rephrasing of three-in-a-tree, so from here on,
we assume $k\geq 4$.  Our approach is similar to that of Chudnovsky
and Seymour for three-in-a-tree.  We give a structural answer to the
following question: what does a graph of girth at least $k$ look like
if no induced tree covers $k$ given vertices $x_1, \dots, x_k$?  We
were expecting very regular and easy structures, but interestingly,
for $k=4$ and $k=6$, sporadic structures popped up from the proofs.
On Figures~\ref{fig:cubeEx}, \ref{fig:squareEx}, \ref{f:k4s}
and~\ref{f:cycle} five examples of graphs with no tree covering the
``$x$'' vertices are represented.  The sequel will show that in a
sense, any graph of girth at least $k$ where no induced tree covers
$k$ given vertices $x_1, \dots, x_k$ looks like one of these five
examples.

\begin{figure}[p]
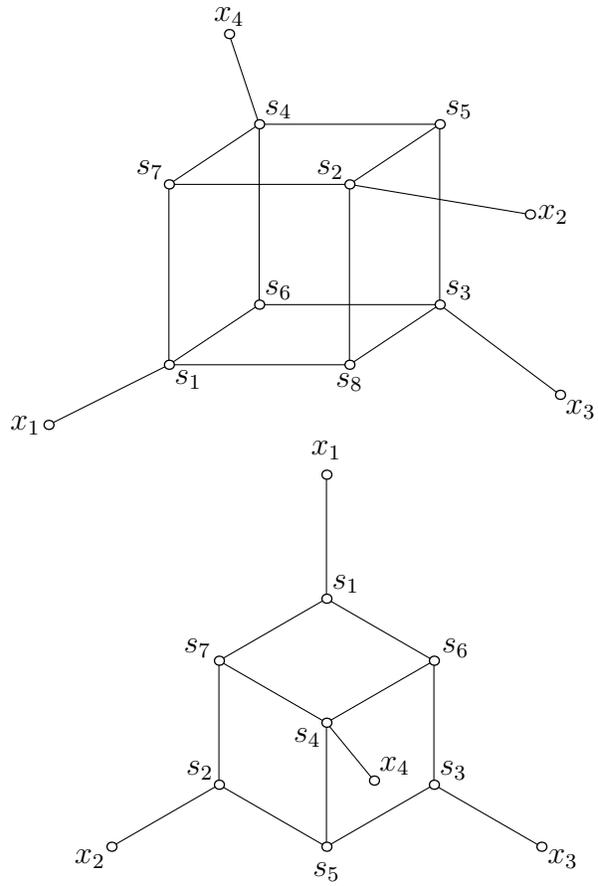

  \center
  \includegraphics{figFtree.1}
  \rule{0.5cm}{0cm}
 \includegraphics{figFtree.3}
 \caption{Two examples of cubic structure\label{fig:cubic}\label{fig:cubeEx}}
\end{figure}

\begin{figure}[p]
  \center
  \includegraphics{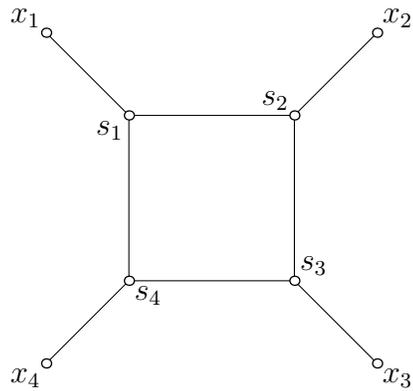}
  \caption{The smallest square structure\label{fig:square}\label{fig:squareEx}}
\end{figure}

\begin{figure}[p]
  \begin{center}
    \includegraphics{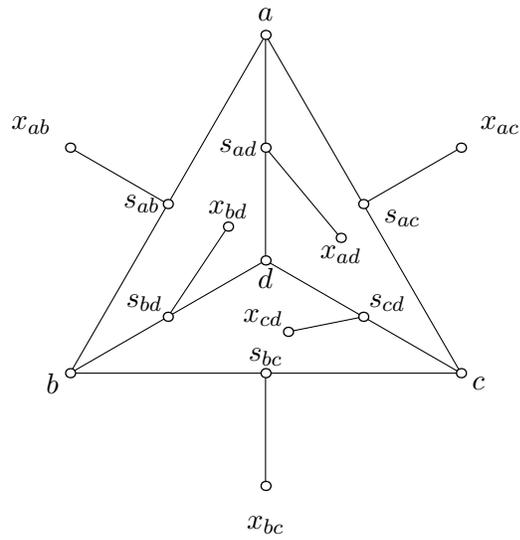}
    \caption{a $K_4$-structure \label{f:k4s}}
  \end{center}
\end{figure}

\begin{figure}[p]
\begin{center}
\includegraphics{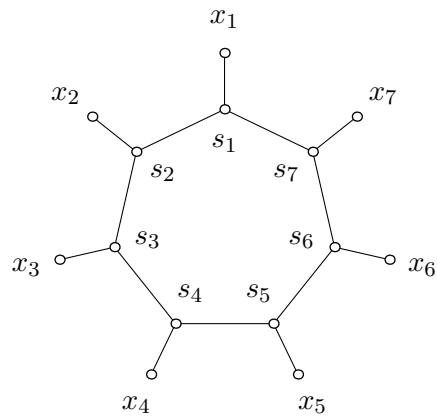}
\caption{a $k$-structure ($k=7)$\label{f:cycle}}
\end{center}
\end{figure}

We call \emph{terminal}\index{terminal} of a graph any vertex of degree one.  Solving
$k$-vertices-in-a-tree or $k$-terminals-in-a-tree are equivalent
problems.  Because if $k$ vertices $x_1, \dots, x_k$ of graph $G$ are
given, we build the graph $G'$ obtained from $G$ by adding a pending
neighbor $y_i$ to $x_i$, $i=1, \dots, k$.  An induced tree of $G$
covers $x_1, \dots, x_k$ if and only an induced tree of $G'$ covers
$y_1, \dots, y_k$.  So, in the rest of the section we assume for
convenience that the vertices to be covered are all terminals.  A
\emph{branch-vertex}\index{branch-vertex} is a vertex of degree at least~3.  The following
is a basic fact whose proof is omitted.

\begin{lemma}
  \label{l:descTree}
  A tree $T$ with $k$ terminals contains at most $k-2$
  branch-vertices.  Moreover if $T$ contains exactly $k-2$
  branch-vertices then every branch-vertex is of degree 3.
\end{lemma}

The key lemma to our results is the following. 

\begin{lemma}[with Liu Wei \cite{nicolas.wei:kTree}]
  \label{l:linkTree}
  Let $k, l$ be integers such that $k\geq 3$ and $2 \leq l \leq k$.
  Let $G$ be a graph of girth at least $k$ and $x_1, \dots, x_l$ be
  $l$ distinct terminals of $G$.  Let $T$ be an induced tree of $G$
  whose terminals are $x_1, \dots, x_{l-1}$.  Let $Q$ be a path from
  $x_l$ to $w$ such that $w$ has at least one neighbor in $T$ and no
  vertex of $Q\sm w$ has neighbors in $T$. Then one and only one of
  the following outcomes holds:

  \begin{itemize}
  \item $T \cup Q$ contains a tree of $G$ that covers $x_1, \dots,
    x_l$.
  \item $k=l$. Moreover, $T$ and $Q$ can be described as follows (up
    to a relabelling of $x_1, \dots, x_{k-1}$):
    \begin{enumerate}
    \item $T$ is the union of $k-1$ vertex-disjoint paths $s_1 \tp
      \cdots \tp x_1$, $s_2 \tp \cdots \tp x_2$, \dots, $s_{k-1} \tp
      \cdots \tp x_{k-1}$;
    \item the only edges between these paths are such that $s_1\tp s_2
      \tp \cdots \tp s_{k-1}$ is a path;
    \item $N_T(w) = \{s_1, s_{k-1}\}$.
    \end{enumerate}
  \end{itemize}

  This is algorithmic in the sense that when $T$ and $Q$ are given,
  the tree of the first outcome or the relabelling of the second can
  be computed in time $O(n^3)$.
\end{lemma}

\begin{proof}
  Clearly, at most one of the outcomes holds (because if the second
  holds then no tree of $T\cup Q$ can cover $x_1, \dots, x_l$).  Let
  us prove that at least one of the outcomes holds.
  
  Let $W = \{w_1, \dots, w_i\}$ be the set of the neighbors of $w$ in
  $T$.  If $i=1$ then $T\cup Q$ is a tree that covers $x_1, \dots,
  x_{l}$ so let us suppose that $i\geq 2$.  Let us call \emph{basic
    path} any subpath of $T$ linking two distinct vertices of $W$ and
  with no interior vertices in $W$.  All the basic paths are on at
  least $k-1$ vertices because the girth of $G$ is at least $k$.  Now
  we consider two cases:

  \noindent{\bf Case 1:} for all basic paths $R$ of $T$ there exists
  an interior vertex $v_R$ of $R$ that has degree two in $T$.  Then,
  let $S \leftarrow T\cup Q$.  For all basic paths $R$, if $R\subseteq
  S$, then let $v_R$ be a vertex of degree two (in $T$) of $R$, let $S
  \leftarrow S\sm \{v_R\}$ and go the next path $R$.  At the end of
  this loop, one vertex of degree two is deleted from all basic paths.
  So, $S$ contains no more cycle, but is still connected because the
  deleted vertices have all degree~2 and exactly one is deleted in
  each basic path.  Hence, we obtain a tree $S$ that covers $x_1,
  \dots, x_{l}$.  This takes time $O(n^3)$ because we enumerate all
  the pairs $w_i, w_j$ to find the basic paths.

  \noindent{\bf Case 2:} we are not in Case~1, so there exists a basic
  path $R$ whose interior vertices are all of degree at least~3 in
  $T$.  Then, since $T$ has $l-1$ terminals, Lemma~\ref{l:descTree}
  says that it has at most $l-3$ branch-vertices.  On the other hand,
  since a basic path is on at least $k-1$ vertices (because the girth
  is at least~$k$), $R$ contains at least $k-3$ branch-vertices of
  $T$.  So in fact, because $l\leq k$, we have $k=l$ and $R$ contains
  all the $k-3$ branch-vertices of $T$.  Since $R$ has no interior
  vertex of degree~2, in fact $R$ contains $k-1$ vertices.  We name
  $s_1,\cdots ,s_{k-1}$ the vertices of $R$.  Note that $w$ is
  adjacent to $s_1$ and $s_{k-1}$ because $R$ is a basic path.  In
  particular, $s_1$ and $s_{k-1}$ are not terminals of~$G$.

  For all $1\leq i \leq k-1$, $s_i$ is a cutvertex of $T$ that
  isolates one terminal among $x_1, \dots, x_{k-1}$ from all the other
  terminals.  Up to a relabelling, we suppose that this terminal is
  $x_i$.  We name $P_i$ the unique path of $T$ between $x_i$ and
  $s_i$.

  Note that $w$ is not adjacent to $s_2, \dots, s_{k-2}$ (because $R$
  is a basic path).  So the second outcome of our lemma holds, unless
  $w$ has at least one neighbor in some $P_i \sm s_i$.  For $i=1,
  \dots, k-1$ , we let $s'_i$ be the neighbor of $s_i$ along $P_i$, if
  $w$ has a neighbor in $P_i$ then we name $w_i$ the neighbor of $w$
  closest to $x_i$ along $P_i$ and if no such neighbor exists, we put
  $w_i = s_i$.

  Suppose that for all $i=1, \dots, k-1$ we have $w_i\neq s'_i$.
  Then, the paths $x_i \tp P_i \tp w_i$, $i=1, \dots, k-1$ together
  with $Q$ and $s_1, \dots, s_{k-1}$ form a graph with a unique cycle:
  $w s_1 \dots s_{k-1}w$.  By deleting a vertex $s_j$ such that
  $w_j\neq s_j$, we obtain a tree that covers $x_1, \dots, x_k$.

  So, we may assume that for some $i$, $w_i=s'_i$ and up to symmetry
  we suppose $i\leq k/2$.  Then $ws_1\dots s_i s'_iw$ is a cycle on
  $i+2$ vertices, so $i+2\geq k$ because of the girth.  So, $k-2\leq
  k/2$, so $k\leq 4$.  Then the paths $x_i \tp P_i \tp w_i$, $i=1,
  \dots, k-1$, together with $Q$ form a tree that covers $x_1, \dots,
  x_k$.
\end{proof}

A graph is a \emph{$k$-structure}\index{K@$k$-structure} with respect to $k$ distinct
terminals $x_1, \dots, x_k$ if it is made of $k$ vertex-disjoints paths
of length at least one $P_1 = x_1 \tp \cdots \tp s_1$, \dots, $P_k =
x_k\tp \cdots \tp s_k$ such that the only edges between them are
$s_1s_2$, $s_2s_3$, \dots, $s_{k-1}s_k$, $s_ks_1$.

\begin{lemma}[with Liu Wei \cite{nicolas.wei:kTree}]
  \label{l:firststep}
  Let $k\geq 3$ be an integer.  Let $G$ be a connected graph of girth
  at least $k$ and $x_1, \dots, x_l$ be $l$ terminals where $1\leq l \leq
  k$.  Then either $G$ contains a tree that covers the $l$ terminals
  or $l=k$ and $G$ contains a $k$-structure with respect to $x_1,
  \dots, x_k$.

  This is algorithmic in the sense that we provide an $O(n^4)$
  algorithm that finds the tree or the $k$-structure. 
\end{lemma}

\begin{proof}
  We suppose that $k$ is fixed and we prove the statement by induction
  on $l$.  For $l=1$ and $l=2$, the lemma is clear: a tree exists (for
  instance, a shortest path linking the two terminals).  Suppose the
  lemma holds for some $l-1<k$ and let us prove it for~$l$.  By the
  induction hypothesis there exists an induced tree $T$ of $G$ that
  covers $x_1, \dots, x_{l-1}$.  Let $Q$ be a path from $x_l$ to some
  vertex $w$ that has neighbors in $T$, and suppose that $Q$ is
  minimal with respect to this property.  Then, no vertex of $Q\sm w$
  has a neighbor in $T$.

  We apply Lemma~\ref{l:linkTree}.  If the first outcome holds, we
  have our tree.  Else, $T\cup Q$ is a $k$-structure.  All this can be
  implemented in time $O(n^4)$ because terminals are taken one by one,
  there are at most $n$ of them and for each of them we rely on basic
  subroutines like BFS (Breadth First Search, see~\cite{gibbons:agt})
  to find $Q$ and on the $O(n^3)$ algorithm of Lemma~\ref{l:linkTree}.
\end{proof}

Now, the idea to solve $k$-in-a-tree is the following.  Run the
algorithm described in the Lemma above.  This gives either the desired
tree or a $k$-structure.  In this last case, take all the remaining
vertices, and add them one by one.  In a series of four lemmas, we
show that these vertices attach in way that grows a kind of structure
certifying that no tree exists, or gives a tree.  We omit the proofs
of the three first lemmas because they are of the same kind: classical
attachment lemmas (we leave the fourth one which the shortest).

The four structures that we need correspond to the four
figures~\ref{fig:cubeEx}, \ref{fig:squareEx}, \ref{f:k4s}
and~\ref{f:cycle}.  They are respectively \emph{cubic structure},
\emph{square structure}, \emph{$K_4$-structure that decomposes the
  graph} and \emph{$k$-structure that decomposes the graph}.

\subsection*{Cubic structure}

A graph that has the same structure as the graph represented on
Fig~\ref{fig:cubeEx} is what we call a cubic structure: a graph $G$ is
said to be a \emph{cubic structure}\index{cubic structure} with respect to a 4-tuple of
distinct terminals $(x_1, x_2, x_3, x_4)$ if there exist sets $A_1,
\dots A_4, $ $B_1, \dots B_4$, $S_1, \dots, S_8$ and $R$ such that:

\begin{enumerate}
  \item
    $A \cup B \cup S \cup R = V(G)$;
  \item
    $A_1, \dots, A_4, B_1, \dots, B_4, S_1, \dots, S_8, R$ are
    pairwise disjoint;
  \item
    $x_i \in A_i$, $i= 1, \dots, 4$;
  \item
    $S_i$ is a stable set, $i= 1, \dots, 8$;
  \item
    $S_i$ is non-empty, $i= 1, \dots, 4$;
  \item
    \label{i:cubeNe}
    at most one of $S_5, S_6, S_7, S_8$ is empty;
  \item
    $S_i$ is complete to $(S_5 \cup S_6 \cup S_7 \cup S_8) \sm
    S_{i+4}$, $i= 1, 2, 3, 4$;
  \item
    $S_i$ is anticomplete to $S_{i+4}$, $i= 1, 2, 3, 4$;
  \item
    $S_i$ is anticomplete to $S_j$, $1 \leq i < j \leq 4$;
  \item
    $S_i$ is anticomplete to $S_j$, $5 \leq i < j \leq 8$;
  \item
    \label{i:cubeNAi}
    $N(A_i) = S_i$, $i= 1, 2, 3, 4$;
  \item
    \label{i:cubeNBi}
    $N(B_i) \subseteq S_i \cup N_S(S_i)$,  $i= 1, 2, 3, 4$;
  \item
    \label{i:cubeNR}
    $N(R) \subseteq S_5 \cup S_6 \cup S_7 \cup S_8$;
  \item
    \label{i:cubeACon}
     $G[A_i]$ is connected, $i= 1, 2, 3, 4$.
\end{enumerate}

A 17-tuple $(A_1, \dots A_4, $ $B_1, \dots B_4$, $S_1, \dots, S_8, R)$
of sets like in the definition above is a \emph{split}\index{split!of
  a cubic structure} of the cubic
structure.  On Figure~\ref{fig:cubic}, two cubic structures are
represented.  It is a routine matter to check that a cubic structure
contains no tree covering $x_1$, $x_2$, $x_3$ and $x_4$.  A \emph{cubic
  structure of a graph $G$} is a subset $Z$ of $V(G)$ such that $G[Z]$
is a cubic structure.  The following lemma whose proof is omitted shows
that if a cubic structure is discovered in a triangle-free graph, then
one can repeatedly add vertices to it, unless at some step a tree
covering $x_1, x_2, x_3, x_4$ is found.

\begin{lemma}[with Dehry and Picouleau \cite{nicolas.d.p:fourTree}]
  \label{l:cube}
  There is an algorithm with the following specification:

  \begin{description}
  \item
    {\sc Input:} a triangle-free graph $G$, four terminals $x_1, x_2,
    x_3, x_4$, a split of a cubic structure $Z$ of $G$, and a vertex
    $v\notin Z$.
  \item
    {\sc Output:} a tree of $G[Z \cup \{v\}]$ that covers $x_1, x_2,
    x_3, x_4$ or a split of the cubic structure $G[Z \cup \{v\}]$.
  \item
    {\sc Complexity: } $O(m)$.
  \end{description}
\end{lemma}

\subsection*{Square structure}

A graph that has the same structure as the graph represented on
Figure~\ref{fig:squareEx} is what we call a square structure: a graph
$G$ is said to be a \emph{square structure}\index{square structure} with respect to a 4-tuple
$(x_1, x_2, x_3, x_4)$ of distinct terminals if there are sets $A_1,
A_2, A_3, A_4, S_1, S_2, S_3, S_4, R$ such that:

\begin{enumerate}
  \item
    $A \cup S \cup R = V(G)$;
  \item
    $A_1, A_2, A_3, A_4, S_1, S_2, S_3, S_4, R$ are pairwise disjoint;
  \item
    $x_i \in A_i$, $i= 1, \dots, 4$;
  \item
    $S_i$ is a stable set, $i= 1, \dots, 4$;
  \item
    $S_1, S_2, S_3, S_4 \neq \emptyset$;
  \item
    $S_i$ is complete to $S_{i+1}$, where the addition of subscripts is
    taken modulo 4, $i= 1, 2, 3, 4$;
  \item
    $S_i$ is anticomplete to $S_{i+2}$, $i= 1, 2$;
  \item
    \label{i:squareNAi}
    $N(A_i) = S_i$ , $i= 1, 2, 3, 4$;
  \item
    $N(R) \subseteq S_1 \cup S_2 \cup S_3 \cup S_4$;
  \item
    \label{i:squareACon}
    $G[A_i]$ is connected, $i= 1, \dots, 4$.
 \end{enumerate}

 A 9-tuple $(A_1, \dots A_4, S_1, \dots, S_4, R)$ of sets like in the
 definition above is a \emph{split}\index{split!of a square structure} of the square structure.  On
 Figure~\ref{fig:square}, the smallest square structure is
 represented.  A square structure contains no tree covering $x_1$,
 $x_2$, $x_3$ and $x_4$.  A \emph{square structure of a graph $G$} is
 a subset $Z$ of $V(G)$ such that $G[Z]$ is a square structure. The
 following lemma, whose proof is omitted, shows that if a square
 structure is discovered in a triangle-free graph, then one can
 repeatedly add vertices to it, unless at some step a cubic structure
 or a tree covering $x_1, x_2, x_3, x_4$ is found:

 \begin{lemma}[with Dehry and Picouleau \cite{nicolas.d.p:fourTree}]
   \label{l:square}
  There is an algorithm with the following specification:

  \begin{description}
  \item
    {\sc Input:} a triangle-free graph $G$, four terminals $x_1, x_2,
    x_3, x_4$, a split of a square structure $Z$ of $G$, and a vertex
    $v \notin Z$.
  \item
    {\sc Output:} a tree of $G[Z \cup \{v\}]$ that covers $x_1, x_2,
    x_3, x_4$ or a split of some cubic structure of $G$ or a split of
    the square structure $G[Z \cup \{v\}]$.
  \item
    {\sc Complexity: } $O(m)$.
  \end{description}
\end{lemma}

\subsection*{$K_4$-structure that decomposes a graph}

A graph is a \emph{$K_4$-structure}\index{K@$K_4$-structure} with
respect to 6 distinct terminals $x_{ab}, x_{ac}, x_{ad}, x_{bc},
x_{bd}, x_{cd}$ if it is made of 6 vertex-disjoints path of length at
least one $P_{ab} = x_{ab}\tp \cdots \tp s_{ab}$, $P_{ac} = x_{ac}\tp
\cdots \tp s_{ac}$, $P_{ad} = x_{ad}\tp \cdots \tp s_{ad}$, $P_{bc} =
x_{bc}\tp \cdots \tp s_{bc}$, $P_{bd} = x_{bd}\tp \cdots \tp s_{bd}$,
$P_{cd} = x_{cd}\tp \cdots \tp s_{cd}$ and four vertices $a, b, c, d$
such that the only edges between them are $as_{ab}$, $as_{ac}$,
$as_{ad}$, $bs_{ab}$, $bs_{bc}$, $bs_{bd}$, $cs_{ac}$, $cs_{bc}$,
$cs_{cd}$, $ds_{ad}$, $ds_{bd}$, $ds_{cd}$.  See Figure~\ref{f:k4s}.
We put $X= \{x_{ab}, x_{ac}, x_{ad}, x_{bc}, x_{bd}, x_{cd}\}$.

We use the following ordering of the letters $a$, $b$, $c$, $d$: $a <
b < c < d$.  We say that a $K_4$-structure $K$ in a graph $G$
\emph{decomposes}\index{decompose!$K_4$-structure that} $G$ if the two following conditions hold:
\begin{enumerate}
\item\label{i:d1} for all $i, j$ such that $a\leq i < j \leq d$, $\{i, j\}$ is a
 cutset of $G$ that separates $x_{ij}$ from $X\sm \{x_{ij}\}$;
\item\label{i:d2} for all $i, j$ such that $a\leq i< j \leq d$,
  $\{s_{ij}\}$ is a cutset of $G$ that separates $x_{ij}$ from $X\sm
  \{x_{ij}\}$.
\end{enumerate}

When a $K_4$-structure decomposes a graph, then the graph contains no
tree covering $X$.  Because suppose that $H$ is a connected induced
subgraph of $G$ covering $X$. Then $H$ must contain at least three
vertices among $a, b, c, d$, because if it fails to contain two of
them, say $a, b$, then $x_{ab}$ is isolated from the rest of the graph
because of Condition~\ref{i:d1}.  So, we may assume that $H$ contains
$a, b, c$.  Also, because of Condition~\ref{i:d2}, $H$ must contain
$s_{ab}$, $s_{bc}$ and $s_{ac}$.  So, $H$ contains the cycle
$as_{ab}bs_{bc}cs_{ac}a$.  Hence, $H$ cannot be a tree.

\begin{lemma}[with Liu Wei \cite{nicolas.wei:kTree}]
  \label{l:k4}
  If a graph $G$ of girth 6 contains a $K_4$-structure $K$ with
  respect to 6 terminals $x_{ab}, x_{ac}, x_{ad}, x_{bc}, x_{bd},
  x_{cd}$ then one and only one of the following outcomes holds:
  \begin{itemize}
  \item $K$ decomposes $G$;
  \item $G$ contains a tree that covers $x_{ab}, x_{ac}, x_{ad},
    x_{bc}, x_{bd}, x_{cd}$.
  \end{itemize}
  This is algorithmic in the sense that if $K$ is given, testing
  whether $K$ decomposes $G$ or outputting the tree can be performed in
  time $O(n^4)$.
\end{lemma}

\subsection*{$k$-structure that decomposes the graph}

For $k$-structures, we assume that notation like in the definition is
used.  We put $X= \{x_1, \dots, x_k\}$.  We say that a $k$-structure
$K$ in a graph $G$ \emph{decomposes}\index{decompose!$k$-structure
  that} $G$ if for all $i$ such that $1\leq i \leq k$, $\{s_i\}$ is a
cutset of $G$ that separates $x_{i}$ from $X\sm \{x_{i}\}$.  When a
$k$-structure decomposes a graph, it is clear that the graph contains
no tree covering $X$.

\begin{lemma}[with Liu Wei \cite{nicolas.wei:kTree}]
  \label{l:k}
  Let $k\geq 5$ be an integer. If a graph $G$ of girth at least $k$
  contains a $k$-structure $K$ with respect to $k$ terminals $x_1,
  \dots, x_k$ then one of the following outcomes holds:
  \begin{itemize}
  \item $K$ decomposes $G$;
  \item $k=6$ and there exists a vertex $v$ of $G\sm K$ such that
    $K\cup\{v\}$ is a $K_4$-structure with respect to $x_1, \dots,
    x_6$;
  \item $G$ contains a tree that covers $X$.
  \end{itemize}
  This is algorithmic in the sense that testing whether $K$ decomposes
  $G$ or outputting the tree or outputting a relabelling showing that
  $K\cup\{v\}$ is a $K_4$-structure can be performed in time $O(n^4)$.
 \end{lemma}

\begin{proof}
  Let $H$ be an induced subgraph of $G$ that contains $K$ and such
  that $K$ decomposes $H$ ($H$ exists since $K$ decomposes $K$). We
  show that for any vertex $v$ of $G\sm H$, $H\cup \{v\}$ either is
  decomposed by $K$ or is a $K_4$-structure or contains a tree
  covering $X$.  This will prove the theorem by induction and be the
  description of an $O(n^4)$ algorithm since for each~$v$, the proof
  gives the way to actually build the tree or the relabelling by
  calling the algorithm of Lemma~\ref{l:linkTree} and searching the
  graph (with BFS for instance).  Note also that testing whether $K$
  decomposes some graph can be performed in time $O(km)$, or $O(nm)$
  since $k\leq n$, by $k$ checks of connectivity.

  Suppose that $H\cup \{v\}$ is not decomposed by $K$.  Let $Y$
  (resp. $Z$) be the connected component of $H\sm \{s_1\}$ that
  contains $x_{1}$ (resp. that contains $K' = K\sm P_{1}$).  Up to
  symmetry, we may assume that $v$ has a neighbor in $Y$ and a
  neighbor in $Z$.  Let $Q$ be a shortest path in $Y\cup Z \cup \{v\}$
  from $x_1$ to some vertex $w$ that has a neighbor in $K'$.  Note
  that $Q$ must go through~$v$.  Because $K'$ is a tree that covers
  $X\sm \{x_{1}\}$, we may apply Lemma~\ref{l:linkTree} to $K'$ and
  $Q$ in $Q \cup K'$.  So, either we find the tree or $w$ has exactly
  two neighbors in $K'$ and $N_{K'}(w)$ must be one of the following:
  $\{s_{2}, s_{k}\}$, $\{s_{2}, s'_{k-1}\}$, $\{s'_{3}, s_{k}\}$,
  $\{s'_{3}, s'_{k-1}\}$ where $s'_i$ denotes the neighbor of $s_i$
  along $P_i$.

  When $N_{K'}(w) = \{s_{2}, s_{k}\}$, we observe that $s_2s_1s_kw$ is
  a square, contradicting our assumption on the girth. 

  When $N_{K'}(w) = \{s_{2}, s'_{k-1}\}$ (or symmetrically $\{s'_{3},
  s_{k}\}$) then $w$ is not adjacent to $s'_1$ (else $s'_1s_1s_2w$ is
  a square).  If $w$ has a neighbor in $P_{1}$, we let $P$ be a
  shortest path from $w$ to $x_{1}$ in $P_{1} \cup \{w\}$.  Else, we
  let $P = P_{1}$.  We observe that $\{w\} \cup P \cup (K' \sm
  \{s_{k-1}\})$ is a tree that covers $X$.

  We are left with the case when $N_{K'}(w) = \{s'_{3}, s'_{k-1}\}$.
  Suppose first that $w$ has no neighbor in~$P_1$.  Then $\{w\} \cup K
  \sm \{s_3\}$ is a tree that covers $X$.  Suppose now that $w$ has a
  neighbor in $P_1\sm \{s_1, s'_1\}$.  We let $P$ be a shortest path
  from $w$ to $x_{1}$ in $\{w\} \cup (P_{1} \sm \{s_1, s'_1\})$.  If
  $ws_1 \notin E(G)$ then $P \cup \{s_1\} \cup (K\sm (P_1 \cup
  \{s_3\}))$ induces a tree that covers~$X$.  If $ws_1 \in E(G)$ then
  we observe that $P \cup \{s_1\} \cup (K\sm (P_1 \cup \{s_3,
  s_{k-1}\}))$ induces a tree that covers~$X$.  

  So we may assume that $N_{P_1}(w)$ is one of $\{s_1\}$,
  $\{s'_1\}$. If $N_{P_1}(w) = \{s_1\}$ then $s_1ws'_3s_3s_2$ is a
  $C_5$ so $k=5$ because of the girth assumption.  Hence $\{w\} \cup K
  \sm \{s_3, s_4\}$ is a tree that covers $X$.  So we are left with
  the case when $N_{P_1}(w) = \{s'_1\}$.  Then $ws'_1s_1s_2s_3s'_3$ is
  a $C_6$, so $k=5$ or $6$ because of the girth.  If $k=5$ then $\{w\}
  \cup K \sm \{s_3, s_4\}$ is a tree that covers $X$.  If $k=6$ then
  $K\cup \{w\}$ is a $K_4$-structure as shown by the following
  relabelling: $x_{ab} \leftarrow x_1$, $x_{ac} \leftarrow x_3$,
  $x_{ad} \leftarrow x_5$, $x_{bc} \leftarrow x_2$, $x_{bd} \leftarrow
  x_6$, $x_{cd} \leftarrow x_4$, $a \leftarrow w$, $b \leftarrow s_1$,
  $c \leftarrow s_3$, $d \leftarrow s_5$, $s_{ab} \leftarrow s'_1$,
  $s_{ac} \leftarrow s'_3$, $s_{ad} \leftarrow s'_5$, $s_{bc}
  \leftarrow s_2$, $s_{bd} \leftarrow s_6$, $s_{cd} \leftarrow s_4$.
\end{proof}

\subsection*{The main theorem}

\begin{theorem}[with Dehry and Picouleau \cite{nicolas.d.p:fourTree}
  and Liu Wei \cite{nicolas.wei:kTree}]
  \label{th:main}
  Let $k\geq 4$ be an integer.  Let $G$ be a connected graph of girth
  at least $k$ and $x_1, \dots, x_k$ be terminals of $G$.  Then one
  and only one of the following outcomes holds:

  \begin{itemize}
  \item $G$ contains $k$-structure $K$ with respect to $x_1, \dots,
    x_k$ and $K$ decomposes~$G$;
  \item $k=4$ and $G$ is a square structure or a cubic structure with
    respect to $x_1, x_2, x_3, x_4$;
  \item $k=6$, $G$ contains a $K_4$-structure $K$ with respect to
    $x_1, \dots, x_6$ and $K$ decomposes $G$;
  \item $G$ contains a tree covering $x_1, \dots, x_k$.
  \end{itemize}

  This is algorithmic in the sense that we provide an algorithm that
  output the tree or the structure certifying that no such tree exists
  in time $O(n^4)$.
\end{theorem}

\begin{proof}
  It is easy to check that at most one of the outcomes holds.  The
  description of the algorithm will show that at least one of the
  outcomes holds.  By Lemma~\ref{l:firststep}, we can output a tree
  covering $X$ or a $k$-structure of $G$ in time $O(n^4)$.

  When $k=4$, the 4-structure is in fact a square structure $Z$ of
  $G$.  While there exists a vertex $v$ not in $Z$, we use the
  algorithm of Lemma~\ref{l:square} to add $v$ to $Z$, keeping a
  square structure.  If we manage to put every vertex of $G$ in $Z$
  then we have found that $G$ is a square structure that we output.
  Else, Lemma~\ref{l:square} says that at some step we have found
  either a tree covering $x_1, x_2, x_3, x_4$ that we output, or a
  cubic structure $Z'$, together with a split for it. In this last
  case, we go to the next step: while there exists a vertex $v$ not in
  $Z'$, we use the algorithm of Lemma~\ref{l:cube} to add $v$ to $Z'$,
  keeping a cubic structure.  If we manage to put every vertex of $G$
  in $Z'$ then we have found that $G$ is a cubic structure that we
  output.  Else, Lemma~\ref{l:cube} says that at some step we have
  found a tree covering $x_1, x_2, x_3, x_4$ that we output.

  When $k\geq 5$, if a $k$-structure $K$ is output then by
  Lemma~\ref{l:k}, we can check whether $K$ decomposes $G$ (in which
  case no tree exists) or find a tree, or find a $K_4$-structure $K'$.
  In this last case, by Lemma~\ref{l:k4}, we can check whether $K'$
  decomposes $G$ or find a tree.
\end{proof}

\begin{theorem}
\label{th:algo}
Let $k\geq 3$ be an integer.  Let $G$ be a connected graph of girth at
least $k$ and $x_1, \dots, x_k$ be vertices of $G$.  Deciding whether
$G$ contains an induced tree covering $x_1, \dots, x_k$ can be
performed in time $O(n^4)$. 
\end{theorem}

\begin{proof}
  Follows from three-in-a-tree for $k=3$, from and from
  Theorem~\ref{th:main} for $k\geq 4$.
\end{proof}

The condition on the girth in this section is artificial.  It is here
only to help the proof to go smoothly.  It seems that the nastiest
cycles are these on 3, 4 and 6 vertices.  Indeed, these cycles can be
glued in ways that create no tree: triangles into line-graphs, squares
into a cube and hexagons into a $K_4$-structure.  Interestingly,
triangles, squares and hexagons are precisely these polygons that
yield regular tessellations of the plane (where a \emph{regular
  tessellation}\index{tessallation} is made up of congruent regular
polygons).  Is this coincidental?  Unfortunately, I know little about
combinatorial geometry, so I have no idea.  Anyway, a step toward an
understanding of $k$-in-a-tree could be the following.

\begin{question}
  Give a polynomial time algorithm for $k$-in-a-tree in graphs of
  girth at least~7. 
\end{question}

\subsection*{An NP-complete generalization of three-in-a-tree}

Another generalisation of three-in-a-tree would be interesting.  Let
us call \emph{centered tree}\index{centered tree} any tree that
contains at most one vertex of degree greater than two.  Note that any
minimal tree covering three vertices of a graph is centered.  Hence,
three-in-a-tree and three-in-a-centered-tree are in fact the same
problem.  So four-in-a-centered-tree is also an interesting
generalisation of three-in-a-tree.  But with Derhy and Picouleau, we
proved that it is NP-complete, even when restricted to several classes
of graphs, including triangle-free graphs.  The proof (slightly
simplified by Seymour) is by a simple reduction from $\Gamma_{K_3}$,
see~\cite{nicolas.d.p:fourTree}.

\chapter{Two decomposition theorems}
\label{chap:2dec}

At the end of Section~\ref{sec:3-in-tree}, twelve detection problems
are given, each coming from an s-graph on 4 vertices.  Four of these
problems are easily proved to be polynomial by three-in-a-tree or
NP-complete by Bienstock's construction.  I was hoping that by
studying the eight remaining ones, some general method would appear to
approach detection problems, but it turned out instead that as shown
by this chapter, each problem has its own flavour.  This was bad in a
sense, but the good news is that on two s-graphs,
interesting results were obtained:

\noindent
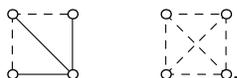
\begin{center}
\includegraphics{bigraphs.2}\rule{2.5em}{0ex}
\includegraphics{bigraphs.12}.
\end{center}

The first one yields a detection problem that can be rephrased as
\emph{the detection of cycles with a unique chord}.  After discussing
the eight problems with Maffray, we identified it as the potentially
easiest.  So, at the Oberwolfach meeting of march 2007, we discussed
it further with Chudnovsky, Seymour and Vu\v skovi\'c.  After a few
minutes, we thought we have solved it using a theorem of Conforti,
Cornu{\'e}jols, Kapoor and Vu{\v s}kovi{\'c} on a related class,
\emph{cap-free graphs}\index{cap-free graphs}
\cite{conforti.c.k.v:capfree}.  I was a bit disappointed that a
problem I was thinking of for a while was so easy, but several weeks
later I started to write the easy proof anyway, just to see whether
this could be the start of something.  It turned out that we were
completely mistaken: we were assuming that the graph contains a
triangle, a very easy subcase that can be done directly (see
Lemma~\ref{c:Ctriangle} below).  So, I emailed Vu\v skovi\'c to see
whether she could work on that, since we were looking for a good
subject to work together.  Six months later, we had a structure
theorem for graphs with no cycle with a unique chord, a recognition
algorithm, and results on coloring.  All this is explained in
Section~\ref{sec:oneChord}.

The second s-graph yields a detection problem that can be rephrased as
\emph{the detection of induced subdivision of $K_4$}.  We call
\emph{ISK4-free graphs}\index{ISK4} these graphs that do not contain a subdivision
of $K_4$.  In the discussion with Maffray on the eight s-graphs, this
one was identified as the most interesting one for many reasons.
First, it looks more natural than the others because of its
symmetries.  Moreover, detecting subdivisions of $K_3$ is trivial, and
detecting induced subdivisions of $K_5$ is NP-complete as mentioned in
Section~\ref{sec:bienstock}, so $K_4$ is in the middle of something.
Also, ISK4-free graphs generalize series-parallel graphs, so bounding
their chromatic number would be a kind of generalisation of Hajos and
Hadwiger conjectures which are the same easy statements for
$K_4$-minors.  But my main interest for ISK4-free graphs was something
less obvious.  They contain no pyramid just like Berge graphs, so some
ideas from the decomposition of Berge graphs can be applied to them.
Moreover, line-graphs of graph of maximum degree 3 form a basic class.
So, ideas from the decomposition of claw-free graphs can be applied
also.  These remarks was a real motivation for me because the
decomposition of Berge graphs and claw-free graph are very difficult
results.  Instead of spending months reading the papers from the first
line to the last one, it is sometimes better to work on easier
problems of the same flavour, trying to apply the ideas by your own.
ISK4-free graphs seemed to be a nice playground for that.

So, with L\'ev\^eque and Maffray, we started to work on ISK4-free
graphs.  We have been less lucky than I was with Vu\v skovi\'c for
cycles with a unique chord: we did find a decomposition theorem but
because of strong cutsets, we could not find a recognition algorithm
or a nice bound for $\chi$.  Yet, if we exclude \emph{wheels}, that
are holes plus a vertex adjacent to at least three vertices of the
hole, we obtain a very precise structure theorem.  Also, Alex Scott
noticed that our work implies a proof of Conjecture~\ref{conj:scott}
for $K_4$.  All this is explained Section~\ref{sec:ISK4}.

The last section is devoted to the wheel-free graphs, a class on which
we have open problems that could generalise results of the two first
sections.

\section{Cycles with a unique chord}
\label{sec:oneChord}

In this section we call $\cal C$ the class of graphs that do not
contain cycles with a unique chord.  The decomposition theorem that we
give now for $\cal C$ is a generalisation of Theorem~\ref{th:nochord}
since all chordless graphs are clearly in $\cal C$.  It has two
sporadic basic graphs: the Petersen and Heawood graphs.  Both these
graphs were discovered at the end of the XIXth century in the research
on the four color conjecture, see~\cite{petersen:98}
and~\cite{heawood:90}.  It was very nice and surprising to discover
these as basic graphs.  They share several properties (as being
\emph{cages}\index{cage} that are smallest cubic graphs of a given girth), but in
a sense we still do not \emph{understand} why they are in our theorem
(maybe there is nothing to understand).

The \emph{Petersen graph}\index{Petersen graph} is the graph on  $\{a_1, \dots a_5,
b_1, \dots, b_5\}$ so that $\{a_1, \dots, a_5\}$ and $\{b_1, \dots,
b_5\}$ both induce a $C_5$ with vertices in their natural order, and such
that the only edges between the $a_i$'s and the $b_i$'s are $a_1b_1$,
$a_2b_4$, $a_3b_2$, $a_4b_5$, $a_5b_3$. See Fig.~\ref{f:petersen}.

\begin{figure}[h]
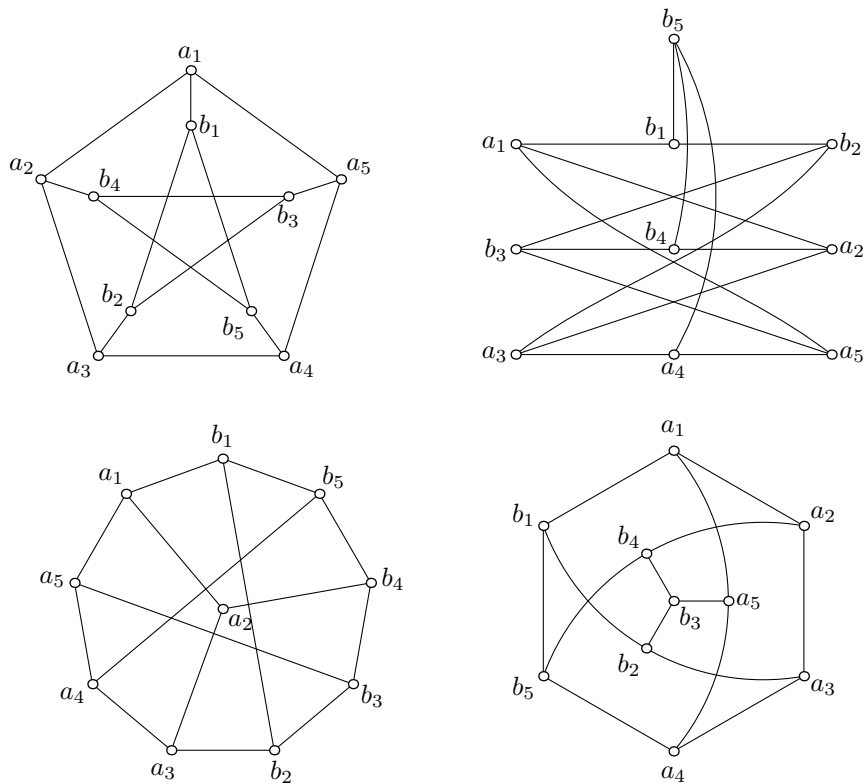

\begin{center}
  \includegraphics{figChord.4} \rule{3em}{0ex}
  \includegraphics{figChord.5}

 \rule{0em}{1ex}

  \includegraphics{figChord.6}  \rule{3em}{0ex}
  \includegraphics{figChord.7}
\end{center}
\caption{Four ways to draw the Petersen graph\label{f:petersen}}
\end{figure}

The \emph{Heawood graph}\index{Heawood graph} is the graph on $\{a_1, \dots , a_{14}\}$ so
that $\{a_1, \dots , a_{14}\}$ is a Hamiltonian cycle with vertices in
their natural order, and such that the only other edges are
$a_{1}a_{10}$, $a_{2}a_{7}$, $a_{3}a_{12}$, $a_{4}a_{9}$,
$a_{5}a_{14}$, $a_{6}a_{11}$, $a_{8}a_{13}$. See Fig.~\ref{f:heawood}.

It can be checked that both Petersen and Heawood graph are in~$\cal
C$.  Note that since the Petersen graph and the Heawood graph are both
vertex-transitive, and are not themselves a cycle with a unique chord,
to check that they are in ${\cal C}$, it suffices to delete one vertex,
and then check that there is no cycle with a unique chord.  Also the
Petersen graph has girth~5 so a cycle with a unique chord in it must
contain at least 8 vertices.  The Heawood graph has girth~6 so a cycle
with a unique chord in it must contain at least 10 vertices.
For the Petersen graph, deleting a vertex yields an Hamiltonian graph,
and it is easy to check that it does not contain a cycle with a unique
chord.  For the Heawood graph, it is useful to notice that deleting
one vertex yields the Petersen graph with edges $a_1b_1,b_3b_4,a_3a_4$
subdivided.

\begin{figure}[h]
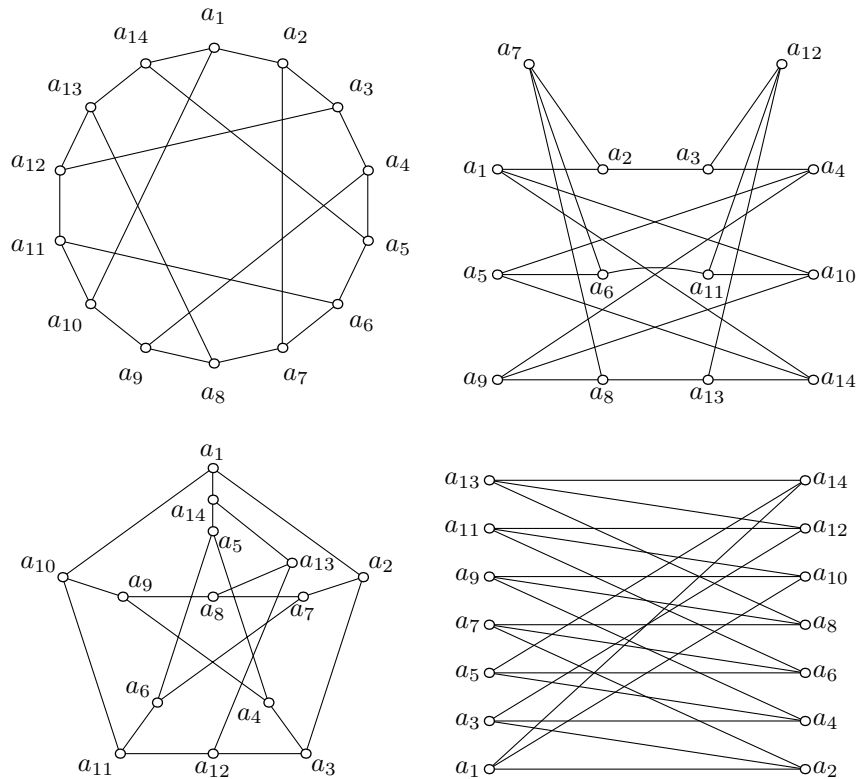

\begin{center}
  \includegraphics{figChord.1}  \rule{1em}{0ex}
  \includegraphics{figChord.2}

 \rule{0em}{1ex}

  \includegraphics{figChord.8} \rule{1em}{0ex}
  \includegraphics{figChord.10}
\end{center}
\caption{Four ways to draw the Heawood graph\label{f:heawood}}
\end{figure}

Our last basic graphs are sparse graphs, already defined
Section~\ref{sec:chordless}.  Recall that a graph is sparse if for all
pairs $u, v$ of vertices of degree at least~3, $uv\notin E(G)$.  A
sparse graph is clearly in $\cal C$ because any chord of a cycle is an
edge linking two vertices of degree at least three.

We  now define cutsets used in our decomposition theorem:

\begin{itemize}
\item A \emph{1-cutset}\index{1-cutset} of a connected graph $G$ is a
  vertex $v$ such that $V(G)$ can be partitioned into non-empty sets
  $X$, $Y$ and $\{ v \}$, so that there is no edge between $X$ and
  $Y$.  We say that $(X, Y,v)$ is a \emph{split}\index{split!of a
    1-cutset} of this 1-cutset.

\item A \emph {special 2-cutset}\index{2-cutset!special}\index{special
    2-cutset} of a connected graph $G$ is a pair of non-adjacent
  vertices $a, b$, both of degree at least three, such that $V(G)$ can
  be partitioned into non-empty sets $X$, $Y$ and $\{ a,b \}$ so that:
  $|X|\geq 2$, $|Y| \geq 2$; there are no edges between $X$ and $Y$;
  and both $G[X \cup \{ a,b \}]$ and $G[Y \cup \{ a,b \}]$ contain an
  $ab$-path.  We say that $(X, Y, a, b)$ is a
  \emph{split}\index{split!of a special 2-cutset} of this special
  2-cutset.
  
\item
  A \emph{1-join}\index{1-join} of a graph $G$ is a partition of $V(G)$ into sets
  $X$ and $Y$ such that there exist sets $A, B$ satisfying:

  \begin{itemize}
  \item
    $\emptyset \neq A \subseteq X$, $\emptyset \neq B \subseteq Y$;
  \item
  $|X| \geq 2$ and $|Y| \geq 2$;
  \item 
    there are all possible edges between $A$ and $B$;
  \item
    there are no other edges between $X$ and $Y$.
  \end{itemize}
  
  We say that $(X, Y, A, B)$ is a \emph{split}\index{split!of a
    1-join} of this 1-join.  The sets $A, B$ are \emph{special
    sets}\index{special sets!w.r.t.\ 1-join}
  with respect to this 1-join.

  1-Joins were first introduced by Cunningham~\cite{cunningham:1join}.
  Here we will use a special type of a 1-join called a {\em
    proper 1-join}: a 1-join such that $A$ and $B$ are stable sets of
  $G$ of size at least two. Note that a square admits a proper 1-join.
\end{itemize}

The main result of this section is the following decomposition
theorem:

\begin{theorem}[with Vu\v skovi\'c \cite{nicolas.kristina:one}]
  \label{th:OneChord}
  Let $G$ be a connected graph that does not contain a cycle with a
  unique chord. Then either $G$ is sparse, or $G$ is a clique, or $G$
  is an induced subgraph of the Petersen or the Heawood graph, or $G$
  has a 1-cutset, a special 2-cutset, or a proper 1-join.
\end{theorem}

Note that in \cite{nicolas.kristina:one}, the so-called \emph{strongly
  2-bipartite graphs}\index{strongly 2-bipartite graphs} are used
instead of sparse graphs.  This is because sparse graphs can have
special 2-cutsets, so they can be decomposed further into these
strongly 2-bipartite graphs (not worth defining here) and cycles (that
do not appear in this version of the theorem because they are sparse).
But let us forget these technicalities.  Let us sketch the proof of
the theorem.  We first need a lemma:

\begin{lemma}
  \label{c:neighHole}
  Let $G$ be a graph in $\cal C$, $H$ a hole of $G$ and $v$ a
  vertex of $G\sm H$. Then $v$ has at most two neighbors in $H$, and
  these two neighbors are not adjacent.
\end{lemma}

\begin{proof}
  If $v$ has at least three neighbors in $H$, then $H$ contains a
  subpath $P$ with exactly three neighbors of $v$ and $V(P) \cup
  \{v\}$ induces a cycle of $G$ with a unique chord, a
  contradiction. If $v$ has two neighbors in $H$, they must be
  non-adjacent for otherwise $H\cup \{v\}$ is a cycle with a unique
  chord.
\end{proof}

If $H$ is any induced subgraph of $G$ and $D$ is a subset of vertices of
$G\setminus H$, the \emph{attachment of $D$ over $H$}\index{attachment!of $D$ over $H$} is the set of
all vertices of $H$ that have at least one neighbor in $D$.  The first
step of the proof is to get rid of triangles.

\begin{lemma}
  \label{c:Ctriangle}
  If $G$ is a graph of $\cal C$ that contains a triangle then either
 $G$ is a clique or $G$ admits a 1-cutset. 
\end{lemma}

\begin{proof}
  Suppose that $G$ contains a triangle, and let $C$ be a maximal
  clique of $G$ that contains this triangle.  If $G \neq C$ and if no
  vertex of $C$ is a 1-cutset of $G$ then let $D$ be a connected
  induced subgraph of $G\sm C$, whose attachment over $C$ contains at
  least two vertices, and that is minimal with respect to this
  property. So, $D$ is a path with one end adjacent to $a\in C$, the
  other end adjacent to $b \in C\sm \{a\}$ and $D \cup \{a, b\}$
  induces a chordless cycle.  If $D$ has length zero, then its unique
  vertex (say $u$) must have a non-neighbor $c\in C$ since $C$ is
  maximal. Hence, $\{u, a, b, c\}$ induces a diamond, a contradiction.
  If $D$ has length at least one then let $c \neq a, b$ be any vertex
  of $C$.  Then the hole induced by $D\cup \{a, b\}$ and vertex $c$
  contradict Lemma~\ref{c:neighHole}.
\end{proof}

The second step is to get rid of squares. 

\begin{lemma}
  \label{c:noSquare0}
  If a graph $G$ of $\cal C$ that contains a square then either $G$ is
  clique or $G$ admits a 1-cutset or $G$ admits a proper 1-join.
\end{lemma}

\begin{proof}
  By Lemma~\ref{c:Ctriangle}, we may assume that $G$ contains no
  triangle.  Assume $G$ contains a square. Then $G$ contains disjoint
  sets of vertices $A$ and $B$ such that $G[A]$ and $G[B]$ are both
  stable graphs, $|A|, |B| \geq 2$ and every vertex of $A$ is adjacent
  to every vertex of $B$. Let us suppose that $A \cup B$ is chosen to be
  maximal with respect to this property. If $V(G) = A\cup B$ then $(A,
  B)$ is a proper 1-join of $G$, so we may assume that there are vertices
  in $G\sm (A\cup B)$.

  \begin{claim}
    \label{c:noSquare1}
    Every component of $G \sm (A \cup B)$ has neighbors only in
    $A$ or only in $B$.
  \end{claim}
    
  \begin{proofclaim}
    Else, let us take a connected induced subgraph $D$ of $G \sm
    (A\cup B)$, whose attachment over $A \cup B$ contains vertices of
    both $A$ and $B$, and that is minimal with respect to this
    property.  So $D = u \tp \cdots \tp v$ is a path, no interior
    vertex of which has a neighbor in $A\cup B$ and there exists $a\in
    A$, $b\in B$ such that $ua, vb \in E(G)$. Since $G$ is
    triangle-free, $u \neq v$, $u$ has no neighbor in $B$ and $v$ has
    no neighbor in $A$. By maximality of $A \cup B$, $u$ has a
    non-neighbor $a' \in A$ and $v$ has a non-neighbor $b' \in
    B$. Now, $D \cup \{a, b, a', b'\}$ is a cycle with a unique chord
    (namely $ab$), a contradiction.
  \end{proofclaim}
  
  From~(\ref{c:noSquare1}), it follows that $G$ has a proper 1-join with
  special sets $A, B$.
\end{proof}

The next step is to get rid of the Petersen graph, but the reader
could get tired of these steps since there are about ten more of them.
So let us sum them up.  It is easy to see after some checking that if
a graph of $\cal C$ with no triangle and no square contains the
Petersen or the Heawood graph as an induced subgraph $K$ then any
component of $G\sm K$ must attach to a single vertex of $K$.  So,
containing $K$ implies being basic or having a 1-cutset.

The next steps all follow from the Petersen-or-Heawood graph.  To see
how, let us say that a graph $H$ \emph{arises}\index{arise!$H$ arises
  from $G$} from a graph $G$ if $H$ is obtained from $G$ by repeatedly
applying the following: deleting a vertex of degree~3 and subdividing
any number of times edges incident to at least one vertex of degree 2.
It is easy to see that if $G$ is in $\cal C$ then so is $H$.  A graph
that arises from the Petersen-or-Heawood graph is made of vertices of
degree~3 linked either by edges or by \emph{flat
  paths}\index{flat!path}, that are paths of length at least~2 whose
ends are non-adjacent and whose interior vertices are all of degree~2.
Up to symmetries, there are 5 types of 2-connected graphs that arise
from the Petersen-or-Heawood graphs.  For each of them, (taken in the
order in which they appear when eliminating vertices of degree~3), it
can be proved that a connected component of what remains attaches
either to single vertex (so there is a 1-cutset), or to a flat path
(so the ends of the flat path form a special 2-cutset), or in a way
that gives a previously studied configuration.  All this is a bit
tedious because each of the five cases needs about 2 pages of
checking.

The next steps are less tedious: eliminating cycles with exactly three
chords, and finally all cycles with at least one chord.  After this,
the graph is chordless, so we may rely on Theorem~\ref{th:nochord} (be
careful, special and proper 2-cutsets are different, but this is a
technical detail).  In \cite{nicolas.kristina:one}, things are
presented slightly differently, because we did not identify
Theorem~\ref{th:nochord} as a separate result.

\subsection*{Applications}

Theorem~\ref{th:OneChord} has several applications.  First, it is
easily seen to be a structure theorem.  We do not give the list of the
reverse operations, see~\cite{nicolas.kristina:one} for all details.

From Theorem~\ref{th:OneChord}, it is easy to derive a recognition
algorithm for Class~$\cal C$.  An ${O}(n^5)$ or a slightly more
involved ${O}(n^4)$ algorithm could be obtained from first
principles, but in~\cite{nicolas.kristina:one}, sophisticated
algorithms from other authors, namely Dahlhaus~\cite{dahlhaus:split},
Hopcroft and
Tarjan~\cite{hopcroft.tarjan:447,tarjan:dfs,hopcroft.tarjan:3con} get
our algorithm to run in ${O}(nm)$-time.  In view of
Chapter~\ref{chap:reco}, it would be nice to have a recognition
algorithm that relies on something like three-in-a-tree.  Here are
NP-completeness results supporting the idea that no such approach
exists.

Let $H_{1|1}$ be the s-graph on vertices $a, b, c, d$ with real edges
$ab$, $ac$, $ad$ and subdivisible edges $bd$, $cd$. We also define for
$k, l \geq 1$ the s-graph $H_{k|l}$ obtained from $H_{1|1}$ by
subdividing the edge $ab$ into a path of length $k$ and the edge $ac$
into a path of length $l$.  See Fig.~\ref{f:sgraph}.  Problem
$\Pi_{H_{1|1}}$ can be rephrased as ``Does $G$ contain a cycle with a
unique chord?''  or ``Is $G$ not in~$\cal C$?''.

\begin{figure}
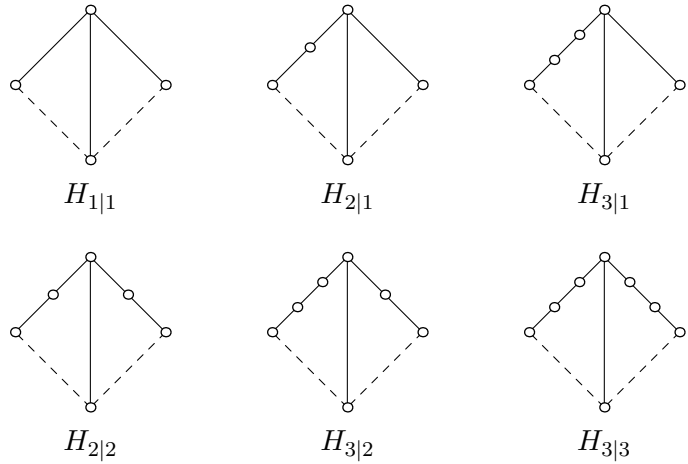

\center
\begin{tabular}{ccccc}
\includegraphics{sgraphs.1}&
\rule{1em}{0ex}&
\includegraphics{sgraphs.2}&
\rule{1em}{0ex}&
\includegraphics{sgraphs.4}\\
$H_{1|1}$&&
$H_{2|1}$&&
$H_{3|1}$
\\
\rule{0em}{1ex}&&&&\\
\includegraphics{sgraphs.3}&
\rule{1em}{0ex}&
\includegraphics{sgraphs.5}&
\rule{1em}{0ex}&
\includegraphics{sgraphs.6}\\
$H_{2|2}$&&
$H_{3|2}$&&
$H_{3|3}$
\end{tabular}
\caption{Some s-graphs\label{f:sgraph}}
\end{figure}

We claim that our polynomial time algorithm for $\Pi_{H_{1|1}}$ exists
not because of some hidden general method but thanks to what we call
\emph{degeneracy}\index{degenaracy}.  Let us explain this.  Degeneracy has to deal with
the following question: does putting bounds on the lengths of the
paths in realisations of an s-graph affects the complexity of the
related detection problem?  First, putting upper bounds may turn the
complexity from NP-complete to polynomial. This follows from a simple
observation: let $B$ be any s-graph. A realisation of $B$, where the
lengths of the paths arising from the subdivisions of subdivisible
edges are bounded by an integer $N$, has a number of vertices bounded
by a fixed integer $N'$ (that depends only on $N$ and the size of
$B$). So, such a realisation can be detected in time ${O}(n^{N'})$ by
a brute force enumeration.  But surprisingly, putting upper bounds in
another way may also turn the complexity from polynomial to
NP-complete: in~\cite{chudnovsky.c.l.s.v:reco}, a polynomial time
algorithm for $\Pi_{K}$ is given, while in
\cite{maffray.t.v:3pcsquare} it is proved that $\Pi_{K'}$ is
NP-complete, where $K, K'$ are the s-graphs represented in
Figure~\ref{f:pp}.

\begin{figure}
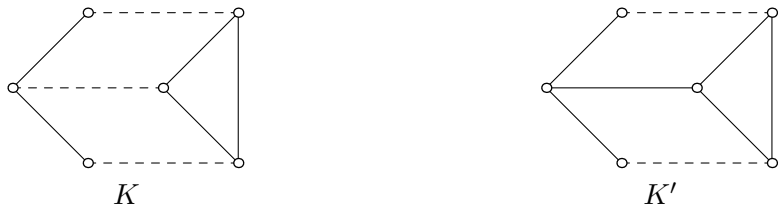

\center
\begin{tabular}{ccc}
\includegraphics{sgraphs.7}&
\rule{8em}{0ex}&
\includegraphics{sgraphs.8}\\
$K$&&
$K'$
\end{tabular}
\caption{Some s-graphs\label{f:pp}}
\end{figure}

Can putting lower bounds turn the complexity from polynomial to
NP-complete? Our recognition algorithm for $\cal C$ shows that the
answer is yes since we also prove that the problem $\Pi_{H_{3|3}}$ is
NP-complete (this relies on Bienstock's construction,
see~\cite{nicolas.kristina:one} for a proof).  A realisation of
$H_{3|3}$ is simply a realisation of $H_{1|1}$ where every
subdivisible edge is subdivided into a path of length at least three.
A ``general'' approach (like three-in-a-tree) that would recognize
class $\cal C$ would perhaps solve also $\Pi_{H_{3|3}}$ and therefore
should not exist.

We also believe that a satisfactory structural description of the
class~$\cal C'$ of graphs that do not contain a realisation of
$H_{3|3}$ is hopeless because $\Pi_{H_{3|3}}$ is NP-complete.  So why
is there a decomposition theorem for~$\cal C$?  Simply because
degenerate small graphs like the \emph{diamond}\index{diamond} (that
is the cycle on four vertices with exactly one chord) are forbidden
in~$\cal C$, not in~$\cal C'$, and this helps a lot in our proof of
Theorem~\ref{th:OneChord}.  This is what we call the \emph{degeneracy}
of the class~$\cal C$. It is clear that degeneracy can help in solving
detection problems, and our results give a first example of this
phenomenon.

Note that we do not know the answer to the following: can putting
lower bounds turn the complexity from NP-complete to polynomial?  Also
$\Pi_{H_{2|1}}$, $\Pi_{H_{3|1}}$, $\Pi_{H_{2|2}}$ and~$\Pi_{H_{3|2}}$
are open problems.  The related classes of graphs are not degenerate
enough to allow us to decompose, and they are too degenerate to allow
us to find an NP-completeness proof.

Theorem~\ref{th:OneChord} has consequences on combinatorial
optimization problems.  The fact that triangles entail 1-cutsets gives
easily a linear-time algorithm for the maximum clique problem.  The
presence of sparse graphs in $\cal C$ shows that computing maximum
stable sets is NP-hard by a very useful construction of
Poljak~\cite{poljak:74} that is worth recalling.  A
\emph{2-subdivision}\index{2-subdivision} is a graph obtained from any
graph by subdividing twice every edge. More precisely, every edge $uv$
of a graph $G$ is replaced by an induced path $uabv$ where $a$ and $b$
are of degree two. Let $F$ be the resulting graph. It is easy to see
that $\alpha(F) = \alpha(G) + |E(G)|$.  So, computing maximum stable
sets in 2-subdivisions is NP-hard.  Since a 2-subdivision is sparse,
our NP-hardness result follows.

Coloring graphs from $\cal C$ is more difficult.  If the graph to be
colored contains a triangle, there is a 1-cutset so coloring can be
performed recursively on blocks.  The difficult case is when there is
no triangle.  Because coloring recursively along 1-joins and 2-cutsets
is not direct (as already mentioned in Section~\ref{sec:chordless}).
The first idea we had is to use an extreme decomposition: a
decomposition such that one of the block is a basic graph.  But we soon
realised that some graphs in $\cal C$ do not admit an extreme
decomposition.  The graph in Fig.~\ref{f:noextr} has no special
2-cutset and a unique 1-join.  No block with respect to this proper
1-join is basic, but both blocks have a special 2-cutset (in a block
with respect to a 1-join, one side is replaced by a vertex).

\begin{figure}[h]
\center
\includegraphics{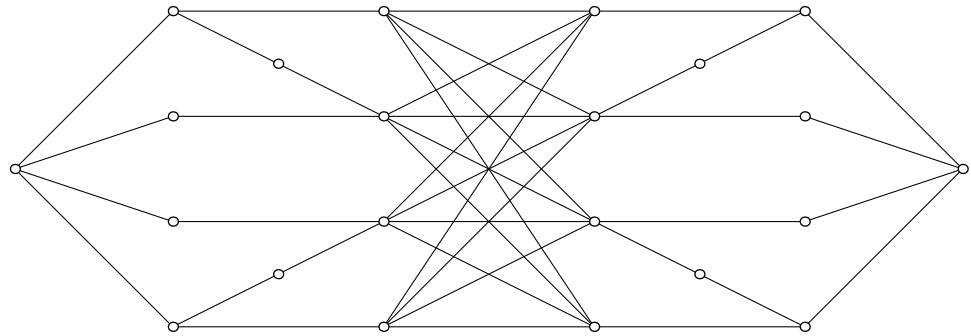}
\caption{A graph in $\cal C$ with no extreme decomposition\label{f:noextr}}
\end{figure}

So, we tried to prove that every triangle-free graph from $\cal C$ is
3-colorable directly by induction.  To do this, we had to prove
something stronger, and finding the right induction hypothesis had
been quite painful.  Let us call \emph{third color}\index{third color}
of a graph any stable set that contains at least one vertex of every
odd cycle.  Finding a third color is equivalent to 3-color a graph:
give color~3 to the third color; since $G\sm S$ contains no odd cycle,
it is bipartite: color it with colors~1, 2.

Let us call \emph{strong third color}\index{strong third color} of a
graph any stable set that contains at least one vertex of every cycle
(odd or even).  Recall that $N[v]$ denotes $\{v\} \cup N(v)$. When $v$
is a vertex of a graph $G$, a pair of disjoint subsets $(R, T)$ of
$V(G)$ is \emph{admissible with respect to $G$ and
  $v$}\index{admissible} if one of the following holds (see
Fig.~\ref{f:admis}):

\begin{itemize}
\item
  $T = N(v)$ and $R = \{v\}$;
\item
  $T = \emptyset$ and $R = N[v]$;
\item
  $v$ is of degree two, $N(v) = \{u, w\}$, $T = \{u\}$, $R = \{v,
  w\}$;
\item
  $v$ is of degree two, $N(v) = \{u, w\}$, $T = \{u\}$, $R = N[w]$;
\item
  $v$ is of degree two, $N(v) = \{u, w\}$, $T = \emptyset$, $R = \{u\}
  \cup N[w]$.
\end{itemize}

\begin{figure}
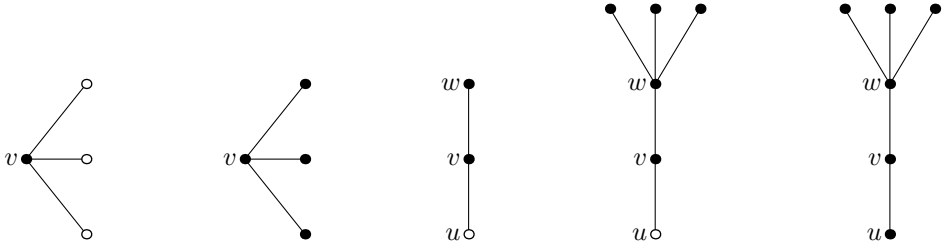

\center
\includegraphics{figChord.11}\rule{4.1em}{0ex}
\includegraphics{figChord.12}\rule{4.1em}{0ex}
\includegraphics{figChord.13}\rule{4.1em}{0ex}
\includegraphics{figChord.14}\rule{4.1em}{0ex}
\includegraphics{figChord.15}
\caption{Five examples of admissible pairs (vertices of $T$ are white,
  vertices of $R$ are black)\label{f:admis}}
\end{figure}

We say that a pair of disjoint subsets $(R,T)$ is an {\em admissible
  pair} of $G$ if for some $v \in V(G)$, $(R,T)$ is admissible with
respect to $G,v$.  An admissible pair $(R, T)$ should be seen as a
constraint for coloring: we will look for third colors (sometimes
strong, sometimes not, we have to do this) that must contain every
vertex of $T$ and no vertex of $R$.  We do this first in basic graphs,
and then by induction in all triangle-free graphs of $\cal C$, thus
proving that they are 3-colorable.  We give here the main statements
but we omit the proofs. 

\begin{lemma}
  \label{th:tcSP}
  Let $G$ be a non-basic, connected, triangle-free, square-free and
  Petersen-free graph in $\cal C$ that has no 1-cutset and no proper
  1-join.  Let $(R, T)$ be an admissible pair of~$G$.  Then $G$ admits
  a strong third color $S$ such that $T \subseteq S$ and $S \cap R =
  \emptyset$.
\end{lemma}

In the proof of the lemma above we really need to find \emph{strong}
third colors for the sake of induction.  But we cannot do this for all
graphs in $\cal C$ because it is false in general that triangle-free
graphs from $\cal C$ admit strong third colors.  A counter-example is
the graph $G$ obtained as follows: take four disjoint copies $\Pi_1,
\dots, \Pi_4$ of the Petersen graph minus one vertex. So $\Pi_i$
contains a set~$X_i$ of three vertices of degree two ($i= 1, \dots,
4$). We add all edges between $X_1, X_2$, between $X_2, X_3$, between
$X_3, X_4$ and between $X_4, X_1$.  Note that $G$ can be obtained by
gluing one square $S=s_1s_2s_3s_4$ and four disjoint copies of the
Petersen graph along 1-joins, so it is a graph from $\cal C$.

We claim now that $G$ does not contain a stable set that intersects
every cycle. Indeed, if $S$ is such a stable set then $S$ must contain
all vertices in one of the $X_i$'s for otherwise we build a $C_4$ of
$G\sm S$ by choosing a vertex in every~$X_i$. So $X_1 \subseteq S$
say. We suppose that $\Pi_1$ has vertices $\{a_2, \dots, a_5, b_1,
\dots, b_5\}$ with our usual notation. So $X_1 = \{b_1, a_2, a_5\}
\subseteq S$ and we observe that every vertex in $C = V(\Pi_1\sm S)$
has a neighbor in $X_1$. Hence $C \cap S = \emptyset$ while $G[C]$ is
a cycle on six vertices, a contradiction.  

Yet, we can prove that all graphs from $\cal C$ admit a third color.

\begin{lemma}
  \label{th:thirdcolor}
  Let $G$ be a non-basic connected triangle-free graph in $\cal C$ and
  $(R, T)$ be an admissible pair of~$G$.  Then $G$ admits a third
  color $S$ such that $T \subseteq S$ and $S \cap R = \emptyset$.
\end{lemma}

All this shows that graphs in $\cal C$ have the best $\chi$-bounding
function that one may expect for a class not included in perfect
graphs: $f(2)=3$ and $f(x) = x$ for $x\geq 3$. 

\begin{theorem}[with Vu\v skovi\'c \cite{nicolas.kristina:one}]
  \label{th:color1chord}
  Let $G$ be a graph that does not contain a cycle with a unique
  chord.  Then either $\chi(G) = 3$ or $\chi(G) = \omega(G)$.
\end{theorem}

The question of characterizing these graphs that are $\chi$-bounded
with this function $f$ is widely open and Section~\ref{sec:noWheel}
gives a conjecture about them.

\section{Induced subdivision of $K_4$}
\label{sec:ISK4}

Graphs with no subdivision of $K_4$ as a possibly \emph{non-induced}
subgraph are called \emph{series-parallel
  graphs}\index{series-parallel}.  The structure of series-parallel
graphs is well-known.

\begin{theorem}[Duffin~\cite{duffin:SP}, Dirac~\cite{dirac:SP}]
  Series-parallel graphs are graphs that arise from a forest by
  repeatedly applying the following operations:

  \begin{itemize}
  \item
    add a parallel edge to an existing edge and subdivide this new edge
    (this is the \emph{parallel extension});
  \item  
    subdivide an edge (this is the \emph{series extension}).
  \end{itemize}
\end{theorem}

\begin{exercise}  
  Let $G$ be a 2-connected triangle-free series-parallel graph. Then
  either $G$ is isomorphic to $C_4$, to $C_5$ or to $K_{2, 3}$, or $G$
  contains at least 4 vertices of degree~2.
\end{exercise}

Of course, series-parallel graphs form an obvious basic class for
ISK4-free graphs.  So, an ISK4-free graph that is not basic should
contain a subdivision $H$ of $K_4$ that has chords.  Assuming that
$H$ has a minimum number of vertex, we obtain the following.  Recall
that a \emph{wheel}\index{wheel} is obtained by taking a hole and adding a vertex
with at least three neighbors in the hole.

\begin{lemma}[with L\'ev\^eque and Maffray \cite{nicolas:isk4}]
  \label{l:begin}
  Let $G$ be an ISK4-free graph. Then either $G$ is a series-parallel
  graph,  or $G$ contains a prism, or $G$ contains a wheel or $G$
  contains $K_{3, 3}$. 
\end{lemma}

This lemma gives an obvious way to find a decomposition theorem for
ISK4-free graphs: by proving attachment lemmas for prisms, wheels and
$K_{3, 3}$, and hoping that containing one of these entails a
decomposition.  The key question is the order in which the subgraphs
should be considered.  Hopefully, $K_{3, 3}$ is quite easy.  The
following is not very difficult to prove.

\begin{lemma}[with L\'ev\^eque and Maffray \cite{nicolas:isk4}]
  \label{l:decK33}
  Let $G$ be an ISK4-free graph that contains $K_{3, 3}$. Then either
  $G$ is a bipartite complete graph, or $G$ is a tripartite complete
  graph, or $G$ has a clique-cutset of size at most~3.
\end{lemma}

The Lemma above already gives us something.  Since Scott's
Conjecture~\ref{conj:scott} is known to be true for $K_{3, 3}$-free
graphs by a theorem of K\"uhn and Osthus~\cite{kuhnOsthus:04}, we
easily prove the following. 

\begin{theorem}[Scott \cite{scott:pc}]
  \label{th:ScottK4}
  ${\rm Forb}^*(K_4)$ is $\chi$-bounded.
\end{theorem}

The next question is to decide whether we should consider prisms or
wheels first.  Here, knowing a bit about Berge and claw-free graphs
helps a lot to decide something (foolish or wise, but we have to
decide).  Because from them, we know that prisms are line-graphs that
are likely to grow into bigger line-graphs while wheels are more
messy.  Indeed, if one attaches a vertex to a prism, it is very likely
that a kind of pyramid (which is an ISK4) is created, there are only
few exceptions to this and most of them actually give some bigger
line-graphs.  So, considering an inclusion-wise maximal induced
subgraph isomorphic to a line-graph should give something.  Note that
the idea of considering a maximal line-graph in such a context is one
of the key steps that lead to the proof of the Strong Perfect Graph
Conjecture, but it was used before by Conforti and
Cornu\'ejols~\cite{conforti.c:wp}.

Precise statements are postponed because we need long definitions.
Let us say now that one attachment to a prism is pathological.
Suppose that a prism contains a square and add a vertex complete to
this square.  This gives an ISK4-free graph, which is in fact a
line-graph.  But, add again another vertex complete to the same
square.  Then, no ISK4 is created, and the graph obtained is not a
line-graph.  So, some prisms are likely to grow in other directions
than just line-graphs.  It has been quite long to figure out that in
fact, a single basic class arises from these pathological attachments,
the so-called rich-squares.  The smallest example of a rich-square is
$L(K_4)$, more famous as the octahedron.  Even if it does not appear
too much in the statements below, it behaves a bit as a sporadic
ISK4-free graph.

\subsection*{Definitions}

A \emph{separation}\index{separation} of a graph $H$ is a pair $(A,
B)$ of subsets of $V(H)$ such that $A \cup B = V(H)$ and there are no
edges between $A \setminus B$ and $B \setminus A$. It is
\emph{proper}\index{proper!separation} if both $A \setminus B$ and
$B\setminus A$ are non-empty. The \emph{order}\index{order of a
  separation} of the separation is $|A\cap B|$.  A
\emph{$k$-separation}\index{K@$k$-separation} is a separation $(A, B)$
such that $|A \cap B| \leq k$. A separation $(A, B)$ is
\emph{cyclic}\index{cyclic} if both $H[A]$ and $H[B]$ has cycles.  A
graph $H$ is \emph{cyclically 3-connected}\index{cyclically
  3-connected} if it is 2-connected, not a cycle, and there is no
cyclic 2-separation.  Note that a cyclic 2-separation of any graph is
proper.  The reader might prefer to think about cyclically 3-connected
graphs as follows.

\begin{exercise}
  \label{l:descc3c}
  A graph $H$ is a cyclically 3-connected graph if and only if it is
  either a theta or a subdivision of a 3-connected graph.
\end{exercise}

Recall that a \emph{branch-vertex}\index{branch-vertex} is a vertex of
degree at least~3.  A \emph{branch}\index{branch} is a path whose ends
are branch-vertices and whose interior vertices are not.  A
\emph{basic branch}\index{basic!branch} in a graph is branch such that
no two vertices in it are members of a triangle.  Note that a branch
in a graph is either basic, or is an edge $uv$ and there is a common
neighbor of $u, v$.

A \emph{triangular}\index{triangular subdivision} subdivision of $K_4$
is a subdivision of $K_4$ that contains a triangle.  Note that a
pyramid is a triangular ISK4.  A \emph{square}\index{square!theta}
theta is a theta that contains a square.  Rephrased: a theta with two
branches of length two. A \emph{square}\index{square!prism} prism is a
prism that contains a square. Rephrased, a prism with two basic
branches of length one. Rephrased again, the line-graph of a square
theta.  A \emph{square}\index{square!subdivision of $K_4$} subdivision
of $K_4$ is a subdivision of $K_4$ such that the four vertices of
degree three in it lie in a possibly non-induced square.  Rephrased: a
subdivision of $K_4$ where only two edges with no common ends are
possibly subdivided. An induced square in a graph is
\emph{basic}\index{basic!square} if
an even number of edges in it lie in a triangle of the graph. It
easily checked that the line-graph of a subdivision $H$ of $K_4$
contains a basic square if and only if $H$ is a square subdivision of
$K_4$, and that the vertices in any basic square of $L(H)$ arise from
the edges of a square on the branch-vertices of $H$. It is easily
checked that a prism contains only basic squares.

If $X, Y$ are two basic branches of a graph $G$, a \emph{connection
  between $X$, $Y$} \index{connection!between $X$, $Y$} is a path $P =
p \tp \cdots \tp p'$ such that $p$ has neighbors in $X$, $p'$ has
neighbors in $Y$, no interior vertex of $P$ has neighbors in $X\cup
Y$, and if $p\neq p'$ then $p$ has no neighbor in $Y$ and $p'$ has no
neighbor in $X$.

When $S$ = $\{u_1, u_2, u_3, u_4\}$ induces a square in a graph $G$
with $u_1, u_2, u_3, u_4$ in this order along the square, a
\emph{connection}\index{connection!of a square} of $S$ is an induced
path of $G$, no interior vertex of which has a neighbor in $S$, with
ends $p, p'$ such that either $p=p'$ and $N_S(p) = S$; or $N_S(p) =
\{u_1, u_2\}$ and $N_S(p') = \{u_3, u_4 \}$; or $N_S(p) = \{u_1,
u_4\}$ and $N_S(p') = \{u_2, u_3\}$.

The line-graph of $K_4$ is isomorphic to $K_{2, 2, 2}$. It is usually
called the \emph{octahedron}\index{octahedron}. It contains three basic squares. For
every basic square $S$ of an octahedron $G$, the two vertices of $G\sm
S$ are both connexions of $S$.  Note also that when $K$ is a square
prism with a square $S$, then $V(K) \setminus S$ is a connection of
$S$.

When $G$ is a graph, $H$ is a graph such that $L(H)$ is an induced
subgraph of $G$, and $C$ is a connected induced subgraph of $V(G)
\setminus L(H)$, we define several types that $C$ can have, according
to its attachment over $L(H)$:

\begin{itemize}
\item
  $C$ is of type \emph{branch}\index{branch!type} if the attachment of $C$ over $L(H)$
  is included in a basic branch of $L(H)$;

\item
  $C$ is of type \emph{triangle}\index{triangle!type} if the attachment of $C$ over $K$ is
  included in a triangle of $L(H)$;

\item
  $C$ is of type \emph{augmenting}\index{augmenting!type} if $C$ contains a connection $p \tp
  \cdots \tp p'$ of two distinct basic branches $X, Y$ of $L(H)$,
  $N_X(p)$ is an edge of $X$ and $N_Y(p')$ is an edge of~$Y$.
  Moreover, there are no edges between $L(H)\setminus (X\cup Y)$ and
  $P$.
\item
  $C$ is of type \emph{square}\index{square!type} if $L(H)$ contains a basic square $S$,
  and $C$ contains a connection $P$ of $S$. Moreover, there are no
  edges between $L(H)\setminus S$ and $P$.
\end{itemize}

Note that the types may overlap: a subgraph $C$ can be of more than
one type. Since we view a vertex of $G$ as a connected induced
subgraph of $G$, we may speak about the type of vertex with respect to
$L(H)$.

\subsection*{Getting rid of prisms}

Now the idea to decompose ISK4-free graphs further is the following:
suppose that an ISK4-free graph $G$ contains a prism.  So, $G$
contains a maximal line-graph $L(H)$.  If there is nothing else in the
graph, then $G$ itself is a line graph, so it is basic.  Else, we are
going to prove that any component $C$ of $G\sm H$ is of one of the
type defined above.  Type branch gives a proper 2-cutset, type
triangle gives a clique cutset and type augmenting gives a bigger
line-graph, so a contradiction to the maximality of $L(H)$.

Type square is slightly more tricky.  Sometimes, it gives a bigger
line-graphs (when it is the only type square attachment to a given
square).  When used a second time, it often gives an ISK4, with two
degenerate exceptions: when $L(H)$ is a square prisms or the
line-graphs of a square subdivision of $K_4$.  These two exceptions
give the basic class of rich-squares: a \emph{rich square}\index{rich square} is a graph
$K$ that contains a square $S$ as an induced subgraph, and such that
$K\setminus S$ has at least two components and every component of
$K\setminus S$ is a connection of~$S$.

We give now the sequence of attachments lemmas that we need to show
that when an ISK4-free graphs contains a prism then it is either basic
or admits a decomposition.  We give all the statements without proofs
so that the readers feels the fun of trying to decompose a class,
which is a bit like getting out of a maze.  The slight problem is that
trying a direction or another costs a lot of time, for instance, the
proof of the Lemma~\ref{l:compsk4} is 6 pages long despite efforts
towards avoiding symmetric cases.  But the next one is not so long and
is a good exercise.

\begin{lemma}
  \label{l:compprism}
  Let $G$ be a graph that contains no triangular ISK4. Let $K$ be a
  prism that is an induced subgraph of $G$ and let $C$ be a connected
  induced subgraph of $G\setminus K$. Then $C$ is either of type
  branch, triangle, augmenting or square with respect to $K$.
\end{lemma}

\begin{lemma}
  \label{l:compsk4}
  Let $G$ be a graph that contains no triangular ISK4. Let $H$ be a
  subdivision of $K_4$ such that $L(H)$ is an induced subgraph of $G$.
  Let $C$ be a connected induced subgraph of $G\setminus L(H)$. Then
  $C$ is either of type branch, triangle, augmenting or square with
  respect to $L(H)$.
\end{lemma}

A graph $H$ is \emph{substantial}\index{substantial} if it is cyclically 3-connected and
neither a square theta nor a square subdivision of $K_4$.  Note that
in the next lemma, the type square is not needed. 

\begin{lemma}
  \label{l:compsubstan}
  Let $G$ be a graph that contains no triangular ISK4. Let $H$ be a
  substantial graph such that $L(H)$ is an induced subgraph of $G$.
  Let $C$ be a component of $G\setminus L(H)$. Then $C$ is either of
  type branch, triangle or augmenting with respect to $L(H)$.
\end{lemma}

The next lemmas are not attachment lemmas, but decomposition
results that use the lemmas above.

\begin{lemma}
  \label{l:substdecomp}
  Let $G$ be a graph that contains no triangular ISK4.  Let $H$ be a
  substantial graph such that $L(H)$ is an induced subgraph of
  $G$. Suppose $L(H)$ inclusion-wise maximum with respect to that
  property.

  Then either $G = L(H)$ or $G$ has a clique-cutset of size at most
  three, or $G$ has a proper 2-cutset.
\end{lemma}

The next two lemmas sort out which prisms give line-graphs and which
prisms give rich-squares.

\begin{lemma}
  \label{l:richsquaredecomp}
  Let $G$ be an ISK4-free graph that contains no line-graph of a
  substantial graph. Let $K$ be a rich square that is an induced
  subgraph of $G$, and maximal with respect to this property. Then
  either $G=K$ or $G$ has a clique-cutset of size at most three or $G$
  has a proper 2-cutset.
\end{lemma}

\begin{lemma}
  \label{l:prismdecomp}
  Let $G$ be an ISK4-free graph that contains no line-graph of a
  substantial graph and no rich square as an induced subgraph. Let $K$
  be a prism that is an induced subgraph of $G$. Then either $G=K$ or
  $G$ has a clique-cutset of size at most three or $G$ has a proper
  2-cutset.
\end{lemma}

The next theorem was our goal.  It is easy to prove with all the
lemmas above.  Thanks to it, we ``get rid of prisms''.

\begin{theorem}[with L\'ev\^eque and Maffray \cite{nicolas:isk4}]
  \label{l:mainlinegraphdecomp}
  Let $G$ be an ISK4-free graph that contains a prism or an
  octahedron. Then either $G$ is the line-graph of a graph with
  maximum degree three, or $G$ is a rich square, or $G$ has a
  clique-cutset of size at most three or $G$ has a proper 2-cutset.
\end{theorem}

\subsection*{Getting rid of wheels}

In view of Lemma~\ref{l:begin}, our next to goal is to produce
something like Theorem~\ref{l:mainlinegraphdecomp} with ``wheel''
instead of ``prism''.  The only result we could prove is the
following.

A \emph{double star cutset}\index{double star cutset} of a graph is a set $S$ of vertices such
that $G\sm S$ is disconnected and such that $S$ contains two vertices
$u, v$ and every vertex of $G$ is adjacent at least one of $u, v$. In
particular $u$ and $v$  are adjacent. Note that a star-cutset is
either a double star cutset or consists of one vertex.

\begin{lemma}[with L\'ev\^eque and Maffray \cite{nicolas:isk4}]
  \label{thm:wheel}
  Let $G$ be an ISK4-free graph that does not contain a prism or an
  octahedron.  If $G$ contains a wheel, then $G$ has a star-cutset or
  a double star cutset.
\end{lemma}

Thanks to the lemma above, we can easily devise a decomposition
theorem for ISK4-free graphs that uses double-star cutsets, but I do
not want even to state this here because it is useless.  These double
star cutsets are strong cutsets.  In some situations, they are very
useful: they are the best known cutsets for odd hole-free graphs
(Theorem~\ref{th.ccv}) and they are used by Cornu\'ejols, Liu and Vu\v
skovi\'c~\cite{chudnovsky.c.l.s.v:reco} to give an interesting
recognition algorithm of Berge graphs based on decomposition.  But
from a bunch examples, we feel that ISK4-free graphs should be
decomposable with better cutsets, as star cutset or perhaps 3-cutsets.
This research is still in progress.  We are not very optimistic for
the recognition algorithm, but we hope to prove the following.

\begin{conjecture}
  Any ISK4-free graph is 4-colorable. 
\end{conjecture}

\section{Graphs with no wheels}
\label{sec:noWheel}

The previous section suggests that wheels are annoying.  Also in the
original proof of the Strong Perfect Graph Theorem, the most difficult
part was devoted to \emph{wheel systems}\index{wheel system}, a very complicated
construction to decompose a Berge graph containing a wheel.  So,
perhaps nice things happen if we exclude wheels?  Indeed, from the lemmas in
the section above, we obtain easily a very precise decomposition
theorem. 

\begin{theorem}[with L\'ev\^eque and Maffray \cite{nicolas:isk4}]
  \label{th:ISK4Wheel}
  Let $G$ be an ISK4-free graph with no wheel. Then either:
  \begin{itemize}
  \item $G$ is series-parallel;
  \item $G$ is the line-graph of a graph with maximum degree three;
  \item  $G$ is a rich square; 
  \item $G$ is a complete bipartite graph; 
  \item $G$ is a complete tripartite graph; 
  \item $G$ has a clique-cutset of size at most three;
  \item $G$ has a proper 2-cutset.
  \end{itemize}
\end{theorem}

In several simple classes, we also obtain results.  

\begin{lemma}[with L\'ev\^eque and Maffray \cite{nicolas:isk4}]
  \label{l:WF3c-rs}
  Let $G$ be a rich square that contains no wheel.  Then $G$ is
  $3$-colorable.
\end{lemma}

It is easy to check that a line-graph $G=L(R)$ contains a wheels if
and only if $R$ contains a cycle with at least one chord.  So, a
line-graph with no wheel is the line-graph of a chordless graph.  In
fact, with L\'ev\^eque and Maffray this was our motivation to study
chordless graphs, and it was lucky that a few months before they were
studied implicitly in the paper~\cite{nicolas.kristina:one} with Vu\v
skovi\'c on cycles with a unique chord.  From this remark and
Theorem~\ref{th:nochord} we get the following.

\begin{lemma}[with L\'ev\^eque and Maffray \cite{nicolas:isk4}]
  \label{l:WF3c-lg}
  Let $G$ be a graph that contains no ISK4, no wheel and such that $G$
  is a line-graph. Then $G$ is $3$-colorable.
\end{lemma}

\begin{exercise}
  Prove or disprove that for all line-graphs $G$ with no wheel
  $\chi(G) = 3$ or $\chi(G) = \omega(G)$.
\end{exercise}

Finally, by combining all the results above, we obtain the following.

\begin{theorem}[with L\'ev\^eque and Maffray \cite{nicolas:isk4}]
  \label{l:WF3c}
  Let $G$ be a graph that contains no ISK4 and no wheel.  Then $G$ is
  $3$-colorable.
\end{theorem}

This suggests that some bounds on $\chi$ could be obtained for graphs
with no wheels.  We have no counter-example to the following.

\begin{conjecture}[with L\'ev\^eque and Maffray \cite{nicolas:isk4}]
  A graph $G$ with no wheel satisfies $\chi(G) = 3$ or $\chi(G) =
  \omega(G)$. 
\end{conjecture}

In fact, a wheel always contains a cycle with a unique chord.  So,
Theorem~\ref{th:color1chord} is a weaker form of the conjecture.  The
following also could be a not so difficult weaker form.

\begin{question}
  A planar graph with no wheel as a subgraph (possibly
  non-induced) admits a vertex of degree at most~2.
\end{question}

The following is certainly very difficult and
Theorems~\ref{th:OneChord} and~\ref{th:ISK4Wheel} can be seen as
weaker forms of an answer to it.

\begin{question}
  Find a decomposition theorem for wheel-free graphs.
\end{question} 

\chapter{Berge graphs}
\label{chap:Berge}
\newcommand{\even}{balanced}
\newcommand{\Even}{Balanced}
\newcommand{\an}{a}
\newcommand{\An}{A}
\newcommand{\ESPD}{BSP}

Berge graphs is perhaps the most studied class of graphs closed under
taking induced subgraphs, especially if one takes into account that
many works on other classes are motivated by their inclusion in Berge
graphs.  Writing a survey about Berge graphs is uncomfortable because
a long list of definitions is needed just to state the results.  If
these definitions are skipped, the survey becomes fuzzy.  Let us try
both approaches by keeping the fuzziness for this short introduction
and postponing the formal definitions to the next section.

In the eighties, the idea of proving the Strong Perfect Graph
Conjecture by decomposing Berge graphs started to be advocated by
researchers around Va\v sek Chv\'atal.  During a talk given in July
2004 at the conference held in Paris honoring the memory of Claude
Berge, Chv\'atal said that as far as he remembered, Sue Whitesides was
the first researcher who told him about that.  Guessing the right
decomposition statement has been a long process.  Inventing operations
that preserve perfectness started from the very beginning of the study
of perfect graphs with the clique cutset, continued with Lov\'asz'
replication of vertices, Fonlupt and Urhy's contraction of even pairs,
Burlet and Fonlupt's amalgam and finally Cornu\'ejols and Cunningham's
2-join and Chv\'atal and Sbihi's homogeneous pair.  But keeping the
focus on these operations means in fact looking for a \emph{structure
  theorem}.  So, a key step toward a more realistic
\emph{decomposition} statement applying to all Berge graphs was the
introduction of strong cutsets by Chv\'atal: star cutsets and skew
partitions.  So, at the end of the eighties, most of the ingredients to
state a decomposition theorem for Berge graphs were present.  But as
any cooker knows, having the ingredients does not mean having your
dinner ready.  

Interesting attempts to state reasonable decomposition theorems were
done by Bruce Reed and later disproved by Irena Rusu.  All this is
explained in the introduction of~\cite{chudnovsky.r.s.t:spgt}.
Finally, what seems today to be the good statement was guessed around
2000 by Michele Conforti, G\'erard Cornu\'ejols and Kristina Vu\v
skovi\'c.  Thanks to their experience on the decomposition of
balanced matrices (arising from integer programming), they had a good
intuition of what was going on.  But more than just guessing the
statement, they provided realistic plans toward a proof: a thorough
study of a similar class (even-hole-free graphs, that they studied
with Ajai Kapoor) and a serie of papers proving special cases of the
Strong Perfect Graph Conjecture.  These papers give an extremely
useful tool-box: attachment to prisms, how to deal with line-graphs,
attempts to use the cleaning which lead naturally to
Roussel-and-Rubio-like lemmas\dots\ Also, they point out a theorem of
Truemper on signing the edges of a graph and the related key-role
played by the \emph{Truemper configurations}\index{Truemper configurations} which are 3PC's (a.k.a.\
prisms, pyramids and thetas) and wheels.  All this culminated with a
decomposition theorem for square-free Berge graphs and a decomposition
theorem for odd-hole-free graphs, that does not imply the Strong
Perfect Graph Conjecture but is interesting in its own right.

The methods used in the group of G\'erard Cornu\'ejols were precisely
the kind of things where the group around Paul Seymour was very good
at.  So, it was natural for him to try to decompose Berge graphs.
Robertson and Seymour got a grant for that (this story is interesting,
see~\cite{seymour:how}), and they were joined by Robin Thomas and
later Maria Chudnovsky.  Finally, thanks to an impressive list of
brilliant technicalities (found in only 2 years) and after a very long
proof, they could settle the Strong Perfect Graph Conjecture in May
2002.  The decomposition theorem of Chudnovsky et al. states that
every Berge graph is either basic or decomposable.  Here basic means
bipartite, line-graph of bipartite, or complement of these.  There is
another class, double split graphs, whose definition is postponed to
the next section.  Decomposable means having a 2-join, a homogeneous
pair or a balanced skew partition.  A
\emph{balanced}\index{balanced!skew partition} skew partition is a
skew partition with parity constraints on the paths and antipaths
linking the different boxes of the partition.  Note that all basic
classes are really basic for many questions.  In particular, they are
easy to recognize and fast algorithms exist to optimize them.  The
2-join and the homogeneous pairs are known to be good decompositions,
in the sense that they are reversible and possibly useful for
optimization algorithms.  But as the sequel will show, very few
researches are devoted to optimize with them.  In contrast, so far,
no one seems to know how skew partitions (balanced or not) could be used
in optimization algorithms.

Let me now explain how the work I did takes place in the story summed
up above.  I did my thesis in 2001--2004, during the exciting time
where all these results were presented in several conferences that I
was lucky to attend to.  I was lucky also to be supervised by
Fr\'ed\'eric Maffray who had a broad knowledge of perfect graphs and
spent hours passing it on to me.  He was wise enough to be confident
that structural methods could finally settle the Conjecture; he
experienced this himself by proving a structural description of
claw-free perfect graphs~\cite{maffray.reed:claw}.  So, he was paying
a very careful attention to the work done in the two leading teams.
In the same time, he was trying to use even pairs and in November
2002, during the conference on perfect graphs of Palo Alto, we could
prove together the existence of even pairs in Artemis graphs,
something conjectured by Hazel Everett and Bruce Reed.  Our method was
something Maffray already used with Cl\'audia Linhares Sales for
square-free Artemis graphs, and what we did was using
Roussel-and-Rubio-like lemmas to handle squares.  The same method has
been used recently by Chudnovsky and Seymour to shorten the proof of
the Strong Perfect Graph Theorem.

This was nice, but at the end of my thesis, I had a little problem:
how to find a good research direction in such a field and after all
this?  Giving up perfect graphs and doing something else was an
option, but after spending so many hours reading difficult papers, I
felt it would be a kind of waste.  Pushing further the method that
worked for Artemis graphs was something we already tried quite hard
with Maffray.  The last important open question on perfect graphs,
coloring them by a combinatorial algorithm, was too difficult.  Among
the list of outcomes of the decomposition theorem for Berge graphs,
all were testable in polynomial time except one: the presence of a
balanced skew partition.  But there was a polynomial time algorithm
for skew partition (balanced or not) due to de Figueiredo, Klein,
Kohayakawa and Reed.  Also, algorithms for partitioning graphs into
boxes with prescribed adjacencies in various ways was a subject with
interesting recent progress, see for
instance~\cite{feder.h.k.m:partition}.  So, the detection of balanced
skew partitions seemed to be a problem of reasonable interest and
difficulty.

The obvious way to attack the problem was to see how the people
working on the partition of vertices into boxes solve their problems,
and then see how to handle the ``balanced'' constraint on the parity
of paths.  But this was likely to be a dead end (or to be too
difficult), because the main result of the paper of
Bienstock~\cite{bienstock:evenpair} is that deciding whether there
exists a path of odd length between two given vertices of a graph is
NP-complete.  So, it was clear to me that detecting balanced skew
partitions should be NP-hard, and indeed, it is easy to prove this
with Bienstock's construction, see Theorem~\ref{th:ESPDHard}.

After attending a talk of Celina de Figueiredo in Grenoble (in June
2004), I realised that the problem could be solved by decomposition
methods for \emph{Berge} graphs.  If the graph is basic, there is a
hope to detect balanced skew partitions directly.  If the graph has a
2-join, there is a hope to decompose and to solve the question
recursively.  If a graph is not basic and has no 2-join, just answer
``there is a  balanced skew partition'', the correct answer because of a
theorem of Chudnovsky stating that homogeneous pairs are in fact
useless to decompose Berge graphs.  This blind use of decomposition is
a kind of cheating, but is formally correct.  All this could have been
easy, and I was planning to write something short about that.  But it
was more difficult than expected.  Building the blocks of a 2-join in
the usual way does not preserve balanced skew partitions because of the
so-called path 2-joins; this issue and how to bypass it is explained
Section~\ref{algomotiv}.  Finally, a structure theorem for graphs with
no balanced skew partitions was needed and resulted into 50 pages of
proof sketched in Section~\ref{decth}.  The detection of
balanced skew partition in Berge graphs was solved but the
intermediate structure theorem to solve it was more interesting than
the result itself.

A natural next step was trying to use the structure theorem to
optimize Berge graphs with no balanced skew partitions.  I tried
alone, with Gautier Stauffer, with Sulamita Klein and discussed this a
bit with Maffray and Vu\v skovi\'c.  But all these attempts failed.  I
discussed this again in November 2008 in Rio, with Simone Dantas,
Sulamita Klein and Celina de Figueiredo.  I explained them that
homogeneous pairs were in fact a problem despite the theorem of
Chudnovsky.  Then, Celina suggested to work on Berge graphs with no
balanced skew partitions and no homogeneous pairs.  This was a very
good hint, since with Vu\v skovi\'c we could in fact color these
graphs, see Sections~\ref{sec:structBergenoSP} devoted to the
structure of the class, \ref{sec:colorBergenoSP} devoted to the
computation of maximum cliques and stable set, and~\ref{sec:Lovasz}
devoted to coloring the class.

Before explaining all this, let me say that I am grateful to Celina de
Figueiredo because she gave me good hints for the two papers I wrote
on decomposing Berge graphs.

\section{Decomposing Berge graphs: main results}
\label{sec:decompBerge}

Skew partitions were first introduced by
Chv\'atal~\cite{chvatal:starcutset}.  A \emph{complete
  pair}\index{complete pair} in a graph is a pair $(A, B)$ made of two
disjoint sets of vertices with all possible edges between them. An
\emph{anticomplete pair}\index{anticomplete pair} is a pair that is
complete in the complement of the graph.  A \emph{skew
  partition}\index{skew partition} of a graph $G = (V,E)$ is a
partition of $V$ into two sets $A$ and $B$ such that $A$ induces a
graph that is not connected, and $B$ induces a graph that is not
anticonnected. When $A_1, A_2, B_1, B_2$ are non-empty sets such that
$(A_1, A_2)$ partitions $A$, $(A_1, A_2)$ is an anticomplete pair,
$(B_1, B_2)$ partitions $B$, and ($B_1, B_2$) is a complete pair, we
say that $(A_1, A_2, B_1, B_2)$ is a \emph{split}\index{split!of a
  balanced skew partition} of the skew partition $(A, B)$. \An\
\emph{\even\ skew partition}\index{balanced!skew partition} (first
defined in~\cite{chudnovsky.r.s.t:spgt}) is a skew partition $(A, B)$
with the additional property that every induced path of length at
least~2 with ends in $B$, interior in $A$ has even length, and every
antipath of length at least~2 with ends in $A$, interior in $B$ has
even length. If $(A, B)$ is a skew partition, we say that $B$ is a
\emph{skew cutset}\index{skew cutset}. If $(A, B)$ is \even\ we say
that the skew cutset $B$ is \emph{\even}\index{balanced!skew
  cutset}. Note that Chudnovsky et al.~\cite{chudnovsky.r.s.t:spgt}
proved that no minimum counter-example to the strong perfect graph
conjecture has \an\ \even\ skew partition.

Call \emph{double split graph}\index{double split graph} (first defined
in~\cite{chudnovsky.r.s.t:spgt}) any graph $G$ that can be constructed
as follows.  Let $m,n \geq 2$ be integers. Let $A = \{a_1, \dots,
a_m\}$, $B= \{b_1, \dots, b_m\}$, $C= \{c_1, \dots, c_n\}$, $D= \{d_1,
\dots, d_n\}$ be four disjoint sets. Let $G$ have vertex set $A\cup B
\cup C \cup D$ and edges in such a way that:

\begin{itemize}
\item 
  $a_i$ is adjacent to $b_i$ for $1 \leq i \leq m$.  There are no
  edges between $\{a_i, b_i\}$ and $\{a_{i'}, b_{i'}\}$ for $1\leq i <
  i' \leq m$;
\item 
  $c_j$ is non-adjacent to $d_j$ for $1 \leq j \leq n$. There are all
  four edges between $\{c_j, d_j\}$ and $\{c_{j'}, b_{j'}\}$ for
  $1\leq j < j' \leq n$;
\item
  there are exactly two edges between $\{a_i, b_i\}$ and $\{c_j,
  d_j\}$ for $1\leq i \leq m$, $1 \leq j \leq n$ and these two
  edges are disjoint.
\end{itemize}

Note that $C\cup D$ is a non-\even\ skew cutset of $G$ and that
$\overline{G}$ is a double split graph. Note that in a double split
graph, vertices in $A \cup B$ all have degree $n+1$ and vertices in
$C\cup D$ all have degree $2n + m - 2$. Since $n \geq 2, m \geq 2$
implies $2n -2 + m > 1 + n$, it is clear that given a double split
graph the partition $(A\cup B, C \cup D)$ is unique. 

A graph is said to be \emph{basic}\index{basic!graph} if one of $G, \overline{G}$ is
either a bipartite graph, the line-graph of a bipartite graph or a
double split graph.

The 2-join was first defined by Cornu\'ejols and
Cunningham~\cite{cornuejols.cunningham:2join}.  A partition $(X_1,
X_2)$ of the vertex set is a \emph{2-join}\index{2-join} when there exist disjoint
non-empty $A_i, B_i \subseteq X_i$ ($i=1, 2$) satisfying:

\begin{itemize} 
\item
  every vertex of $A_1$ is adjacent to every vertex of $A_2$ and every
  vertex of $B_1$ is adjacent to every vertex of $B_2$;
\item
  there are no other edges between $X_1$ and $X_2$.
\end{itemize}

The sets $X_1, X_2$ are the two \emph{sides}\index{side!of a 2-join}
of the 2-join.  When sets $A_i$'s $B_i$'s are like in the definition
we say that $(X_1, X_2, A_1, B_1, A_2, B_2)$ is a
\emph{split}\index{split!of a 2-join} of $(X_1, X_2)$.  Implicitly,
for $i= 1, 2$, we will denote by $C_i$ the set $X_i \setminus (A_i
\cup B_i)$.

A 2-join $(X_1, X_2)$ in a graph $G$ is said to be \emph{connected}\index{connected!2-join}\index{2-join!connected}
when for $i= 1, 2$, every component of $G[X_i]$ meets both $A_i$ and
$B_i$.  A 2-join $(X_1, X_2)$ is said to be \emph{substantial}\index{substantial!2-join}\index{2-join!substantial} when
for $i= 1, 2$, $|X_i| \geq 3$ and $X_i$ is not a path of length~2 with
an end in $A_i$, an end in $B_i$ and its unique interior vertex in
$C_i$.  A 2-join $(X_1, X_2)$ in a graph $G$ is said to be
\emph{proper}\index{proper!2-join}\index{2-join!proper} when it is connected and substantial.

A 2-join is said to be a \emph{path 2-join}\index{path
  2-join}\index{2-join!path} if it has a split $(X_1, X_2, A_1, B_1,
A_2, B_2)$ such that $G[X_1]$ is a path with an end in $A_1$, an end
in $B_1$ and interior in $C_1$. Implicitly we will then denote by
$a_1$ the unique vertex in $A_1$ and by $b_1$ the unique vertex in
$B_1$. We say that $X_1$ is the \emph{path-side}\index{path-side of a
  2-join} of the 2-join. Note that when $G$ is not a hole then only
one of $X_1, X_2$ is a path side of $(X_1, X_2)$. A \emph{non-path
  2-join}\index{non-path 2-join}\index{2-join!non-path} is a 2-join
that is not a path 2-join.

The homogeneous pair was first defined by Chv\'atal and
Sbihi~\cite{chvatal.sbihi:bullfree}. The definition that we give here
is a slight variation used in~\cite{chudnovsky.r.s.t:spgt}. A
\emph{homogeneous pair}\index{homogeneous pair} is a partition of
$V(G)$ into six non-empty sets $(A, B, C, D, E, F)$ such that:

\begin{itemize} 
\item 
  every vertex in $A$ has a neighbor in $B$ and a non-neighbor in $B$,
  and vice versa; 
\item the pairs $(C,A)$, $(A,F)$, $(F,B)$, $(B,D)$ are complete; 
\item the pairs $(D,A)$, $(A,E)$, $(E,B)$, $(B,C)$ are anticomplete. 
\end{itemize}

A graph $G$ is path-cobipartite if it is a Berge graph obtained by
subdividing an edge between the two cliques that partitions a
cobipartite graph.  More precisely, a graph is
\emph{path-cobipartite}\index{path-cobipartite} if its vertex set can
be partitioned into three sets $A, B, P$ where $A$ and $B$ are
non-empty cliques and $P$ consists of vertices of degree~2, each of
which belongs to the interior of a unique path of odd length with one
end $a$ in $A$, the other one $b$ in $B$. Moreover, $a$ has neighbors
only in $A \cup P$ and $b$ has neighbors only in $B \cup P$. Note that
a path-cobipartite graph such that $P$ is empty is the complement of
bipartite graph. Note that our path-cobipartite graphs are simply the
complement of the \emph{path-bipartite}\index{path-bipartite} graphs
defined by Chudnovsky in~\cite{chudnovsky:these}. For convenience, we
prefer to think about them in the complement as we do.

 A \emph{double star}\index{double star cutset} in a graph is a subset $D$ of the
vertices such that there is an edge $ab$ in $G[D]$ satisfying: $D
\subset N(a) \cup N(b)$.

Now we can state the known decomposition theorems of Berge graphs. The
first decomposition theorem for Berge graph ever proved is the
following:

\begin{theorem}[Conforti, Cornu\'ejols and Vu\v skovi\'c, 
    2001, \cite{conforti.c.v:dstrarcut}]
  \label{th.ccv}
  Every graph with no odd hole is either basic or has a proper 2-join
  or has a double star cutset.
\end{theorem}

It could be thought that this theorem is useless to prove the Strong
Perfect Graph Theorem since there are minimal imperfect graphs that
have double star cutsets: the odd antiholes of length at
least~7. However, by the Strong Perfect Graph Theorem, we know that
the following fact is true: for any minimal non-perfect graph $G$, one
of $G, \overline{G}$ has no double star cutset. A direct proof of this
--- of which we have no idea --- would yield together with
Theorem~\ref{th.ccv} a new proof of the Strong Perfect Graph Theorem.

The following theorem was first conjectured in a slightly different
form by Conforti, Cornu\'ejols and Vu\v skovi\'c, who proved it in the
particular case of square-free graphs~\cite{conforti.c.v:square}.  A
corollary of it is the Strong Perfect Graph Theorem.

\begin{theorem}[Chudnovsky, Robertson, Seymour and Thomas, 2002, \cite{chudnovsky.r.s.t:spgt}]
  \label{th.0}
  Let $G$ be a Berge graph. Then either $G$ is basic or $G$ has a
  homogeneous pair, or $G$ has \an\  \even\  skew partition or one of $G,
  \overline{G}$ has a proper 2-join.
\end{theorem} 

 The two theorems that we state now are due to Chudnovsky who proved
 them from scratch, that is without assuming Theorem~\ref{th.0}. Her
 proof uses the notion of \emph{trigraph}\index{trigraph}.  The first theorem shows
 that homogeneous pairs are not necessary to decompose Berge
 graphs. Thus it is a result stronger than Theorem~\ref{th.0}. The
 second one shows that path 2-joins are not necessary to decompose
 Berge graphs, but at the expense of extending \even\ skew partitions
 to general skew partitions and introducing a new basic class. Note
 that a third theorem can be obtained by viewing the second one in the
 complement of $G$.

\begin{theorem}[Chudnovsky, 2003, \cite{chudnovsky:trigraphs,chudnovsky:these}]
\label{th.1}
  Let $G$ be a Berge graph. Then either $G$ is basic, or one of $G,
  \overline{G}$ has a proper 2-join or $G$ has \an\  \even\  skew partition.
\end{theorem}

\begin{theorem}[Chudnovsky, 2003, \cite{chudnovsky:these}]
\label{th.2}
  Let $G$ be a Berge graph. Then either $G$ is basic, or one of $G,
  \overline{G}$ is path-bipartite, or $G$ has a proper non-path
  2-join, or $\overline{G}$ has a proper 2-join, or $G$ has a
  homogeneous pair or $G$ has a skew partition.
\end{theorem}

\section{Detection of balanced skew partitions}
\label{algomotiv}

Before explaining Theorem~\ref{th.th}, a new decomposition for Berge
graphs that is a generalization of Theorems~\ref{th.0},~\ref{th.1}
and~\ref{th.2}, let us explain its initial motivation: the detection
of balanced skew partitions.  De Figueiredo, Klein, Kohayakawa and
Reed devised an algorithm that given a graph $G$ computes in
polynomial time a skew partition if $G$ has
one~\cite{figuereido.k.k.r:sp}.  See also a recent work by Kennedy and
Reed~\cite{kennedyreed:skew}.  Let us call \ESPD\ the decision problem
whose input is a graph and whose answer is YES if the graph has \an\
\even\ skew partition and NO otherwise. Using a construction due to
Bienstock~\cite{bienstock:evenpair}, we can prove the following.

\begin{theorem}[\cite{nicolas:bsp}]
  \label{th:ESPDHard}
  \ESPD\ is NP-hard.
\end{theorem}

Note that we are not able to prove that \ESPD\ is in NP or in
coNP. But using Theorem~\ref{th.th}, we give an $O(n^9)$-time algorithm
for \ESPD\ restricted to Berge graphs.  Let us explain how.

In 2002, Chudnovsky, Cornu\'ejols, Liu, Seymour and Vu\v
skovi\'c~\cite{chudnovsky.c.l.s.v:reco} gave an algorithm that
recognizes Berge graphs in time $O(n^9)$.  This algorithm may be used
to prove that, when restricted to Berge graphs, \ESPD\ is in NP.
Indeed, \an\  \even\  skew partition is a good certificate for \ESPD:
given a Berge graph and a partition $(A,B)$ of its vertices, one can
easily check that $(A,B)$ is a skew partition; to check that it is
\even, it suffices to add a vertex adjacent to every vertex of $B$, to
no vertex of $A$, and to check that this new graph is still Berge.

Proving that \ESPD\ is actually in P by a decomposition theorem uses a
classical idea, used for instance in~\cite{conforti.c.k.v:eh2} to
check whether a given graph has or not an even hole. First, solve
\ESPD\ for each class of basic graph. This can be done in time
$O(n^5)$. Note that bipartite graphs are the most difficult to handle
efficiently.  For them, we use an algorithm due to
Reed~\cite{reed:skewhist} for general skew partitions; and since
balanced and general skew partitions are the same thing in bipartite
graphs (see below), an algorithm follows.

\begin{lemma}[\cite{nicolas:bsp}]
  Let $G$ be a bipartite graph. Then $(A,B)$ is a skew partition of
  $G$ if and only if it is \an\ \even\ skew partition of $G$.
\end{lemma}

The following two lemmas show how to handle the other basic classes.
Note that the balanced skew partition is a self-complementary notion,
so the complements of the classes need no special treatment.

\begin{lemma}[\cite{nicolas:bsp}] 
  \label{l.dsg}
  A double split graph $G$ has exactly one skew partition and
  this skew partition is not \even. 
\end{lemma}

\begin{lemma}[\cite{nicolas:bsp}]
  \label{butterequiv}
  Let $G$ be the line-graph of a bipartite graph. Suppose that $G$ has
  at least one edge and at least~5 vertices.  Then $G$ has \an\  \even\  skew
  partition if and only if $G$ has a star cutset.
\end{lemma}

For a graph $G$ such that one of $G, \overline{G}$ has a 2-join, we
try to break $G$ into smaller blocks in such a way that $G$ has \an\
\even\ skew partition if and only if one of the blocks has one,
allowing us to run recursively the algorithm. And when a graph is not
basic and has no 2-join, we simply answer ``the graph has \an\ \even\
skew partition'', which is the correct answer because of the
Decomposition Theorem~\ref{th.1}. This blind use of decomposition is
not safe from criticism, but this will be discussed later.

Let us now explain why without further work, this idea fails to solve
\ESPD.  Building the blocks of a 2-join preserves existing \even\ skew
partitions, but some 2-joins can create \even\ skew partitions when
building the blocks carelessly.  For non-path 2-join, there is no
problem: sides of the 2-join can be replaced safely by sufficiently
long paths.  But for path 2-joins, there is a problem.  In the graph
represented in Fig.~\ref{figP3loose} on the left, we have to simplify
somehow the left part of the obvious path 2-join to build one of the
blocks.  The most reasonable way to do so seems to be replacing $X_1$
by a path of length~1.  But this creates a skew cutset: the black
vertices on the right. Of course, this graph is bipartite but one can
find more complicated examples based on the same template, and another
template exists.  These bad 2-joins will be described in more details
in Section~\ref{decth} and called~\emph{cutting 2-joins}\index{cutting
  2-join}\index{2-join!cutting}. All of them are path 2-joins.

\begin{figure}[h]
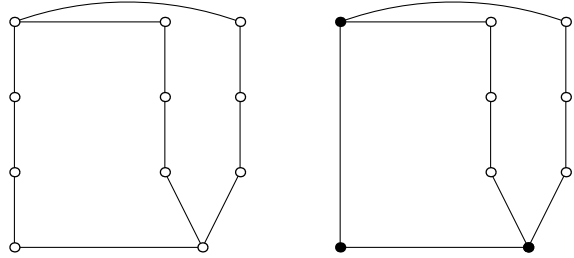
  
  \center
  \includegraphics{evenskew.13}\hspace{3em}\includegraphics{evenskew.14}
  \caption{Contracting a path creates a skew cutset\label{figP3loose}}
\end{figure}

Theorem~\ref{th.th} shows that cutting 2-joins are not necessary to
decompose Berge graphs.  A more general statement is proved, that
makes use of a new basic class and of a new kind of decomposition that
are quite long to describe. But an interesting corollary can be stated
with no new notions. By \emph{contracting a path
  $P$}\index{contracting!a path} that is the side of a proper path
2-join of a graph we mean delete the interior vertices of $P$, and
link the ends of $P$ with a path of length~1 or~2 according to the
original parity of the length of $P$.

\begin{theorem}[\cite{nicolas:bsp}]
  \label{th.case}
  Let $G$ be a Berge graph. Then either:
  \begin{itemize}
    \item $G$ is basic;
    \item one of $G, \overline{G}$ has a non-path proper 2-join;
    \item $G$ has no \even\ skew partition and exactly one of $G,
      \overline{G}$ (say $G$) has a proper path 2-join. Moreover, for
      every proper path 2-join of $G$, the graph obtained by
      contracting its path-side has no \even\ skew partition;
    \item $G$ has \an\  \even\  skew partition.
  \end{itemize}
\end{theorem}

The algorithm for detecting \even\ skew partitions is now easy to
sketch. Since the \even\ skew partition is a self-complementary
notion, we may switch from the graph to its complement as often as
needed. First check whether the input graph is basic, and if so look
directly for \an\ \even\ skew partition. Else, try to decompose along
non-path 2-joins (they preserve the existence of \even\ skew
partitions).  If there are none of them, try to decompose along path
2-joins (possibly, this creates \even\ skew partitions but does not
destroy them). At the end of this process, one of the leaves of the
decomposition tree has \an\ \even\ skew partition if and only if the
root has one. Note that \an\ \even\ skew partition in a leaf may have
been created by the contraction of a cutting 2-join since such 2-joins
do exist (we are not able to recognize all of them, it seems to be a
difficult task). But Theorem~\ref{th.case} shows that when such a bad
contraction occurs, the graph has anyway \an\ \even\ skew cutset
somewhere.  

All this runs in time $O(n^9)$.  The complexity is the same as the
recognition of Berge graphs but this is purely coincidental.  The
complexity bottleneck of our algorithm is the routine to search for a
non-path 2-join.  This takes $O(n^8)$, and speed it up to $O(n^5)$
would speed up the algorithm.  Speed it up further would be useless
because of the second bottleneck which is detecting a skew partition
in a bipartite graph in time $O(n^5)$.

Let us come back to the weak point of our recognition algorithm, which
is when it answers ``the graph has \an\ \even\ skew-partition'' using
blindly some decomposition theorem. This weakness is the reason why we
are not able to find explicitly \an\ \even\ skew partition when there
is one.  However, our result suggests that an explicit algorithm might
exist. The proof of Theorem~\ref{th.0} or Theorem~\ref{th.1} might
contain its main steps. However, we would like to point out that if
someone manage to read algorithmically the proof of Theorem~\ref{th.0}
or of Theorem~\ref{th.1}, (s)he will probably end up with an algorithm
that given a graph, either finds an odd hole/antihole, or certifies
that the graph is basic, or finds some decomposition. If the
decomposition found is not \an\ \even\ skew partition, the algorithm
will probably not certify that there is no \even\ skew partition in
the graph, and thus \ESPD\ will not be solved entirely. To solve it,
one will have to think about the detection of \even\ skew partitions
in basic graphs, and in graphs having a 2-join: this is what we are
doing here.  Thus an effective algorithm might have to use much of the
present work.

\section{A new decomposition theorem for Berge graphs}
\label{decth}

As stated in Section~\ref{algomotiv}, our main problem for the
detection of \even\ skew partitions is the possibility of path 2-joins
in Berge graphs. One could hope that these 2-joins are actually not
necessary to decompose Berge graphs. Theorem~\ref{th.2} indicates that
such a hope is realistic, but this theorem allows non-\even\ skew
partitions, so it is useless for our purpose. What we would like is to
prove something like Theorem~\ref{th.1} with ``non-path 2-join''
instead of ``2-join''. Let us call this statement our
\emph{conjecture}. A simple idea to prove the conjecture would be to
consider a minimum counter-example $G$, that is: a Berge graph,
non-basic, with no \even\ skew partition, and no non-path 2-join. Such
a graph must have a path 2-join by Theorem~\ref{th.1} (possibly after
taking the complement).  Here is why we need Theorem~\ref{th.1} in our
proof.  The idea is now to use this path 2-join to build a smaller
graph $G'$ that is also a counter-example, and this is a contradiction
which proves the conjecture.

So, given $G$ with its path 2-join, how can we build a smaller graph
that will have ``almost'' the same structure as $G$ ?  Obviously,
this can be done by contracting the path-side of the 2-join. Let us
call $G_c$ the graph that we obtain. It has to be proved that $G_c$ is
still a counter-example to the conjecture. But we know that this can
be false. Indeed, if the path 2-join of $G$ is cutting,
\an\ \even\ skew partition can be created in $G_c$, so $G_c$ is not a
counter-example. We need now to be more specific and to define cutting
2-joins.

\begin{figure}[ht]
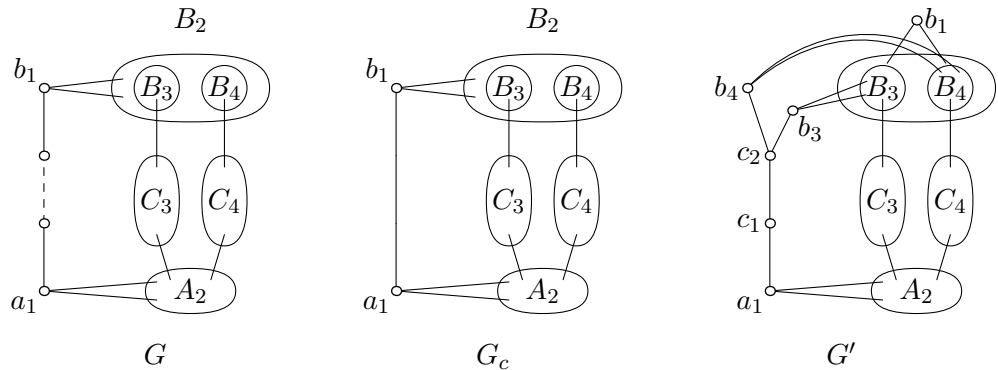

  \begin{center}
    \begin{tabular}{ccc}
      \includegraphics{evenskew.15}\rule{1em}{0ex}&
      \rule{1em}{0ex}\includegraphics{evenskew.19}\rule{1em}{0ex}&
      \rule{1em}{0ex}\includegraphics{evenskew.16}\\
      \rule{0em}{3ex}$G$&\rule{0em}{3ex}$G_c$&\rule{0em}{3ex}$G'$
    \end{tabular}
  \end{center}
  \caption{A graph $G$ with a cutting 2-join of type~1 and the
    associated graph $G'$\label{fig:cut1}}
\end{figure}

A 2-join is said to be \emph{cutting of type~1}\index{cutting
  2-join!type 1}\index{2-join!cutting of type 1} if it has a split
$(X_1,$ $X_2,$ $A_1,$ $B_1,$ $A_2,$ $B_2)$ such that:

  \begin{enumerate}
  \item
    $(X_1, X_2)$ is a path 2-join with path-side $X_1$; 

  \item
    $G[X_2 \setminus A_2]$ is disconnected.
  \end{enumerate}

  In Fig.~\ref{fig:cut1} the structure of a graph $G$ with a cutting
  2-join of type~1 is represented. Obviously, after contracting the
  path-side into an edge $a_1b_1$, we obtain a graph $G_c$ with a
  potentially-\even\ skew cutset $\{a_1, b_1\} \cup A_2$ that
  separates $C_3\cup B_3$ from $C_4 \cup B_4$. So, how can we find a
  graph smaller than $G$ that is still a counter-example to the
  conjecture? Our idea is to build the graph $G'$, also represented in
  Fig.~\ref{fig:cut1}.  If we count vertices, $G'$ is not ``smaller''
  than $G$, but in fact, by ``minimum counter-example'' we mean
  counter-example with a minimum number of path 2-joins. We can prove
  that $G'$ is smaller in this sense (this is not trivial because we
  have to prove that path 2-joins cannot be created in $G'$, but
  clearly, one path 2-join is destroyed in $G'$). We can also prove
  that $G'$ is a counter-example which gives the desired
  contradiction. This is the first case of the proof of
  Theorem~\ref{th.th}.  Note that finding the right graph $G'$ has
  been long and painful (in particular the strange little hat $b_1$
  that is so useful).  But once $G'$ is found, the proof is not really
  difficult.  Yet, it is quite long (about 10 pages) and relies on
  many easy lemmas on parities of various paths and antipaths.
  Proving that $G'$ has no non-path 2-join is the main difficulty and
  is very tedious.  A good exercise is to prove that $G'$ is Berge.

\begin{figure}
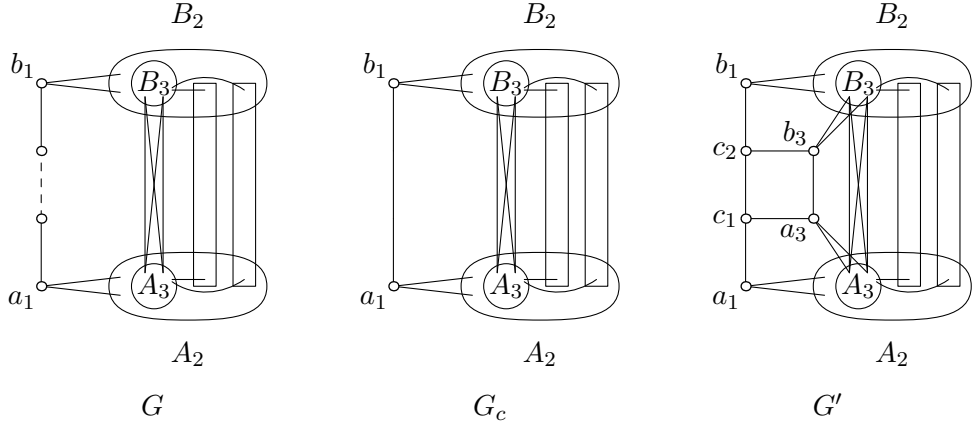

  \begin{center}
    \begin{tabular}{ccc}
      \includegraphics{evenskew.17}\rule{1em}{0ex}&
      \rule{1em}{0ex}\includegraphics{evenskew.20}\rule{1em}{0ex}&
      \rule{1em}{0ex}\includegraphics{evenskew.18}\\
      \rule{0em}{3ex}$G$&\rule{0em}{3ex}$G_c$&\rule{0em}{3ex}$G'$
    \end{tabular}
  \end{center}
  \caption{A graph $G$ with a cutting 2-join of type~2 and the
    associated graph $G'$\label{fig:cut2}}
\end{figure}

Unfortunately, there is another kind of path 2-join that can create
\even\ skew partitions when contracting the path-side.  A 2-join is
said to be \emph{cutting of type~2}\index{cutting 2-join!type
  2}\index{2-join!cutting of type 2} if it has a split $(X_1,$ $X_2,$
$A_1,$ $B_1,$ $A_2,$ $B_2)$ such that there exist sets $A_3$, $B_3$
satisfying:

  \begin{enumerate}
  \item \label{cond.first}
    $(X_1, X_2)$ is a path 2-join with path-side $X_1$; 
  \item
    $A_3 \neq \emptyset$, $B_3 \neq \emptyset$, $A_3 \subset A_2$,
    $B_3 \subset B_2$;
  \item
    $A_3$ is complete to $B_3$;
  \item
    every outgoing path from $B_3\cup \{a_1\}$ to $B_3 \cup\{a_1\}$
    (resp. from $A_3\cup \{b_1\}$ to $A_3 \cup\{b_1\}$) has even
    length;
  \item \label{cond.penul} every antipath of length at least~2 with
    its ends outside  $B_3\cup \{a_1\}$ (resp. $A_3 \cup\{b_1\}$)
    and its interior in $B_3 \cup\{a_1\}$ (resp. $A_3\cup \{b_1\}$)
    has even length;
  \item \label{cond.disconnect}
    $G \setminus (X_1 \cup A_3 \cup B_3)$ is disconnected.
  \end{enumerate}

In Fig.~\ref{fig:cut2}, the structure of a graph $G$ with a cutting
2-join of type~2 is represented. After contracting the path-side into an
edge $a_1b_1$, we obtain a graph $G_c$ with \an\ \even\ skew cutset $\{a_1,
b_1\} \cup A_3 \cup B_3$. It is ``skew'' because $a_1\cup B_3$ is
complete to $b_1 \cup A_3$, and it is \even\ by the parity constraints
in the definition.  How can we find a graph smaller than $G$ that is
still a counter-example to the conjecture ?  Again, we find a graph
$G'$, also represented in Fig.~\ref{fig:cut2}.  Again, we prove that $G'$ is a
smaller counter-example, a contradiction. This is the second case of
the proof of Theorem~\ref{th.th}. 

\label{why}

Note that we are not able to prove something like
Theorem~\ref{th.case} with ``skew partition'' instead of ``\even\ skew
partition''. Following our frame, we would have to give up the
conditions on the parity of paths in the definition of cutting 2-joins
of type~2. But then we would not be able to prove that $G'$ is Berge,
making the whole proof collapse. Also we would like to explain a
little twist in our proof. In fact Case~2 is not ``the 2-join is
cutting of type 2'', but something slightly more general: ``the 2-join
is such that there are sets $A_3$, $B_3$ satisfying the
items~\ref{cond.first}--\ref{cond.penul} of the definition of cutting
2-joins of type~2''. Indeed, in Case~2, we do not need to use the last
item. And this has to be done, since in Case~3, at some place where we
need a contradiction, we find a 2-join that is almost of type~2, that
satisfies items~\ref{cond.first}--\ref{cond.penul}, and not the last
one.

A 2-join is said to be \emph{cutting}\index{cutting
  2-join}\index{2-join!cutting} if it is either cutting of type~1 or
cutting of type~2. So, in our main proof we can get rid of cutting
2-joins as explained above.  In \cite{nicolas:bsp}, we also study how
a 2-join and \an\ \even\ skew partition can overlap in a Berge
graph. The main result of this study  says that when contracting the
path side of a non-cutting 2-join, no \even\ skew partition is
created. So if we come back to our main proof, we can at last build
$G'$ ``naturally'', that is by contracting the path-side of the 2-join
in $G$. This is the third case of the proof. Transforming $G$ into
$G'$ will not create a \even\ skew partition as just mentioned.  We
need to prove also that no 2-join is created.  This might happen but
then, an analysis of the adjacencies in $G$ shows that $G$ has a
2-join that is almost cutting of type~2 (``almost'' because the last
item of the definition of cutting 2-joins of type~2 does not
hold). This is a contradiction since we are in Case~3.  But the
contraction may create other nasty things.

For instance suppose that $G$ is obtained by subdividing an edge of
the complement of a bipartite graph. Then, contracting the path-side
of the path 2-join of $G$ yields the complement of a bipartite
graph. This is why we have to view path-cobipartite graphs as basic in
our main theorem. Note that Chudnovsky also has to consider these
graphs as basic in her Theorem~\ref{th.2}.
 
Suppose now that $G$ is obtained from a double split graph $H$ by
subdividing matching edges of $H$ into paths of odd length.  Such a
graph has a path 2-join whose contraction may yield a basic graph,
namely a double split graph. Let us define this more precisely.

We call \emph{flat path of a graph $H$}\index{flat!path} any path
whose interior vertices all have degree~2 in $H$ and whose ends have
no common neighbors outside the path.  A \emph{path-double split
  graph}\index{path-double split graph} is any graph $H$ that may be constructed as follows.  Let
$m,n \geq 2$ be integers. Let $A = \{a_1, \dots, a_m\}$, $B= \{b_1,
\dots, b_m\}$, $C= \{c_1, \dots, c_n\}$, $D= \{d_1, \dots, d_n\}$ be
four disjoint sets. Let $E$ be another possibly empty set disjoint
from $A$, $B$, $C$, $D$. Let $H$ have vertex set $A\cup B \cup C \cup
D \cup E$ and edges in such a way that:

\begin{itemize}
\item for every  vertex $v$ in $E$, $v$ has degree~2 and there exists
  $i \in \{1, \dots m\}$ such that $v$ lies on a
  path of odd length from $a_i$ to $b_i$; 
\item 
  for $1 \leq i \leq m$, there is a unique path of odd length
  (possibly~1) between $a_i$ and $b_i$ whose interior is in $E$.
  There are no edges between $\{a_i, b_i\}$ and $\{a_{i'}, b_{i'}\}$
  for $1\leq i < i' \leq m$;
\item 
  $c_j$ is non-adjacent to $d_j$ for $1 \leq j \leq n$. There are all
  four edges between $\{c_j, d_j\}$ and $\{c_{j'}, b_{j'}\}$ for
  $1\leq j < j' \leq n$;
\item
  there are exactly two edges between $\{a_i, b_i\}$ and $\{c_j,
  d_j\}$ for $1\leq i \leq m$, $1 \leq j \leq n$ and these two
  edges are disjoint.
\end{itemize}

Let us come back to our main proof. Adding path-cobipartite graphs and
path-double split graphs as basic graphs in our conjecture is not
enough.  Because we need to prove that when contracting a path 2-join,
no 2-join in the complement is created, and that the counter-example
is not transformed into the complement of the line-graph of a
bipartite graph. And, unfortunately, both things may happen. But a
careful analysis of these phenomenons shows that such graphs have a
special structure that we must add to our conjecture: a
\emph{homogeneous 2-join}\index{homogeneous
  2-join}\index{2-join!homogeneous} is a partition of $V(G)$ into six
non-empty sets $(A,$ $B,$ $C,$ $D,$ $E,$ $F)$ such that:

\begin{itemize} 
\item 
  $(A, B, C, D, E, F)$ is a homogeneous pair;
\item 
  every vertex in $E$ has degree~2 and belongs to a flat path of odd
  length with an end in $C$, an end in $D$ and whose interior is in
  $E$;
 \item 
   every flat path outgoing from $C$ to $D$ and whose interior is in
   $E$ is the path-side of a non-cutting proper 2-join of $G$.
\end{itemize}

\noindent Now, we have defined all the new basic classes and
decompositions that we need.  Our main result is the following.

\begin{theorem}[\cite{nicolas:bsp}]
  \label{th.th}
  Let $G$ be a Berge graph. Then either $G$ is basic, or one of $G,
  \overline{G}$ is a path-cobipartite graph,  or one of $G,
  \overline{G}$ is a path-double split graph, or one of $G,
  \overline{G}$ has a homogeneous 2-join, or one of $G, \overline{G}$
  has a non-path proper 2-join, or $G$ has \an\  \even\  skew partition.
\end{theorem}

Of course, in the proof sketched above, the graph $G$ is a
counter-example to Theorem~\ref{th.th}, not to the original
conjecture: ``Theorem~\ref{th.1} where path 2-joins are not
allowed''. So we need to be careful that our construction of graphs
$G'$ in cases~1, 2, 3 does not create a homogeneous 2-join and does not
yield a path-double split graph or a path-cobipartite graph. This
might have happened, and we would then have had to classify the
exceptions by defining new basic classes and decompositions, and this
would have lead us to a perhaps endless process. Luckily this process
ends up after just one step.

\begin{figure}[h]
  \begin{center}
    \includegraphics{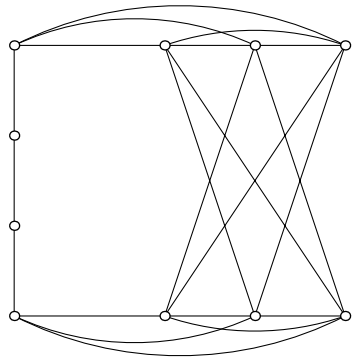}
    \caption{A path-cobipartite graph\label{fig:contrex1}}
  \end{center}
\end{figure}

\begin{figure}[h]
  \begin{center}
    \includegraphics{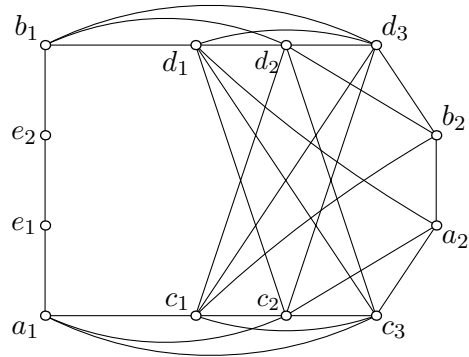}
    \caption{A path-double split graph\label{fig:contrex2}}
  \end{center}
\end{figure}

\begin{figure}[h]
  \begin{center}
    \includegraphics{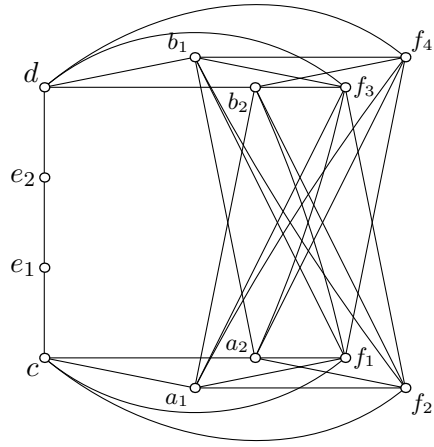}
    \caption{A graph that has a homogeneous 2-join $(\{a_1, a_2\},$
      $\{b_1, b_2\},$ $\{c\}, \{d\},$ $\{e_1, e_2\},$ $\{f_1, f_2,$ $f_3,
      f_4\})$\label{fig:contrex3}}
  \end{center}
\end{figure}

Theorem~\ref{th.th} generalizes Theorems~\ref{th.0},~\ref{th.1}
and~\ref{th.2}: path-cobipartite graphs may be seen either as graphs
having a proper path 2-join (Theorems~\ref{th.0} and~\ref{th.1}) or as
a new basic class (Theorem~\ref{th.2}). Path-double split graphs may
be seen as graphs having a proper path 2-join (Theorems~\ref{th.0}
and~\ref{th.1}) or as graphs having a non-\even\ skew partition
(Theorem~\ref{th.2}). And graphs having a homogeneous 2-join may be
seen as graphs having a homogeneous pair (Theorems~\ref{th.2} and
perhaps~\ref{th.0}) or as graphs having a proper path 2-join
(Theorems~\ref{th.1} and perhaps~\ref{th.0}). Formally all these
remarks are not always true: it may happen in special cases that
path-cobipartite graphs and path-double split graphs have no proper
2-join because the ``proper'' condition fails. But such graphs are
easily established to be basic or to have \an\ \even\ skew partition.

Note also that our new basic classes and decomposition yield
counter-examples to reckless extensions of Theorems~\ref{th.1}
and~\ref{th.2}.  The three graphs represented in
Fig.~\ref{fig:contrex1}, \ref{fig:contrex2}, \ref{fig:contrex3} are
counter-examples to our original conjecture, that is the extension of
Theorem~\ref{th.1} where path 2-joins are not allowed. Path-double
split graphs yield counter-examples to Theorem~\ref{th.2} with
``\even\ skew partition'' instead of ``skew partition'' (see
Fig.~\ref{fig:contrex2}). Graphs with a homogeneous 2-join yield
counter-examples to Theorem~\ref{th.2} where homogeneous pairs are not
allowed (see Fig.~\ref{fig:contrex3}). This shows that
Theorems~\ref{th.1},~\ref{th.2} are in a sense best possible, and that
to improve them, we need to do what we have done: add more basic
classes and decomposition. The three graphs represented in
Fig.~\ref{fig:contrex1}, \ref{fig:contrex2}, \ref{fig:contrex3} also
show that path cobipartite graphs, path-double split graphs and
homogeneous 2-join must somehow appear in our theorem, that is also in
a sense best possible.

This work suggests an algorithm for \ESPD\ with no reference to a new
decomposition theorem. Indeed, the graph $G'$ represented in
Fig.~\ref{fig:cut1} (resp. in Fig.~\ref{fig:cut2}) is a good candidate
to serve as a block of a cutting 2-joins of type~1 (resp. of
type~2). The fact that $G'$ is bigger than $G$ is not really a
problem, since the number of path 2-joins in a graph can be an
ingredient of a good notion of size. So, an algorithm might try to
deal with path 2-joins by constructing the appropriate block when the
2-join is recognized to be cutting. In fact this was our original idea
but it fails: we are not able to recognize cutting 2-joins of type~2.
To do this, we would have to guess somehow the sets $A_3, B_3$. But
this seems to be exactly the problem of detecting \even\ skew
partitions, so we are sent back to our original question. Perhaps an
astute recursive call to the algorithm would finally bypass this
difficulty, at the possible expense of a worse running time. Anyway,
we prefer to proceed as we have done, since a new decomposition for
Berge graphs is valuable in itself.

Theorem~\ref{th.case} gives a structural description of Berge graphs
that have no \even\ skew partitions: these graphs can be decomposed
along 2-joins (and complements of 2-join) till reaching basic graphs.
This could be used to solve algorithmic problems for the class of
Berge graphs with no \even\ skew partitions (together with the Berge
graphs recognition algorithm~\cite{chudnovsky.c.l.s.v:reco}, our work
solves the recognition in $O(n^9)$).  This class has an unusual
feature in the field of perfect graphs: it is not closed under taking
induced subgraphs. Theorem~\ref{th.case} also gives a structural
information on every Berge graph: it can be decomposed in a first step
by using only \even\ skew partitions, and in a second step by using
only 2-joins, possibly in the complement.

A strange feature of the algorithm is that it does not work for
general skew partitions in Berge graphs.  Because as explained above,
we are not able to prove Theorem~\ref{th.case} with ``skew partition''
instead of ``\even\ skew partition''.  Rather than a failure, we
consider this as a further indication that \even\ skew partition is a
relevant decomposition for Berge graphs.

\section{Decomposing Berge graphs with no skew partitions and no
  homogeneous pairs}
\label{sec:structBergenoSP}

\newcommand{\Cl}[2]{{\cal C}^{\text{\scriptsize\sc #1}}_{\text{\scriptsize\sc #2}}}

Theorem~\ref{th.th} could be used to devise optimization algorithms
for Berge graphs with no balanced skew partitions.  But these graphs
may have homogeneous 2-joins; recall that they are a mixture of path
2-joins and homogeneous pairs, both decompositions that we do not know
how to handle.  Note that the graph represented in
Figure~\ref{fig:contrex3} is uniquely decomposable by a path 2-join or
a homogeneous pair.  The aim of this section is to show that excluding
homogeneous pairs gives a class with a very useful structure theorem.
So, call $\Cl{Berge}{no cutset}$ the class of Berge graphs with no
balanced skew partition and no homogeneous pair.

The following theorem is an obvious corollary of Theorem~\ref{th.th}.
This theorem uses non-path 2-joins and, thanks to them, allow to solve
the four classical optimization problems: maximum clique, maximum
stable set, coloring and coloring of the complement.  Note that using
2-joins for optimization seems to be something new, despite folklore
rumors saying that it should be possible.  The sequel shows that to
use 2-joins, many technicalities have to be dealt with.

Call $\Cl{Berge}{basic}$ the class of bipartite, line-graphs of
bipartite, path-cobipartite, path-double split graphs and the
complements of all these graphs.

\begin{theorem}
  \label{th.3}
  If $G$ is in $\Cl{Berge}{no cutset}$, then either $G$ is in
  $\Cl{Berge}{basic}$ or one of $G, \overline{G}$ has a proper
  non-path 2-join.
\end{theorem}

Let $\Cl{parity}{}$ the class of these graphs such that all holes have
same parity.  We need to define the blocks of decomposition with
respect to a 2-join.  They are built by replacing each side of the
2-join by a path and the lemma below shows that for graphs in
$\Cl{parity}{}$ there exists a unique way to choose the parity of that
path.

\begin{lemma}
  \label{l.2jAiBi}
  Let $G$ be a graph in $\Cl{parity}{}$ and $(X_1,X_2,A_1,B_1,A_2,B_2)$
  be a split of a proper 2-join of $G$.  Then for $i=1,2$, all the
  paths with an end in $A_i$, an end in $B_i$ and interior in $C_i$
  have the same parity.
\end{lemma}

\label{sec.defpiece}
Let $G$ be a graph and $(X_1, X_2, A_1, B_1, A_2, B_2)$ be a split of
a proper 2-join of $G$.  Let $k_1, k_2 \geq 1$ be integers.  The
\emph{blocks of decomposition}\index{blocks of
  decomposition!w.r.t. 2-join} of $G$ with respect to $(X_1, X_2)$ are
the two graphs $G^{k_1}_1, G^{k_2}_2$ that we describe now.  We obtain
$G^{k_1}_1$ by replacing $X_2$ by a \emph{marker path}\index{marker path} $P_2$, of
length $k_1$, from a vertex $a_2$ complete to $A_1$, to a vertex $b_2$
complete to $B_1$ (the interior of $P_2$ has no neighbor in $X_1$).
The block $G_2^{k_2}$ is obtained similarly by replacing $X_1$ by a
marker path $P_1$ of length $k_2$.  We say that $G^{k_1}_1$ and
$G^{k_2}_2$ are {\em parity-preserving} if $G$ is in $\Cl{parity}{}$,
for $i=1,2$ and for a path $Q_i$ from $A_i$ to $B_i$ whose
intermediate vertices are in $C_i$ (and such a path exists since
$(X_1,X_2)$ is connected), the marker path $P_i$ has the same parity
as $Q_i$.  Note that by Lemma~\ref{l.2jAiBi}, our definition does not
depend on the choice of a particular $Q_i$.

The following lemma shows that when decomposing a graph in
$\Cl{Berge}{no cutset}$, a unique complementation is possibly needed
at the beginning of the process.

\begin{lemma}[with Vu\v skovi\'c \cite{nicolas.kristina:2-join}]
  \label{l:recurseBerge}
  Let $G\in \Cl{Berge}{no cutset}$ and let $(X_1, X_2)$ be a proper
  non-path 2-join of $G$.  Let $G^{k_1}_1$ and $G^{k_2}_2$ be
  parity-preserving blocks of decomposition of $G$ w.r.t.\ $(X_1,
  X_2)$ where $3\leq k_1, k_2 \leq 4$.  Then $G^{k_1}_1$ and
  $G^{k_2}_2$ are in $\Cl{Berge}{no cutset}$.  Moreover,
  $\overline{G^{k_1}_1}$ and $\overline{G^{k_2}_2}$ have no proper
  2-join.
\end{lemma}

But we need more.  Indeed, when trying to solve an optimization
problem with 2-joins, the idea is to build the 2 blocks and to ask
questions recursively to the blocks.  Usually, for at least one block,
two questions have to be asked (see for instance how we handle maximum
cliques in the next section).  So, even with a linear number of
steps, the number of questions to ask becomes exponential.  A
classical idea to bypass this problem is the notion of
\emph{extreme}\index{extreme!2-join}
2-join.   

We say that $(X_1,X_2)$ is an {\em extreme 2-join} if for some $i\in
\{ 1,2 \}$ and all $k \geq 3$ the block of decomposition $G_i^{k}$ has
no proper non-path 2-join.  We say that $X_i$ is an \emph{extreme
  side}\index{extreme!side} of such a 2-join.  Figure~\ref{fige2j} shows that graphs in
general do not have an extreme 2-join, but as we now show, graphs with
no star cutset do.  Having a star-cutset implies having a
balanced skew partition (there are a few trivial exceptions: graphs
with no edges and graphs on less than 5 vertices).  This was noticed
by Zambelli~\cite{zambelli:these}, a proof can be found
in~\cite{nicolas:bsp}.  So, what follows applies to $\Cl{Berge}{no
  cutset}$ since graphs from this class have no star cutste.

 \begin{figure}[h!]
  \begin{center}
    \psset{xunit=6.0mm,yunit=6.0mm,radius=0.1,labelsep=0.1}
    \def\uputnode(#1,#2)#3#4{\Cnode(#1,#2){#3}\nput{ 90}{#3}{\small$#4$}}
    \def\dputnode(#1,#2)#3#4{\Cnode(#1,#2){#3}\nput{270}{#3}{\small$#4$}}
    \def\lputnode(#1,#2)#3#4{\Cnode(#1,#2){#3}\nput{180}{#3}{\small$#4$}}
    \def\rputnode(#1,#2)#3#4{\Cnode(#1,#2){#3}\nput{0}{#3}{\small$#4$}}
    \def\sputnode(#1,#2)#3#4{\Cnode(#1,#2){#3}\nput{45}{#3}{\small$#4$}}
    \begin{pspicture}(15,9)

      \dputnode(3,1){b1'}{}
      \uputnode(5,1){b2'}{}
      \uputnode(2,2){w}{}
      \lputnode(6,2){z}{}
      \sputnode(1,3){w1}{}
      \rputnode(3,3){b1}{}
      \dputnode(5,3){b2}{}
      \dputnode(7,3){z1}{}
      \dputnode(1,5){x1}{}

      \dputnode(3,5){a1'}{}
      \uputnode(5,5){a2'}{}
      \uputnode(7,5){y1}{}
      \lputnode(2,6){x}{}
      \sputnode(6,6){y}{}
      \rputnode(3,7){a1}{}
      \dputnode(5,7){a2}{}

      \rput(4,8){\rnode{g}{$G$}}

      \dputnode(11,1){B1'}{}
      \uputnode(10,2){W}{}
      \sputnode(9,3){W1}{}
      \rputnode(11,3){B1}{}
      \dputnode(13,3){B2}{}
      \dputnode(9,5){X1}{}

      \dputnode(11,5){A1'}{}
      \uputnode(15,5){Y1}{}
      \lputnode(10,6){X}{}
      \sputnode(14,6){Y}{}
      \rputnode(11,7){A1}{}
      \dputnode(13,7){A2}{}

      \rput(12,8){\rnode{h}{$H$}}

      \ncline{x}{x1}
      \ncline{x}{a1}
      \ncline{x1}{w1}
      \ncline{w}{w1}
      \ncline{w}{b1'}
      \ncline{w1}{a1'}
      \ncline{x1}{b1}
      \ncline{a1'}{b1}

      \ncline{y}{a2}
      \ncline{y}{y1}
      \ncline{y1}{z1}
      \ncline{z}{z1}
      \ncline{z}{b2'}
      \ncline{b2}{y1}
      \ncline{a2'}{z1}
      \ncline{a2'}{b2}

      \ncline{X}{X1}
      \ncline{X}{A1}
      \ncline{W}{W1}
      \ncline{W}{B1'}
      \ncline{Y}{A2}
      \ncline{B2}{Y1}
      \ncline{A2}{A1}
      \ncline{A2}{A1'}
      \ncline{B2}{B1}
      \ncline{B2}{B1'}

      \psset{linewidth=0.8mm}

      \ncline{a2}{a1}
      \ncline{a2'}{a1}
      \ncline{a2}{a1'}
      \ncline{a2'}{a1'}
      \ncline{b2}{b1}
      \ncline{b2'}{b1}
      \ncline{b2}{b1'}
      \ncline{b1'}{b2'}

      \ncline{Y}{Y1}
      \ncline{A1'}{B1}
      \ncline{W1}{A1'}
      \ncline{X1}{W1}
      \ncline{X1}{B1}     

    \end{pspicture}
\end{center}
\caption{\label{fige2j}Graph $G$ that has a star cutset, but does not
  have an extreme proper non-path 2-join. $G$ has a proper
  non-path 2-join represented with bold lines, and all proper
  non-path 2-joins are equivalent to this one. Both of the blocks of
  decomposition are isomorphic to graph $H$, and $H$ has a proper
  non-path 2-join whose edges are represented with bold lines. 
}
\end{figure}
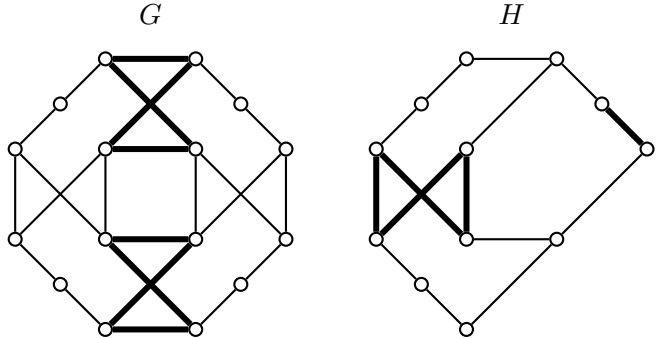

Let $(X_1,X_2)$ be a proper non-path 2-join of a graph $G$.  We say
that $(X_1,X_2)$ is a {\em minimally-sided proper non-path 2-join}
if for some $i\in \{ 1,2 \}$, the following holds: for every proper
non-path 2-join $(X_1',X_2')$ of $G$, neither $X_1' \subsetneq X_i$
nor $X'_2 \subsetneq X_i$ holds.  We call $X_i$ a {\em minimal side}
of this minimally-sided 2-join.  Note that minimally-sided proper
non-path 2-joins exist in any graph that admits a proper non-path
2-join.

By $\Cl{}{no sc}$ we mean the class of graphs with no star cutset. 

\begin{lemma}[with Vu\v skovi\'c \cite{nicolas.kristina:2-join}]
  \label{k4}
  Let $(X_1, X_2, A_1, B_1, A_2, B_2)$ be a split of a minimally-sided
  proper non-path 2-join of a graph $G$, with $X_1$ being a minimal
  side.  Assume that $G$ and all the blocks of decomposition of $G$
  w.r.t.\ $(X_1,X_2)$ whose marker paths are of length at least 3, all
  belong to $\Cl{}{no sc}$.  Then $(X_1,X_2)$ is an extreme 2-join and
  $X_1$ is an extreme side.
\end{lemma}

Since the classical algorithm to find a
2-join~\cite{conforti.c.k.v:eh2} actually finds a minimally sided
2-join, we have an algorithm to detect extreme 2-joins.  But again,
this is not enough.  Because in the sequel, we need to use special
blocks with respect to a 2-join (not defined yet).  These blocks are
not class preserving.  So, to use them, we use the following trick: we
first decompose with the classical parity-preserving block.  We keep
track of all marker paths. After that, we reprocess the decomposition
tree: we replace each marker path with the right block.  But to do
this, we need marker paths to keep disjoint.  This motivates the
following definition and lemma.

When ${\cal M}$ is a collection of vertex-disjoint flat paths, a
2-join $(X_1, X_2)$ is \emph{{$\cal
    M$}-independent}\index{independent!{$\cal M$}-independent} if for
every path $P$ from $\cal M$ we have either $V(P) \subseteq X_1$ or
$V(P) \subseteq X_2$.

\begin{lemma}[with Vu\v skovi\'c \cite{nicolas.kristina:2-join}]
  \label{k23}
  Let $(X_1,X_2,A_1,B_1,A_2,B_2)$ be a split of a minimally-sided
  proper non-path 2-join of a graph $G$, with $X_1$ being a minimal
  side.  Assume that $G$ and all the blocks of decomposition of $G$
  w.r.t.\ $(X_1,X_2)$ whose marker paths are of length at least 3, all
  belong to $\Cl{}{no sc}$.  Let ${\cal M}$ be a set of vertex-disjoint
  flat paths of length at least~3 of $G$.  If there exists a path $P
  \in {\cal M}$ such that $P \cap A_1 \neq \emptyset$ and $P \cap A_2
  \neq \emptyset$, then let $A_1'=A_2$, and otherwise let $A_1'=A_1$.
  If there exists a path $P \in {\cal M}$ such that $P \cap B_1 \neq
  \emptyset$ and $P \cap B_2 \neq \emptyset$, then let $B_1'=B_2$, and
  otherwise let $B_1'=B_1$.  Let $X_1'=X_1 \cup A_1' \cup B_1'$ and
  $X_2'=V(G) \setminus X_1'$.  Then the following hold:
  \begin{enumerate}
  \item\label{i:k23con} $(X_1',X_2')$ is a proper non-path 2-join
    of $G$.
  \item\label{i:k23indep} $(X'_1, X'_2)$ is ${\cal M}$-independent.
  \item\label{i:k23extr} $(X_1',X_2')$ is an extreme 2-join of $G$ and
    $X'_1$ is an extreme side of this 2-join.
  \end{enumerate}
\end{lemma}

\section{Cliques and stable sets in Berge graphs with no skew
  partitions and no homogeneous
  pairs}
\label{sec:colorBergenoSP}

Before seeing how the machinery above is used to optimize, let us give
an NP-hardness result showing that the 2-join is usually not a good
tool to compute maximum stable sets.  This is just to show that all
the longcuts\footnote{A \emph{longcut} is the opposite of a shortcut}
taken below have to be taken.

\subsection*{NP-hardness}

We define a class $\cal C$ of graph for which computing a maximum
stable set is NP-hard.  The interesting feature of class $\cal C$ is
that all graphs in $\cal C$ are decomposable along extreme 2-joins
into one bipartite graph and several gem-wheels where a
\emph{gem-wheel}\index{gem-wheel} is any graph made of an induced
cycle of length at least 5 together with a vertex adjacent to exactly
four consecutive vertices of the cycle.  Note that a gem-wheel is a
line-graph (of a cycle with one chord).  Our NP-completeness result
(proved jointly with Guyslain Naves) shows that being able to
decompose along extreme 2-joins is not enough in general to compute
stables sets. This suggests that being Berge is essential for
computing stable sets along 2-joins and that the inequalities used
below capture some features of Berge graphs.

Here, \emph{extending} a flat path $P = p_1\tp \cdots \tp p_k$ of a
graph means deleting the interior vertices of $P$ and adding three
vertices $x, y, z$ and the following edges: $p_1x$, $xy$, $yp_k$,
$zp_1$, $zx$, $zy$, $zp_k$.  By extending a graph $G$ we mean
extending all paths of $\cal M$ where $\cal M$ is a set of disjoint
flat paths of length at least 3 of $G$.  Class $\cal C$ is the class
of all graphs obtained by extending 2-connected bipartite graphs.
From the definition, it is clear that all graphs of $\cal C$ are
decomposable along extreme proper non-path 2-joins.  One leaf of the
decomposition tree will be the underlying bipartite graph.  All the
others leaves will be gem-wheels.

We call \emph{4-subdivision}\index{4-subdivision} any graph $G$ obtained from a graph $H$
by subdividing four times every edge.  More precisely, every edge $uv$
of $H$ is replaced by an induced path $u\tp a \tp b\tp c\tp d \tp v$
where $a, b, c, d$ are of degree two.  It is easy to see that
$\alpha(G) = \alpha(H) + 2|E(H)|$. This construction, essentially due
to Poljak~\cite{poljak:74}, yields as observed by Guyslain Naves:

\begin{theorem}[Naves, \cite{naves:pc}]
  \label{th:npHard}
  The problem whose instance is a graph $G$ from $\cal C$ and an
  integer $k$, and whose question is ``Does $G$ contain a stable set
  of size at least~$k$'' is NP-complete.
\end{theorem}

\begin{proof}
  Let $H$ be any graph.  First we subdivide 5 times every edge of~$H$.
  So each edge $ab$ is replaced by $P_7 = a \tp p_1 \tp \cdots \tp p_5
  \tp b$.  The graph $H'$ obtained is bipartite.  Now we build an
  extension $G$ of $H'$ by replacing all the $P_5$'s $p_1 \tp \cdots
  \tp p_5$ arising from the subdivisions in the previous step by
  $P_4$'s.  And for each $P_4$ we add a new vertex complete to it and
  we call \emph{apex vertices}\index{apex-vertex} all these new
  vertices.  The graph $G$ that we obtain is in $\cal C$.  It is easy
  to see that there exists a maximum stable set of $G$ that contain no
  app-ex vertex because an apex vertex of a maximum stable set can be
  replaced by one vertex of its neighborhood.  So, we call $G'$ the
  graph obtained from $G$ by deleting all the apex vertices and see
  that $\alpha(G') = \alpha(G)$.  Also, $G'$ is the 4-subdivision
  arising from $H$.  So from the remark above, maximum stable sets in
  $H$ and $G$ have sizes that differ by $2|E(H)|$.
\end{proof}

\subsection*{Keeping track of cliques}
\label{sec:clique}

Here we show how to find a maximum clique in a graph using 2-joins.
For the sake of induction we have to solve the weighted version of the
problem.

From here on, by graph we mean a graph with weights on the vertices.
Weights are numbers from $K$ where $K$ means either the set $\R_+$ of
non-negative real numbers or the set $\N_+$ of non negative integers.
The statements of the theorems will be true for $K= \R_+$ but the
algorithms are to be implemented with $K= \N_+$.  Let $G$ be a
weighted graph with a weight function $w$ on $V(G)$.  When $H$ is an
induced subgraph of $G$ or a subset of $V(G)$, $w(H)$ denotes the sum
of the weights of vertices in $H$.  Note that we view a graph where no
weight is assigned to the vertices as a weighted graph whose vertices
have all weight~1.  Here, $\omega (G)$ denotes the weight of a maximum
weighted clique of $G$.

Let $(X_1, X_2, A_1, B_1, A_2, B_2)$ be a split of a proper 2-join
of $G$.  We define for $k\geq 3$ the \emph{clique-block}\index{clique-block} $G_2^k$ of
$G$ with respect to $(X_1, X_2)$.  It is obtained from the block
$G^k_2$ by giving weights to the vertices.  Let $P_1= a_1 \tp x_1 \tp
\cdots \tp x_{k-1} \tp b_1$ be the marker path of $G^k_2$.  We assign
the following weights to the vertices of $G^k_2$:

\begin{itemize}
\item for every $u \in X_2$, $w_{G_2^k}(u) = w_{G}(u)$;
\item $w_{G^k_2}(a_1)=\omega (G[A_1])$;
\item $w_{G^k_2}(b_1)=\omega (G[B_1])$;
\item $w_{G^k_2}(x_1)=\omega (G[X_1]) - \omega (G[A_1])$;
\item $w_{G^k_2}(x_i)=0$, for $i=2, \ldots, k-1$.
\end{itemize}

\begin{lemma}[with Vu\v skovi\'c \cite{nicolas.kristina:2-join}]
  \label{komega}
  $\omega (G)=\omega(G_2^k)$.
\end{lemma}

\subsection*{Keeping track of stable sets}
\label{sec:alphaTrack}

Here we show how to use 2-joins to compute maximum stable sets.  This
is more difficult than cliques mainly because stable sets may
completely overlap both sides of a 2-join.  For the sake of induction
we need to put weights on the vertices.  But even with weights, there
is an issue: we are not able to compute maximum weighted stable set of
a graph assuming that some computations are done on its blocks.  So we
need to enlarge slightly our blocks to encode information, and this
causes some trouble.  First, the extended blocks may fail to be in the
class we are working on.  This problem will be solved by building the
decomposition tree in two steps.  Also in a decomposition tree built
with our unusual blocks, the leaves may fail to be basic graphs, so
computing something in the leaves of the tree is a problem solved
later.

Throughout this subsection, $G$ is a fixed graph with a weight function
$w$ on the vertices and $(X_1,X_2,A_1,B_1,A_2,B_2)$ is a split of a
2-join of~$G$.  For $i=1,2$, $C_i=X_i \setminus (A_i \cup B_i)$.  For
any graph $H$, $\alpha (H)$ denotes the weight of a maximum weighted
stable set of $H$.  We define $a = \alpha(G[{A_1 \cup C_1}])$, $b=
\alpha(G[{B_1 \cup C_1}])$, $c = \alpha(G[{C_1}])$ and $d =
\alpha(G[{X_1}])$.  The following follows easily from the definitions.

\begin{lemma}[with Vu\v skovi\'c \cite{nicolas.kristina:2-join}]
  \label{l:4cases}
  Let $S$ be a stable set of $G$ of maximum weight.  Then one of the
  following holds:

  \begin{enumerate}
  \item\label{i:4c1} $S \cap A_1 \neq \emptyset$, $S \cap B_1 =
    \emptyset$, $S\cap X_1$ is a maximum weighted stable set of $G[A_1
    \cup C_1]$ and $w(S \cap X_1) = a$;
  \item\label{i:4c2} $S \cap A_1 = \emptyset$, $S \cap B_1 \neq
    \emptyset$, $S\cap X_1$ is a maximum weighted stable set of $G[B_1
    \cup C_1]$ and $w(S \cap X_1) = b$;
  \item\label{i:4c3} $S \cap A_1 = \emptyset$, $S \cap B_1 =
    \emptyset$, $S\cap X_1$ is a maximum weighted stable set of
    $G[C_1]$ and $w(S \cap X_1) = c$;
  \item\label{i:4c4} $S \cap A_1 \neq \emptyset$, $S \cap B_1 \neq
    \emptyset$, $S\cap X_1$ is a maximum weighted stable set of
    $G[X_1]$ and $w(S \cap X_1) = d$.
  \end{enumerate}
\end{lemma}

We need kinds of blocks that preserve being in $\Cl{Berge}{}$.  To
define them we need several inequalities that tell more about how
stable sets and 2-joins overlap.

\begin{lemma}[with Vu\v skovi\'c \cite{nicolas.kristina:2-join}]
  \label{l:ineqbasic}
  $0 \leq c \leq a, b \leq d \leq a+b$.
\end{lemma}

\begin{proof}
  The inequalities $0 \leq c \leq a, b \leq d$ are trivially true. Let
  $D$ be a maximum weighted stable set of $G[X_1]$.  We have:
  $$
  d = w(D) = w(D\cap A_1) + w(D\cap (C_1 \cup B_1)) \leq a + b.
  $$
\end{proof}

A 2-join with split $(X_1, X_2,A_1,B_1,A_2,B_2)$ is said to be
\emph{$X_1$-even} (resp.\ \emph{$X_1$-odd})\index{even!$X_1$-even
  2-join}\index{odd!$X_1$-odd 2-join}\index{2-join!$X_1$ even or odd}
if all paths from $A_1$ to $B_1$ with interior in $C_1$ are of even
length (resp.\ odd length).  Note that from Lemma~\ref{l.2jAiBi}, if
$G$ is in $\Cl{parity}{}$ and $(X_1,X_2)$ is proper, then $(X_1, X_2)$
must be either $X_1$-even or $X_1$-odd.  The following two lemmas
present important inequalities.  The proofs are included.

\begin{lemma}[with Vu\v skovi\'c \cite{nicolas.kristina:2-join}]
  \label{l:ineqEven}
  If $(X_1, X_2)$ is an $X_1$-even 2-join of $G$, then $a+b \leq c+d$.
\end{lemma}

\begin{proof}
  Let $A$ be a stable set of $G[A_1 \cup C_1]$ of weight $a$ and $B$ a
  stable set of $G[B_1 \cup C_1]$ of weight $b$.  In the bipartite
  graph $G[A\cup B]$, we denote by $Y_A$ (resp.\ $Y_B$) the set of
  those vertices of $A\cup B$ such that there exists a path in $G[A
  \cup B]$ joining them to some vertex of $A \cap A_1$ (resp.\ $B \cap
  B_1$).  Note that from the definition, $A \cap A_1 \subseteq Y_A$,
  $B \cap B_1 \subseteq Y_B$ and no edges exist between $Y_A\cup Y_B$
  and $(A\cup B)\sm (Y_A \cup Y_B)$.  Also, $Y_A$ and $Y_B$ are
  disjoint with no edges between them because else, there is some path
  in $G[A\cup B]$ from some vertex of $A \cap A_1$ to some vertex of
  $B \cap B_1$.  If such a path is minimal with respect to this
  property, its interior is in $C_1$ and it is of odd length because
  $G[A\cup B]$ is bipartite.  This contradicts the assumption that
  $(X_1, X_2)$ is $X_1$-even.  Now we put:

  \begin{itemize}
  \item $Z_D = (A \cap Y_A) \cup (B \cap Y_B) \cup (A \sm (Y_A \cup Y_B))$;
  \item $Z_C = (A \cap Y_B) \cup (B \cap Y_A) \cup (B \sm (Y_A \cup Y_B))$.
  \end{itemize}

  From all the definitions and properties above, $Z_D$ and $Z_C$ are
  stable sets and $Z_D \subseteq X_1$ and $Z_C \subseteq C_1$.  So,
  $a+b = w(Z_C) + w(Z_D) \leq c+d$.
\end{proof}

\begin{lemma}[with Vu\v skovi\'c \cite{nicolas.kristina:2-join}]
  \label{l:ineqOdd}
  If $(X_1, X_2)$ is an $X_1$-odd 2-join of $G$, then $c+d \leq a+b$.
\end{lemma}

\begin{proof}
  Let $D$ be a stable set of $G[X_1]$ of weight $d$ and $C$ a stable
  set of $G[C_1]$ of weight $c$.  In the bipartite graph $G[C\cup D]$,
  we denote by $Y_A$ (resp.\ $Y_B$) the set of those vertices of
  $C\cup D$ such that there exists a path in $G[C \cup D]$ joining
  them to some vertex of $D\cap A_1$ (resp.\ $D \cap B_1$).  Note that
  from the definition, $D \cap A_1 \subseteq Y_A$, $D \cap B_1
  \subseteq Y_B$ and no edges exist between $Y_A \cup Y_B$ and $(C\cup
  D)\sm (Y_A \cup Y_B)$.  Also, $Y_A$ and $Y_B$ are disjoint with no
  edges between them because else, there is some path in $G[C\cup D]$
  from some vertex of $D \cap A_1$ to some vertex of $D \cap B_1$.  If
  such a path is minimal with respect to this property, its interior
  is in $C_1$ and it is of even length because $G[C\cup D]$ is
  bipartite.  This contradicts the assumption that $(X_1, X_2)$ is
  $X_1$-odd.  Now we put:

  \begin{itemize}
  \item $Z_A = (D \cap Y_A) \cup (C \cap Y_B) \cup (C \sm (Y_A \cup
    Y_B))$;
  \item $Z_B = (D \cap Y_B) \cup (C \cap Y_A) \cup (D \sm (Y_A \cup Y_B)$.
  \end{itemize}

  From all the definitions and properties above, $Z_A$ and $Z_B$ are
  stable sets and $Z_A \subseteq A_1 \cup C_1$ and $Z_B \subseteq B_1
  \cup C_1$.  So, $c+d = w(Z_A) + w(Z_B) \leq a+b$.     
\end{proof}

\subsection*{Even and odd blocks}

We call \emph{flat claw}\index{flat!claw} of a weighted graph $G$ any set $\{q_1, q_2,
q_3, q_4\}$ of vertices such that:
\begin{itemize}
\item the only edges between the $q_i$'s are $q_1q_2$, $q_2q_3$ and
  $q_4q_2$;
\item $q_1$ and $q_3$ have no common neighbor in $V(G) \sm \{q_2\}$;
\item $q_4$ has degree~1 in $G$ and $q_2$ has degree~3 in $G$.
\end{itemize}

We define now the \emph{even block}\index{even!block} $G_2$ with
respect to a 2-join $(X_1, X_2)$.  We keep $X_2$ and replace $X_1$ by
a flat claw on $q_1, \dots, q_4$ where $q_1$ is complete to $A_2$ and
$q_3$ is complete to $B_2$.  We give the following weights: $w(q_1) =
d-b$, $w(q_2) = c$, $w(q_3) = d-a$, $w(q_4) = a+b-d$.  From
Lemma~\ref{l:ineqbasic}, all weights are non-negative.  By
Lemma~\ref{l:ineqEven}, the following Lemma applies in particular if
$(X_1, X_2)$ is a proper $X_1$-even 2-join.  The reader can check that
we really need $a+b \leq c+d$.

\begin{lemma}[with Vu\v skovi\'c \cite{nicolas.kristina:2-join}]
  \label{l:evenBlock}
  If $a+b \leq c+d$ and if $G_2$ is the even block of $G$, then
  $\alpha(G_2) = \alpha (G)$.
\end{lemma}

We call \emph{flat vault}\index{flat!vault} of graph $G$ any set $\{r_1, r_2,
r_3, r_4, r_5, r_6\}$ of vertices such that:

\begin{itemize}
\item the only edges between the $r_i$'s are such that
  $r_3,r_4,r_5,r_6,r_3$ is a 4-hole;
\item $N(r_1) = N(r_5)\sm \{r_4, r_6\}$; 
\item $N(r_2) = N(r_6)\sm \{r_3, r_5\}$; 
\item $r_1$ and $r_2$ have no common neighbors;
\item $r_3$ and $r_4$ have degree 2 in $G$.
\end{itemize}

Let us now define the \emph{odd block}\index{odd block} $G_2$ with respect to a 2-join
$(X_1,X_2)$.  We replace $X_1$ by a flat vault on $r_1, \dots, r_6$.
Moreover $r_1, r_5$ are complete to $A_2$ and $r_2, r_6$ are complete
to $B_2$.  We give the following weights: $w(r_1) = d-b$, $w(r_2) =
d-a$, $w(r_3) = w(r_4) = c$, $w(r_5) = w(r_6) = a+b-c-d$.  Note that
if we suppose $c+d \leq a+b$ (which holds in particular if $(X_1,
X_2)$ is an $X_1$-odd proper 2-join by Lemma~\ref{l:ineqOdd}), all
the weights are non-negative by Lemma~\ref{l:ineqbasic}.  Finding this
odd block was difficult, we tried a lot of different ideas. It seems
that the two isolated vertices have to be there and this is a bit
strange.

\begin{lemma}[with Vu\v skovi\'c \cite{nicolas.kristina:2-join}]
  \label{l:oddBlock}
  If $c+d \leq a+b$ and if $G_2$ is the odd block of $G$, then
  $\alpha(G_2) = \alpha (G)$.
 \end{lemma}

\subsection*{Extension of basic classes}

To build a decomposition tree that allows keeping track of maximum
stable sets we use the even and odd blocks defined above.  As a
consequence, the leaves of our decomposition tree may fail to be
basic, but are what we call \emph{extensions}\index{extension!of a
  basic graph} of basic graphs.  Let us define this.

Let $P = p_1 \tp \cdots \tp p_k$, $k\geq 4$, be a flat path of a graph
$G$.  \emph{Extending} $P$ means:

\begin{itemize}
\item Either:
  \begin{enumerate}
  \item replace the vertices of $P$ by a flat claw on $q_1, \dots, q_4$
    where $q_1$ is complete to $N_G(p_1)\sm \{p_2\}$ and $q_3$ is
    complete to $N_G(p_k) \sm \{p_{k-1}\}$;
  \item replace $X_1$ by a flat vault on $r_1, \dots, r_6$ where $r_1,
    r_5$ are complete to $N_G(p_1)\sm \{p_2\}$ and $r_2, r_6$ are complete
    to $N_G(p_k) \sm \{p_{k-1}\}$.
  \end{enumerate}
\item Mark the vertices of the flat claw (or vault) with the integer
  $k$.
\item Give any non-negative weights to the vertices. 
\end{itemize}

An \emph{extension} of a pair $(G, {\cal M})$, where $G$ is a graph
and $\cal M$ is a set of vertex-disjoint flat paths of length at
least~3 of $G$, is any weighted graph obtained by extending the flat
paths of $\cal M$.  Note that since $\cal M$ is a set of
\emph{vertex-disjoint} paths, the extensions of the paths from $\cal
M$ can be done in any order and lead to the same graph.  An
\emph{extension} of a graph $G$ is any graph that is an extension of
$(G, {\cal M})$ for some $\cal M$.

We say that the extension of $P$ is
\emph{parity-preserving}\index{parity-preserving!extension}\index{extension!parity-preserving}
when $P$ has even length and is replaced by a flat claw, or when $P$
has odd length and is replaced by a flat vault. We define the
\emph{parity-preserving} extension of a pair $(G, {\cal M})$ and of a
graph $G$ by requiring that all extensions of paths are
parity-preserving.  The following is easy.

\begin{lemma}[with Vu\v skovi\'c \cite{nicolas.kristina:2-join}]
  \label{l:staybip}
  A parity-preserving extension of a bipartite graph is a bipartite
  graph.
\end{lemma}

The following lemma shows that maximum weighted stable sets can be
computed for all parity-preserving extensions of Berge basic classes,
except line-graphs of bipartite graphs.  It is easy for bipartite
graphs (that stay bipartite).  For the other classes, it relies on the
fact that there is a bounded number of long flat paths.

\begin{lemma}[with Vu\v skovi\'c \cite{nicolas.kristina:2-join}]
  \label{l:optBergeBasic}
  There is an algorithm with the following specification:
  \begin{description}
  \item[Input: ] A weighted graph $G$ that is a parity preserving
    extension of either a bipartite graph, the complement of a
    bipartite graph, the complement of a line-graph of a bipartite
    graph, a path-cobipartite graph, the complement of non-trivial
    path-cobipartite graph, a path-double split graph or the
    complement of a path-double split graph.
  \item[Output: ] A maximum weighted stable set of $G$.
  \item[Running time: ] ${\cal O} (n^5)$
  \end{description}
\end{lemma}

Extensions of line-graphs are more difficult to handle than other
extensions because an extension of a line-graph may fail to be a
line-graph and a line-graph may contain arbitrarily many disjoint
long flat paths.  Note that in what follows, extensions are not
required to be parity-preserving.

Let $G'$ be a weighted graph that is an extension of a line-graph
$G=L(R)$.  We now define the \emph{transformation}\index{transformation} $G''$ of $G'$.  The
structure of $G''$, i.e.,\ its vertices and edges, depends only on $G$
but the weights given to its vertices depend only on $G'$.  Let $\cal
M$ be the set of vertex-disjoint flat paths of length at least~3 of
$G$ that are extended to get~$G'$.  So, ${\cal M} = \{P^1, \dots,
P^k\}$ and we put $P^i = p^i_1 \tp \cdots \tp p^i_{l_i}$.  For all
$1\leq i \leq k$, path $P^i$ is replaced in $G'$ by a set $Q^i$ that
induces either a flat claw on vertices $q^i_1, q^i_2, q^i_3, q^i_4$ or
a flat vault on vertices $r^i_1, r^i_2, r^i_3, r^i_4, r^i_5, r^i_6$.
For all flat paths $P^i$ of $\cal M$, we put $A^i_2= N_{G}(p^i_1)\sm
\{p^i_2\}$, $B^i_2 = N_{G}(p^i_{l_i}) \sm \{p^i_{l_i - 1}\}$.

For all $1\leq i \leq k$, we prepare a set $S^i$ of four new vertices
$p^i, p'^i, x^i, y^i$.  The graph $G''$ has vertex set:
$$V(G'') = (S^1 \cup \cdots \cup S^k ) \cup V(G) \sm (P^1 \cup \cdots \cup P^k).$$

\noindent Edges of $G''$ depend only on edges of $G$.  They are:

\begin{itemize}
\item $p^ip'^i$, $x^ip^i$, $p^iy^i$, $y^ip'^i$, $p'^ix^i$, $i= 1,
  \dots, k$;
\item $uv$ for all $u, v \in V(G) \cap V(G'')$ such that $uv\in
  E(G)$;
\item $p^iu$ for all $u\in A^i_2 \cap V(G'')$, $i=1, \dots, k$;
\item $p'^iu$ for all $u\in B^i_2 \cap V(G'')$, $i=1, \dots, k$;
\item $x^iu$ for all $u\in (A^i_2 \cup B^i_2) \cap V(G'')$, $i=1,
  \dots, k$;
\item $p^ip^j$ for all $i, j$ such that $p^i_1p^j_1 \in E(G)$;
\item $p'^ip^j$ for all $i, j$ such that $p^i_{l_i}p^j_1 \in E(G)$;
\item $p'^ip'^j$ for all $i, j$ such that $p^i_{l_i}p^j_{l_j} \in
  E(G)$;
\item $x^ip^j$ for all $i, j$ such that $p^i_1p^j_1 \in E(G)$ or
  $p^i_{l_i}p^j_1 \in E(G)$;
\item $x^ip'^j$ for all $i, j$ such that $p^i_1p^j_{l_j} \in E(G)$ or
  $p^i_{l_i}p^j_{l_j} \in E(G)$;
\item $x^ix^j$ for all $i, j$ such that $p^i_1p^j_{1} \in E(G)$ or
  $p^i_{1}p^j_{l_j} \in E(G)$ or $p^i_{l_i}p^j_{1} \in E(G)$ or
  $p^i_{l_i}p^j_{l_j} \in E(G)$.
\end{itemize}

\noindent We define the following numbers that depend only on $G'$:

\begin{itemize}
\item $a^i = \alpha(G'[\{q^i_1, q^i_2, q^i_4\}])$ for all $i$ such
  that $Q^i$ is a flat claw of $G'$;
\item $a^i = \alpha(G'[\{r^i_1, r^i_3, r^i_4, r^i_5\}])$ for all
  $i$ such that $Q^i$ is a flat vault of $G'$;
\item $b^i = \alpha(G'[\{q^i_2, q^i_3, q^i_4\}])$ for all $i$ such
  that $Q^i$ is a flat claw of $G'$;
\item $b^i = \alpha(G'[\{r^i_2, r^i_3, r^i_4, r^i_6\}])$ for all
  $i$ such that $Q^i$ is a flat vault of $G'$;
\item $c^i = \alpha(G'[\{q^i_2, q^i_4\}])$ for all $i$ such that
  $Q^i$ is a flat claw of $G'$;
\item $c^i = \alpha(G'[\{r^i_3, r^i_4\}])$ for all $i$ such that
  $Q^i$ is a flat vault of $G'$;
\item $d^i = \alpha(G'[\{q^i_1, q^i_2, q^i_3, q^i_4\}])$ for all
  $i$ such that $Q^i$ is a flat claw of $G'$;
\item $d^i = \alpha(G'[\{r^i_1, r^i_2, r^i_3, r^i_4, r^i_5,
  r^i_6\}])$ for all $i$ such that $Q^i$ is a flat vault of $G'$.
\end{itemize}

\noindent Note that from the definitions, $c^i \leq a^i, b^i \leq d^i$
for all $i=1, \dots, k$.  We give the following weights to the
vertices of $G''$ (they depend on the weights in $G'$):

\begin{itemize}
\item $w_{G''}(u) = w_{G'}(u)$ for all $u\in V(G) \cap V(G'')$;
\item $w_{G''}(p^i) = a^i$, $i= 1, \dots, k$;
\item $w_{G''}(p'^i) =b^i$, $i= 1, \dots, k$;
\item $w_{G''}(y^i) = c^i$, $i= 1, \dots, k$;
\item $w_{G''}(x^i) = d^i-c^i$, $i= 1, \dots, k$.
\end{itemize}

A \emph{multigraph}\index{multigraph} is a graph where multiple edges between vertices
are allowed (but we do not allow loops). 

\begin{lemma}[with Vu\v skovi\'c \cite{nicolas.kristina:2-join}]
  \label{l:L(multi)}
  $G''$ is the line-graph of a multigraph.
\end{lemma}

\begin{proof}
  Path $P^i$ of $G$ corresponds to a path $R^i = r^i_1 \tp \cdots \tp
  r^i_{l_i+1}$ of $R$.  For all $i=1, \dots, k$, path $R^i$ is induced
  and has interior vertices of degree 2 in $R$ because $P^i$ is a flat
  path of $G$.  Since paths of $\cal M$ are vertex-disjoints, paths
  $R^1, \dots, R^k$ are edge-disjoint (but they may share
  end-vertices).  Now let us build a multigraph $R''$ from $R$.  We
  delete the interior vertices of all $R^i$'s.  For each $R^i$, we add
  two vertices $u^i$, $v^i$ and the edges $r^i_1r^i_{l_i+1}$,
  $u^ir^i_1$, $u^ir^i_{l_i+1}$ and $u^iv^i$.

  It is a routine matter to check that $L(R'')$ is isomorphic to
  $G''$.  Edge $r^i_1r^i_{l_i+1}$ corresponds to vertex $x^i$, edge
  $u^ir^i_1$ corresponds to vertex $p^i$, edge $u^ir^i_{l_i+1}$
  corresponds to $p'^i$ and edge $u^iv^i$ corresponds to vertex $y^i$.
  Note that possibly, two paths $R^i$ and $R^j$ have the same ends.
  For instance $r^i_1=r^j_1$ and $r^i_{l_i+1} = r^j_{l_j+1}$ is
  possible.  Then, the edge $r^i_1r^i_{l_i+1}$ is added twice.  This
  is why we need $R''$ to be a multigraph.
\end{proof}

\begin{lemma}[with Vu\v skovi\'c \cite{nicolas.kristina:2-join}]
  \label{l:Linealpha}
  $\alpha(G'') = \alpha(G')$.
\end{lemma}

\begin{lemma}[with Vu\v skovi\'c \cite{nicolas.kristina:2-join}]
  \label{l:optLineGraphs}
  There is an algorithm with the following specification:
  \begin{description}
  \item[Input: ] A weighted graph $G'$ that is an extension of a
    line-graph $G$. 
  \item[Output: ] A maximum weighted stable set of $G'$.
  \item[Running time: ] ${\cal O} (n^3)$
  \end{description}
\end{lemma}

With all the results above, we can now sketch the algorithm for
cliques and stable sets.  First, build a decomposition tree with usual
blocks of decomposition.  By Theorem~\ref{th.3} and
Lemma~\ref{l:recurseBerge}, such a tree exists (up to a unique
complementation at the beginning).  By Lemmas~\ref{k4} and~\ref{k23},
we can be sure that the tree uses only extreme 2-joins and that all
the marker path in the nodes of the tree are vertex-disjoint.  Then
reprocess this tree to use the relevant even or odd block according to
the parity of the 2-join (by replacing marker paths by flat claws or
flat vaults).  Note that the leaves of the tree are parity preserving
extension of graphs from $\Cl{Berge}{basic}$, so by
Lemmas~\ref{l:optBergeBasic} and~\ref{l:optLineGraphs} we know how to
optimize them.  Then, by Lemmas~\ref{komega}, \ref{l:evenBlock}
and~\ref{l:oddBlock} we can backtrack a maximum weighted clique or
stable set to the root of the tree.

We can sum up all the results in a single theorem.  Note that we
include also even-hole-free graphs in the theorem.  We did not mention
them above, but a similar approach, relying on~\cite{dsv:ehf}, allows
to deal with them.

\begin{theorem}[with Vu\v skovi\'c \cite{nicolas.kristina:2-join}]
  \label{l:CS}
  There is an algorithm with the following specification:
  \begin{description}
  \item[Input: ] A weighted graph $G$ that is either a Berge graph
    with no balanced skew partition and no homogeneous pair or an
    even-hole-free graph with no star cutset. 
  \item[Output: ] A maximum weighted stable set and a maximum weighted
    clique of $G$.
  \item[Running time: ] ${\cal O} (n^9)$
  \end{description}
\end{theorem}

\section{Coloring Berge graphs using subroutines for cliques and stable sets}
\label{sec:Lovasz}

In this section, we present results of Gr\"ostchel, Lov\'asz and
Schrijver~\cite{grostchel.l.s:color} in the context of our class
$\Cl{Berge}{no cutset}$.  Their usual proofs rely on the stability of
perfect graphs under taking induced subgraphs, complementation, and
Lov\'asz' replication of vertices.  Our class $\Cl{Berge}{no cutset}$
is not closed under taking induced subgraphs and the Replication Lemma
also fails here.  But we can simulate these operations with weights on
the vertices.  The proofs of Gr\"ostchel, Lov\'asz and Schrijver are
very short and beautiful, so we include them, rewritten in our
context.  More generally, the results of this section should hold for
any class of perfect graphs under the assumption that we can find
maximum weighted cliques and maximum weighted stable sets for the
class.  We do not use other structural properties.  Note that the
following theorem, known as the Perfect Graph Theorem, was proved
without assuming the Strong Perfect Graph Theorem.

\begin{theorem}[Lov\'asz \cite{lovasz:nh}]
  \label{th:lovasz}
  A graph is perfect if and only if its complement is perfect.
\end{theorem}

\begin{lemma}
  \label{l:compS}
  There is an algorithm with the following specification:
  \begin{description}
  \item[Input: ] A graph $G$ in $\Cl{Berge}{no cutset}$, and a sequence
    $K_1, \dots, K_t$ of maximum cliques of $G$ where $t\leq n$.
  \item[Output: ] A stable set of $G$ that intersects each $K_i$,
    $i=1, \dots, t$. 
  \item[Running time: ] ${\cal O}(n^9)$
  \end{description} 
\end{lemma}

\begin{proof}
  By $\omega(G)$ we mean here the maximum \emph{cardinality} of a
  clique in~$G$.  Give to each vertex $v$ the weight $y_v= |\{
  i\suchthat v \in K_i \}|$.  Note that this weight is possibly zero.
  By Theorem~\ref{l:CS}, compute a maximum weighted stable set $S$
  of~$G$.

  Let us consider the graph $G'$ obtained from $G$ by replicating
  $y_v$ times each vertex $v$.  So each vertex $v$ in $G$ becomes a
  stable set $Y_v$ of size $y_v$ in $G'$ and two such stable sets
  $Y_u$, $Y_v$ are complete to one another if $uv\in E(G)$ and
  anticomplete otherwise.  Note that vertices of weight zero in $G$
  are not in $V(G')$.  Note also that $G'$ may fail to be in
  $\Cl{Berge}{no cutset}$, but it is easily seen to be perfect.  By
  replicating $y_v$ times each vertex $v$ of $S$, we obtain a stable
  set $S'$ of $G'$ of maximum cardinality.

  By construction, $V(G')$ can be partitioned into $t$ cliques of size
  $\omega (G)$ that form an optimal coloring of $\overline{G'}$
  because $\omega(G') = \omega(G)$.  Since by Theorem~\ref{th:lovasz}
  $\overline{G'}$ is perfect, $|S'|=t$.  So, in $G$, $S$ intersects
  every $K_i$, $i \in \{ 1, \ldots ,t\}$.
\end{proof}

\begin{theorem}[with Vu\v skovi\'c \cite{nicolas.kristina:2-join}]
  \label{th:color}
  There exists an algorithm of complexity ${\cal O}(n^{10})$ whose
  input is a Berge graph $G$ with no balanced skew partition and no
  homogeneous pair, and whose output is an optimal coloring of $G$.
\end{theorem}

\begin{proof}
  We only need to show how to find a stable set $S$ intersecting all
  maximum cliques of $G$, since we can apply recursion to $G \setminus
  S$ (by giving weight~0 to vertices of $S$).  Start with $t=0$. At
  each iteration, we have a list of $t$ maximum cliques $K_1, \ldots
  ,K_t$ and we compute by the algorithm in Lemma~\ref{l:compS} a
  stable set $S$ that intersects every $K_i$, $i \in \{ 1, \ldots ,t
  \}$.  If $\omega (G \setminus S) < \omega (G)$ then $S$ intersects
  every maximum clique, otherwise we can compute a maximum clique
  $K_{t+1}$ of $G \setminus S$ (by giving weight~0 to vertices of
  $S$).  This will finally find the desired stable set, the only
  problem being the number of iterations.  We show that this number is
  bounded by $n = |V(G)|$.

  Let $M_t$ be the incidence matrix of the cliques $K_1, \dots, K_t$.
  So the columns of $M_t$ correspond to the vertices of $G$ and each
  row is a clique (we see $K_i$ as row vector).  We prove by induction
  that the rows of $M_t$ are independent.  So, we assume that the rows
  of $M_t$ are independent and prove that this holds again for $M_{t+1}$.

  The incidence vector $x$ of $S$ is a solution to $M_tx = \mathbf{1}$
  but not to $M_{t+1}x = \mathbf{1}$.  If the rows of $M_{t+1}$
  are not independent, we have $K_{t+1} = \lambda_1 K_1 + \cdots +
  \lambda_t K_t$.  Multiplying by $x$, we obtain $K_{t+1}x = \lambda_1
  + \cdots + \lambda_t \neq 1$.  Multiplying by $\mathbf{1}$, we
  obtain $\omega = K_{t+1}\mathbf{1} = \lambda_1 \omega + \cdots +
  \lambda_t \omega$, so $\lambda_1 + \cdots + \lambda_t = 1$, a
  contradiction.

  So the matrices $M_1, M_2, \dots$ cannot have more than $|V(G)|$
  rows. Hence, there are at most $|V(G)|$ iterations.
\end{proof}

\section{A bit of science-fiction: coloring Berge graphs}
\label{sec:SF}

Recall that an \emph{even pair}\index{even!pair} in a graph is a pair
of vertices $u, v$ such that all induced paths linking $u$ to $v$ are
of even length.  Fonlupt and Uhry~\cite{fonlupt.uhry:82} proved that
contracting an even pair preserves the chromatic number of a graph.
Even pairs contractions are used in several efficient algorithms that
color classes of perfect graphs, see my thesis~\cite{nicolas:these} or
the more recent one by L\'ev\^eque~\cite{leveque:these}.  We recall
several theorems and conjectures suggesting that even pairs might
replace balanced skew partitions in several decomposition results.

Recall that a graph is a \emph{prism}\index{prism} if it consists of two
vertex-disjoint triangles (cliques of size $3$) with three
vertex-disjoint paths between them, and with no other edges than those
in the two triangles and in the three paths.  A graph is
\emph{Artemis}\index{Artemis graph} if it contains no odd hole, no prism and no antihole on
at least 5 vertices.

\begin{theorem}
  The following are equivalent:
  \begin{enumerate}
  \item\label{o:1i} $G$ is Artemis;
  \item\label{o:1ii} for every induced subgraph $G'$ of $G$ and every
    vertex $t$ of $G'$, either $N_{G'}(t)$ is a clique or $N_{G'}(t)$
    contains an even pair of $G'$;
  \item\label{o:1iii} for every induced subgraph $G'$ of $G$, either
    $G'$ has a balanced skew partition, or $G'$ is a clique or $G'$ is
    bipartite.
  \end{enumerate}
\end{theorem}

\begin{proof}
  Implication from \ref{o:1i} to \ref{o:1ii} is the main Theorem of
  \cite{nicolas:artemis}.  The converse is easy since antiholes and
  odd holes have no even pairs.  And vertices of degree~3 of prisms
  have no even pairs in their neighborhood.
  
  Implication from \ref{o:1i} to \ref{o:1iii} follows from statements
  1.8.7--1.8.12 of~\cite{chudnovsky.r.s.t:spgt}.  The converse is clear.
\end{proof}

\begin{figure}
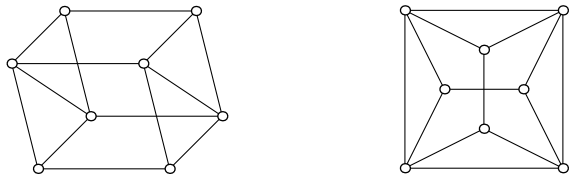

  \begin{center}
  \includegraphics{oddpair.1}
  \rule{2cm}{0cm}
  \includegraphics{oddpair.3}
  \caption{The double-diamond and $L(K_{3,3}\setminus e)$\label{fig:ddlk33e}}
  \end{center}
\end{figure}

A prism is \emph{long}\index{long!prism} if at least one of the three
paths in the definition is of length at least~2.  The double-diamond
and $L(K_{3,3}\setminus e)$ are the graphs represented in
Figure~\ref{fig:ddlk33e}.  A graph is
\emph{bipartisan}\index{bipartisan graph} if none of $G, \overline{G}$
contains an odd hole, a long prism, a double diamond or $L(K_{3, 3}
\sm e)$.

\begin{theorem}
  The following are equivalent:
  \begin{enumerate}
  \item\label{o:2i} $G$ is bipartisan;
  \item\label{o:2ii} for every induced subgraph $G'$ of $G$, either $G'$ has a
    balanced skew partition, or one of $G, \overline{G}$ is bipartite.
  \end{enumerate}
\end{theorem}

\begin{proof}
  Implication from \ref{o:2i} to \ref{o:2ii} follows from statements
  1.8.7--1.8.12 of~\cite{chudnovsky.r.s.t:spgt}.  The converse is
  clear.
\end{proof}

Maffray conjectures that for every bipartisan graph $G$ on at least
two vertices, one of $G, \overline{G}$ contains an even pair,
see~\cite{nicolas:cliques} where this conjecture is erroneously given
as a conjecture of Maffray and Thomas.  Chudnovsky and Seymour proved
a variant of this conjecture that allows a substantial simplification
of the proof of the Strong Perfect Graph Conjecture,
see~\cite{chudnovsky.seymour:even}.

Thomas conjectures that every Berge graph with no even pairs and no
even pairs in the complement can be explicitly constructed from basic
pieces by several simple operations including 2-joins.  It is tempting
to conjecture more, that Berge graphs with no even pairs have a structure, but
very recently Chudnovsky and Seymour found a counter-example, a Berge
graph with no even pair and no known decomposition.  Thomas'
conjecture was proved by Chudnovsky and Seymour for Berge graphs with
no $K_4$, see~\cite{chudnovsky.seymour:k4freeEven}.  And
Theorem~\ref{th.th} is a similar statement with ``balanced skew
partition'' instead of ``even pairs''.

All these theorems and conjectures suggest that Berge graphs with no
even pairs and Berge graphs with no balanced skew partition could be
close.  Since the even pair is a good tool to color a graph, coloring
a Berge graph could be performed by first contracting even pairs until
no more exist, and then use other decompositions, such as 2-joins,
homogeneous pairs and other things to be discovered \dots  It must be
pointed out that the counter-example of Chudnosky and Seymour
mentioned above is a serious problem.

\appendix

\chapter{Curriculum vit\ae} 
{\parindent0cm

\vspace{-1cm}
 {\large\bf\hspace{-0em}\'Etudes}

  \espSR
  {\bf  1995}  \parbox[t]{12cm}{Dipl\^ome de statisticien-\'economiste
    obtenu \`a 
    l'ENSAE    (\'Ecole Nationale
    de la Statistique et de l'Administration \'Economique).
  }

  \espSR
  {\bf 1997}  \parbox[t]{12cm}{Agr\'egation de Math\'ematiques.}

  \espSR {\bf 2001} \parbox[t]{12cm}{DEA Recherche Op\'erationnelle et
    Combinatoire, Universit\'e Grenoble~1 Joseph Fourier, Grenoble.

    Mention Tr\`es Bien.  M\'emoire~: \emph{L'argument de compacit\'e
      en combinatoire}, sous la direction de Sylvain Gravier.}
 
  \espSR
  {\bf  2004}  \parbox[t]{12cm}{Doctorat en Math\'ematiques et Informatique.

  \emph{Graphes parfaits : structure et algorithmes}, sous la direction
  de Fr\'ed\'eric Maffray et Michel Burlet.

  Universit\'e Grenoble~1 Joseph Fourier, laboratoire Leibniz, IMAG,  
  Grenoble, France.   Soutenue le $28$ Septembre 2004.}

  \espR

  {\large\bf\hspace{-0em}Postes occup\'es}
 
  \espSR
  {\bf 1998--1999} 

  Professeur de math\'ematiques, Lyc\'ee Charles de Gaulle,
  Rosny-sous-Bois.

  \espSR
  {\bf 1999--2005} 

  Professeur Agr\'eg\'e (PRAG) de math\'ematiques, Universit\'e
  Grenoble~2 Pierre Mend\`es-France, Grenoble.

  \espSR
  {\bf 2005--2008} 

  Ma\^itre de conf\'erence, Universit\'e Paris~1 Panth\'eon-Sorbonne,
  Paris. Chercheur au Centre d'\'Economie de la Sorbonne. 

  \espSR
  {\bf 2008--pr\'esent} 

  Charg\'e de recherche au CNRS, Universit\'e Paris~7  Paris-Diderot,
  LIAFA, Paris.

}
\chapter{Activit\'es d'encadrement} 

Ces derni\`eres ann\'ees, j'ai encadr\'e de nombreux TER en M1 et
quelques stages de M2.  J'ai aussi collabor\'e avec des doctorants et
des post-docs.  J'indique ci-dessous les stages d'\'etudiants et les
collaborations avec de jeunes chercheurs qui ont amen\'e des
r\'esultats int\'eressants scientifiquement.

\begin{itemize}
\item Master Recherche Op\'erationnelle et Combinatoire, Universit\'e
  Grenoble~1 Joseph Fourier, Grenoble.  \emph{G\'en\'eration des
    graphes autocompl\'ementaires} par Sylvain Bauchau, 2005.

  Sylvain Bauchau est parvenu \`a g\'en\'erer assez efficacement tous
  les graphes autocompl\'ementaires jusqu\`a 16 sommets, \`a
  isomorphisme pr\`es.  Pour cela, il a cr\'eé une interface entre un
  programme en langage C que j'avais \'ecrit auparavant et le
  programme Nauty de Brendan McKay.  Gr\^ace \`a cela, j'ai pu
  obtenir des contre-exemples pour des questions un peu na\"ives que
  je me posais \`a l'\'epoque sur les graphes autocompl\' ementaires. 

\item Master Recherche Op\'erationnelle et Combinatoire, Universit\'e
  Grenoble~1 Joseph Fourier, Grenoble.  \emph{D\'etection d'arbres
    induits} par Liu Wei, 2009.

  Liu Wei a travaill\'e sur les questions de type ``three-in-a-tree'',
  voir chapitre 3 de ce document.  Elle a trouv\'e seule le
  contre-exemple conduisant \`a la notion de $K_4$-structure
  (figure~\ref{f:k4s}), ce dont j'\' etais tr\`es content.  Les
  th\'eor\`emes que nous avons d\'emontr\'es ensemble ont \'et\'e
  tr\`es bien accueillis par les experts du jury de M2 et nous les
  avons soumis pour publication.

\item Master Recherche Math\'ematiques Avanc\'ees, \'Ecole normale
  sup\'erieure de Lyon.  \emph{Quelques cas simples d'une conjecture
    de Scott} par Amine Abdelkader, 2009.

  Amine Abdelkader a candidat\' e spontan\'ement pour un faire un
  stage de fin d'\'etude de l'\'ecole centrale de Lyon sur un sujet
  tr\`es th\'eorique.  Il suit parallèlement un M2 à l'\'Ecole Normale
  Sup\'erieure de Lyon.  Il est passionn\'e par les conjectures les
  plus difficiles de la th\'eorie des graphes, comme la conjecture
  d'Hadwiger.  Un tel sujet ne pouvant \^etre rai\-so\-na\-blement attaqu\'e
  lors d'un stage, je lui ai propos\'e d'\'etudier quelques cas simples
  de la conjecture de Scott sur les subdivisions induites.  Cela nous
  a conduit \`a beaucoup de pistes de recherche int\'eressantes,
  ainsi qu'\`a un r\'esultat pr\'esent\'e \`a la section 2.1 de ce
  document.

\item Th\`ese de Benjamin L\'ev\^eque, encadr\'ee par Fr\'ed\'eric Maffray,
  Universit\'e Grenoble 1, 2004--2007. 

  Benjamin L\'ev\^eque a commenc\'e sa th\`ese quand je terminais la
  mienne.  Nous avons discut\'e ensemble de quasiment tous les
  nombreux th\`emes de recherche qu'il a abord\'es avec succ\`es durant sa
  th\`ese.  Nous avons collabor\'e sur plusieurs r\'esultats pr\'esent\'es dans
  ce document.

\item Th\`ese de Nicolas Dehry, encadr\'ee par Christophe Picouleau au
  CNAM, Paris, 2007--2008.  

  J'ai rencontr\'e Christophe Picouleau en 2007 \`a une conf\'erence
  o\`u nous parlions tous les deux de probl\`emes li\'es \`a
  three-in-a-tree.  Nous avons alors entam\'e une collaboration avec
  son \'etudiant de th\`ese, Nicolas Dehry, qui nous a conduit \`a des
  r\'esultats pr\'esent\'es dans le chapitre 3 de ce document et
  publi\'es dans {\it Graphs and Combinatorics}.

\item PostDoc de Juraj Stacho, effectu\'e au LIAFA, 2008--2009. 

Juraj Stacho m'a demand\'e qu'on travaille ensemble durant son
PostDoc.  Je lui ai propos\'e de r\'efl\'echir \`a l'application de
m\'ethodes structurelles pour r\'esoudre des cas sp\'eciaux de la
conjecture de Gy\'arf\'as sur les arbres induits.  Nous sommes
parvenus \`a r\'esoudre le cas le plus simple signal\'e comme encore
ouvert par Gy\'arf\'as.  Nous avons ensuite r\'ealis\'e que ce cas
avait \'et\'e r\'esolu entre temps. Notre tentative est expliqu\'ee
\`a la section 6.2 de ce document.  Nous esp\'erons des
d\'eveloppement futurs.

\end{itemize}

\chapter{R\'esum\'e sur l'originalit\'e des recherches} 

Les chapitres 3, 4 et 5 de ce document pr\'esentent mes recherches les
plus originales.  Les principaux r\'esultats sont publi\'es dans des
journaux internationaux, gage d'originalit\'e.  En dehors de quelques
r\'esultats que j'ai d\^u inclure pour la coh\'erence du propos,
ces chapitres pr\'esentent des recherches conduites
apr\`es ma th\`ese.

\chapter{Expos\'e synth\'etique des recherches} 

Le chapitre~2 de ce document se veut p\'edagogique~: il pr\'esente des
th\'eor\`emes moins difficiles que par la suite, dans le but d'illustrer
les m\'ethodes et d'introduire des notions.  Il est donc synth\'etique
en ce sens qu'il permet d'aborder rapidement mon domaine de recherche.
Les trois chapitres sui\-vants r\'esument les r\'esultats principaux que
j'ai obtenus dans le th\`eme principal de mes recherches~: les graphes
parfaits, les bornes exactes pour le nombre chromatique et la
d\'etection de sous-graphes induits.

\chapter{Perspectives} 

Tout au long du document, des questions sont pos\'ees (elle sont
num\'erot\'ees comme les th\'eor\`emes).  Elles constituent un programme de
recherche pour les ann\'ees futures.

\chapter{Liste des travaux} 
 {\parindent0cm
      {\large\bf\hspace{-0em}Articles dans des journaux}

  \espSR

      {\bf Publiés}

      \espSR

      [1] S.~Gravier, F.~Maffray, J.~Renault, and N.~Trotignon.
      \newblock Ramsey-type results on singletons, co-singletons and monotone
      sequences in large collections of sets.
      \newblock {\em European Journal of Combinatorics}, 25(5):719--734, 2004.

     \espSSR
    
    Les résultats de cet article ne sont pas mentionnés dans ce document.

      \espSR

      [2] F.~Maffray and N.~Trotignon.
      \newblock Algorithms for perfectly contractile graphs.
      \newblock {\em SIAM Journal on Discrete Mathematics}, 19(3):553--574, 2005.

     \espSSR
    
      Certains résultats de cet article sont mentionnés chapitre~\ref{chap:reco}.

      \espSR

      [3] F.~Maffray and N.~Trotignon.  \newblock A class of perfectly
      contractile graphs.  \newblock {\em Journal of Combinatorial
        Theory Series B}, 96(1):1--19, 2006.

     \espSSR
    
      Certains résultats de cet article sont mentionnés section~\ref{sec:wt}

      \espSR

      [4] M.~Burlet, F.~Maffray, and N.~Trotignon.  \newblock Odd
      pairs of cliques.  \newblock In A.~Bondy, J. Fonlupt, J-L.
      Fouquet, J-C. Fournier, and J. L.~Ram{\'i}rez
      Alfons\'in, editors, {\em Graph Theory in Paris, Proceedings of
        a Conference in Memory of Claude Berge}, pages 85--95.
      Birkh{\"a}user, 2007.

     \espSSR
     
     Les résultats de cet article ne sont pas mentionnés dans ce document.

      \espSR

      [5] N.~Trotignon.  \newblock Decomposing {B}erge graphs and
      detecting balanced skew partitions.  \newblock {\em Journal of
        Combinatorial Theory Series B}, 98(1):173--225, 2008.

     \espSSR
    
      Les résultats de cet article sont mentionnés sections~\ref{algomotiv} et \ref{decth}.

  \newpage    \espSR

      [6] F.~Maffray, N.~Trotignon, and K.~Vu{\v s}kovi{\'c}.
      \newblock Algorithms for square-{$3PC(\cdot, \cdot)$}-free
                {B}erge graphs.  \newblock {\em SIAM Journal on
                  Discrete Mathematics}, 22(1):51--71, 2008.

     \espSSR
     
     Les résultats de cet article ne sont pas mentionnés dans ce document.

      \espSR
  
      [7] B.~L{\'e}v{\^e}que, F.~Maffray, B.~Reed, and N.~Trotignon.
      \newblock Coloring {A}rtemis graphs.  \newblock {\em Theoretical
        Computer Science}, 410:2234--2240, 2009.

     \espSSR
     
     Les résultats de cet article ne sont pas mentionnés dans ce document.

 \espR

 {\bf Acceptés}

      \espSR $\bullet$ N. Trotignon and K.~Vu\v skovi\'c. A structure
      theorem for graphs with no cycle with a unique chord and its
      consequences. \emph{Journal of Graph Theory}. Accepted.

     \espSSR
    
      Les résultats de cet article sont mentionnés section~\ref{sec:oneChord}.

      \espSR $\bullet$ B.~L\'ev\^eque, D.~Lin, F.~Maffray, and
      N. Trotignon, Detecting induced subgraphs. \emph{Discrete
        Applied Mathematics}. Accepted.
     \espSSR
    
      Certains résultats de cet article sont mentionnés chapitre~\ref{chap:reco}.

      \espSR 
      $\bullet$ N.~Dehry, C.~Picouleau and N. Trotignon,
      The four-in-a-tree problem for triangle-free graphs.
      \emph{Graphs and Combinatorics}.  Accepted.

     \espSSR
    
      Les résultats de cet article sont mentionnés section~\ref{sec:gen3}.

      \espR
 
{\bf Soumis}

     \espSR 
      $\bullet$ N. Trotignon and Liu Wei,
      {\it The $k$-in-a-tree problem for graphs of girth at least~$k$.}
      Submitted.

     \espSSR
    
      Les résultats de cet article sont mentionnés section~\ref{sec:gen3}. 

     \espSR
     $\bullet$ N.~Trotignon and K.~Vu{\v s}kovi{\'c},
    \newblock {\it Combinatorial optimization with 2-joins}, 2009.

     \espSSR
    
      Les résultats de cet article sont mentionnés sections~\ref{sec:structBergenoSP} à \ref{sec:Lovasz}.

     \espR

  {\bf Manuscrits}

     \espSR

     $\bullet$ N.~Trotignon, 
      \newblock {\it On the structure of self-complementary graphs}, 2004.
      
     \espSSR
     
     Les résultats de cet article ne sont pas mentionnés dans ce document.

      \espSR

      $\bullet$ B.~L\'ev\^eque, F.~Maffray, and N. Trotignon,
     \newblock {\it On graphs that do not contain a subdivision of the
     complete graph on four vertices as an induced subgraph}, 2007,
     revised in 2009.

     \espSSR
    
      Les résultats de cet article sont mentionnés section~\ref{sec:ISK4}.

      \espSR 
      $\bullet$ J.~Stacho and N. Trotignon,
      {\it Graphs with no triangle and no $F$ are 8-colourable}, 2009.

     \espSSR
    
      Les résultats de cet article sont mentionnés section~\ref{sec:juraj}.

     \espR

  {\bf En cours}

     \espSR

     $\bullet$ A.~Gy{\'a}rf{\'a}s, A. Seb{\H o} and N. Trotignon
     \newblock {\it Problems and results on the gap of graphs}, 2009.

     \espSSR
    
      Les résultats de cet article sont mentionnés chapitre~\ref{chap:gap}.

     \espR

      {\large\bf\hspace{-0em}Mémoires}

      \espSR
      $\bullet$ N.~Trotignon.
      \newblock {\em Pascal, Fermat et la g\'eom\'etrie du hasard}.
      \newblock IUFM de Cr\'eteil, 1999.
      \newblock Sous la direction d'\'Evelyne Barbin.

      \espSR
      $\bullet$ N.~Trotignon.
      \newblock {\it L'argument de compacit{\'e} en combinatoire.}
      \newblock Master's thesis, Universit{\'e} Joseph Fourier --- Grenoble I, 2001.
      \newblock Sous la direction de  Sylvain Gravier.

      \espSR
      $\bullet$ N.~Trotignon.
      \newblock {\em Graphes parfaits : structure et algorithmes}.
      \newblock PhD thesis, Universit{\'e} Joseph Fourier --- Grenoble I, 2004.
      \newblock Sous la direction de  Frédéric Maffray et Michel Burlet.
      \espSR
}

\newpage
\addcontentsline{toc}{chapter}{Bibliography}

\addcontentsline{toc}{chapter}{Index}
\printindex

\newpage
\mbox{}
\thispagestyle{empty}

\newpage
\mbox{}

\thispagestyle{empty}

\vspace{13ex}

{\noindent\bf\hfill\Large Résumé --- Abstract\hfill}

\vspace{8ex}

  Ce document présente les recherches de l'auteur durant ces dix
  dernieres années, sur les classes de graphes définies en excluant
  des sous-graphes induits.
  
\vspace{3ex}

  This document presents the work of the author over the last ten
  years on classes of graphs defined by forbidding induced subgraphs.

\end{document}